\documentclass[a4paper, 11pt, oneside]{article}
\usepackage[english]{babel}
\usepackage[T1]{fontenc}
\usepackage[utf8]{inputenc}
\usepackage{graphicx} 
\usepackage{hyperref, fancyhdr, amsmath, amssymb, amsthm, enumerate, array, tabularray, float,vmargin, bbm, mathtools,cleveref}
\usepackage{enumitem}
\usepackage{enumerate}
\usepackage[dvipsnames]{xcolor}
\usepackage{float}
\usepackage{subfig}
\usepackage{wrapfig}
\usepackage{geometry}
\usepackage{fancyhdr}
\usepackage{tikz}
\usetikzlibrary{patterns}
\usetikzlibrary{positioning}
\usetikzlibrary{decorations.pathreplacing}
\usetikzlibrary{fit}
\usepackage{mathrsfs}
\usepackage{siunitx}
\usepackage{fontawesome5}
\usepackage{setspace}

\usepackage[toc,page]{appendix}

\usepackage{comment}
\usepackage{soul}

\usepackage{resmes}
\usepackage{titlesec}

\usepackage[
backend=biber,
style=alphabetic,
]{biblatex}
\theoremstyle{definition}
\newtheorem{teorema}{Theorem}[section]
\newtheorem{lemma}[teorema]{Lemma}
\newtheorem{defi}[teorema]{Definition}

\newtheorem{obs}[teorema]{Remark}
\newtheorem{exemple}[teorema]{Example}

\newcommand{\R}{\mathbb{R}}

\makeatletter
\newcommand{\subalign}[1]{%
  \vcenter{%
    \Let@ \restore@math@cr \default@tag
    \baselineskip\fontdimen10 \scriptfont\tw@
    \advance\baselineskip\fontdimen12 \scriptfont\tw@
    \lineskip\thr@@\fontdimen8 \scriptfont\thr@@
    \lineskiplimit\lineskip
    \ialign{\hfil$\m@th\scriptstyle##$&$\m@th\scriptstyle{}##$\hfil\crcr
      #1\crcr
    }%
  }%
}
\makeatother

\title{\textbf{A $Tb$ TYPE THEOREM\\ FOR SUPPRESSED KERNELS}}
\author{MARINA FERNÀNDEZ-VILASECA\thanks{This manuscript is based on the author's Master's thesis, defended in June 2026, in the programme of Advanced Mathematics at Universitat de Barcelona and Universitat Autònoma de Barcelona.}\\[0.1cm] Universitat Autònoma de Barcelona (UAB)}
\date{}

\begin{document}

\pagestyle{plain}

\maketitle
\vspace{-0.2CM}
\begin{abstract}
    In this article, a non-homogeneous $Tb$ type theorem for arbitrary dimensional Calderón-Zygmund singular integral operators is proved. This is an extension of an analogous non-homogeneous $Tb$ theorem for the Cauchy transform, in the planar setting, due to Nazarov, Treil and Volberg. The novelties of the present work are the change of dimension and the fact that the operators to which the theorem applies are not necessarily antisymmetric. The techniques used in the proof include, among others, suppressed kernels, decompositions in $L^2(\mu)$, where $\mu$ is a Radon measure in $\mathbb{R}^d$,  and a probabilistic argument resulting from taking averages of the operators involved.  
\end{abstract}

\setcounter{page}{1}
\tableofcontents

\section*{Introduction}

A classical problem that has led to the development of many techniques in Analysis is Painlevé's problem on the metric or geometric characterization of removable sets for bounded analytic functions. More precisely, given a compact set $E\subset \mathbb{C}$, we say that it is removable for bounded analytic functions, or simply, removable, if for every open set $\Omega\supset E$, every bounded and analytic function on $\Omega\setminus E$ admits an analytic extension to the whole of $\Omega$. For example, if $E$ is a finite collection of points, it is removable, but if we take $E$ to be a disk, it is not. 

A key tool in the description of removable sets is that of analytic capacity, which was introduced by Ahlfors \cite{Ahlfors_1947} and is defined as follows. For a compact set $E\subset \mathbb{C}$, we let its analytic capacity be the number
$$
\gamma (E) := \sup|f'(\infty)| = \sup\bigg|\lim_{|z|\to +\infty}z\left(f(z)-f(\infty)\right)\bigg|,
$$
where the supremum is taken over all $f\in \mathcal{H}(\mathbb{C}\setminus E)$ with $|f|\leq 1$ in $\mathbb{C}\setminus E$. The connection between analytic capacity and removability is described by a result, due to Ahlfors, which states that a compact set $E\subset \mathbb{C}$ is removable if and only if $\gamma (E)=0$. Even though this is a positive result in the characterization of removable sets, since it is a purely analytic characterization, rather than a metric or geometric one, it fails to be the sought answer to Painlevé's problem.

Another capacity, which turns out to have a better geometric interpretation, is the analytic capacity $\gamma_+$ (or capacity $\gamma_+$), which is defined, for $E\subset \mathbb{C}$ compact, 
$$
\gamma_+(E) = \sup\mu(E),
$$
where the supremum is taken over all positive Radon measures $\mu$, supported on $E$, and such that their Cauchy transform,
$$
\mathcal{C}\mu(x) = \int\frac{1}{y-x}\,d\mu(y), 
$$
is in $L^{\infty}(\mathbb{C})$, with modulus bounded above by $1$. One advantage of this new capacity is that it has a precise description, due to Tolsa \cite{Tolsa_duke}, using previous ideas from Melnikov and Verdera \cite{MV}, in terms of curvature of measures, which is, indeed, a geometric notion, which was first introduced by Melnikov \cite{MelnikovMS1995Acda}.

From this, we see that the missing piece to solve Painlevé's problem is how $\gamma$ and $\gamma_+$ are related. This was dealt with by Tolsa in \cite{Tolsa_2003}, in which he proved that there exists an absolute constant $A$ such that for any compact set $E\subset \mathbb{C}$,
$$
\gamma_+(E)\leq \gamma(E)\leq  A\gamma_+(E),
$$
which we write $\gamma_+\approx \gamma$, and we say that the capacities $\gamma$ and $\gamma_+$ are comparable. A key consequence of this result is that a compact set $E\subset \mathbb{C}$ is non-removable if and only if it supports a positive Radon measure that satisfies a certain growth condition and has finite curvature (more precisely, see \cite[Theorem 1.2]{Tolsa_2003}). This description closes Painlevé's problem.

A key ingredient in Tolsa's proof of the comparability between the capacities $\gamma$ and $\gamma_+$ is a $Tb$ type theorem, which is originally due to Nazarov, Treil and Volberg \cite{NTV_preprint}. See Tolsa \cite[Section 5]{llibre_xavi} for a careful explanation of this result in the case of the Cauchy transform. The aim of the present work is to extend this theorem to a broader class of operators. Specifically, in this work it is shown that the reasoning can be adapted for higher dimension and that one can drop the assumption that the kernel that we deal with is antisymmetric. We will comment on the organization of the proof after a brief exposition on the topic of $Tb$ theorems. 

Broadly speaking, the notion of ``a $Tb$ type theorem'' refers to a criterion for the boundedness of a special type of integral operators, which is obtained via their action on a suitable bounded test function, which is usually denoted by $b$. The operators that appear in this context are called singular integral operators (SIOs). More precisely, we consider a function $k$, which we call kernel, defined on $\R^d\times \R^d$ except, perhaps, on the diagonal $\{(x,y)\in \R^d \times \R^d: x=y\}$, and we let
$$
T\mu(x) = \int k(x,y)\,d\mu(y),
$$
where $\mu$ is a Radon measure in $\R^d$, and for $f\in L^1_{\text{loc}}(\mu)$, we set $T_\mu f(x) = T(f\mu)(x)$. Of course, the convergence of these integrals deserves some thought, which will be done in the first section of this work, in which we will also introduce Calderón-Zygmund kernels, a particular case of which is the kernel $\frac{1}{x-y}$ that appears in the definition of the Cauchy transform above. These are the kernels that determine the operators appearing in $Tb$ type theorems.

The family of theorems which we refer to as $Tb$ type was developed as a tool to study the $L^2(\mu)$ boundedness of singular integral operators. For instance, the ones in David, Journé and Semmes \cite{DJS} and in Christ \cite{ChristMichael1990ATtw}. These theorems fall into the category known as classical Calderón-Zygmund theory. The adjective is due to the fact that most of these theorems assume that the underlying measure $\mu$ is doubling, which means that there is some constant $c>0$ such that for all $x\in \text{supp}(\mu)$ and $r>0$,
$$
\mu(B(x,2r))\leq c\,\mu(B(x,r)).
$$
Originally, it was believed that reasonable results for Calderón-Zygmund singular integral operators could only be obtained in spaces for which the underlying measure satisfies the condition above. These kinds of spaces are called homogeneous. However, when working with analytic capacity, it is not uncommon to find a situation where the doubling assumption above fails. For instance, if the measure $\mu$ is the Hausdorff $1$-dimensional measure restricted to a compact set. Nevertheless, in this type of context, an alternative and natural assumption is that $\mu$ has polynomial growth of degree $n\leq d$, meaning that for $r>0$ and $x\in \R^d$,
\begin{equation}
\mu(B(x,r)) \leq c_0\, r^n.\label{polynomial_growth}
\end{equation}
For example, in Tolsa \cite{Tolsa_duke}, a non-doubling version for the $T1$ theorem, that is, for the simpler case that $b=1$, for the Cauchy transform is proved.

The $Tb$ theorem in the present work falls into the non-homogeneous setting. It will concern a finite Radon measure $\mu$ without any doubling assumption. Moreover, the polynomial growth condition (\ref{polynomial_growth}) will not necessarily be true for all the balls $B(x,r)$. Instead, we will allow a controlled family of balls to have bigger measure. In this sense, we will not obtain that our singular integral operator $T$ is bounded in $L^2(\mu)$, but rather in $L^2(\mu\lfloor G)$, where the $\mu$-measure of $G$ is not too small compared to that of $\text{supp}(\mu)$, and $\mu\lfloor G$ will have the desired polynomial growth. Moreover, since the antisymmetric assumption is removed, the theorem deals with two bounded functions $b_1, b_2$, satisfying that
$$
\int T_*(b_1\mu)\,d\mu \leq c_*\, \|\mu\|, \quad \int T_*^*(b_2\mu)\,d\mu \leq c_*\, \|\mu\|,
$$
where the integrals are taken in the complementary of a distinguished set. The complete details about the hypothesis and conclusions can be found in the precise statement, which is Theorem \ref{TEOREMA} from Section \ref{chap:tb}.

Let us remark that the proof of the $Tb$ theorem is going to be highly technical, as is the case in \cite{NTV_preprint} and even in the simpler case \cite{llibre_xavi}. This is the reason why we have decided to split the proof into several sections, which we outline now.

As we mentioned above, the first section is an introduction to the theory of Calderón-Zygmund kernels and singular integral operators. We give the basic definitions of the concepts and statements of results that will be key in the rest of the thesis. Some are very well-known and are the same in both the classical and the non-homogeneous setting, such as the dyadic Hardy-Littlewood differentiation theorem, but not all are valid in both cases. For instance, we will need a Calderón-Zygmund decomposition tailored to the non-homogeneous case. The proof of most of the results from this preliminary section can be found in \cite[Section 2]{llibre_xavi}.

The second section contains the complete statement of the $Tb$ theorem that we prove, which is Theorem \ref{TEOREMA}. After stating it, the following sections are devoted to establishing the necessary tools for the rather long proof. We define what we call an exceptional set, which, morally, controls the portion of $\text{supp}(\mu)$ where the maximal operator associated to $T$ is big. Afterwards, we consider, for a Lipschitz function $\Theta\, \colon \R^d \to [0,+\infty)$, an appropriate modification of the Calderón-Zygmund kernel that defines $T$, which we call $\widetilde{k}_\Theta$, the ``suppressed'' version of $k$, with associated operator $K_{\Theta}$. This notion was introduced by Nazarov, Treil and Volberg \cite{NTV_acta} and is helpful for overcoming the difficulties associated with the non-homogeneous setting.

Still in the second section, we introduce the essential concepts related to dyadic lattices. Instead of fixing one lattice, we will work with its translates, and for each of these we are going to classify its cubes according to their position with respect to some of the sets that are singled out in the statement of our $Tb$ theorem. Later, we will see that this classification, along with a family of operators indexed by the cubes from the lattice, yields a decomposition in $L^2(\mu)$, which is Lemma \ref{lema_descomposicio_L2}, and will be central in the rest of the proof.

In the last section of the second section, we introduce the notion of good and bad cubes, from one dyadic lattice with respect to another. The third section is devoted to proving a lemma concerning the action of $K_{\Theta}$ on good functions, which are essentially the ones in which the $L^2(\mu)$ decomposition only sees the good cubes. The proof of this lemma is long and technical, and consists of using different strategies depending the type of cubes appearing in the $L^2(\mu)$ decomposition.

The fourth section is quite brief and deals with a Cotlar type theorem for a certain type of Calderón-Zygmund operators. Its proof will be independent of the classifications of dyadic cubes mentioned above and it will use the non-doubling Calderón-Zygmund decomposition stated in the first section.

The last section is devoted to concluding the proof of the $Tb$ theorem through a probabilistic argument. It consists of showing that the probability that of being in a bad situation, in the sense of bad cubes and functions, can be made arbitrarily close to zero. Moreover, here we define the set $G$ on which our original operator $T$ is going to be bounded in $L^2$. 

One of the motivations for introducing the non-homogeneous $Tb$ theorem was the proof of the comparability between the capacities $\gamma$ and $\gamma_+$. Of course, this was a result for the planar case, so it did not need the generalization that we prove in this work. Thus, it is reasonable to ask whether the additional work that we have done to prove the theorem in arbitrary dimension has a similar justification. The answer is affirmative and it concerns a reformulation of the problem of removability. 

This reformulation refers to changing the type of functions that we study, switching from analytic functions to Lipschitz harmonic functions. So far, we have considered functions defined on $\Omega\setminus E$, where $\Omega\subset \mathbb{C}$ is open and $E$ is compact. Then, we look for conditions on the set $E$ that ensure that for $f$ bounded and analytic in $\Omega\setminus E$, we actually have that $\overline{\partial}f(z)=0$ for all $z\in \Omega$. We can write this problem in terms of only real functions, as follows. Instead of holomorphic functions, we consider harmonic functions: $u\,\colon \Omega \setminus E\to \R$ such that $\Delta u =0$. Equivalently, such that $\overline{\partial}\partial\, u =0$. We can call $f= \partial u$, and we see that $u$ is harmonic if and only if $\partial u$ is analytic.  Using this, one sees that if $\gamma(E)=0$, then $E$ is removable for Lipschitz harmonic functions. The converse implication is also true, but it requires using the fact that $\gamma\approx \gamma_+$.

In this context, it is natural to look for a Lipschitz harmonic capacity of the set $E$, a quantity that, similarly to analytic capacity, will vanish if and only if $E$ is removable for Lipschitz harmonic functions. By our previous comment, for $E\subset \R^2$, this capacity vanishes if and only if $\gamma (E)=0$. Now, note that this new capacity can be defined in $\R^d$, for $d\geq 2$, that is, it extends the notion of analytic capacity to a higher dimensional setting. Since the planar $Tb$ theorem was useful for studying analytic capacity, it seems reasonable to think that a $d$-dimensional analogue of this theorem would be a useful tool to study Lipschitz harmonic capacity. This is indeed the case, for more details, see Section $2$ of Volberg's book \cite{Volberg_2003} or \cite{Nazarov_Tolsa_Volberg_2014}.

As a final remark, after completing the present work, I was informed by my advisor that Andrea Merlo, Mihalis Mourgoglou and Carmelo Puliatti \cite{MP} notified him that, simultaneously and independently of my work, they have proved a very general local $Tb$ theorem, from which one could derive the $Tb$ theorem explained in this thesis.

\textbf{Acknowledgement.} I would like to thank Dr. Xavier Tolsa for supervising my Master's thesis, for suggesting the topic, and for his guidance and support throughout the development of this work.

\addcontentsline{toc}{section}{Introduction}

\pagestyle{plain}

\pagestyle{plain}

\newpage
\pagestyle{fancy}

\section{Preliminaries}

In this section we introduce the basic notions related to the theory of singular integral operators in $\R^d$. We will discuss Calderón-Zygmund kernels, their adjoints and their maximal versions. We will also introduce some basic aspects of non-doubling Calderón-Zygmund theory and dyadic lattices, which are a key element in the main theorem of this thesis.

\subsection{Calderón-Zygmund kernels and singular integral operators}

\begin{defi}
    We say that $k (\cdot, \cdot )\colon \R^d\times \R^d\setminus\{(x,y)\in \R^d\times \R^d: x=y\}\to \mathbb{C}$ is an $n$-dimensional \textbf{Calderón-Zygmund kernel} if there exist constants $c>0$ and $0<\eta\leq 1$, such that the following inequalities hold for all $x,y\in \R^d, x\neq y$,
    \begin{equation}\label{CZ}
    \begin{aligned}
        |k(x,y)| &\leq \frac{c}{|x-y|^n},\quad \text{and,} \\ 
        \text{if}\,  |x-x'|&\leq \frac{|x-y|}{2},\, \,|k(x,y)-k(x',y)| + |k(y,x)-k(y,x')| \leq \frac{c|x-x'|^{\eta}}{|x-y|^{n+\eta}}.
    \end{aligned}
\end{equation} \label{definicio_CZ_kernel}
\end{defi}
Note that, without the condition that $|x-x'|\leq \frac{1}{2}|x-y|$, the second inequality in the second line above also holds if $|x-y|\approx |x'-y|$, possibly with a different constant $c$.

A Calderón-Zygmund kernel allows us to define a linear operator on the space of Radon measures in $\R^d$ in the following way.

\begin{defi}
    Given an $n$-dimensional Calderón-Zygmund kernel $k$, for $\nu\in M(\R^d)$, we define
    \begin{equation}
    T\nu(x):=\int_{\R^d}k(x,y)\,d\nu(y), \quad \text{for }\,x\in \R^d\setminus \text{supp}(\nu).\label{SIO}
    \end{equation}
    We also define $T\nu(x)$ for any $x\in \R^d$ for which the integral above is well-defined. We say that $T$ is an $n$-dimensional \textbf{singular integral operator} (SIO) with kernel $k(\cdot,\cdot)$.\label{definicio_SIO}
\end{defi}
Note that the integral in (\ref{SIO}) may fail to be absolutely convergent if $x\in \text{supp}(\nu)$. This is why we will also consider the following $\varepsilon$-truncated operators, $T_{\varepsilon}$, for $\varepsilon>0$,
$$
T_{\varepsilon}\nu(x):=\int_{|x-y|>\varepsilon}k(x,y)\,d\nu(y), \quad x\in \R^d.
$$
Now, the integral is absolutely convergent if, for example, $|\nu|(\R^d)<\infty$. Moreover, we define the maximal operator as
$$
T_*\nu(x) := \sup_{\varepsilon>0}|T_\varepsilon\nu(x)|.
$$
Also, the $\delta$-truncated maximal operator is
$$
T_{*,\delta}\nu(x)= \sup_{\varepsilon>\delta}|T_{\varepsilon}\nu(x)|,
$$
and we set, for $f\in L^1_{\text{loc}}(\mu)$, $T_{\mu,*}f = T_*(f\mu)$, and $T_{\mu,*,\delta}f = T_{*,\delta}(f\mu)$.

Given a fixed positive Radon measure $\mu$ on $\R^d$, we can define, for $f\in L^1_{\text{loc}}(\mu)$,
\begin{alignat*}{2}
    T_\mu f(x)\,&:=T(f\mu)(x), \qquad &&x\in \R^d\setminus\text{supp}(f\mu),\\
    T_{\mu,\varepsilon}f(x)\,&:= T_{\varepsilon}(f\mu)(x), \qquad &&x\in\R^d. 
\end{alignat*}
The last integral is absolutely convergent for all $x\in \R^d$ if, for instance, $f\in L^p(\mu)$ and $\mu$ has growth of degree $n$, meaning that there is $c>0$ such that for any $x\in \R^d$ and $r>0$, $\mu(B(x,r))\le c\,r^n$.

We say that the operator $T_{\mu}$ is bounded in $L^p(\mu)$ if the operators $T_{\mu,\varepsilon}$ are bounded in $L^p(\mu)$  uniformly on $\varepsilon>0$ and analogously with respect to the boundedness from $L^1(\mu)$ to $L^{1,\infty}(\mu)$. Moreover, we say that $T$ is bounded from $M(\R^d)$ to $L^{1,\infty}(\mu)$ if there is a constant $C>0$ such that for all $\nu\in M(\R^d)$ and all $\lambda>0$,
$$
\mu(\{x\in \R^d: |T_{\varepsilon}\nu(x)|>\lambda\})\leq C\frac{\|\nu\|}{\lambda},
$$
uniformly on $\varepsilon>0$. Lastly, we say that a singular integral operator is a \textbf{Calderón-Zygmund operator} (CZO) if it is bounded in $L^2(\mu)$.

Another notion that will be useful for our theorem is the following. Given a SIO $T$, originated from a kernel $k(\cdot,\cdot)$, we denote by $T^*$ the SIO that arises from the kernel $\widetilde{k}(x,y)=k(y,x)$. That is, for $\nu\in M(\R^d)$, we put
$$
T^*\nu(x) := \int_{\R^d}k(y,x)\,d\nu(y), \quad \text{for }\,x\in \R^d\setminus \text{supp}(\nu).
$$

We will call $T^*$ the \textbf{adjoint} of $T$. The reason for this notation is the following. Recall that if $H$ is a Hilbert space with inner product $\langle\cdot, \cdot \rangle$ and $S\colon H \to H$ is a continuous linear operator, the adjoint of $S$, which we denote by $S^*$, is the continuous linear operator $S^*\colon H\to H$ determined by the relation
$$
\langle Sf, g\rangle = \langle f, S^*g\rangle, \quad \text{for all }\,f,g\in H.
$$
It is straightforward to check, using Fubini's theorem, that for $f,g\in \mathcal{C}_c^0(\R^d)$, we have
$$
\langle T_{\mu,\varepsilon}f, g\rangle = \int T_{\mu, \varepsilon}f(x)g(x)\,d\mu(x) = \int f(x)T^*_{\mu,\varepsilon}g(x)\,d\mu(x) = \langle f, T^*_{\mu,\varepsilon}g\rangle, 
$$
which is precisely the relation that characterizes the adjoint operator.

In the special case that we have $k(y,x)=-k(x,y)$, we say that the kernel $k(\cdot, \cdot)$ is \textbf{antisymmetric}.

Lastly, we give two examples of singular integral operators.

\begin{exemple}
    The Cauchy transform is the SIO in $\mathbb{C}$ obtained from the $1$-dimensional antisymmetric Calderón-Zygmund kernel
    $$
    k(x,y) := \frac{1}{y-x}, \quad x,y\in \mathbb{C}.
    $$
    If instead we consider 
    $$
    k(x,y) = \frac{1}{|x-y|},
    $$
    we obtain a $1$-dimensional Calderón-Zygmund kernel which is not antisymmetric.
\end{exemple}

\begin{exemple}
    In $\R^d$, for an integer $0<n\leq d$, we consider the Riesz kernels,
    $$
    k(x,y) = \frac{x_j-y_j}{|x-y|^{n+1}}, \quad j=1,\dots, d,
    $$
    where we write $x\in \R^d$ as $x=(x_1,\dots, x_d)$ and analogously for $y\in \R^d$. The $n$-dimensional Riesz kernels are the SIOs originated by these kernels. Note that these kernels are also antisymmetric.
\end{exemple}

\subsection{The Calderón-Zygmund decomposition}

As mentioned in the Introduction, the $Tb$ theorem in the present work belongs to the ambit of non-homogeneous Calderón-Zygmund theory, that is, we work with measures $\mu$ in $\R^d$ such that the doubling condition
$$
\mu(B(x,2r)) \leq c\,\mu(B(x,r)), \quad \text{for any }x\in \R^d, \,r>0,
$$
with a constant $c$, independent of the balls, is not necessarily true for all balls $B(x,r)$ or cubes. This is why we introduce, the following notion. For $\alpha, \beta>1$, we say that a cube $Q\subset \R^d$ is $(\alpha,\beta)$\textbf{-doubling} if 
$$
\mu(\alpha Q) \leq \beta\mu(Q),
$$
where $\alpha Q$ denotes the cube concentric with $Q$ with diameter $\alpha\, \text{diam}(Q)$. We claim that, although $\mu$ may not satisfy any growth condition like the one above, there are a lot of small doubling cubes. For a proof, see \cite[Lemma 2.8]{llibre_xavi}.

\begin{lemma}\label{lema_cubs_doblants}
    Let $\beta>\alpha^d$. If $\mu$ is a Radon measure in $\R^d$, then for $\mu$-a.e. $x\in \R^d$, there exists a sequence of $(\alpha,\beta)$-doubling cubes $\{Q_k\}$ centered at the point $x$ with $\ell(Q_k)\to 0$ as $k\to +\infty$.
\end{lemma}

A key tool in many results concerning SIOs is the following lemma, called the \textbf{Calderón-Zygmund decomposition}. Morally, it says that if we consider a Radon measure $\mu$ on $\R^d$, then for any finite measure $\nu$ with compact support, we can split it into a ``good'' part, where it is controlled by the ambient measure $\mu$ and a ``bad'' part, in which $|\nu|$ is bigger than $\mu$. Moreover, the latter is a union of cubes with finite overlap. For a proof, see \cite[Lemma 2.14]{llibre_xavi}.

\begin{lemma}\label{descomposicio_calderon_zygmund}
Let $\mu$ be a Radon measure on $\mathbb{R}^d$. For every
$\nu \in M(\mathbb{R}^d)$ with compact support and every
$\lambda > 2^{d+1}\frac{\|\nu\|}{\|\mu\|}$,
we have:
\begin{enumerate}
\item[(a)] There exists a family of almost disjoint cubes, that is, having bounded overlap, $\{Q_i\}_i$,
and a function $f\in L^1(\mu)$ such that
\begin{equation}
|\nu|(Q_i)
>
\frac{\lambda}{2^{d+1}}\,\mu(2Q_i),
\label{2.13}
\end{equation}
\begin{equation}
|\nu|(\eta Q_i)
\le
\frac{\lambda}{2^{d+1}}\,\mu(2\eta Q_i),
\qquad \text{for } \eta>2,
\label{2.14}
\end{equation}
\begin{equation}
\nu=f\,\mu
\qquad \text{in } \mathbb{R}^d\setminus \bigcup_i Q_i,
\qquad \text{with } |f|\le \lambda
\quad \mu\text{-a.e.}
\label{2.15}
\end{equation}

\item[(b)] For each $i$, let $R_i$ be a $(6,6^{d+1})$-doubling cube
concentric with $Q_i$, with $\ell(R_i)>4\,\ell(Q_i)$,
and let $w_i=\frac{\chi_{Q_i}}{\sum_k \chi_{Q_k}}$. Then, there exists a family of functions $\varphi_i$ with
$\operatorname{supp}(\varphi_i)\subset R_i$, each $\varphi_i$
with constant sign, satisfying
\begin{equation}
\int \varphi_i\,d\mu
=
\int_{Q_i} w_i\,d\nu,
\label{2.16}
\end{equation}
\begin{equation}
\sum_i |\varphi_i|
\le
B\,\lambda,
\label{2.17}
\end{equation}
(where $B$ is some fixed constant depending only on $d$ and $n$),
and
\begin{equation}
\|\varphi_i\|_{L^\infty(\mu)}\,\mu(R_i)
\le
c\,|\nu|(Q_i).
\label{2.18}
\end{equation}
\end{enumerate}
\end{lemma}

Lastly, we need a standard integral estimate concerning doubling cubes, which is easily proved by splitting the domain of integration adequately. 

\begin{lemma}\label{lema:2.15}
    Let $\mu$ be a Radon measure on $\R^d$. If $Q\subset R$ are concentric cubes such that there are no $(\alpha,\beta)$-doubling cubes (with $\beta>\alpha^d$) of the form $\alpha^kQ$, $k\geq 0$, with $Q\subset \alpha^k Q\subset R$, and $x_Q$ denotes the center of $Q$, then
    $$
    \int_{R\setminus Q} \frac{1}{|x-x_Q|^n}d\mu(x) \leq c_1\frac{\mu(R)}{\ell(R)^n},
    $$
    where $c_1$ depends only on $\alpha, \beta, n$, and $d$.
\end{lemma}

\subsection{The Hardy-Littlewood maximal operator}
\label{sec:HL}

Let us introduce briefly an operator which is very well known in Harmonic Analysis: for a Radon measure $\mu$ in $\R^d$, the \textbf{centered maximal Hardy-Littlewood operator} $M_\mu$ applied to a complex Radon measure $\nu$ is defined as
$$
M_\mu\nu(x) = \sup_{r>0}\frac{|\nu|(B(x,r))}{\mu(B(x,r))},
$$
so for $f\in L^1_{\text{loc}}(\mu)$, we have $M_\mu f = M_\mu (f\mu)$. A classical result, which can be proved using the covering theorem for bounded subsets of $\R^d$, due to Besicovitch, is the following.
\begin{teorema}\label{HL}
    Let $\mu$ be a Radon measure in $\R^d$. The centered maximal Hardy-Littlewood operator $M_\mu$ is bounded from $M(\R^d)$ to $L^{1,\infty}(\mu)$ and in $L^p(\mu)$, for $1<p\leq \infty$.
\end{teorema}
For our purposes, this theorem will be useful because it will enable us to relate boundedness of two integral operators whenever their difference is controlled by the maximal Hardy-Littlewood operator.

Another result concerning this maximal operator, which we will use in Section \ref{chap:cotlar}, is Cotlar's inequality, which we state next.

\begin{teorema}\label{cotlar_referenciar}
Let $\mu$ be a positive Radon measure on $\mathbb{R}^d$ and let $T$ be an $n$-dimensional SIO. Let $0<s\le 1$, and $0<\delta\le \varepsilon$. Suppose that for some fixed $x\in\mathbb{R}^d$,
\begin{equation}
\mu(B(x,r)) \le c_0 r^n
\qquad \text{for } r\ge \varepsilon ,
\label{Cotlar:creixement}
\end{equation}
and that $T_\delta$ is bounded from $M(\mathbb{R}^d)$ to
$L^{1,\infty}(\mu)$. Then we have
\begin{equation}
|T_\varepsilon \nu(x)|
\le
c_s \,{M}_{\mu}\!\left(|T_\delta \nu|^s\right)(x)^{1/s}
+
c_{s,T}\, M_\mu \nu(x),
\qquad \text{for } \nu\in M(\mathbb{R}^d).
\label{Cotlar:primera_desigualtat}
\end{equation}

Thus, if $\mu$ has growth of degree $n$, then for all $x\in\mathbb{R}^d$,
\begin{equation}
T_{*,\delta}\nu(x)
\le
c_s \,{M}_{\mu}\!\left(|T_\delta \nu|^s\right)(x)^{1/s}
+
c_{s,T}\, M_\mu \nu(x),
\qquad \text{for } \nu\in M(\mathbb{R}^d).
\label{Cotlar:segona_desigualtat}
\end{equation}

The constant $c_s$ depends only on the constant $c_0$ in \ref{Cotlar:creixement},
$s$, $n$, and $d$, and
\[
c_{s,T}
=
c\left(1+\|T_\delta\|_{M(\mathbb{R}^d)\to L^{1,\infty}(\mu)}\right),
\]
with $c$ depending only on $c_0$, $s$, $n$, and $d$.
\end{teorema}

Note that in the theorem above we impose no doubling conditions on the measure $\mu$. Hence, this theorem is valid for the non-homogeneous setting, which will be the situation in the $Tb$ theorem that we will prove. This result is due to Nazarov, Treil and Volberg \cite{NTV_cotlar}, although the form in which we have stated it here is from Tolsa \cite{tolsa_cotlar}.

Lastly, we introduce the one-dimensional radial maximal operator,
$$
M_{r}\nu (x) = \sup_{r>0}\frac{|\nu|(B(x,r))}{r^n}.
$$
Note that if $\mu$ has linear growth, i.e. there is some constant $c_0>0$ such that $\mu(B(x,r))\leq c_0r$ for any $r>0$, then $M_R\nu(x)\leq c_0 M_\mu \nu(x)$, and so $M_R$ is bounded in $L^p(\mu)$ for $1<p\leq \infty$ and from $M(\R^d)$ to $L^{1,\infty}(\mu)$.

\subsection{Dyadic lattices}

In the proof of our main theorem \ref{TEOREMA}, dyadic lattices, which we introduce now, will play a key role. For $j\in \mathbb{Z}$, $\mathcal{D}_j$ is the collection of all cubes of the form
$$
Q^j_k := \left\{(x_1,\dots, x_d)\in \R^d : \frac{k_i}{2^j}\leq x_i< \frac{k_i+1}{2^{j}}, 1\leq i\leq d\right\},  \quad k=(k_1,\dots, k_d)\in \mathbb{Z}^d.
$$
The \textbf{standard dyadic lattice} (from $\R^d$) is the union of all the $\mathcal{D}_j$,
$$
\mathcal{D}_0:= \bigcup_{j\in \mathbb{Z}}\mathcal{D}_j.
$$
Note that for $k,s\in \mathbb{Z}^d$, if $k\neq s$, then $Q^j_k\cap Q^j_s=\varnothing$. For $Q^j_k$, we say that $j$ is the \textbf{generation} of $Q^j_k$ and $\ell(Q^j_k)=2^{-j}$ is its \textbf{side length}. Moreover, for any $w\in \R^d$, we can translate the standard dyadic lattice and we obtain another dyadic lattice,
\begin{equation}
\mathcal{D}(w) := w + \mathcal{D}_0.\label{xarxes_traslladades}
\end{equation}
For any dyadic lattice $\mathcal{D}$ in $\R^d$, we have the following useful properties:
\begin{enumerate}
    \item For any $x\in \R^d$, $j\in \mathbb{Z}$, there is a unique $Q\in \mathcal{D}_j$ that contains $x$.
    \item Given $Q, R \in \mathcal{D}$, either $Q\cap R=\varnothing$ or one is contained inside the other.
    \item Each $Q\in \mathcal{D}_j$ is a disjoint union of $2^d$ cubes from $\mathcal{D}_{j+1}$,
    $$
    Q = \bigcup_{i=1}^{2^d}P_i.
    $$
    We say that $\mathcal{CH}(Q):=\{P_i\}_{i=1}^{2^d}$ are the \textbf{children} of $Q$ and that $Q$ is the \textbf{parent} of each $P_i$, which we denote $\widehat{P}_i=Q$.
\end{enumerate}
A terminology that is often used when working with dyadic lattices and cubes, related to the second property, is the following. Suppose that we have some family $I\subset \mathcal{D}$ of not necessarily disjoint dyadic cubes. We say that a cube in this family is \textbf{maximal} if it is not contained in any other cube from the family. We denote by $I^{\text{max}}\subseteq I$ the subfamily of the maximal cubes. The advantage of considering the maximal family is that the cubes that form it are pairwise disjoint.

Given a Radon measure $\sigma$ on $\R^d$ and some dyadic lattice $\mathcal{D}$ from $\R^d$, we define the associated \textbf{dyadic maximal Hardy-Littlewood operator} (with respect to $\sigma$ and $\mathcal{D}$) by
$$
M_{\sigma, d}\nu(x) = \sup_{Q\in \mathcal{D}: Q\ni x}\frac{|\nu|(Q)}{\sigma(Q)},
$$
for any complex measure $\nu\in M(\R^d)$. For $f\in L^1_{\text{loc}}(\sigma)$, we put $M_{\sigma,d}f(x)=M_{\sigma,d}(f\sigma)(x)$. We show the boundedness properties of this operator in the following theorem.

\begin{teorema}\label{diferenciacio_diadic}
    The operator $M_{\sigma,d}$ is bounded from $M(\R^d)$ to $L^{1,\infty}(\sigma)$ and also in $L^p(\sigma)$, for $1<p\leq \infty$.
\end{teorema}

The boundedness of this maximal operator is the key ingredient in proving the dyadic version of the well-known Lebesgue differentiation theorem. The proof is the same as for centered balls in $\R^d$.

\begin{teorema}
    Let $\mu$ be a Radon measure in $\R^d$ and $f\in L^1_{\text{loc}}(\mu)$. Then, for $\mu$-a.e. $x\in \R^d$,
    $$
    \lim_{n\to+\infty}\frac{1}{\mu(Q_n(x))}\int_{Q_n(x)}|f(y)-f(x)|\,d\mu(y)=0,
    $$
    where for each $n\geq 1$, $Q_n(x)\in \mathcal{D}$ is the only cube in the dyadic lattice $\mathcal{D}$ with side length $2^{-n}$ that contains $x$.
\end{teorema}

In particular, using the $L^2(\sigma)$-boundedness of $M_{\sigma,d}$, one can prove the following Dyadic Carleson embedding theorem, which we state below. For the proof, see \cite[Theorem 5.8]{llibre_xavi}. Before stating it, we introduce the following notation. For $\sigma$ a Radon measure in $\R^d$, $f\in L^1_{\text{loc}}(\sigma)$ and $Q\subset \R^d$ a cube, we let
$$
\langle f\rangle_{\sigma, Q}=\frac{1}{\sigma (Q)}\int_Q fd\sigma.
$$

\begin{teorema}
    \label{carleson_diadic} Let $\sigma$ be a Radon measure on $\R^d$. Let $\mathcal{D}$ be some dyadic lattice from $\R^d$ and let $\{a_Q\}_{Q\in \mathcal{D}}$ be a family of non-negative numbers. Suppose that for every cube $R\in \mathcal{D}$ we have
    \begin{equation}
    \sum_{Q\in \mathcal{D}: Q\subset R} a_Q \leq c_2\sigma(R).\label{condicio_carleson}
    \end{equation}
    Then, every family of non-negative numbers $\{w_Q\}_{Q\in \mathcal{D}}$ satisfies
    $$
    \sum_{Q\in \mathcal{D}}w_Qa_Q \leq c_2 \int\sup_{Q\ni x}w_Qd\sigma(x).
    $$
    Also, if $f\in L^2(\sigma)$,
    $$
    \sum_{Q\in \mathcal{D}}|\langle f\rangle_{\sigma, Q} |^2 a_Q \leq c c_2\|f\|^2_{L^2(\sigma)},
    $$
    where $c$ is an absolute constant.
\end{teorema}

\section{A $Tb$ type theorem}

\label{chap:tb}

Below we state the main theorem of the present work. As stated in the Introduction, this is an extension of an analogous result by Nazarov, Treil and Volberg \cite{NTV_preprint}. The theorem concerns a SIO associated to, in the sense of Definition \ref{definicio_SIO}, an $n$-dimensional Calderón-Zygmund kernel $k(\cdot,\cdot)$, defined on $\R^d\times \R^d\setminus \{(x,y)\in \R^d\times \R^d: x=y\}$, with associated constants $C_{CZ}>0$, which, for commodity, will be denoted also by $c$, which will absorb other constants, and $0<\eta \leq 1$ (in the sense of Definition \ref{definicio_CZ_kernel}). Let us fix these constants throughout the rest of the article. 

\begin{teorema}
    Let $\mu$ be a finite measure supported on a compact set $F\subset \R^d$. Suppose that there exist two complex measures $\nu_1, \nu_2$ and, for each $w\in \R^d$, three subsets $H_{\mathcal{D}(w)}, T^1_{\mathcal{D}(w)}, T^2_{\mathcal{D}(w)}\subset \R^d$ made of dyadic cubes from $\mathcal{D}(w)$ such that 
    \begin{enumerate}[label=(\alph*)]
        \item Every ball $B_r$ of radius $r$ such that $\mu(B_r)>c_0r^n$ is contained in $\bigcap_{w\in \R^d}H_{\mathcal{D}(w)}$.
        \item $\nu_1 = b_1\mu$ and $\nu_2 = b_2\mu$, where $b_1, b_2$ are functions in $L^{\infty}(\mu)$, that is, such that $\|b_1\|_{L^{\infty}(\mu)}, \|b_2\|_{L^{\infty}(\mu)}\leq c_b$.
        \item $\int_{\R^d\setminus H_{\mathcal{D}(w)}}T_{*}\nu_1 \,d\mu\leq c_*\mu(F)$, and $\int_{\R^d\setminus H_{\mathcal{D}(w)}}T^*_{*}\nu_2 \,d\mu\leq c_*\mu(F)$ for all $w\in \R^d$.
        \item If $Q\in \mathcal{D}(w)$ is such that $Q\not\subset T^i_{\mathcal{D}(w)}$, then $\mu(Q)\leq c_{\text{acc}}|\nu_i(Q)|$, for $i=1,2$ (we say that $Q$ is an accretive cube).
        \item $\mu(H_{\mathcal{D}(w)}\cup T^1_{\mathcal{D}(w)}\cup T^2_{\mathcal{D}(w)})\leq \delta_0 \mu(F)$, for all $w\in \R^d$ and some $\delta_0<1$.
    \end{enumerate}
    Then, there exists a subset 
    \begin{align}
    G\subset F\setminus \bigcap_{w\in \R^d}\left(H_{\mathcal{D}(w)}\cup T^1_{\mathcal{D}(w)}\cup T^2_{\mathcal{D}(w)}\right)\label{on_viu_G}
    \end{align}
    such that
    \begin{enumerate}[label=(\roman*)]
        \item $\mu(G)\geq c_1^{-1}\mu(F)$.
        \item $\mu\lfloor G$ has $c_0$-polynomial growth of degree $n$.
        \item The SIO $T$ is bounded in $L^2(\mu\lfloor G)$.
    \end{enumerate}
    The constant $c_1$ and the bound for the $L^2(\mu\lfloor G)$ boundedness depend only on $C_{CZ}, c_0,c_b,c_{*}$, $c_{\text{acc}}$ and $\delta_0$. \label{TEOREMA}
\end{teorema}

Note that from condition (a) above, from condition (\ref{on_viu_G}) we infer that $\mu\lfloor G$ has $c_0$-polynomial growth of degree $n$. So, to check (ii), it suffices to prove (\ref{on_viu_G}).

Before moving on to the proof, let us remark that it will be highly technical and quite long. This is why it has been divided in three sections. In the present one, we introduce some objects, which will be the right tools to prove Theorem \ref{TEOREMA}, and we describe and prove their useful properties. In Section \ref{chap:good_functions}, we prove a key lemma concerning one of these objects. Section \ref{chap:cotlar} is brief and consists in proving a theorem that is analogous to one of the results in the section containing the preliminaries. Lastly, Section \ref{chap:prob} combines all the work from the previous ones to conclude the proof of Theorem \ref{TEOREMA} via a probabilistic argument.

\subsection{The exceptional set $S$}
\label{sec:exceptional_set}

The long proof of Theorem \ref{TEOREMA} starts by defining a set $S$, which we call \textit{exceptional}, through the identification of the points from the support of $\mu$ in which the maximal functions $T_*\nu_1$ and $T^*_*\nu_2$ are uniformly bounded. In this section, we are going to show that $\mu(S\setminus H_{\mathcal{D}(w)})$ can be controlled by the measure of the set $F$ by choosing a big enough bound for the maximal functions.

For now, let $\alpha\gg c_0c_b>0$ be a fixed constant, to be chosen below. We set
\begin{alignat*}{2}
    S^1_0 &&:= \{x\in F : T_*\nu_1(x)>\alpha\} &= \left\{x\in F : \sup_{\varepsilon>0}|T_\varepsilon \nu_1(x)|>\alpha\right\}, \\
    S^2_0 &&:= \{x\in F : T^*_*\nu_2(x)>\alpha\} &= \left\{x\in F : \sup_{\varepsilon>0}|T^*_\varepsilon \nu_2(x)|>\alpha\right\}.
\end{alignat*}

In addition, we let
\begin{align}
    e_1(x)&:=\begin{cases}
        \sup\{\varepsilon>0: |T_{\varepsilon}\nu_1(x)|>\alpha\}, \quad &x\in S_0^1,\\
        0, &x\in F\setminus S_0^1,
    \end{cases}\label{e1}\\
    e_2(x)&:=\begin{cases}
        \sup\{\varepsilon>0: |T^*_{\varepsilon}\nu_2(x)|>\alpha\}, \quad &x\in S_0^2,\\
        0, &x\in F\setminus S_0^2,
    \end{cases}\label{e2}
\end{align}

We define the \textbf{exceptional sets} $S_1$ and $S_2$ as
$$
S_1 := \bigcup_{x\in S^1_0}B(x,e_1(x)), \quad S_2 := \bigcup_{x\in S^2_0}B(x,e_2(x)).
$$
Moreover, we put $S:=S_1\cup S_2$. In the next lemma, we are going to see that we can make $\mu\left(S_i\setminus H_{\mathcal{D}(w)}\right)$ small by taking $\alpha$ big enough, for $i=1,2$. For simplicity, we denote
$$
A = 2^n C_{CZ}\left(1+ \frac{1}{2^{\eta}-1}\right).
$$

\begin{lemma}
   Let $w\in \R^d$. If $y\in S_1\setminus H_{\mathcal{D}(w)}$, then $T_*\nu_1(y)> \alpha -A c_0c_b$. Thus, if $\alpha> 2Ac_0c_b$, then
    $$
    \mu\left(S_1\setminus H_{\mathcal{D}(w)}\right)\leq \frac{2c_*}{\alpha}\mu(F).
    $$
\end{lemma}
\begin{proof}
    First, note that if $y\in S_1\setminus H_{\mathcal{D}(w)}$, then $y\in B(x,e_1(x))$ for some $x\in S^1_0$, and so there is $0<r<e_1(x)$ such that $y\in B(x,r)$. By definition of $e_1(x)$, there is some $r<\varepsilon_0(x)\leq e_1(x)$ such that 
    $$
    \left|T_{\varepsilon_0(x)}\nu_1(x)\right|> \alpha \quad \text{and}\quad y\in B(x,\varepsilon_0(x)).
    $$
    To prove the first statement in the lemma, it suffices to show that
    \begin{equation}
    \left|T_{\varepsilon_0(x)}\nu_1(x)-T_{\varepsilon_0(x)}\nu_1(y)\right|\leq Ac_0c_b.\label{desigualtat_A}
    \end{equation}
    Indeed, if the previous inequality holds, we deduce that
    \begin{align*}
        T_{*}\nu_1(y)\geq  \left|T_{\varepsilon_0(x)}\nu_1(y)\right|\geq \left|T_{\varepsilon_0(x)}\nu_1(x)\right|- Ac_0c_b > \alpha - Ac_0c_b. 
    \end{align*}
    Let us prove (\ref{desigualtat_A}). We have that $\left|T_{\varepsilon_0(x)}\nu_1(x)-T_{\varepsilon_0(x)}\nu_1(y)\right|$ is bounded above by
    \begin{align}
        &\left|T_{\varepsilon_0(x)}\left(\nu_1\lfloor B(y,2\varepsilon_0(x))\right)(x)\right|+ \left|T_{\varepsilon_0(x)}\left(\nu_1\lfloor B(y,2\varepsilon_0(x))\right)(y)\right|\label{primera_linia}\\
        &\qquad\qquad  + \left|T_{\varepsilon_0(x)}\left(\nu_1\lfloor B(y,2\varepsilon_0(x))^c\right)(x)-T_{\varepsilon_0(x)}\left(\nu_1\lfloor B(y,2\varepsilon_0(x))^c\right)(y)\right|\label{segona_linia}
    \end{align}
    
    To bound the first term on the right-hand side of (\ref{primera_linia}), we write
    \begin{align*}
        |T_{\varepsilon_0(x)}(\nu_1\lfloor B(y,2\varepsilon_0(x)))(x)| &= \left|\int_{|x-z|\geq \varepsilon_0(x)}k(x,z)d\left(\nu_1\lfloor B(y,2\varepsilon_0(x))\right)(z)\right|\\
        &\leq \int_{|x-z|\geq \varepsilon_0(x)}\frac{C_{CZ}}{|x-z|^n}d\left|\nu_1\lfloor B(y,2\varepsilon_0(x))\right|(z)\\
        &\leq C_{CZ}\frac{|\nu|(B(y,2\varepsilon_0(x)))}{\varepsilon_0(x)^n}\\
        &\leq 2^nC_{CZ}c_bc_0,
    \end{align*}
    where in the last inequality we have used that $y\not\in H_{\mathcal{D}(w)}$, so $\mu(B(y,2\varepsilon_0(y)))\leq c_0(2\varepsilon_0(x))^n$. The second term on the right-hand side of (\ref{primera_linia}) is estimated analogously and it can also be bounded by $2^nC_{CZ}c_bc_0$. The last term, in (\ref{segona_linia}), can be bounded by
    \begin{align}
        \int_{\R^d\setminus B(y,2\varepsilon_0(x))}\left|k(x,z)-k(y,z)\right|d|\nu_1|(z).\label{modul_dins_int}
    \end{align}
    For $z\not\in B(y,2\varepsilon_0(x))$, since $y\in B(x,\varepsilon_0(x))$, we have that
    $$
    |x-y|\leq \varepsilon_0(x)\leq \frac{|z-y|}{2},
    $$
    so (\ref{modul_dins_int}) does not exceed
    \begin{align*}
        &\int_{\R^d\setminus B(y,2\varepsilon_0(x))}C_{CZ}\frac{|x-y|^{\eta}}{|z-y|^{n+\eta}}d|\nu_1|(z) \\
        &\qquad\qquad\leq C_{CZ}c_b\varepsilon_0(x)^{\eta}\sum_{k=1}^{\infty}\int_{2^k\varepsilon_0(x)\leq |z-y|<2^{k+1}\varepsilon_0(x)}\frac{1}{|z-y|^{n+\eta}}d\mu(z)\\
        &\qquad \qquad \leq C_{CZ}c_b\varepsilon_0(x)^{\eta}\sum_{k=1}^{\infty}\frac{\mu(B(y,2^{k+1}\varepsilon_0(x)))}{2^{k(n+\eta)}\varepsilon_0(x)^{n+\eta}}\\
        &\qquad \qquad =2^nC_{CZ}c_bc_0\sum_{k=1}^{\infty}\frac{1}{2^{k\eta}} = \frac{2^n}{2^{\eta}-1}C_{CZ} c_bc_0,
    \end{align*}
    using again that $y\not\in H_{\mathcal{D}(w)}$. So we get
    \begin{align*}
    |T_{\varepsilon_0(x)}\nu_1(x)-T_{\varepsilon_0}\nu_1(y)|&\leq 2\cdot 2^nC_{CZ}c_bc_0+\frac{2^n}{2^{\eta}-1}C_{CZ}c_bc_0\\
    &= 2^nC_{CZ}\left(1+\frac{1}{2^{\eta}-1}\right)c_bc_0 = Ac_0c_b, 
    \end{align*}
    which is (\ref{desigualtat_A}). Now, if $\alpha\geq 2Ac_0c_b$,
    $$
    T_{*}\nu_1(y)> \alpha - Ac_0c_b \geq \frac{\alpha}{2}.
    $$
    Using Chebyshev's inequality,
    \begin{align*}
        \mu(S_1\setminus H_{\mathcal{D}(w)}) &\leq \mu\left(\left\{y\in F\setminus H_{\mathcal{D}(w)}: T_{*}\nu_1(y)\geq \frac{\alpha}{2}\right\}\right)\\
        &\leq \frac{2}{\alpha}\int_{F\setminus H_{\mathcal{D}(w)}}T_{*}\nu_1(y)\,d\mu(y)\\
        &\leq \frac{2}{\alpha}c_*\mu(F),
    \end{align*}
    where in the last inequality we have used the first inequality from assumption (c) in Theorem \ref{TEOREMA}.
\end{proof}

Arguing analogously, we have the following estimate.

\begin{lemma}
    Let $w\in \R^d$. If $y\in S_2\setminus H_{\mathcal{D}(w)}$, then $T^*_*\nu_2(y)> \alpha -A c_0c_b$. Thus, if $\alpha> 2Ac_0c_b$, then
    $$
    \mu\left(S_2\setminus H_{\mathcal{D}(w)}\right)\leq \frac{2c_*}{\alpha}\mu(F).
    $$
\end{lemma}
As a consequence, for $\alpha > 2Ac_0c_b$, we have that
$$
\mu\left(S\setminus H_{\mathcal{D}(w)}\right) \leq \mu\left(S_1\setminus H_{\mathcal{D}(w)}\right) + \mu\left(S_2\setminus H_{\mathcal{D}(w)}\right) \leq \frac{4c_*}{\alpha}\mu(F).
$$

Now, set $\delta_1 = (\delta_0+1)/2$ (so that $\delta_0<\delta_1<1$). Choose $\alpha = \max(2Ac_0c_b, 8c_*/(1-\delta_0))$, so that
\begin{equation}
\mu\left(H_{\mathcal{D}(w)}\cup T^1_{\mathcal{D}(w)}\cup T^2_{\mathcal{D}(w)}\right) + \mu\left(S\setminus H_{\mathcal{D}(w)}\right) \leq \left(\delta_0+\frac{4c_*}{\alpha}\right)\mu(F)\leq \delta_1\mu(F),\label{mu_htt}
\end{equation}
for any $w\in \R^d$.

\subsection{The suppressed operators}
\label{sec:suppressed_operators}

As in \cite{NTV_acta} and \cite{llibre_xavi}, we are going to consider a \textbf{suppressed} version of our kernel $k(\cdot,\cdot)$. That is, we are going to modify it in such a way that we obtain a new kernel with properties that make it easier to work with, while keeping it similar enough to the original one so that it enables us to deduce properties which, without the modification, would be much harder to prove.

For this, let $\Theta\colon \R^d\to [0,\infty)$ be such that
$$
|\Theta(x)-\Theta(y)|\leq |x-y|, \quad \text{for all}\quad x,y\in \R^d,
$$
(i.e. $\Theta$ is $1$-Lipschitz). Moreover, let $\widetilde{\chi}\colon[0,\infty)\to [0,1]$ be a smooth cut-off function which vanishes identically in $[0,1/2]$ and equals $1$ in $[1,\infty)$. We consider a new kernel
\begin{equation}
\widetilde{k}_{\Theta}(x,y) = \widetilde{\chi}\left(\frac{|x-y|^2}{\Theta (x)\Theta(y)}\right)k(x,y), \label{suppressed_kernel}
\end{equation}
and we put $\widetilde{\chi}(\infty)=1$. Similarly to what happened in \cite{NTV_acta} and \cite{llibre_xavi}, we have that $\widetilde{k}_{\Theta}(x,y)=k(x,y)$ if $\Theta(x)$ or $\Theta(y)$ vanish. We now give three useful properties of the new kernel.

\begin{lemma}
    We have that $
    \widetilde{k}_{\Theta}(x,y)=0$ if $|x-y|\leq \frac{1}{2}\,\text{max}(\Theta(x),\Theta(y))$. \label{1/2_max}
\end{lemma}

\begin{proof}
    Since $\widetilde{\chi}$ vanishes in $[0,1/2]$, it is enough to prove that if $|x-y|\leq \frac{1}{2}\text{max}(\Theta(x),\Theta(y))$, then $\frac{|x-y|^2}{\Theta(x)\Theta(y)}\leq \frac{1}{2}$. This is an easy consequence of the fact that $\Theta$ is $1$-Lipschitz. Indeed, assume first that $\Theta(x)\geq \Theta(y)$. Then, $\text{max}(\Theta(x),\Theta(y))=\Theta(x)$ and so
    $$
    \frac{|x-y|^2}{\Theta(x)\Theta(y)} \leq \frac{1}{4}\frac{\max(\Theta(x),\Theta(y))^2}{\Theta(x)\Theta(y)} = \frac{1}{4}\frac{\Theta(x)}{\Theta(y)}.
    $$
    Moreover, using the Lipschitz property, 
    $$
    \Theta(y) \geq \Theta(x)-\frac{1}{2}\max(\Theta(x),\Theta(y)) = \Theta(x)-\frac{1}{2}\Theta(x) = \frac{1}{2}\Theta(x), 
    $$
    and so we deduce that
    $$
    \frac{|x-y|^2}{\Theta(x)\Theta(y)}\leq \frac{1}{4}\frac{\Theta(x)}{\Theta(y)} \leq \frac{2}{4} = \frac{1}{2},
    $$
    which is precisely what we wanted. Now, if $\Theta(x)<\Theta(y)$, the situation is symmetric and the same arguments yield that $\frac{|x-y|^2}{\Theta(x)\Theta(y)}\leq \frac{1}{2}$.
\end{proof}

\begin{lemma}
    \label{desigualtat_maxim_k_tilde} For any $x,y\in \R^d$, we have the bound
    $$
    |\widetilde{k}_{\Theta}(x,y)|\leq \frac{c}{\max(\Theta(x),\Theta(y))^n}.
    $$
\end{lemma}

\begin{proof}
    Using the previous lemma, this is an easy computation. Indeed, if $x,y\in \R^d$ are such that $|x-y|\leq \frac{1}{2}\max(\Theta(x),\Theta(y))$, then $\widetilde{k}_{\Theta}(x,y)=0$ and the inequality in the statement is trivially true. Assume that $|x-y|>\frac{1}{2}\max(\Theta(x), \Theta(y))$. Then, since $\|\widetilde{\chi}\|_{\infty}\leq 1$,
    $$
    |\widetilde{k}_{\Theta}(x,y)| \leq |k(x,y)| \leq \frac{c}{|x-y|^n}\leq \frac{c\,2^n}{\max(\Theta(x),\Theta(y))^n}.
    $$
\end{proof}

\begin{lemma}
   Let $\Theta \colon \R^d \to [0,\infty)$ be a $1$-Lipschitz function. Then, the kernel $\widetilde{k}_{\Theta}$, defined in (\ref{suppressed_kernel}), is also a Calderón-Zygmund kernel.
\end{lemma}

\begin{proof}
   
    First, recall that $k$ being an $n$-dimensional Calderón-Zygmund kernel means that there exist constants $c>0$ and $0<\eta\leq 1$ such that the following inequalities hold for all $x,y\in \R^d$, $x\neq y$,
    \begin{align}
        |k(x,y)| &\leq \frac{c}{|x-y|^n},\quad \text{and,} \label{size_estimate} \\ 
        \text{if}\,  |x-x'|&\leq \frac{|x-y|}{2},\, \,|k(x,y)-k(x',y)| + |k(y,x)-k(y,x')| \leq \frac{c|x-x'|^{\eta}}{|x-y|^{n+\eta}}.\label{gradient_condition}
    \end{align}

    We need to show that the analogous estimates hold for $\widetilde{k}_{\Theta}$. For the size estimate (\ref{size_estimate}), we use that $\|\widetilde{\chi}\|_{\infty}\leq 1$,
    $$
    |\widetilde{k}_{\Theta}(x,y)| = \left|\widetilde{\chi}\left(\frac{|x-y|^2}{\Theta(x)\Theta(y)}\right)k(x,y)\right| \leq |k(x,y)|\leq \frac{c}{|x-y|^n}.
    $$
    
    To check the second condition (\ref{gradient_condition}), first we show that there exists a constant $c'>0$ such that if $|x-x'|\leq \frac{|x-y|}{2}$, 
    \begin{equation}
    |\widetilde{k}_{\Theta}(x,y)-\widetilde{k}_{\Theta}(x',y)| \leq \frac{c'|x-x'|^{\eta}}{|x-y|^{n+\eta}}. \label{volem}
    \end{equation}
    For this, we write,
    \begin{align*}
        |\widetilde{k}_{\Theta}(x,y)-\widetilde{k}_{\Theta}(x',y)| &= \left|\widetilde{\chi}\left(\frac{|x-y|^2}{\Theta(x)\Theta(y)}\right)k(x,y)-\widetilde{\chi}\left(\frac{|x'-y|^2}{\Theta(x)\Theta(y)}\right)k(x',y)\right|.
    \end{align*}
    Now, note that if both
    $$
    \frac{|x-y|^2}{\Theta(x)\Theta(y)}\leq \frac{1}{2} \quad \text{and} \quad \frac{|x'-y|^2}{\Theta(x)\Theta(y)}\leq \frac{1}{2},
    $$
    then $|\widetilde{k}_{\Theta}(x,y)-\widetilde{k}_{\Theta}(x',y)| = 0$ and (\ref{volem}) holds trivially. Moreover, if both
    $$
    \frac{|x-y|^2}{\Theta(x)\Theta(y)}\geq 1 \quad \text{and} \quad \frac{|x'-y|^2}{\Theta(x)\Theta(y)}\geq 1,
    $$
    then $|\widetilde{k}_{\Theta}(x,y)-\widetilde{k}_{\Theta}(x',y)|=|k(x,y)-k(x',y)|$ and we can take $c'=c$ in (\ref{volem}). Thus, we can assume that $\widetilde{k}_{\Theta}(x,y)\neq \widetilde{k}_{\Theta}(x,y)$. With this in mind, we write,
    \begin{align}
        |\widetilde{k}_{\Theta}(x,y)-\widetilde{k}_{\Theta}(x',y)|
        &\leq\left|\widetilde{\chi}\left(\frac{|x-y|^2}{\Theta(x)\Theta(y)}\right)(k(x,y)-k(x',y))\right|\nonumber \\
        &\qquad +\left|\left[\widetilde{\chi}\left(\frac{|x-y|^2}{\Theta(x)\Theta(y)}\right)-\widetilde{\chi}\left(\frac{|x'-y|^2}{\Theta(x)\Theta(y)}\right)\right]k(x',y)\right|\nonumber\\
        &\leq |k(x,y)-k(x',y)| + \|\widetilde{\chi}'\|_{\infty}\left|\frac{|x-y|^2}{\Theta(x)\Theta(y)}-\frac{|x'-y|^2}{\Theta(x')\Theta(y)}\right| \frac{c}{|x-y|^n}.\label{hem_aplicat_TVM}
    \end{align}
    The first term is already in the form that we want and we can bound it by
    $$
    |k(x,y)-k(x',y)| \leq \frac{c|x-x'|^{\eta}}{|x-y|^{n+\eta}}.
    $$
    Now we look at the second term in (\ref{hem_aplicat_TVM}). Since $\widetilde{\chi}$ is smooth and has compact support, the norm $\|\widetilde{\chi}'\|_{\infty}$ is a constant. We are now going to show that there is a constant $\widetilde{c}>0$ such that
    \begin{equation}
    \left|\frac{|x-y|^2}{\Theta(x)\Theta(y)}-\frac{|x'-y|^2}{\Theta(x')\Theta(y)}\right|  \leq \frac{\widetilde{c}|x-x'|}{|x-y|} \leq \frac{\widetilde{c}|x-x'|^{\eta}}{|x-y|^{\eta}},\label{desigualtat_amb_molts_casos}
    \end{equation}
    where in the second inequality we have used that $\eta<1$ and that$|x-x'|/|x-y|\leq 1/2$. This, taking $c'= c\left(1+\|\widetilde{\chi}'\|_{\infty}\widetilde{c}\right)$, gives us (\ref{volem}).

    To show that (\ref{desigualtat_amb_molts_casos}) holds, we denote
    $$
    A(x,y) = \frac{|x-y|^2}{\Theta(x)\Theta(y)}, \quad A(x',y) = \frac{|x'-y|^2}{\Theta(x')\Theta(y)}.
    $$
    Recall that our function $\widetilde{\chi}$ vanishes in $[0,1/2]$, is constant equal to $1$ in $[1,\infty)$ and is smooth with $0\leq \widetilde{\chi}(t)\leq 1$ for $t\in(1/2, 1)$. We will prove (\ref{desigualtat_amb_molts_casos}) distinguishing whether $A(x,y)$ and $A(x',y)$ are in $[0,1/2], (1/2,1)$ or $[1,\infty)$.

    \begin{itemize}
        \item \textbf{Case 1:} $A(x,y)\in (1/2,1)$ or, equivalently, $|x-y|^2<\Theta(x)\Theta(y)<2|x-y|^2$. 
        
        \item \textbf{Case 2:} $A(x,y)\in[0,1/2] $ or, equivalently, $|x-y|^2\leq \frac{1}{2}\Theta(x)\Theta(y)$. 

        \item \textbf{Case 3:} $A(x',y)\in [0,1/2]$ or, equivalently, $|x'-y|^2 \leq \frac{1}{2}\Theta(x')\Theta(y)$. 

        \item \textbf{Case 4:} $A(x',y)\in (1/2,1)$, or, equivalently, $|x'-y|^2<\Theta(x')\Theta(y)<2|x'-y|^2$ and $\Theta(x)\geq \frac{1}{2}\Theta(x')$. 

        \item \textbf{Case 5:} $\Theta(x)\leq \frac{1}{2}\Theta(x')$. 
    \end{itemize}

    Proving the inequality (\ref{desigualtat_amb_molts_casos}) in each case is a matter of computations, writing the same things in different ways according to which case we are in and using that $\Theta$ is $1$-Lipschitz appropriately. For commodity of the reader, we gather all these computations in the five following lemmas. 
    
    Lemmas \ref{lema:cas1} through \ref{lema:cas4} show that in cases $1$ through 4, we have
    \begin{equation*}
    \left|\frac{|x-y|^2}{\Theta(x)\Theta(y)}-\frac{|x'-y|^2}{\Theta(x')\Theta(y)}\right| \leq 40 \frac{|x-x'|}{|x-y|},
    \end{equation*}
    which means that we can take $\widetilde{c}=50$ in (\ref{desigualtat_amb_molts_casos}). For the last case, using Lemma \ref{lema:cas5}, there is no need to use the mean value theorem in the last line of (\ref{hem_aplicat_TVM}), because we can simply bound
    \begin{align*}
        |\widetilde{k}_{\Theta}(x,y)-\widetilde{k}_{\Theta}(x',y)|&\leq |k(x,y)-k(x',y)| + |k(x',y)|\frac{20}{3}\frac{|x-x'|}{|x-y|}\\
        &\leq c\frac{|x-x'|^{\eta}}{|x-y|^{n+\eta}} + \frac{20c}{3}\frac{1}{|x-y|^n}\frac{|x-x'|^{\eta}}{|x-y|^\eta}\\
        &=c\left(1+\frac{20}{3}\right)\frac{|x-x'|^{\eta}}{|x-y|^{\eta}},
    \end{align*}
    which is exactly what we wanted. Lastly, the estimate
    $$
    |\widetilde{k}_{\Theta}(y,x)-\widetilde{k}_{\Theta}(y,x')|\leq c'\frac{|x-x'|^{\eta}}{|x-y|^{n+\eta}}
    $$
    is a consequence of the computations that we have already done, since the roles of $x$ and $y$ are interchangeable in $A(x,y)$ and moreover $|x'-y|\approx |x-y|$.
    
\end{proof}

The next five lemmas are purely technical and their sole purpose is to deal with the different cases that we encountered when trying that the kernel $\widetilde{k}_{\Theta}$ is Calderón-Zygmund, in the lemma above. After these details, we define the SIOs associated to this kernel in page \pageref{k_theta_truncat}.

\begin{lemma}
    Assume that $|x-y|^2<\Theta(x)\Theta(y)<2|x-y|^2$. Then, 
    \begin{equation}
    \left|\frac{|x-y|^2}{\Theta(x)\Theta(y)}-\frac{|x'-y|^2}{\Theta(x')\Theta(y)}\right|  \leq\left(\frac{5}{2}+\frac{45}{2}\right) \frac{|x-x'|}{|x-y|}.\label{cas1}
    \end{equation}\label{lema:cas1}
\end{lemma}

\begin{proof}
    We can write
    \begin{align}
    \frac{|x-y|^2}{\Theta(x)\Theta(y)} - \frac{|x'-y|^2}{\Theta(x')\Theta(y)} &= \frac{|x-y|^2-|x'-y|^2}{\Theta(x)\Theta(y)} + |x'-y|^2\left(\frac{1}{\Theta(x)\Theta(y)}-\frac{1}{\Theta(x')\Theta(y)}\right) \nonumber \\
    &= \frac{|x-y|^2-|x'-y|^2}{\Theta(x)\Theta(y)} + |x'-y|^2\frac{\Theta(x')-\Theta(x)}{\Theta(x)\Theta(y)\Theta(x')}.\label{dos_termes}
    \end{align}
    We can control the numerator in the first term by
    \begin{align}
    \left||x-y|^2-|x'-y|^2\right| & = \left||x-y|+|x'-y|\right|\, ||x-y|-|x'-y||\nonumber \\
    &\leq ||x-y|+|x'-x|+|x-y|| \, |x-x'| \nonumber\\
    &\leq \frac{5}{2}|x-y||x-x'|.\label{desigualtat_trivial}
    \end{align}
    Thus, using that $\Theta(x)\Theta(y)>|x-y|^2$, 
    $$
    \left|\frac{|x-y|^2-|x'-y|^2}{\Theta(x)\Theta(y)}\right| \leq \frac{5}{2}\frac{|x-y||x-x'|}{|x-y|^2} = \frac{5}{2}\frac{|x-x'|}{|x-y|},
    $$
    which is in the form that we wanted. For the second term in (\ref{dos_termes}), we use, in addition, that $\Theta$ is $1$-Lipschitz, to get
    \begin{equation}
    |x'-y|^2\left|\frac{\Theta(x')-\Theta(x)}{\Theta(x)\Theta(y)\Theta(x')}\right| \leq \frac{|x'-y|^2|x'-x|}{|x-y|^2\Theta(x')}\leq \frac{9}{4}\frac{|x'-x|}{\Theta(x')},\label{4_9}
    \end{equation}
    where in the last inequality we have used that $|x'-y|\leq \frac{3}{2}|x-y|$. We see that we need a lower bound for $\Theta(x')$ in terms of $|x-y|$. To obtain such a bound, we first show that
    \begin{equation}
    \min(\Theta(x),\Theta(y)) \geq \frac{3}{5}|x-y|. \label{condicio_minim_original}
    \end{equation}
    Indeed, if we assume that the reverse inequality holds, we obtain that
    $$
    \max(\Theta(x),\Theta(y)) = \frac{\Theta(x)\Theta(y)}{\min(\Theta(x),\Theta(y))} > \frac{5}{3}\frac{|x-y|^2}{|x-y|} = \frac{5}{3}|x-y|,
    $$
    and so
    $$
    |\Theta(x)-\Theta(y)| = \max(\Theta(x),\Theta(y))-\min(\Theta(x),\Theta(y)) \geq \frac{5}{3}|x-y|-\frac{3}{5}|x-y| > |x-y|,
    $$
    which is not possible since $\Theta$ is $1$-Lipschitz, so (\ref{condicio_minim_original}) holds. Lastly, 
    \begin{align*}
        \Theta(x')\geq \Theta(x)-|x-x'| \geq \Theta(x)-\frac{1}{2}|x-y| &\geq \min(\Theta(x),\Theta(y)) - \frac{1}{2}|x-y|\\
        &\geq \frac{3}{5}|x-y|-\frac{1}{2}|x-y| = \frac{1}{10}|x-y|.
    \end{align*}
    We plug this in (\ref{4_9}) and we get
    $$
    |x'-y|^2\left|\frac{\Theta(x')-\Theta(x)}{\Theta(x)\Theta(y)\Theta(x')}\right| \leq \frac{10\cdot 9}{4}\frac{|x-x'|}{|x-y|} = \frac{45}{2}\frac{|x-x'|}{|x-y|}.
    $$
    This proves that (\ref{cas1}) holds.
\end{proof}

\begin{lemma}
    Assume that $|x-y|^2\leq \frac{1}{2}\Theta(x)\Theta(y)$. Then,
    \begin{equation}
    \left|\frac{|x-y|^2}{\Theta(x)\Theta(y)}-\frac{|x'-y|^2}{\Theta(x')\Theta(y)}\right|  \leq\left(\frac{5}{4}+\frac{9}{8}\right) \frac{|x-x'|}{|x-y|}.\label{cas2}
    \end{equation}
\end{lemma}

\begin{proof}
    We write the left-hand side of (\ref{cas2}) as in (\ref{dos_termes}). Using also (\ref{desigualtat_trivial}) and our hypothesis, we can bound the first term in (\ref{dos_termes}) by
    $$
    \left|\frac{|x-y|^2-|x'-y|^2}{\Theta(x)\Theta(y)}\right| \leq \frac{5}{2}\frac{|x-y|\,|x-x'|}{\Theta(x)\Theta(y)} \leq \frac{5}{4}\frac{|x-x'|}{|x-y|}.
    $$
    Moreover, the second term is controlled by
    $$
    \left||x'-y|^2\frac{\Theta(x')-\Theta(x)}{\Theta(x)\Theta(x')\Theta(y)}\right| \leq \frac{|x'-y|^2|x-x'|}{\Theta(x)\Theta(x')\Theta(y)} \leq \frac{|x'-y|^2|x-x'|}{2|x-y|^2\Theta(x')}\leq \frac{9}{8}\frac{|x-x'|}{\Theta(x')},
    $$
    because $|x'-y|\leq \frac{3}{2}|x-y|$. Now, we claim that our hypothesis implies that $\Theta(x)\geq |x-y|$. This, combined with the fact that $\Theta$ is $1$-Lipschitz and that $|x-x'|\leq \frac{1}{2}|x-y|$, implies that $\Theta(x')\geq \frac{1}{2}|x-y|$, and so (\ref{cas2}) holds. To prove that our claim is true, we assume that the opposite inequality holds, i.e., $\Theta(x)<|x-y|$. Then,
    $$
    \frac{1}{2}\Theta(x)\Theta(y) <\frac{1}{2}|x-y|\left(\Theta(x)+|x-y|\right)<|x-y|^2,
    $$
    which contradicts our hypothesis. Thus, the claim holds and we conclude the proof.
\end{proof}

\begin{lemma}
    Assume that $|x'-y|^2\leq \frac{1}{2}\Theta(x')\Theta(y)$. Then,
    \begin{equation}
    \left|\frac{|x-y|^2}{\Theta(x)\Theta(y)}-\frac{|x'-y|^2}{|\Theta(x')\Theta(x)|}\right| \leq \left(5+12\right)\frac{|x-x'|}{|x-y|}.
        \label{cas3}
    \end{equation}
\end{lemma}

\begin{proof}
    The proof in this case is similar to the previous case, but now we need to be slightly more careful. Now it will be more convenient to write the left-hand side of (\ref{cas3}) as
    \begin{equation}
        \frac{|x-y|^2}{\Theta(x)\Theta(y)}- \frac{|x'-y|^2}{\Theta(x')\Theta(y)} = |x-y|^2\frac{\Theta(x')-\Theta(x)}{\Theta(x)\Theta(x')\Theta(y)} - \frac{|x'-y|^2-|x-y|^2}{\Theta(x')\Theta(y)}.\label{dos_termes_versio2}
    \end{equation}
    Using (\ref{desigualtat_trivial}) and our hypothesis, we can bound the second term in (\ref{dos_termes_versio2}) by
    $$
    \left|\frac{|x'-y|^2-|x-y|^2}{\Theta(x')\Theta(y)}\right| \leq \frac{5}{2}\frac{|x-y|\,|x-x'|}{\Theta(x')\Theta(y)} \leq \frac{5}{4}\frac{|x-y|\,|x-x'|}{|x'-y|^2}\leq 5\frac{|x-x'|}{|x-y|},
    $$
    where in the last inequality we have used that $|x'-y|\geq \frac{1}{2}|x-y|$. For the first term, we use the hypothesis to get that
    $$
    \left||x-y|^2\frac{\Theta(x')-\Theta(x)}{\Theta(x)\Theta(x')\Theta(y)}\right| \leq \frac{1}{2}\frac{|x-x'|}{\Theta(x)}.
    $$
    We see that we are in a situation similar to what we had in the previous lemma; we should be able to bound $\Theta(x)\geq a|x-y|$ for some constant $a>0$. For this, we will assume that instead of $|x-x'|\leq \frac{1}{2}|x-y|$, we have that $|x-x'|\leq \frac{1}{3}|x-y|$. This extra assumption causes no problems in showing that the suppressed kernel satisfies the inequalities required to be Calderón-Zygmund. First, since $\Theta$ is $1$-Lipshitz, we clearly have that
    $$
    \Theta(x) \geq \Theta(x')-|x-x'|\geq \Theta(x')-\frac{1}{3}|x-y| \geq \Theta(x') - \frac{2}{3}|x'-y|.
    $$
    We now claim that $\Theta(x')\geq \frac{3}{4}|x'-y|$. If this holds, we have that
    $$
    \Theta(x)\geq \frac{1}{12}|x'-y| \geq \frac{1}{24}|x-y|,
    $$
    and so (\ref{cas3}) holds. To see that our claim is true, assume that the opposite inequality holds, that is, $\Theta(x')<\frac{3}{4}|x'-y|$. In this case, 
    $$
    \frac{1}{2}\Theta(x')\Theta(y) < \frac{3}{8}|x'-y|\left(\Theta(x')+ |x'-y|\right) <\frac{21}{32}|x'-y|^2<|x'-y|^2,
    $$
    which contradicts our hypothesis that $|x'-y|^2\leq \frac{1}{2}\Theta(x')\Theta(y)$. This means that $\Theta(x')\geq \frac{3}{4}|x'-y|$, thus finishing the proof.
\end{proof}

\begin{lemma}
    Assume that $|x'-y|^2<\Theta(x')\Theta(y)<2|x'-y|^2$ and $\Theta(x)\geq \frac{1}{2}\Theta(x')$. Then,
    \begin{equation}
    \left|\frac{|x-y|^2}{\Theta(x)\Theta(y)}-\frac{|x'-y|^2}{\Theta(x')\Theta(y)}\right|  \leq\left(10+\frac{80}{3}\right) \frac{|x-x'|}{|x-y|}.\label{cas4}
    \end{equation}\label{lema:cas4}
\end{lemma}

\begin{proof}
    We write the left-hand side of (\ref{cas4}) as in (\ref{dos_termes_versio2}).
    Exactly as in the proof of Lemma \ref{lema:cas1}, we have that
    $$
    ||x'-y|^2-|x-y|^2| \leq \frac{5}{2}|x-y|\,|x-x'|.
    $$
    Moreover, since from $|x-x'|\leq \frac{1}{2}|x-y|$ we have that
    $$
    \frac{1}{2}|x-y|\leq |x'-y|\leq \frac{3}{2}|x-y|,
    $$
    we can rewrite our hypothesis as
    $$
    \frac{1}{4}|x-y|^2 <\Theta(x')\Theta(y)<\frac{9}{2}|x-y|^2.
    $$
    Using this, we can bound the second term in (\ref{dos_termes_versio2}) by
    $$
    \left|\frac{|x'-y|^2-|x-y|^2}{\Theta(x')\Theta(y)}\right| \leq \frac{5}{2}\frac{|x-y|\,|x-x'|}{\Theta(x')\Theta(y)} \leq 10\frac{|x-y|\,|x-x'|}{|x-y|^2} = 10\frac{|x-x'|}{|x-y|},
    $$
    which is in the form that we want. For the first term in (\ref{dos_termes_versio2}), we have 
    \begin{equation}
    |x-y|^2\left|\frac{\Theta(x')-\Theta(x)}{\Theta(x)\Theta(x')\Theta(y)}\right| \leq \frac{|x-y|^2|x'-x|}{\Theta(x)\Theta(x')\Theta(y)}.\label{el_primer_terme}
    \end{equation}
    Analogously as we saw in the proof of Lemma \ref{lema:cas1}, we have that
    \begin{equation}
    \min(\Theta(x'),\Theta(y))\geq \frac{3}{5}|x'-y|\geq \frac{3}{10}|x-y|.\label{condicio_minim}
    \end{equation}

    Recall that we assume that $\Theta(x)\geq \frac{1}{2}\Theta(x')$. Combining this with the minimum condition (\ref{condicio_minim}), we have that $\Theta(x)\geq \frac{3}{20}|x-y|$. Hence, returning to (\ref{el_primer_terme}),
    $$
    \frac{|x-y|^2|x'-x|}{\Theta(x)\Theta(x')\Theta(y)} \leq 4\frac{|x-y|^2|x'-x|}{\Theta(x)|x-y|^2} = 4\frac{|x-x'|}{\Theta(x)}\leq \frac{80}{3}\frac{|x-x'|}{|x-y|},
    $$
    and we conclude.
\end{proof}

\begin{lemma}
    Assume that $\Theta(x)\leq \frac{1}{2}\Theta(x')$. Then, 
     $$
    \left|\widetilde{\chi}\left(\frac{|x-y|^2}{\Theta(x)\Theta(y)}\right)-\widetilde{\chi}\left(\frac{|x'-y|^2}{\Theta(x)\Theta(y)}\right)\right|\leq \frac{20}{3}\frac{|x-x'|}{|x-y|}.
    $$\label{lema:cas5}
\end{lemma}

\begin{proof}
    First, since $0\leq \widetilde{\chi}(t)\leq 1$ for all $t\in \R$,
    $$
    \left|\widetilde{\chi}\left(\frac{|x-y|^2}{\Theta(x)\Theta(y)}\right)-\widetilde{\chi}\left(\frac{|x'-y|^2}{\Theta(x)\Theta(y)}\right)\right|\leq 1.
    $$
    Once again, we have that $\min(\Theta(x'),\Theta(y))\geq \frac{3}{10}|x-y|$. Moreover, using that $\Theta$ is $1$-Lipschitz,
    $$
    |x-x'|\geq |\Theta(x)-\Theta(x')|\geq \Theta(x')-\Theta(x)\geq \frac{1}{2}\Theta(x')\geq \frac{3}{20}|x-y|.
    $$
    Hence, ${|x-x'|}/{|x-y|}\geq 3/20$, which concludes the proof.
\end{proof}

For $\varepsilon\geq 0$, and a complex Radon measure $\nu$, we set
\begin{align}
K_{\Theta,\varepsilon}\nu(x) = \int_{\R^d\setminus B(x,\varepsilon)}\widetilde{k}_{\Theta}(x,y)\,d\nu(y),\quad 
K_{\Theta, \varepsilon}^*\nu(x)=\int_{\R^d\setminus B(x,\varepsilon)}\widetilde{k}_{\Theta}(y,x)\,d\nu(y).\label{k_theta_truncat}
\end{align}
The operator $K_{\Theta, \varepsilon}$ is the ($\varepsilon$-truncated) $\Theta$\textbf{-suppressed} version of the operator $T$ and analogously for $K^{*}_{\Theta,\varepsilon}$ and $T^*$. We also set
$$
K_{\Theta,*}\nu(x) = \sup_{\varepsilon>0}|K_{\Theta,\varepsilon}\nu(x)|, \quad K_{\Theta}\nu(x) = K_{\Theta,0}\nu(x),
$$
assuming that the integral that defines $K_{\Theta,0}\nu(x)$ exists. Given $f\in L^1(\mu)$, to simplify notation, we will write
$$
K_{\Theta,\varepsilon}f:=K_{\Theta,\varepsilon}(f\mu), \quad K_{\Theta,*}f:=K_{\Theta,*}(f\mu), \quad K_{\Theta}f:=K_{\Theta}(f\mu),
$$
and, as usual, we will use the analogous notation for $K^*_{\Theta, \varepsilon}, K^*_{\Theta, *}$ and $K_{\Theta}^*$.

In the following lemma, we are going to relate the truncated operators $T_\varepsilon$ with the truncated suppressed operators $K_{\Theta, \varepsilon}$. We give estimates for both $T$ and $T^*$.

\begin{lemma}\label{diferencia_truncat_suppressed}
    Let $x\in \R^d$ and suppose that $\Theta(x)\leq \varepsilon$. Let $\sigma$ be a complex Radon measure on $\R^d$. Then, we have the following estimates,
    $$
    |T_{\varepsilon}\sigma(x)-K_{\Theta,\varepsilon}\sigma(x)| \leq c\sup_{r\geq \varepsilon}\frac{|\sigma|(B(x,r))}{r^n}, \quad |T^*_{\varepsilon}\sigma(x)-K^*_{\Theta,\varepsilon}\sigma(x)| \leq c\sup_{r\geq \varepsilon}\frac{|\sigma|(B(x,r))}{r^n}.
    $$
\end{lemma}

\begin{proof}
    We have
    \begin{align*}
    |\widetilde{k}_{\Theta}(x,y)-k(x,y)| &= \left|\widetilde{\chi}\left(\frac{|x-y|^2}{\Theta(x)\Theta(y)}\right)k(x,y)-k(x,y)\right|\\
    &= \left|\widetilde{\chi}\left(\frac{|x-y|^2}{\Theta(x)\Theta(y)}\right)-1\right||k(x,y)|\leq \left|\widetilde{\chi}\left(\frac{|x-y|^2}{\Theta(x)\Theta(y)}\right)-1\right|\frac{C_{CZ}}{|x-y|^n}.
    \end{align*}
    We need to estimate the difference inside the absolute value. First, note that if $|x-y|^2\geq \Theta(x)\Theta(y)$, then it is zero. Now, assume that $|x-y|^2<\Theta(x)\Theta(y)$. We apply the mean value theorem,
    $$
    \left|\widetilde{\chi}\left(\frac{|x-y|^2}{\Theta(x)\Theta(y)}\right)-1\right| = \left|\widetilde{\chi}\left(\frac{|x-y|^2}{\Theta(x)\Theta(y)}\right)-\widetilde{\chi}(1)\right| \leq \|\widetilde{\chi}'\|_{\infty}\left|\frac{|x-y|^2}{\Theta(x)\Theta(y)}-1\right|.
    $$
    Using that $|x-y|^2<\Theta(x)\Theta(y)$, we get
    \begin{align*}
        \left|\frac{|x-y|^2}{\Theta(x)\Theta(y)}-1\right| &= \left|\frac{|x-y|^2-\Theta(x)\Theta(y)}{\Theta(x)\Theta(y)}\right| = \frac{\Theta(x)\Theta(y)-|x-y|^2}{\Theta(x)\Theta(y)}\\
        &\leq \frac{\Theta(x)\Theta(y)}{|x-y|^2} \leq \frac{\Theta(x)\left(\Theta(x)+|x-y|\right)}{|x-y|^2} =\frac{\Theta(x)^2}{|x-y|^2} + \frac{\Theta(x)}{|x-y|}.
    \end{align*}
    Hence,
    $$
    |K_{\Theta,\varepsilon}\sigma(x)-T_{\varepsilon}\sigma(x)|\leq \|\widetilde{\chi}'\|_{\infty}C_{CZ}\int_{|x-y|>\varepsilon}\left(\frac{\Theta(x)^2}{|x-y|^{n+2}}+\frac{\Theta(x)}{|x-y|^{n+1}}\right)d|\sigma|(y).
    $$
    To estimate the last integrals, we write $\{x\in \R^d: |x-y|>\varepsilon\}$ as a disjoint union of annuli,
    \begin{align*}
        \int_{|x-y|>\varepsilon}\frac{1}{|x-y|^{n+2}}d|\sigma|(y) &= \sum_{k=0}^{\infty}\int_{2^k\varepsilon<|x-y|\leq 2^{k+1}\varepsilon}\frac{1}{|x-y|^{n+2}}d|\sigma|(y)\\
        &\leq \sum_{k=0}^{\infty}\frac{|\sigma|(B(x, 2^{k+1}\varepsilon))}{(2^k\varepsilon)^{n+2}}\\
        &= \frac{2^{n+2}}{\varepsilon^2}\sum_{k=0}^{\infty}\frac{|\sigma|(B(x,2^{k+1}\varepsilon))}{(2^{k+1}\varepsilon)^n}\frac{1}{4^{k+1}}\\
        &=\frac{2^{n+2}}{\varepsilon^2}\sup_{r\geq\varepsilon}\frac{|\sigma|(B(x,r))}{r^n}\sum_{k=0}^{\infty}\frac{1}{4^{k+1}}\\
        &=\frac{C_1}{\varepsilon^2}\sup_{r\geq \varepsilon}\frac{|\sigma|(B(x,r))}{r^n}.
    \end{align*}
    Analogously,
    $$
    \int_{|x-y|>\varepsilon}\frac{1}{|x-y|^{n+1}}d|\sigma|(y) \leq \frac{C_2}{\varepsilon}\sup_{r\geq \varepsilon}\frac{|\sigma|(B(x,r))}{r^n}.
    $$
    Using these two estimates, we can bound the difference
    \begin{align*}
        |K_{\Theta,\varepsilon}\sigma(x)-T_{\varepsilon}\sigma(x)|&\leq \|\widetilde{\chi}'\|_{\infty}\left(\frac{C_1}{\varepsilon^2}\Theta(x)^2+\frac{C_2}{\varepsilon}\Theta(x)\right)\sup_{r\geq\varepsilon}\frac{|\sigma|(B(x,r))}{r^n}\\
        &\leq c\sup_{r\geq \varepsilon}\frac{|\sigma|(B(x,r))}{r^n},
    \end{align*}
    where we have used our assumption that $\Theta(x)\leq \varepsilon$. The proof for the estimate involving $T^*_\varepsilon$ and $K^*_{\Theta,\varepsilon}$ is the same, since the role of $x$ and $y$ inside $\widetilde{\chi}$ can be interchanged in the definition of $\widetilde{k}_{\Theta}(x,y)$.
\end{proof}

Continuing in the line of the relation between the truncated operators and both truncated and suppressed operators, in the following lemma we are going to see how knowledge about the size of the first ones gives information about the size of the others.

\begin{lemma}
    Let $x\in \R^d$ and $r_1,r_2\geq 0$, possibly depending on $x$, be such that 
    \begin{itemize}
        \item $\mu(B(x,r))\leq c_0r^n$ for $r\geq \text{min}(r_1,r_2)$,
        \item $|T_\varepsilon\nu_1(x)|\leq \alpha$ for $\alpha \geq r_1$ and $|T^*_{\varepsilon}\nu_2(x)|\leq \alpha$ for $\alpha \geq r_2$.
    \end{itemize}
    Then,
    \begin{enumerate}[label=(\alph*)]
        \item If $\Theta(x)\geq r_1$, we have $ |K_{\Theta,\varepsilon}\nu_1(x)|\leq \alpha + cc_bc_0$ for all $\varepsilon>0$.
        \item If $\Theta(x)\geq r_2$, we have $ |K^*_{\Theta,\varepsilon}\nu_2(x)|\leq \alpha + cc_bc_0$ for all $\varepsilon>0$.
    \end{enumerate}
    In particular, if $\Theta(x)\geq \max(r_1,r_2)$, both estimates hold simultaneously for all $\varepsilon>0$. \label{lema_r0}
\end{lemma}

\begin{proof}
    For the commodity of the reader, we will only give the details of the proof of (a). A careful inspection of our argument quickly reveals that the proof of (b) is almost identical.
    
    That being said, in order to prove (a) we will treat two cases separately. First, if $\varepsilon\geq \Theta(x)\geq r_1$, we use the previous lemma and we infer that in this situation we can bound
    $$
    |K_{\Theta,\varepsilon}\nu_1(x)| \leq |T_{\varepsilon}\nu_1(x)| + c\sup_{r\geq \varepsilon}\frac{|\nu_1|(B(x,r))}{r^n} \leq \alpha +c c_b\sup_{r\geq\varepsilon}\frac{\mu(B(x,r))}{r^n} \leq \alpha +cc_bc_0.
    $$
    Now, if the opposite inequality holds, that is, if $\varepsilon<\Theta(x)$, we can bound
    \begin{equation}
    |K_{\Theta,\varepsilon}\nu_1(x)| \leq c_b \int_{B(x,\Theta(x))}|\widetilde{k}_{\Theta}(x,y)|\,d\mu(y) + \left|\int_{\R^d\setminus B(x,\Theta(x))}\widetilde{k}_{\Theta}(x,y)\,d\nu_1(y)\right|.\label{dues_integrals}
    \end{equation}
    For the first integral, we write
    $$
    \int_{B(x,\Theta(x))}|\widetilde{k}_{\Theta}(x,y)|d\mu(y)  = \int_{B(x,\frac{1}{2}\Theta(x))}|\widetilde{k}_{\Theta}(x,y)|d\mu(y) + \int_{\frac{1}{2}\Theta(x)\leq |x-y|<\Theta(x)}|\widetilde{k}_{\Theta}(x,y)|d\mu(y).
    $$
    Since $\frac{1}{2}\Theta(x)\leq \frac{1}{2}\max(\Theta(x),\Theta(y))$, by Lemma \ref{1/2_max}, we have that $\widetilde{k}_{\Theta}(x,y)=0$ for all $y\in B(x,\frac{1}{2}\Theta(x))$. Thus, the right-hand side of the previous equation is bounded by
    $$
    \int_{\frac{1}{2}\Theta(x)\leq |x-y|<\Theta(x)}|\widetilde{k}_{\Theta}(x,y)|d\mu(y) \leq 2^n\frac{\mu(B(x,\Theta(x)))}{\Theta(x)^n} \leq 2^nc_0 = cc_0,
    $$
    since $r_1\leq \Theta(x)$, we have that $\mu(B(x,\Theta(x)))\leq c_0\Theta(x)^n$.
    
    Lastly, the second integral in (\ref{dues_integrals}) equals $K_{\Theta,\Theta(x)}\nu_1(x)$ and, by our computation above, it is bounded by $\alpha + cc_bc_0$.
\end{proof}

Now, recall point (a) of Theorem \ref{TEOREMA}, which is the converse inequality to the one in the first hypothesis from the previous lemma,
\begin{center}
    (a) Every ball $B_r$ of radius $r$ such that $\mu(B_r)>c_0r^n$ is contained in $\bigcap_{w\in \R^d}H_{\mathcal{D}(w)}$.
\end{center}
The behavior of these balls motivates the following definition.
\begin{defi}
    Let $x\in F$. We say that a ball $B(x,r)$ is \textbf{non-}$n$\textbf{-Ahlfors} if 
    $$\mu(B(x,r))>c_0r^n.$$
    In addition, we define the $n$\textbf{-Ahlfors radius} of the point $x$ as
    \begin{align*}
        \mathcal{R}(x) &= \sup\{r>0 : \mu(B(x,r))>c_0r^n\} \\
        &=\sup\{r>0 : B(x,r)\text{\, is a non-$n$-Ahlfors ball\,}\}.
    \end{align*}
    If there does not exist any $r>0$ such that $\mu(B(x,r))>c_0r^n$, we simply set $\mathcal{R}(x)=0$.
\end{defi}
Note that for any $w\in \R^d$, all non-$n$-Ahlfors balls and the balls $B(x,e_1(x)), B(x, e_2(x))$, defined in (\ref{e1}) and (\ref{e2}), are contained in $H_{\mathcal{D}(w)}\cup S$. This observation will be essential in proving the following lemma.

\begin{lemma}
    \label{lema_distancia} Let $w\in \R^d$ and let $\Theta\colon \R^d\to [0,\infty)$ be a $1$-Lipschitz function such that 
    $$
    \Theta(x)\geq \text{dist}(x, \R^d\setminus (H_{\mathcal{D}(w)}\cup S)),
    $$
    for all $x\in \R^d$. Then, $K_{\Theta, *}\nu_1(x)\leq c_4$ and $K_{\Theta,*}^*\nu_2(x)\leq c_4$ for all $x\in F$, with $c_4$ depending on $c_0,c_b, c_*$, and $\delta_0$. 
\end{lemma}
\begin{proof}
    Let $x\in F$ and call $r_0 = \max(\mathcal{R}(x), e_1(x), e_2(x))$. We claim that $\Theta(x)\geq r_0$. 
    
    Obviously, if $x\not\in H_{\mathcal{D}(w)}\cup S$, $r_0=\mathcal{R}(x)=e_1(x)=e_2(x)=0$, so the inequality holds trivially. If $x\in H_{\mathcal{D}(w)}\cup S$, we can see that $\Theta(x)\geq \mathcal{R}(x)$ and $\Theta(x)\geq \max( e_1(x),e_2(x))$ separately.
    
    First, we show that $\Theta(x)\geq \mathcal{R}(x)$. Indeed, if $\mathcal{R}(x)=0$, it is clear. If $\mathcal{R}(x)>0$, for any $0<r<\mathcal{R}(x)$, by definition of supremum there is some $r<r'\leq \mathcal{R}(x)$ such that $B(x,r')$ is non-$n$-Ahlfors, so $B(x,r')\subset H_{\mathcal{D}(w)}$. Hence, 
    $$
    \Theta(x)\geq \text{dist}(x,\R^d\setminus (H_{\mathcal{D}(w)}\cup S))\geq r'>r, \quad \text{for all}\quad r<\mathcal{R}(x),
    $$
    from which we deduce that $\Theta(x)\geq \mathcal{R}(x)$. 
    
    Now, to prove that $\Theta(x)\geq \max(e_1(x),e_2(x))$, we will only show the details of the proof of $\Theta(x)\geq e_1(x)$. The other inequality is proved analogously.
    
    If $x\not \in S^1_0$, we have that $e_1(x)=0$ and so the inequality is clear. If $x\in S_0$, then $x\in B(x,e_1(x))\subset S$, so
    $$\text{dist}(x, \R^d\setminus(H_{\mathcal{D}(w)}\cup S))\geq \text{dist}(x, \R^d\setminus (H_{\mathcal{D}(w)}\cup S_1)) \geq e_1(x),$$ 
    and in this case the inequality is also satisfied. Analogously, $\Theta(x)\geq e_2(x)$.

    We can apply the last part of Lemma \ref{lema_r0} with this choice of $r_0$. Indeed, since $r_0\geq \mathcal{R}(x)$, for $r\geq r_0$ we have that $\mu(B(x,r))\leq c_0r^n$. Also, if $\varepsilon\geq r_0\geq \max( e_1(x), e_2(x))$, by definition of $e_1(x),e_2(x)$, we have that $|T_{\varepsilon}\nu_1(x)|, |T^*_{\varepsilon}\nu_2(x)|\leq \alpha$. This allows us to deduce that for all $\varepsilon>0$, $|K_{\Theta,\varepsilon}\nu_1(x)|, |K^*_{\Theta,\varepsilon}\nu_2(x)|\leq \alpha+ cc_bc_0$. We conclude by taking supremum over $\varepsilon>0$.
\end{proof}

\subsection{Dyadic lattices and the martingale decomposition}

\subsubsection{Random dyadic lattices}
\label{sec:random_dyadic_lattices}

As we will see in the end of the proof of Theorem \ref{TEOREMA}, a probabilistic argument will play a key role in showing the $L^2(\mu\lfloor G)$ boundedness of $T$. The motivation for introducing randomness is that we are going to relate boundedness of our original operators with that of other operators arising from taking averages (or, in probabilistic terms, expectations) over the points of a set $\Omega\subset \R^d$ with which we translate the usual dyadic lattice. In this section, we define the set $\Omega$.

Let $N\geq 1$ be some big integer, which we will choose later, and define $S^0=\left[0,2^N\right]^d$. In the rest of the proof of Theorem \ref{TEOREMA}, we assume that
$$
\text{supp}(\mu)= F \subset \frac{1}{8}S^0. 
$$
Moreover, we set
$$
\Omega = \left[-2^{N-4}, 2^{N-4}\right]^d\subseteq \R^d,
$$
and we consider dyadic lattices $\mathcal{D}(w)$ as in (\ref{xarxes_traslladades}), with $w\in \Omega$. We write
\begin{equation}
Q^0(w) = w + S^0,\label{definicio_Q0}
\end{equation}
so that $Q^0(w)\in \mathcal{D}(w)$. Moreover, by our choice of $\Omega$, we have the following.

\begin{lemma}
    For all $w\in \Omega$, we have $F \subseteq \frac{1}{4}Q^0(w)$.\label{cub_1/4}
\end{lemma}
\begin{proof}
    Indeed, if we denote by $z_{Q^0(w)}$ and $z_{S^0}$ the centers of the cubes $Q^0(w)$ and $S^0$, respectively, we have that
    $$
    \frac{1}{4}Q^0(w) = w+z_{S^0}+\left[-2^{N-3}, 2^{N-3}\right]^d.
    $$
    Let $x\in F$ be arbitrary. In order to show that $x\in \frac{1}{4}Q^0(w)$, we have to prove that 
    $$
    \|x-(w-z_{S^0})\|_{\infty} \leq \frac{\ell\left(\frac{1}{4}Q^0(w)\right)}{2} = 2^{N-3}.
    $$
    Using that $x\in \frac{1}{8}S^0$, we have that $\|x-z_{S^0}\|_{\infty}\leq \frac{\ell\left(\frac{1}{8}S^0\right)}{2}=2^{N-4}$. Hence,
    $$
    \|x-(w-z_{S^0})\|_{\infty} \leq \|x-z_{S^0}\|_{\infty}+\|w\|_{\infty}\leq 2^{N-4}+2^{N-4},
    $$
    where we have used that $\|w\|_{\infty}\leq 2^{N-4}$, thus completing the proof.
\end{proof}

We will denote by $P^\Omega$ the uniform probability on $\Omega$, that is, the normalized Lebesgue measure on the cube $\Omega$.

\subsubsection{Transit and terminal cubes}

Recall that in the first section we have singled out an exceptional set with the intention of controlling the maximal functions $T_*\nu_1$ and $T^*_*\nu_2$. This will not be the only instance of the use of such a technique in the proof of Theorem \ref{TEOREMA}. In the present section, we are going to classify the dyadic cubes of any given lattice, $\mathcal{D}$, taking into account whether or not they are contained in the special sets $H_\mathcal{D}, T^1_\mathcal{D}$ or $T^2_\mathcal{D}$. We will see that, on the one hand, the cubes that are not contained inside these sets have some desirable properties regarding the growth of the $\mu$-measure of their dilates and their side lengths. On the other hand, by considering the maximal cubes contained inside the special sets, we will obtain a disjoint family of cubes with side lengths comparable to those of the better-behaved ones. We give this classification now.

Fix $w\in \Omega$ and consider the dyadic lattice $\mathcal{D}\equiv \mathcal{D}(w)$. Let $Q\in \mathcal{D}$ be contained in $Q^0_{\mathcal{D}}=Q^0(w)$ with $\mu(Q)>0$. We say that the cube $Q$ is 
\begin{itemize}
    \item \textbf{terminal of the first} (resp. \textbf{second}) \textbf{kind} if $2Q\subset H_{\mathcal{D}}$ or if $Q\subset  T^1_{\mathcal{D}}$ (resp. $2Q\subset H_\mathcal{D}$ or $Q\subset  T^2_\mathcal{D}$),
    \item \textbf{transit of the first} (resp. \textbf{second}) \textbf{kind} if it is not terminal of the first (resp. second) kind.
\end{itemize}

We denote the set of cubes that are terminal of the first kind by $\mathcal{D}^{\text{term},1}$, the set of those that are terminal of the second kind by $\mathcal{D}^{\text{term},2}$ and $\mathcal{D}^{\text{term}} = \mathcal{D}^{\text{term},1}\cup \mathcal{D}^{\text{term},2}$ and analogously with $\mathcal{D}^{\text{tr},1}$ and $\mathcal{D}^{\text{tr},2}$ for transit cubes.

Note that we assume all transit and terminal (of the first and second kind) cubes to be contained in $Q^0_{\mathcal{D}}$. This is because if a cube (or part of it) is outside $Q^0_{\mathcal{D}}$, then it cannot contribute any $\mu$-measure (or the only contribution comes from inside $Q^0_{\mathcal{D}}$).

From assumption (e) in Theorem \ref{TEOREMA}, $Q^0_{\mathcal{D}}$ is always transit of both kinds. Indeed, if it were terminal of any kind, it would be contained in $H_{\mathcal{D}}\cup T^1_{\mathcal{D}}\cup T_{\mathcal{D}}^2$, and so by (e),
$$
\mu(Q^0_{\mathcal{D}})\leq \mu(H_{\mathcal{D}}\cup T^1_{\mathcal{D}}\cup T^2_{\mathcal{D}})\leq \delta_0\mu(F)<\mu(F).
$$
However, by Lemma \ref{cub_1/4},
$$
\mu(F)\leq \mu\left(\frac{1}{4}Q^0_{\mathcal{D}}\right)\leq \mu(Q^0_{\mathcal{D}}),
$$
which contradicts the previous inequality.

\begin{lemma}
    \label{lema:comportament_cubs_transit} Let $Q\in \mathcal{D}^{\text{tr},1}\cup \mathcal{D}^{\text{tr},2}$. Then,
    $$
    \mu(\lambda Q)\leq c_0\ell(\lambda Q)^n, \quad \text{for all }\,\lambda\geq 1.
    $$
    Moreover, if $Q\in \mathcal{D}^{\text{tr},1}$, we have $\mu(Q)\leq c_{\text{acc}}|\nu_1 (Q)|$, and for $Q\in \mathcal{D}^{\text{tr},2}$, the corresponding inequality is $\mu(Q)\leq c_{\text{acc}}|\nu_2(Q)|$.
\end{lemma}
\begin{proof}
    Since $Q$ is transit, $Q\not\subset H_{\mathcal{D}}\cup T^1_{\mathcal{D}}\cup T^2_{\mathcal{D}}$. In particular, $Q\not\subset H_{\mathcal{D}}$ and so by assumption (a) of Theorem \ref{TEOREMA}, any ball containing $Q$ satisfies $\mu(B(x,r))\leq c_0r^n$. Hence, if $z_Q$ is the center of $Q$,
    $$
    \mu(\lambda Q)\leq \mu(B(z_Q, \lambda \ell(Q)))\leq c_0\ell(Q)^n.
    $$
    Also, from $Q\not\subset T^1_{\mathcal{D}}\cup T^2_{\mathcal{D}}$, from assumption (d) in Theorem \ref{TEOREMA}, we have that $\mu(Q)\leq c_{\text{acc}}|\nu_i(Q)|$, for $i=1,2$.
\end{proof}

\subsubsection{The martingale decomposition}

So far, our classification of dyadic cubes of a given lattice has only given us information about the measure $\mu$ on the dilates of the cubes and the relation between the $\mu$-measure and the $\nu_i$-measure of said cubes. However, the distinction between transit and terminal cubes has much more to offer. Indeed, in the following pages, we are going to define for each transit cube an operator on $L^2(\mu)$ and with a little work we will be able to express any function from this space as a sum of these operators acting on our function. This is a standard technique in Harmonic Analysis and will involve the averages of our functions over the transit cubes, together with those of the functions $b_1$ and $b_2$.

Let $f\in L^1_{\text{loc}}(\mu)$, and let $Q$ be any cube with $\mu(Q)\neq 0$, we put
$$
\langle f\rangle_Q = \frac{1}{\mu(Q)}\int_Qf\,d\mu.
$$
\begin{obs}
    An observation which will be useful below is what follows. If $Q\in \mathcal{D}^{\text{tr},i}$, then $|\langle b_i\rangle_Q|\geq c_{\text{acc}}^{-1}$. Indeed, we can write
    $$
    |\langle b_i\rangle_Q| = \frac{1}{|\mu(Q)|}\left|\int_Qb_i\,d\mu\right| = \frac{|\nu_i(Q)|}{|\mu(Q)|}\geq \frac{|\nu_i(Q)|}{c_{\text{acc}}|\nu_i(Q)|}= c_{\text{acc}}^{-1},
    $$
    because if $Q\in \mathcal{D}^{\text{tr},i}$, in particular, $Q\not \subset T^i_{\mathcal{D}}$. \label{obs:martingala_b_transit}
\end{obs}

Moreover, we define the operators $\Xi_1$ and $\Xi_2$ as
$$
\Xi_1 f = \frac{\langle f\rangle_{Q^0_{\mathcal{D}}}}{\langle b_1\rangle_{Q^0_{\mathcal{D}}}}b_1, \quad \Xi_2 f = \frac{\langle f\rangle_{Q^0_{\mathcal{D}}}}{\langle b_2\rangle_{Q^0_{\mathcal{D}}}}b_2,
$$
where $b_1$ and $b_2$ are the functions from Theorem \ref{TEOREMA}. Since $Q^0_{\mathcal{D}}$ is always transit, by Lemma \ref{lema:comportament_cubs_transit}, we have that neither $\langle b_1\rangle_{Q_\mathcal{D}^0}$ nor $\langle b_2\rangle_{Q^0_{\mathcal{D}}}$ vanish. Moreover, the definition of $\Xi_i$ does not depend on the choice of the lattice $\mathcal{D}(w)$, $w\in \Omega$, since $Q^0_\mathcal{D}$ contains the support of $\mu$. We display some more useful facts concerning these operators in the following lemma.

\begin{lemma}
    \label{lema:propietats_xi}For $i=1,2$, we have that
    \begin{enumerate}[label=(\alph*)]
        \item $\Xi_i\,\colon L^2(\mu)\to L^2(\mu)$ are bounded, with norm depending on $\|b_i\|_{\infty}$ and $\langle b_i\rangle_{Q^0_{\mathcal{D}}}$.
        \item $\Xi_i^2 =\Xi_i$, where $\Xi_i^2$ means the composition of $\Xi_i$ with itself.
        \item The adjoint of $\Xi_i$ in $L^2(\mu)$ is
        $$
        \Xi_i^*f = \frac{\langle fb_i\rangle_{Q_\mathcal{D}^0}}{\langle b_i\rangle_{Q^0_{\mathcal{D}}}}.
        $$
    \end{enumerate}
\end{lemma}
These properties are verified by means of simple computations, which we omit here.

Recall that, for a fixed dyadic cube $Q\in \mathcal{D}$, the set of, at most $2^d$, children of $Q$ is denoted by $\mathcal{CH}(Q)$. From now on, we will use the same notation to refer to the set of children whose $\mu$-measure is not zero.

For any cube $Q\in \mathcal{D}^{\text{tr},i}$ and any $f\in L^1_{\text{loc}}(\mu)$, we define two new functions, $\Delta_{i, Q}f$, $i=1,2$, as follows,
\begin{align*}
    \Delta_{i,Q}f = \begin{cases}
        0 \qquad &\text{in }\, \R^d\setminus \bigcup_{P\in \mathcal{CH}(Q)}
P,\\
\left(\frac{\langle f\rangle_{P}}{\langle b_i\rangle_P}-\frac{\langle f\rangle_Q}{\langle b_i\rangle_Q}\right)b_i\qquad &\text{in }\, P\,\text{ if }\, P\in \mathcal{CH}(Q)\cap \mathcal{D}^{\text{tr},i},\\
f-\frac{\langle f\rangle_Q}{\langle b_i\rangle_Q}b_i&\text{in }\, P\,\text{ if }P\in \mathcal{CH}(Q)\cap \mathcal{D}^{\text{term},i}.
\end{cases}
\end{align*}

We collect some properties of the operators $\Delta_{i, Q}$ in the next lemma.
\begin{lemma}\label{lema:propietats_delta}
    For any $f\in L^2(\mu)$ and all $Q\in \mathcal{D}^{\text{tr},i}$,
    \begin{enumerate}[label=(\alph*)]
        \item $\Delta_{i,Q}f\in L^2(\mu)$,
        \item $\int \Delta_{i,Q}f d\mu=0$,
        \item $\Delta_{i,Q}$ is a projection, that is, $\Delta^2_{i,Q}= \Delta_{i,Q}$,
        \item $\Delta_{i,Q}\Xi_i = \Xi_i\Delta_{i,Q}=0$,
        \item if $R\in \mathcal{D}^{\text{tr},i}$, then $\Delta_{i,Q}\Delta_{i,R}=0$,
        \item The adjoint of $\Delta_{i,Q}$ is
        \begin{align*}
    \Delta^*_{i,Q}f = \begin{cases}
        0 \qquad &\text{in }\, \R^d\setminus \bigcup_{P\in \mathcal{CH}(Q)}
P,\\
\frac{\langle fb_i\rangle_{P}}{\langle b_i\rangle_P}-\frac{\langle fb_i\rangle_Q}{\langle b_i\rangle_Q}\qquad &\text{in }\, P\,\text{ if }\, P\in \mathcal{CH}(Q)\cap \mathcal{D}^{\text{tr},i},\\
f-\frac{\langle fb_i\rangle_Q}{\langle b_i\rangle_Q}&\text{in }\, P\,\text{ if }P\in \mathcal{CH}(Q)\cap \mathcal{D}^{\text{term},i}.
\end{cases}
\end{align*}
    \end{enumerate}
\end{lemma}

\begin{proof}
    The proof of assertions (a)-(f) is similar to the proof of Lemma \ref{lema:propietats_xi}, which we omitted, but the computations here are somewhat more involved. This is why we show the proofs of (b), (c), (d) and give an indication for (e).

    Starting with (b), for $Q\in \mathcal{D}^{\text{tr}, i}$, we can enumerate its children as $\{P_k\}_{k=1}^{2^d}$. Then,
    \begin{equation}
    \int \Delta_{i,Q}f\,d\mu = \int_{Q} \Delta_{i,Q}f\,d\mu = \sum_{k=1}^{2^d}\int_{P_k}\Delta_{i,Q}f\,d\mu.\label{calcul_b}
    \end{equation}
    First, if $P_k\in \mathcal{CH}(Q)\cap \mathcal{D}^{\text{tr},i}$,
    \begin{align*}
        \int_{P_k}\Delta_{i,Q}f\,d\mu &= \int_{P_k}\left(\frac{\langle f\rangle_{P_k}}{\langle b_i\rangle_{P_k}}-\frac{\langle f\rangle_Q}{\langle b_i\rangle_{Q}}\right)b_i\,d\mu\\
        &= \langle f\rangle_{P_k}\mu(P_k) - \frac{\langle f\rangle_{Q}}{\langle b_i\rangle_Q}\mu(P_k)\langle b_i\rangle_{P_k} = \int_{P_k}f\,d\mu - \frac{\langle f\rangle_Q}{\langle b_i\rangle_Q}\int_{P_k}b_i\,d\mu.
    \end{align*}
    Second, if $P\in \mathcal{CH}(Q)\cap \mathcal{D}^{\text{term},i}$, by definition of $\Delta_{i,Q}$ it is obvious that
    $$
    \int_{P_k}\Delta_{i,Q}f\,d\mu = \int_{P_k}f\,d\mu - \frac{\langle f\rangle_Q}{\langle b_i\rangle_Q}\int_{P_k}b_i\,d\mu.
    $$
    Lastly, if $P_k\in \R^d\setminus \cup_{P\in \mathcal{CH}(Q)}P$,
    $$
    \int_{P_k}\Delta_{i,Q}f\,d\mu=0=\int_{P_k}f\,d\mu - \frac{\langle f\rangle_Q}{\langle b_i\rangle_Q}\int_{P_k}b_i\,d\mu.
    $$
    Hence, we can return to (\ref{calcul_b}) and we find that
    \begin{align*}
        \int \Delta_{i,Q}f \,d\mu = \sum_{k=1}^{2^d}\int_{P_k}f\,d\mu - \frac{\langle f\rangle_Q}{\langle b_i\rangle_Q}\sum_{k=1}^{2^d}\int_{P_k}b_i\,d\mu = \int_Qf\,d\mu - \frac{\langle f\rangle_Q}{\langle b_i\rangle_Q}\int b\,d\mu = 0.
    \end{align*}
    To show (c), note that using (b) we can write
    \begin{align*}
        \Delta^2_{i,Q}f = \Delta_{i,Q}(\Delta_{i,Q}f) = \begin{cases}
            0 \qquad &\text{in }\, \R^d\setminus \bigcup_{P\in \mathcal{CH}(Q)}P,\\
            \frac{\langle \Delta_{i,Q} f\rangle_P}{\langle b_i\rangle_P}, \qquad &\text{in }\, P\,\text{ if }\, P\in \mathcal{CH}(Q)\cap \mathcal{D}^{\text{tr},i},\\
            \Delta_{i,Q}f\qquad &\text{in }\, P\,\text{ if }\, P\in \mathcal{CH}(Q)\cap \mathcal{D}^{\text{term},i}.
        \end{cases}
    \end{align*}
    Hence, we only need to compute, for $P\in \mathcal{CH}(Q)\cap \mathcal{D}^{\text{tr},i}$,
    \begin{align*}
        \langle \Delta_{i,Q}f\rangle_P = \frac{1}{\mu(P)}\left(\frac{\langle f\rangle_P}{\langle b_i\rangle_P}-\frac{\langle f\rangle_Q}{\langle b_i\rangle_Q}\right)b\,d\mu = \langle f\rangle_P- \frac{\langle f\rangle_Q}{\langle b_i\rangle_Q}\langle b_i\rangle_P.
    \end{align*}
    Multiplying by $\frac{1}{\langle b_i\rangle_P}$, we obtain our claim. For (d), first we check that $\Delta_{i,Q}\Xi_if = 0$. We can write
    $$
    \Delta_{i,Q}\Xi_i f = \Delta_{i,Q} \left(\frac{\langle f\rangle_{Q^0_{\mathcal{D}}}}{\langle b_i\rangle_{Q^0_{\mathcal{D}}}}b\right) = \frac{\langle f\rangle_{Q^0_{\mathcal{D}}}}{\langle b_i\rangle_{Q^0_{\mathcal{D}}}}\Delta_{i,Q}b_i.
    $$
    That is, we need to show that $\Delta_{i,Q}b_i=0$. This is verified by a very quick computation that we will omit. Now, using (b), we have that
    $$
    \Xi_i\Delta_{i,Q}f = \frac{\langle \Delta_{i,Q}f\rangle_{Q^0_{\mathcal{D}}}}{\langle b_i\rangle_{Q^0_{\mathcal{D}}}}b_i = 0,
    $$
    which finishes the proof of (d). Lastly, for the proof of (e), one needs to distinguish whether $Q\cap R=\varnothing$ or if they have intersection. Furthermore, one can fix which of $Q$ and $R$ has biggest side length and then use the fact that if two dyadic cubes have intersection and one has smaller side length than the other, the first is contained in only one of the children of the second.
\end{proof}

Now, let us show the main purpose of introducing such operators in $L^2(\mu)$. The following lemma, which will be key in the rest of the proof of Theorem \ref{TEOREMA}, shows that we can write any function $f\in L^2(\mu)$ as the sum of the operators $\Delta_{i,Q}$ and $\Xi_i$ applied to $f$.

\begin{lemma}
    \label{lema_descomposicio_L2} For $i=1,2$, for any $f\in L^2(\mu)$, we have the decomposition
    \begin{equation}
    f = \Xi_i f + \sum_{Q\in \mathcal{D}^{\text{tr},i}}\Delta_{i,Q}f,\label{descomposicio_L2}
    \end{equation}
    where the sum is unconditionally convergent in $L^2(\mu)$. Moreover, there exists some constant $c_3$ depending on $c_b$ and $c_{\text{acc}}$ such that
    \begin{equation}
    c_3^{-1}\|f\|^2_{L^2(\mu)}\leq \|\Xi_i f\|_{L^2(\mu)}^2 + \sum_{Q\in \mathcal{D}^{\text{tr},i}}\|\Delta_{i,Q}f\|^2_{L^2(\mu)}\leq c_3\|f\|^2_{L^2(\mu)}.\label{desigualtats_descomposicio_L2}
    \end{equation}
\end{lemma}

\begin{proof}
    We are going to follow four steps to prove the lemma. First, we will prove the second inequality in (\ref{desigualtats_descomposicio_L2}) and the first one for finite sums. Afterwards, we will prove (\ref{descomposicio_L2}) and lastly, the first inequality in (\ref{desigualtats_descomposicio_L2}) in full generality.

    \underline{\textbf{Step 1:} the second inequality in (\ref{desigualtats_descomposicio_L2}).} For $Q\in \mathcal{D}$, we define the operator $D_Q$ as
    $$
    D_Q f := \begin{cases}
        0 \qquad &\text{in }\,\R^d\setminus Q,\\
        \langle f\rangle_P- \langle f\rangle_Q &\text{in }\, P\in \mathcal{CH}(Q),
    \end{cases}
    $$
    and $Ef :=\langle f\rangle_{Q^0_{\mathcal{D}}}$. Then, because of orthogonality,
    \begin{equation}
        \|Ef\|^2_{L^2(\mu)} + \sum_{Q\in \mathcal{D}}\|D_Qf\|^2_{L^2(\mu)} = \|f\|^2_{L^2(\mu)}.\label{ortogonalitat}
    \end{equation}
    This sort of decomposition is widely used in the context of $Tb$ type theorems. This technique was initiated, in the context of the Cauchy transform, by Coifman, Jones and Semmes, see \cite{CoifmanR.R.1989Tepo}.
    
    We can separate the middle term in (\ref{desigualtats_descomposicio_L2}) as
    \begin{equation}
        \|\Xi_i f\|_{L^2(\mu)}^2 + \sum_{Q\in \mathcal{D}^{\text{tr},i}}\sum_{P\in \mathcal{CH}(Q)\cap \mathcal{D}^{\text{tr},i}}\|\Delta_{i,Q}f\|^2_{L^2(\mu\lfloor P)} +\sum_{Q\in \mathcal{D}^{\text{tr},i}}\sum_{P\in \mathcal{CH}(Q)\cap \mathcal{D}^{\text{term},i}}\|\Delta_{i,Q}f\|^2_{L^2(\mu\lfloor P)} \label{primer_pas}
    \end{equation}
    As noted in Lemma \ref{lema:propietats_xi}, the first term is bounded by a constant times $\|f\|^2_{L^2(\mu)}$. We now focus on the second term in (\ref{primer_pas}). If $P\in \mathcal{CH}(Q)\cap \mathcal{D}^{\text{tr},i}$, we have that
    $$
    \Delta_{i,Q}f|_P = \left(\frac{\langle f\rangle_P}{\langle b_i\rangle_P}-\frac{\langle f\rangle_Q}{\langle b_i\rangle_Q}\right)b_i = \left(\frac{1}{\langle b_i\rangle_P}-\frac{1}{\langle b_i\rangle_Q}\right)\langle f\rangle_Pb_i + \frac{1}{\langle b_i\rangle _Q}\underbrace{\left(\langle f\rangle_P-\langle f\rangle_Q\right)}_{D_Qf}b_i.
    $$
    By Remark \ref{obs:martingala_b_transit}, we have that $|\langle b_i\rangle_P\langle b_i\rangle_Q|\geq c_{\text{acc}}^{-2}$. Moreover, using that $b_i$ is bounded, we get that
    \begin{align*}
    (\Delta_{i,Q}f|_P)^2 &\leq 2\left(\frac{1}{\langle b_i\rangle_P}-\frac{1}{\langle b_i\rangle_Q}\right)^2\langle f\rangle_P^2b_i^2 + \frac{2}{\langle b_i\rangle^2_{Q}}|D_Qf|^2b_i^2\\
    &\leq 2c_b^2\left[\frac{(D_Qb_i|_P)^2}{|\langle b_i\rangle_P\langle b_i\rangle_Q|^2}\langle f\rangle^2_P+ \frac{|D_Qf|^2}{|\langle b_i\rangle_Q |^2}\right]\\
    &\leq c\left[\left|D_Qb_i|_P\right|^2\langle f \rangle_P^2 + |D_Qf|^2\right]
    \end{align*}
    Therefore, we may bound
    \begin{align*}
    \sum_{Q\in \mathcal{D}^{\text{tr},i}}\sum_{P\in \mathcal{CH}(Q)\cap \mathcal{D}^{\text{tr},i}}\|\Delta_{i,Q}f\|^2_{L^2(\mu\lfloor P)}&\leq c\sum_{Q\in \mathcal{D}}\sum_{P\in \mathcal{CH}(Q)}\|\chi_PD_Qb_i\|^2_{L^2(\mu)}\langle f\rangle_P^2\\
    &\quad + c\sum_{Q\in \mathcal{D}}\|D_Qf\|^2_{L^2(\mu)} =: (\text{I})  +(\text{II}).
    \end{align*}
    From (\ref{ortogonalitat}), we have that $(\text{II})\leq \|f \|^2_{L^2(\mu)}$. Now, we can rewrite $(\text{I})$ as 
    $$
    (\text{I}) = \sum_{Q\in \mathcal{D}^{\text{tr},i}}\|\chi_QD_{\widehat{Q}}b_i\|^2_{L^2(\mu)}\langle f\rangle^2_Q,
    $$
    where $\widehat{Q}$ is the parent of $Q$. To estimate this quantity, we will apply the dyadic Carleson embedding Theorem \ref{carleson_diadic}. We define $a_Q:=\|\chi_QD_{\widehat{Q}}b_i\|^2_{L^2(\mu)}$. Let us check that these numbers satisfy the hypothesis (\ref{condicio_carleson}) from Theorem \ref{carleson_diadic}. For any cube $R\in \mathcal{D}$, we can write
    $$
    \sum_{Q\subseteq R}\|\chi_QD_{\widehat{Q}}b_i\|^2_{L^2(\mu)} = \|D_{\widehat{R}}b_i\|^2_{L^2(\mu\lfloor R)} + \sum_{Q\subsetneq R}\|D_{\widehat{Q}}b_i\|^2_{L^2(\mu)}.
    $$
    Since $b_i$ is bounded, the first term can be bounded by $c\mu(R)$. For the second, we can apply the decomposition (\ref{ortogonalitat}) to the function $b_i\chi_R$,
    \begin{align*}
        \sum_{Q\subsetneq R}\|D_{\widehat{Q}b_i}\|^2_{L^2(\mu)} = \sum_{Q\subsetneq R}\|D_{\widehat{Q}}(b_i\chi_R)\|^2_{L^2(\mu)} \leq \sum_{Q\subseteq R}\|D_Q(b_i\chi_R)\|^2_{L^2(\mu)} \leq \|b_i\chi_R\|^2_{L^2(\mu)}\leq c_b^2\mu(R).
    \end{align*}
    This shows that the condition (\ref{condicio_carleson}) holds and thus we have that $(\text{I})\leq \|f\|^2_{L^2(\mu)}$.

    Now we will deal with the last term of (\ref{primer_pas}). For $Q\in \mathcal{D}^{\text{tr},i}$ and $P\in\mathcal{D}^{\text{term},i}$, we can write
    $$
    \Delta_{i,Q}f|_P = f-\frac{\langle f\rangle_Qb_i}{\langle b_i\rangle_Q}= \left(f-\frac{\langle f\rangle_Pb_i}{\langle b_i\rangle_Q}\right) + \left(\frac{\langle f\rangle_P}{\langle b_i\rangle_Q}-\frac{\langle f\rangle_Q}{\langle b_i\rangle_Q}\right)b_i.
    $$
    Using again Remark \ref{obs:martingala_b_transit} and the fact that $b_i$ is bounded, we get
    \begin{align*}
        |\Delta_{i,Q}f|_P|\leq c\left[\left(|f|+ \langle |f|\rangle_P\right)+ |\langle f\rangle_P- \langle f\rangle_Q|\right] = c\left[(|f|+\langle |f|\rangle_P)+|D_Qf|_P|\right].
    \end{align*}
    Hence,
    \begin{align*}
        \|\Delta_{i,Q}f\|^2_{L^2(\mu\lfloor P)} & \leq c\int_P \left((|f|+\langle |f|\rangle_P)^2+ |D_Qf|\right)^2\,d\mu\\
        &\leq 2c\int_{P}(|f|+\langle |f|\rangle_P)^2\,d\mu + 2c \int_P |D_Qf|^2\,d\mu\\
        &\leq 4c\int_P|f|^2\,d\mu + 4c\int_P\langle |f|\rangle_P^2\,d\mu + 2c \|D_Qf\|^2_{L^2(\mu)}\\
        & \leq c\|f\|^2_{L^2(\mu\lfloor P)}+ c \|D_Q f\|^2_{L^2(\mu)},
    \end{align*}
    where in the last inequality we have used Jensen's inequality. Therefore,
    \begin{align}
        &\sum_{Q\in \mathcal{D}^{\text{tr},i}}\sum_{Q\in \mathcal{CH}(Q)\cap \mathcal{D}^{\text{term},i}}\|\Delta_{i,Q}f\|^2_{L^2(\mu\lfloor P)}\nonumber \\
        &\qquad \qquad \leq c\sum_{Q\in \mathcal{D}^{\text{tr},i}}\sum_{P\in \mathcal{CH}(Q)\cap \mathcal{D}^{\text{term},i}}\|f\|^2_{L^2(\mu\lfloor P)} + c\sum_{Q\in \mathcal{D}}\|D_Qf\|^2_{L^2(\mu)}\label{ultim_pas_5.17}
    \end{align}
    The second term above can be bounded by $\|f\|^2_{L^2(\mu)}$, using (\ref{ortogonalitat}). Furthermore, note that the cubes $P\in \mathcal{D}^{\text{term},i}$ whose parent is transit are pairwise disjoint. Indeed, if $P_1,P_2\in \mathcal{D}^{\text{term},i}$, $\widehat{P}_1,\widehat{P}_2\in \mathcal{D}^{\text{tr},i}$ and $P_1\cup P_2\neq \varnothing$, we can assume that $P_2\subseteq P_1$. Since all the descendants of $P_1$ must be in $\mathcal{D}^{\text{term},i}$, in particular $\widehat{P}_2\in \mathcal{D}^{\text{term},i}$, which is a contradiction. This allows us to bound the first term in (\ref{ultim_pas_5.17}) by $c\|f\|^2_{L^2(\mu)}$, finishing the proof of the second inequality in (\ref{desigualtats_descomposicio_L2}).

    \underline{\textbf{Step 2:} the first inequality in (\ref{desigualtats_descomposicio_L2}) for finite sums.} Consider an arbitrary finite subset $\mathcal{F}\subseteq \mathcal{D}^{\text{tr},i}$ and, for the fixed function $f$ in the statement of the lemma, define a new function
    $$
    g = \Xi_i f + \sum_{Q\in \mathcal{F}}\Delta_{i,Q}f.
    $$
    From properties (d) and (e) of Lemma \ref{lema:propietats_delta}, we have that $\Xi_i g = \Xi_i f$, and
    $$
    \Delta_{i,Q}g = \begin{cases}
        \Delta_Q f, \qquad &\text{if }\, Q\in \mathcal{F},\\
        0 &\text{if }\,Q\not\in \mathcal{F}.
    \end{cases}
    $$
    Hence, the decomposition (\ref{descomposicio_L2}) is trivially true for $g$. From property (b) of Lemma \ref{lema:propietats_xi} and property (c) of Lemma \ref{lema:propietats_delta}, we have
    $$
    g = \Xi_i g + \sum_{Q\in \mathcal{D}^{\text{tr},i}}\Delta_{i,Q}g = \Xi^2_i g + \sum_{Q\in \mathcal{D}^{\text{tr},i}}\Delta^2_{i,Q}g.
    $$
    Using this decomposition, assuming that the functions that appear take real values,
    \begin{align*}
        \int|g|^2\,d\mu &= \int\left(\Xi^2_ig+\sum_{Q\in \mathcal{D}^{\text{tr},i}}\Delta^2_{i,Q}g\right)g\,d\mu\\
        &=\int\left((\Xi_ig)(\Xi_i^*g) + \sum_{Q\in \mathcal{D}^{\text{tr},i}} (\Delta_{i,Q}g)(\Delta_{i,Q}^*g)\right)\,d\mu\\
        &\leq \int\left(|\Xi_ig|^2+\left|\sum_{Q\in \mathcal{D}^{\text{tr},i}}\Delta_{i,Q}g\right|^2\right)^{\frac{1}{2}}\left(|\Xi^*_ig|^2+ \left|\sum_{Q\in \mathcal{D}^{\text{tr},i}}\Delta_{i,Q}^*g\right|^2\right)^{\frac{1}{2}}d\mu\\
        &\leq \left[\int\left(|\Xi_ig|^2+\left|\sum_{Q\in \mathcal{D}^{\text{tr},i}}\Delta_{i,Q}g\right|^2\right)d\mu\right]^{\frac{1}{2}}\left[\int\left(|\Xi^*_ig|^2+ \left|\sum_{Q\in \mathcal{D}^{\text{tr},i}}\Delta_{i,Q}^*g\right|^2\right)d\mu\right]^{\frac{1}{2}}\\
        &= \left[\|\Xi_ig\|^2_{L^2(\mu)}+ \left\|\sum_{Q\in \mathcal{D}^{\text{tr},i}}\Delta_{i,Q}g\right\|^2_{L^2(\mu)}\right]^{\frac{1}{2}}\left[\|\Xi^*_ig\|^2_{L^2(\mu)}+ \left\|\sum_{Q\in \mathcal{D}^{\text{tr},i}}\Delta^*_{i,Q}g\right\|^2_{L^2(\mu)}\right]^{\frac{1}{2}}\\
        &\leq \left[\|\Xi_ig\|^2_{L^2(\mu)}+ \sum_{Q\in \mathcal{D}^{\text{tr},i}}\|\Delta_{i,Q}g\|^2_{L^2(\mu)}\right]^{\frac{1}{2}}\left[\|\Xi^*_ig\|^2_{L^2(\mu)}+ \sum_{Q\in \mathcal{D}^{\text{tr},i}}\|\Delta^*_{i,Q}g\|^2_{L^2(\mu)}\right]^{\frac{1}{2}},
    \end{align*}
    where we have used the Cauchy-Schwarz inequality two times, first in $\R^2$ with the Euclidean norm and afterwards, in $L^2(\mu)$. From the last inequality, we see that we will complete the second step if we are able to show that
    $$
    \|\Xi_i^*g\|^2_{L^2(\mu)} + \sum_{Q\in \mathcal{D}^{\text{tr},i}}\|\Delta_{i,Q}^*g\|^2_{L^2(\mu)} \leq c\|g\|^2_{L^2(\mu)}.
    $$
    By the expression of $\Xi^*_i$ from Lemma \ref{lema:propietats_xi}, it is clear that
    $$
    \|\Xi_i^*g\|_{L^2(\mu)} \leq c\|g\|_{L^2(\mu)},
    $$
    which means that the only term left to estimate is $\sum_{Q\in \mathcal{D}^{\text{tr},i}}\|\Delta^*_{i,Q}g\|^2_{L^2(\mu)}$. As in the first step, we separate this sum as
    \begin{equation}
    \sum_{Q\in \mathcal{D}^{\text{tr},i}}\sum_{P\in \mathcal{CH}(Q)\cap \mathcal{D}^{\text{tr},i}}\|\Delta^*_{i,Q}g\|^2_{L^2(\mu\lfloor P)} +  \sum_{Q\in \mathcal{D}^{\text{tr},i}}\sum_{P\in \mathcal{CH}(Q)\cap \mathcal{D}^{\text{term},i}}\|\Delta^*_{i,Q}g\|^2_{L^2(\mu\lfloor P)}\label{transit_terminal_pas2}
    \end{equation}
    Our strategy will be once more to relate $\Delta^*_{i,Q}g$ to $D_Qg$. If $P\in \mathcal{CH}(Q)\cap \mathcal{D}^{\text{tr},i}$, using the expression for $\Delta^*_{i,Q}$ from Lemma \ref{lema:propietats_delta}, we have
    \begin{align*}
        \Delta^*_{i,Q}g|_P &= \frac{\langle gb_i\rangle_P-\langle gb_i\rangle_Q}{\langle b_i\rangle_Q} + \langle gb_i\rangle \left(\frac{1}{\langle b_i\rangle_P}-\frac{1}{\langle b_i\rangle_Q}\right)
        =\frac{D_Q(gb_i)|_P}{\langle b_i\rangle_Q}-\frac{\langle gb_i\rangle_{P}D_Qb_i|_P}{\langle b_i\rangle_P\langle b_i\rangle_Q}.
    \end{align*}
    Again, using Remark \ref{obs:martingala_b_transit}, $|\langle b_i\rangle_Q|,|\langle b_i\rangle_P|\geq c_{\text{acc}}^{-1}$, and so
    \begin{align}
        &\sum_{Q\in \mathcal{D}^{\text{tr},i}}\sum_{P\in \mathcal{CH}(Q)\cap \mathcal{D}^{\text{tr},i}}\|\Delta^*_{i,Q}g\|^2_{L^2(\mu)} \nonumber \\
        &\qquad \qquad \leq c\sum_{Q\in \mathcal{D}}\|D_Q(gb_i)\|^2_{L^2(\mu)} + c\sum_{Q\in \mathcal{D}}\sum_{P\in \mathcal{CH}(Q)}\|\langle gb_i\rangle_PD_Qb_i|_P\|^2_{L^2(\mu)}.\label{transit_segon_pas}
    \end{align}
    Again, from the decomposition (\ref{ortogonalitat}), we deduce
    $$
    \sum_{Q\in \mathcal{D}}\|D_Q(gb_i)\|^2_{L^2(\mu)} \leq \|gb_i\|^2_{L^2(\mu)} \leq c\|g\|^2_{L^2(\mu)}.
    $$
    Note that the last term in (\ref{transit_segon_pas}) can be rewritten as
    $$
    \sum_{Q\in \mathcal{D}}|\langle gb_i\rangle_Q|^2\|\chi_QD_{\widehat{Q}}b_i\|^2_{L^2(\mu)},
    $$
    and to estimate it we will apply again the dyadic Carleson embedding Theorem \ref{carleson_diadic}. If we call $a_Q:=\|\chi_QD_{\widehat{Q}}b_i\|^2_{L^2(\mu)}$, from the estimate in the previous step we have that the packing condition (\ref{condicio_carleson}) is satisfied and thus we obtain
    $$
    \sum_{Q\in \mathcal{D}}\sum_{P\in \mathcal{CH}(Q)}\|\langle gb_i\rangle_P D_{Q}b_i|_P\|^2_{L^2(\mu\lfloor P)}\leq c\|gb_i\|^2_{L^2(\mu)} \leq c \|g\|^2_{L^2(\mu)}.
    $$
    Lastly, we need to deal with the terminal cubes in (\ref{transit_terminal_pas2}). For $Q\in \mathcal{D}^{\text{tr},i}$ and $P\in \mathcal{CH}(Q)\cap \mathcal{D}^{\text{term},i}$, we have
    $$
    \Delta^*_{i,Q}g|_P = g-\frac{\langle gb_i\rangle_Q}{\langle b_i\rangle_Q}=  \left(g-\frac{\langle gb_i\rangle_P}{\langle b_i\rangle_Q}\right) + \left(\frac{\langle gb_i\rangle_P}{\langle b_i\rangle_Q}-\frac{\langle gb_i\rangle_Q}{\langle b_i\rangle_Q}\right).
    $$
    As usual, we bound
    $$
    |\Delta^*_{i,Q}g|_P| \leq c\left[\left(|g|+\langle |g|\rangle_P\right)+\left|\langle gb_i\rangle_P-\langle gb_i\rangle_Q\right|\right] = c\left[\left(|g|+\langle |g|\rangle_P\right)+ \left|D_Q(gb_i)\right|\right]
    $$
    This yields that
    \begin{align*}
        &\sum_{Q\in \mathcal{D}^{\text{tr},i}}\sum_{P\in \mathcal{CH}(Q)\cap \mathcal{D}^{\text{term},i}}\|\Delta^*_{i,Q}g\|^2_{L^2(\mu\lfloor P)}\\
        &\qquad\qquad \leq c\sum_{Q\in \mathcal{D}^{\text{tr},i}}\sum_{Q\in \mathcal{CH}(Q)\cap \mathcal{D}^{\text{term},i}}\|g\|^2_{L^2(\mu\lfloor P)} + c\sum_{Q\in \mathcal{D}}\|D_Q(gb_i)\|^2_{L^2(\mu)}. 
    \end{align*}
    Arguing by the same methods as in (\ref{ultim_pas_5.17}), we get that
    $$
    \sum_{Q\in \mathcal{D}^{\text{tr},i}}\sum_{P\in \mathcal{CH}(Q)\cap \mathcal{D}^{\text{term},i}}\|\Delta^*_{i,Q}g\|^2_{L^2(\mu\lfloor P)}\leq c\|g\|^2_{L^2(\mu)},
    $$
    which finishes the proof of the second inequality in (\ref{desigualtats_descomposicio_L2}) for finite sums.

    \underline{\textbf{Step 3:} proof of (\ref{descomposicio_L2}).} We start by showing the unconditional convergence of the series that appears in (\ref{descomposicio_L2}). For this, let $\mathcal{F}\subset \mathcal{D}^{\text{tr},i}$ be a finite subset, and call $g = \sum_{Q\in \mathcal{F}}\Delta_{i,Q}f$. Then, by the previous step,
    $$
    \bigg\|\sum_{Q\in \mathcal{F}}\Delta_{i,Q}f\bigg\|^2_{L^2(\mu)} = \|g\|^2_{L^2(\mu)} \leq c\sum_{Q\in \mathcal
    D^{\text{tr},i}}\|\Delta_{i,Q}f\|^2_{L^2(\mu)}  = c\sum_{Q\in \mathcal{F}} \|\Delta_{i,Q}f\|^2_{L^2(\mu)}.
    $$
    From this, recalling that by the first step,
    $$
    \sum_{Q\in \mathcal{D}^{\text{tr},i}} \|\Delta_{i,Q}f\|^2_{L^2(\mu)} \leq c\,\|f\|^2_{L^2(\mu)}<\infty,
    $$
    we deduce the unconditional convergence in $L^2(\mu)$ of the series $\sum_{Q\in \mathcal{D}^{\text{tr},i}}\Delta_{i,Q}f$.

    Now, since convergence in $L^2(\mu)$ implies $\mu$-a.e. convergence of a subsequence, to prove (\ref{descomposicio_L2}), it suffices to check that
    $$
    \Xi_if(x) + \lim_{n\to+\infty}\sum_{Q\in \mathcal{D}^{\text{tr},i}:\ell(Q)\geq 2^{-n}}\Delta_{i,Q}f(x) = f(x), \quad \text{for }\,\mu\text{-a.e. }\,x\in F.
    $$
    To do so, we distinguish two cases. Assume first that $x\in Q_0\in \mathcal{D}^{\text{term},i}$, with $\ell(Q_0)=2^{n_0}, n_0\in \mathbb{Z}$ and that $Q_0$ is the biggest terminal cube (of the $i$-th kind) that contains $x$. Such a cube $Q_0$ always exists, because $x\in Q^0_{\mathcal{D}}$, which is not terminal of any kind. If $x\in Q$ and $\ell(Q)<\ell(Q_0)$, $Q\subseteq Q_0$, so $Q\in \mathcal{D}^{\text{term},i}$, so we have that
    $$
    \Xi_if(x) + \lim_{n\to+\infty}\sum_{Q\in \mathcal{D}^{\text{tr},i}:\ell(Q)\geq 2^{-n}}\Delta_{i,Q}f(x)  = \Xi_if(x) + \lim_{n\to+\infty}\sum_{Q\in \mathcal{D}^{\text{tr},i}:\ell(Q)> 2^{n_0}}\Delta_{i,Q}f(x).
    $$
    Call $Q_1$ the cube in $\mathcal{D}^{\text{tr},i}$ with minimal side length that contains $Q$. This means that the unique child of $Q_1$ that contains $x$ is in $\mathcal{D}^{\text{term},i}$, so
    $$
    \Delta_{i,Q_1}f(x) = f(x) - \frac{\langle f\rangle_{Q_1}}{\langle b_i\rangle_{Q_1}}b_i(x). 
    $$
    If we denote by $Q_2$ the parent of $Q_1$, since $Q_2\in \mathcal{D}^{\text{tr},i}$,
    $$
    \Delta_{i,Q_2}f(x) = \left(\frac{\langle f\rangle_{Q_1}}{\langle b_i\rangle_{Q_1}}-\frac{\langle f\rangle_{Q_2}}{\langle b_i\rangle_{Q_2}}\right)b_i(x).
    $$
    Hence,
    $$
    \Delta_{i,Q_1}f(x) + \Delta_{i,Q_2}f(x) = f(x)-\frac{\langle f\rangle_{Q_2}}{\langle b_i\rangle_{Q_2}}b_i(x).
    $$
    We iterate this process until we reach $Q^{0}_{\mathcal{D}}$,
    $$
    \Delta_{i,Q_1}f(x)+ \Delta_{i,Q_2}f(x) + \cdots + \Delta_{i,Q^0_{\mathcal{D}}}f(x) = f(x) - \frac{\langle f\rangle_{Q^0_{\mathcal{D}}}}{\langle b_i\rangle_{Q^0_{\mathcal{D}}}}b_i(x) = f(x)-\Xi_if(x).
    $$
    Lastly, since we always assume transit cubes to be contained in $Q^0_{\mathcal{D}}$, the biggest cube that appears in the limit we have to compute is $Q^0_{\mathcal{D}}$, and so the limit is true for such $x$.

    Assume now that $x$ does not belong to any cube in $\mathcal{D}^{\text{term},i}$. Fix any $n\in \mathbb{Z}$. If $n$ is not too negative, there exists only one cube $Q\ni x$ with $\ell(Q)=2^{-n}$ and $Q\in \mathcal{D}^{\text{tr},i}$. If $P$ be the only child of $Q$ that contains $x$ and $Q^1$ is the parent of $Q$,
    $$
    \Delta_{i,Q}f(x) = \left(\frac{\langle f\rangle_P}{\langle b_i\rangle_P}-\frac{\langle f\rangle_Q}{\langle b_i\rangle_Q}\right)b_i(x), \quad \Delta_{i,Q^1}f(x) = \left(\frac{\langle f\rangle_Q}{\langle b_i\rangle_Q}-\frac{\langle f\rangle_{Q^1}}{\langle b_i\rangle_{Q^1}}\right)b_i(x).
    $$
    Thus,
    $$
    \Delta_{i,Q}f(x) + \Delta_{i,Q^1}f(x) = \frac{\langle f\rangle_P}{\langle b_i\rangle_P}b_i(x) - \frac{\langle f\rangle_{Q^1}}{\langle b_i\rangle_{Q^1}}b_i(x).
    $$
    We can iterate this process, similarly to the previous case, and we get that
    $$
    \Delta_{i,Q}f(x) + \cdots + \Delta_{i,Q^0_{\mathcal{D}}}f(x) = \frac{\langle f\rangle_P}{\langle b_i\rangle_P}b_i(x) - \frac{\langle f\rangle_{Q^0_{\mathcal{D}}}}{\langle b_i\rangle_{Q^0_{\mathcal{D}}}}b_i(x) = \frac{\langle f\rangle_P}{\langle b_i\rangle_P}b_i(x)-\Xi_if(x).
    $$
    We now write
    $$
    \frac{\langle f\rangle_P}{\langle b_i\rangle_P}b_i(x) = \frac{1}{\mu(P)}\int_Pf(y)\,d\mu(y)\left(\frac{1}{\mu(P)}\int_P b_i(y)\,d\mu(y)\right)^{-1}b_i(x).
    $$
    By the dyadic Lebesgue differentiation Theorem \ref{diferenciacio_diadic}, we find that
    $$
    \lim_{n\to+\infty}\frac{\langle f\rangle_P}{\langle b_i\rangle_P}b_i(x) = f(x), \quad \text{for }\,\mu\text{-a.e. }\, x\in F, 
    $$
    which is what we needed to prove.

    \underline{\textbf{Step 4:} proof of the first inequality in (\ref{desigualtats_descomposicio_L2}).} In the second step, the advantage of considering only finite sums was that the decomposition (\ref{descomposicio_L2}) was true. Now we know that it holds for any $f\in L^2(\mu)$, so following exactly the same steps as before, we obtain the first inequality in (\ref{desigualtats_descomposicio_L2}) for any $f\in L^2(\mu)$.
\end{proof}

\begin{obs}
    The previous lemma also holds by replacing $\Xi_i$ and $\Delta_{i,Q}$ by $\Xi_i^*$ and $\Delta^*_{i,Q}$, respectively. That is, for $i=1,2$ and $f\in L^2(\mu)$, we also have the inequalities
    \begin{equation}
        c^{-1}\|f\|^2_{L^2(\mu)} \leq \|\Xi_i^*f\|^2_{L^2(\mu)} + \sum_{Q\in \mathcal{D}^{\text{tr},i}}\|\Delta^*_{i,Q}f\|^2_{L^2(\mu)} \leq c\|f\|^2_{L^2(\mu)}.\label{desigualtats_descomposicio_transposat}
    \end{equation}
\end{obs}

\subsection{Good and bad cubes and functions}
\label{sec:good_functions}

In the preceding section we have seen that the transit cubes, which enjoy some desirable properties with respect to the measure $\mu$, can be used to decompose any function $f\in L^2(\mu)$. The next step to study $L^2$ boundedness of our operator $T$ will be to use this decomposition when working with the suppressed operators $K_\Theta$. However, the distinction between transit and terminal cubes will not be enough, and we will need to further classify the transit cubes.

More precisely, we are now going to define what it means for a dyadic cube $Q\in \mathcal{D}^{\text{tr},i}(w_1)$ to be bad (or good) with respect to another dyadic lattice $\mathcal{D}(w_2)$, with $w_1,w_2\in \Omega$. The purpose of doing so is that, using the decomposition in (\ref{descomposicio_L2}), we will be able to write $f\in L^2(\mu)$ as
\begin{align}
f &= f_{\text{good}}(w_1,w_2) + f_{\text{bad}}(w_1,w_2)\nonumber \\
&= \Xi_i f + \sum_{Q\in \mathcal{D}^{\text{tr},i}(w_1)\cap \text{good}(w_2)}\Delta_{i,Q} f + \sum_{Q\in \mathcal{D}^{\text{tr},i}(w_1)\cap \text{bad}(w_2)}\Delta_{i,Q} f.\label{descomposicio_bo/dolent}
\end{align}
Then, denoting by $P^{\Omega^2}=P^{\Omega}\times P^{\Omega}$ the product probability measure in $\Omega\times \Omega$, and $\text{w}=(w_1,w_2)$, for an appropriately chosen $1$-Lipschitz function $\Psi_{\text{w}}$, we will define the operator
$$
\widetilde{K}f(x) = \int_{\Omega\times \Omega}K_{\Psi_{\text{w}}}f(x)dP^{\Omega^2}(\text{w}).
$$
We will later see that we can prove the $L^2(\mu\lfloor G)$-boundedness of our singular integral operator $T$ by proving that $\widetilde{K}$ is bounded in $L^2(\mu)$. To do so, if we denote by $\mathbb{E}_{P^{\Omega^2}}$ the expectation with respect to $P^{\Omega^2}$, and for $f,g\in L^2(\mu)$,
$$
\langle f,g\rangle =\int fg d\mu,
$$
we will see that we need the bound 
\begin{equation}
\left|\mathbb{E}_{P^{\Omega^2}}\left(\langle K_{\Psi_{\text{w}}}f,g\rangle\right)\right| = |\langle \widetilde{K}f, g\rangle| \leq c\|f\|^2_{L^2(\mu)}\|g\|^{2}_{L^2(\mu)}.\label{desigualtat_esperança}
\end{equation}
Using the decomposition (\ref{descomposicio_bo/dolent}), we can write the expectation as
\begin{align}
\mathbb{E}_{P^{\Omega^2}}\left(\langle K_{\Psi_{\text{w}}}f,g\rangle\right) &= \mathbb{E}_{P^{\Omega^2}}(\langle K_{\Psi_{\text{w}}}f_{\text{good}}, g_{\text{good}} \rangle) + \mathbb{E}_{P^{\Omega^2}}(\langle K_{\Psi_{\text{w}}}f_{\text{good}}, g_{\text{bad}}\rangle)\nonumber\\
&\qquad + \mathbb{E}_{P^{\Omega^2}}(\langle K_{\Psi_\text{w}}f_{\text{bad}},g \rangle).\label{esperances_good/bad}
\end{align}
To obtain the bound in (\ref{desigualtat_esperança}), first we are going to show that for the good part we almost have a bound of the type
$$
|\langle K_{\Psi_{\text{w}}}f_{\text{good}}, g_{\text{good}}\rangle| \leq c\|f\|_{L^2(\mu)}\|g\|_{L^2(\mu)},
$$
where \textit{almost} here means that on the right hand side we will get an extra term, which will be dealt with in the final step of the proof of Theorem \ref{TEOREMA}, using a probabilistic argument.

For the other two terms in (\ref{esperances_good/bad}), which contain the bad part, we are going to choose an appropriate definition of bad cubes such that for each fixed $Q\in \mathcal{D}^{\text{tr},i}(w_1)$, the probability that it is bad with respect to the other dyadic lattice $\mathcal{D}(w_2)$ can be made arbitrarily small. Using this, we will be able to make the two last terms in (\ref{esperances_good/bad}) also small, thus proving the $L^2(\mu\lfloor G)$ boundedness of $\widetilde{K}$. 

To define what it means for a cube to be bad, we need the following notion. If $Q\subset \R^d$, we say that $\partial Q$ is $(\mu,M)$-small if 
$$
\mu\left(\left\{x\in \R^d: \text{dist}(x, \partial Q)\leq \lambda \ell(Q)\right\}\right)\leq\lambda M \,\|\mu\|,
$$
for all $\lambda >0$. If $P\subset \R^d$ is another cube, we say that $\partial Q$ is $(\mu,M,P)$-small if $\partial Q$ is $(\mu\lfloor P, M)$-small.

Let $\mathcal{D}_1=\mathcal{D}(w_1)$ and $\mathcal{D}_2=\mathcal{D}(w_2)$, with $w_1,w_2\in \Omega$, be two dyadic lattices and put $\gamma = \frac{\eta}{2(n+\eta)}$. For $i\in \{1,2\}$, we say that a transit cube (of the $i$-th kind) $Q\in \mathcal{D}^{\text{tr},i}_1$ is $i$-\textbf{bad} (with respect to $\mathcal{D}_2$) if either
\begin{enumerate}[label=(\alph*)]
    \item there exists a cube $R\in \mathcal{D}_2$ such that $\text{dist}(Q,\partial R)\leq \ell(Q)^{\gamma}\ell(R)^{1-\gamma}$ and $2^m\ell(Q)\leq \ell(R)\leq 2^N$ (where $m\geq 1$ is some integer that we will fix below), or
    \item there exists a transit cube $R\in \mathcal{D}^{\text{tr}{\color{red}}}_2$ such that $2^{-m}\ell(Q)\leq \ell(R)\leq 2^m \ell(Q)$, $\text{dist}(Q,R)\leq 2^m \ell(Q)$ such that for some $P\in \mathcal{CH}(Q)$, there is $S\in \mathcal{CH}(R)$ such that $\partial P$ is not $(\mu, M,S)$-small.
\end{enumerate}

If $Q$ is not $i$-bad, then we say that it is $i$-\textbf{good}. It is important to remark that we say that a cube $Q\in \mathcal{D}_1$ is good or bad depending on another lattice $\mathcal{D}_2$. This dependence on another dyadic lattice is new; recall that the definition of transit and terminal cubes did not contain any interaction between two different lattices. Moreover, the definition of bad cubes also depends on the constants $m$ and $M$, but for commodity we will not talk about $(m,M)$-bad cubes, only bad cubes.

With this notion, we can define precisely what it means for a function to be good. Recall that given any fixed dyadic lattice $\mathcal{D}_1=\mathcal{D}(w_1)$, by Lemma \ref{lema_descomposicio_L2}, we can write any $f\in L^2(\mu)$ as
$$
f = \Xi_i f + \sum_{Q\in \mathcal{D}^{\text{tr},i}_1}\Delta_{i,Q}f.
$$
We say that $f$ is $i$-$\mathcal{D}_1$-\textbf{good with respect to} $\mathcal{D}_2$ (or simply, $i$-good) if $\Delta_{i,Q} f =0$ for all $i$-bad cubes $Q\in \mathcal{D}^{\text{tr},i}_1$ (with respect to $\mathcal{D}_2$). 

\section{Estimates for good functions}

\label{chap:good_functions}

\subsection{The main lemma for good functions}
\label{subsec:main_lemma_good_functions}

Recall that in Section \ref{sec:suppressed_operators}, we considered $\Theta\,\colon \R^d \to [0,+\infty)$ and we defined a new Calderón-Zygmund kernel $\widetilde{k}_\Theta$, which had some properties that the original kernel $k$ did not necessarily satisfy (see Lemma \ref{1/2_max} and Lemma \ref{desigualtat_maxim_k_tilde}). Our goal now will be to estimate the action of the SIO $K_\Theta$, which arises from the kernel $\widetilde{k}_{\Theta}$, on good functions, which we defined in the last section of the preceding section. In this direction, we will have to restrict the function $\Theta$ a little more, in the sense that we will need it to be sensible to the sets that are singled out in the statement of Theorem \ref{TEOREMA} and in Section \ref{sec:exceptional_set}. We give the details of this connection between the function and said sets below.

For each dyadic lattice $\mathcal{D}$, we define the \textbf{total exceptional set} as 
$$
W_{\mathcal{D}}=H_{\mathcal{D}}\cup T_{\mathcal{D}}^1\cup T_\mathcal{D}^2\cup S.
$$
Moreover, we define the function
$$
\Phi_\mathcal{D}(x) = \text{dist}(x,\R^d\setminus W_{\mathcal{D}}).
$$
This is a $1$-Lipschitz function that vanishes on $\R^d\setminus W_{\mathcal{D}}$. Moreover, $\Phi_{\mathcal{D}}(x)\geq \text{dist}(x,\R^d\setminus(H_{\mathcal{D}}\cup S))$, so that from Lemma \ref{lema_distancia}, we obtain that
$$
K_{\Phi_{\mathcal{D}},*}\nu_1(x) \leq c_4, \quad  K_{\Phi_{\mathcal{D}},*}^*\nu_2(x)\leq c_4, \quad \text{for all }\,x\in F.
$$
In the next lemma, we show that an inequality of the type in (\ref{desigualtat_esperança}) almost holds for good functions $f,g\in L^2(\mu)$ and an appropriate $1$-Lipschitz function $\Theta$.

\begin{lemma}
    Let $\mathcal{D}_1=\mathcal{D}(w_1)$ and $\mathcal{D}_2=\mathcal{D}(w_2)$, with $w_1,w_2\in \Omega$, be two dyadic lattices. Given $\varepsilon>0$, let $\Theta\colon \R^d\to [\varepsilon,\infty)$ be a $1$-Lipschitz function such that $\Theta(x)\geq \max(\Phi_{\mathcal{D}_1}(x),\Phi_{\mathcal{D}_2}(x))$ for all $x\in \R^d$. If $f$ is $1$-$\mathcal{D}_1$-good with respect to $\mathcal{D}_2$ and $g$ is $2$-$\mathcal{D}_2$-good with respect to $\mathcal{D}_1$, then
    $$
    |\langle K_{\Theta}f,g\rangle|\leq c_5\|f\|_{L^2(\mu)}\|g\|_{L^2(\mu)}+ \sum_{\substack{Q\in \mathcal{D}^{\text{tr},1}_1, \,R\in \mathcal{D}^{\text{tr},2}_2\\ \text{dist}(Q,R)<\text{max}(\ell(Q),\ell(R))\\ 2^{-m}\ell(R)\leq \ell(Q)\leq 2^m\ell(R)}}|\langle K_{\Theta}(\Delta_{1,Q}f), \Delta_{2,R}g\rangle|,
    $$
    where $c_5$ is a constant that depends on $c_0,c_b,c_*$ and $\delta_0$, but not on $\varepsilon$.\label{lemma:good_functions}
\end{lemma}

\begin{obs}
    Note that the assumption that the function $\Theta$ is bounded below by some positive constant $\varepsilon>0$ is only for technical reasons, as it ensures that the operator $K_{\Theta}$ is bounded in $L^p(\mu)$ for all $1\leq p\leq \infty$ (with norm depending on $\varepsilon$). Indeed, if $\Theta(x)\geq \varepsilon$ for all $x\in \R^d$, then we have that
    $$
    \frac{|x-y|^2}{\Theta(x)\Theta(y)}\leq \frac{|x-y|^2}{\varepsilon^2},
    $$
    from which we see that if $|x-y|\leq \frac{\varepsilon}{\sqrt{2}}$, $\widetilde{k}_{\Theta}(x,y)=0$. Hence, 
    \begin{align*}
        \left|K_{\Theta}f(x)\right| &= \left|\int_{|x-y|>\frac{\varepsilon}{\sqrt{2}}}\widetilde{k}_{\Theta}(x,y)f(y)\,d\mu(y)\right|\\
        &\leq \left(\int_{|x-y|>\frac{\varepsilon}{\sqrt{2}}}\left|\widetilde{k}_{\Theta}(x,y)\right|^{p'}\,d\mu(y)\right)^{\frac{1}{p'}}\left(\int_{|x-y|>\frac{\varepsilon}{\sqrt{2}}}|f(y)|^p\,d\mu(y)\right)^{\frac{1}{p}}\\
        &\leq \frac{C_{CZ}2^{\frac{n}{2}}}{\varepsilon^n}\mu(\R^d)^\frac{1}{p'}\|f\|_{L^p(\mu)}.
    \end{align*}
    In particular, $K_{\Theta}f$ is well defined for any $f\in L^2(\mu)$. The main point in the lemma is that the estimate does not depend on $\varepsilon$, unlike above.
\end{obs}

\subsubsection{Beginning of the proof of Lemma \ref{lemma:good_functions}}

In the next few pages, we are going to introduce the strategy for the proof of Lemma \ref{lemma:good_functions}. We are going to use the decomposition in $L^2(\mu)$, given by Lemma \ref{lema_descomposicio_L2}, in terms of the operators $\Delta_{1,Q}$ and $\Delta_{2,R}$ for $Q\in\mathcal{D}^{\text{tr},1}_1$ and $R\in \mathcal{D}^{\text{tr},2}_2$, to show that the bound for $\langle K_{\Theta}f,g\rangle$ reduces to studying a sum of terms of the form $\langle K_{\Theta}(\Delta_{1,Q}f), \Delta_{2,R}g\rangle$. Once we have shown this, we are going to separate the aforementioned sum according to the relation between the cubes $Q$ and $R$. The bounds for each of the terms arising from this separation will be obtained through different arguments for each one and for the most part will be highly technical.

Let $f,g \in L^2(\mu)$ be as in Lemma \ref{lemma:good_functions}; $f$ is $1$-$\mathcal{D}_1$-good with respect to $\mathcal{D}_2$ and $g$ is $2$-$\mathcal{D}_2$ good with respect to $\mathcal{D}_1$. Then, we can write them as
\begin{equation}
    \label{descomposicio_f_g} f = \Xi_1f + \sum_{\substack{Q\in \mathcal{D}_1^{\text{tr},1}\\ 1\text{-}\mathcal{D}_1\text{-good w.r.t. }\mathcal{D}_2}}\Delta_{1,Q}f, \qquad g = \Xi_2g + \sum_{\substack{R\in \mathcal{D}_2^{\text{tr},2}\\ 2\text{-}\mathcal{D}_2\text{-good w.r.t. }\mathcal{D}_1}}\Delta_{2,R}g
\end{equation}
Since we also assume that $\Theta(x)\geq \varepsilon>0$ for all $x\in \R^d$, by our previous remark we have that $\|K_{\Theta}\|_{L^2(\mu)\to L^2(\mu)}<\infty$ (depending on $\varepsilon$), and thus we can also apply Lemma \ref{lema_descomposicio_L2} to $K_{\Theta}f$,
\begin{equation}
    \label{descomposicio_k_theta} K_{\Theta}f = K_{\Theta}(\Xi_1f) + \sum_{\substack{Q\in \mathcal{D}_1^{\text{tr},1}\\ 1\text{-}\mathcal{D}_1\text{-good w.r.t. }\mathcal{D}_2}}K_{\Theta}(\Delta_{1,Q}f),
\end{equation}
with the sum being unconditionally convergent in $L^2(\mu)$, and the analogous identity also holds for $g$.

To prove the lemma, we will assume that $\Delta_{1,Q}f=0$ except for a finite family of cubes $Q\in \mathcal{D}^{\text{tr},1}_1$ and $\Delta_{2,R}g=0$ except for another finite family of cubes $R\in \mathcal{D}^{\text{tr},2}_2$. To our advantage, the estimates that we will obtain will not depend on the number of non-zero terms. This assumption will help us because there will be no problems with convergence and will allow us to change the order of summation in the sums involving the cubes $Q,R$. The lemma in the general case will then follow by taking $L^2(\mu)$ limits, using the fact that $\|K_{\Theta}\|_{L^2(\mu)\to L^2(\mu)}<\infty$ and the identity (\ref{descomposicio_k_theta}).

We will say that $f$ is \textbf{very} $1$-$\mathcal{D}_1$\textbf{-good with respect to} $\mathcal{D}_2$ (or just $1$-very good or even very good) if it is $1$-$\mathcal{D}_1$-good with respect to $\mathcal{D}_2$ and moreover $\Delta_{1,Q}f=0$ except for a finite family of cubes $Q\in \mathcal{D}^{\text{tr},1}_1$. We also consider the analogous notation interchanging $\mathcal{D}_1$ and $\mathcal{D}_2$.

To simplify notation, from now on we are going to write $Q^0=Q^0_{\mathcal{D}_1}$ and $R^0=Q^0_{\mathcal{D}_2}$ (as defined in (\ref{definicio_Q0})). Recall that our goal now is to bound $\langle K_{\Theta}f,g\rangle$. To do so, using the decompositions (\ref{descomposicio_f_g}), we can write
\begin{align}
    \langle K_{\Theta}f, g\rangle &= \langle K_{\Theta}(\Xi_1f + \sum_{Q\in \mathcal{D}^{\text{tr},1}_1}\Delta_{1,Q}f), \Xi_2 g + \sum_{R\in \mathcal{D}^{\text{tr},2}}\Delta_{2,R}g\rangle\nonumber\\
    & = \langle K_\Theta\left(\Xi_1 f\right), g\rangle + \langle K_{\Theta}(\sum_{Q\in \mathcal{D}^{\text{tr},1}_1}\Delta_{1,Q}f), \Xi_2g \rangle  + \langle K_{\Theta}(\sum_{Q\in \mathcal{D}^{\text{tr},1}_1}\Delta_{1,Q}f), \sum_{R\in \mathcal{D}^{\text{tr},2}_2}\Delta_{2,R}g\rangle\nonumber \\
    & = \langle K_{\Theta}(\Xi_1f), g\rangle + \langle K_{\Theta}f, \Xi_2g\rangle -\langle K_{\Theta}(\Xi_1f), \Xi_2g\rangle\nonumber \\[0.5em]
    &\qquad + \langle K_{\Theta}(\sum_{Q\in \mathcal{D}^{\text{tr},1}_1}\Delta_{1,Q}f), \sum_{R\in \mathcal{D}^{\text{tr},2}_2}\Delta_{2,R}g\rangle.\label{cosa_xunga}
\end{align}

We now show that the first two terms in (\ref{cosa_xunga}) can be estimated easily. Indeed, by Lemma \ref{lema_distancia}, $K_{\Theta,*}\nu_1(x)\leq c_4$ and $K_{\Theta, *}^*\nu_2(x)\leq c_4$, for all $x\in F$, with $c_4$ depending on $c_0,c_b, c_*$ and $\delta_0$, so we have
$$
\left|K_{\Theta}b_1(x)\right| = \left| K_{\Theta}\nu_1(x)\right|\leq c_4, \quad \left|K_{\Theta}^*b_2(x)\right| = \left| K_{\Theta}^*\nu_2(x)\right|\leq c_4, \quad x\in F.
$$
Thus, since both $Q^0$ and $R^0$ are transit cubes, we can use Remark \ref{obs:martingala_b_transit} to bound $|\langle b_1\rangle_{Q^0}|$ and $|\langle b_2\rangle_{R^0}|$ from below, and by the previous bounds,
\begin{align*}
    \|K_{\Theta}(\Xi_1f)\|_{L^2(\mu)} &= \frac{\left|\langle f\rangle_{Q^0}\right|}{\left|\langle b_1\rangle_{Q^0}\right|}\|K_{\Theta}b_1\|_{L^2(\mu)} \leq c\left|\langle f\rangle_{Q^0}\right|\mu(Q^0)^\frac{1}{2}\leq c \|f\|_{L^2(\mu)},\\
    \|K_{\Theta}(\Xi_2g)\|_{L^2(\mu)} &= \frac{\left|\langle f\rangle_{R^0}\right|}{\left|\langle b_2\rangle_{R^0}\right|}\|K_{\Theta}b_2\|_{L^2(\mu)} \leq c\left|\langle g\rangle_{R^0}\right|\mu(R^0)^\frac{1}{2}\leq c \|g\|_{L^2(\mu)}.
\end{align*}
Therefore, on the one hand we have that
$$
\left|\langle K_{\Theta}(\Xi_1f), g\rangle \right|\leq c\|f\|_{L^2(\mu)}\|g\|_{L^2(\mu)}.
$$
On the other hand, by duality,
$$
\left|\langle  K_{\Theta}f, \Xi_2g\rangle \right| = \left|\langle f, K_{\Theta}^*(\Xi_2g) \rangle  \right| \leq c \|f \|_{L^2(\mu)}\|g\|_{L^2(\mu)}.
$$
Now, for the third term in (\ref{cosa_xunga}), we can use the same estimate for $f$ and the fact that $\Xi_2$ is bounded in $L^2(\mu)$, and we obtain that
$$
\left|\langle K_{\Theta}(\Xi_1f), \Xi_2g\rangle \right| \leq c \| f\|_{L^2(\mu)}\|g\|_{L^2(\mu)}.
$$
Hence, to prove the lemma it is enough to estimate the quantity
\begin{equation}
    \label{el_que_hem_dacotar} \sum_{Q\in \mathcal{D}^{\text{tr},1}_1, R\in \mathcal{D}^{\text{tr},2}_2}\langle K_{\Theta}(\Delta_{1,Q}f), \Delta_{2,R}g\rangle.
\end{equation}
We are going to use different approaches to bound this last sum depending on the relative positions and sizes of the cubes that are involved. This is why we introduce the following notion. We say that two cubes $Q\in \mathcal{D}_1, R\in \mathcal{D}_2$ are \textbf{distant} if 
$$
\text{dist}(Q,R)\geq \min(\ell(Q), \ell(R))^{\gamma}\max(\ell(Q), \ell(R))^{1-\gamma}.
$$
Recall that $\gamma = \frac{\eta}{2(n+\eta)}$, from the definition of bad cubes.

\begin{obs}
    \label{remark:not_distant_good} If $Q\in \mathcal{D}_1^{\text{tr},1}$ and $R\in \mathcal{D}_2^{\text{tr},2}$ are disjoint cubes which are not distant and both are good (each with respect to the lattice to which the other cube belongs), it turns out that
    $$
    2^{-m}\ell(R)<\ell(Q)<2^m\ell(R),
    $$
    where $m$ is the integer from the definition of bad cubes. Indeed, assume first that $\ell(Q)\leq \ell(R)$. Then, since $Q$ and $R$ are disjoint and not distant, 
    $$
    \text{dist}(Q,R) = \text{dist}(Q, \partial R)\leq \ell(Q)^{\gamma}\ell(R)^{1-\gamma}.
    $$
    Moreover, since $Q$ is good, there does not exist any $\widetilde{R}\in \mathcal{D}_2$ such that
    $$
    \text{dist}(Q, \partial \widetilde{R}) \leq \ell(Q)^{\gamma}\ell(\widetilde{R})^{1-\gamma}\quad \text{and}\quad \ell(\widetilde{R})\geq 2^m\ell(Q).
    $$
    Hence, since $R$ satisfies the first condition, it must be that $\ell(R)<2^m\ell(Q)$, that is, $2^{-m}\ell(R)<\ell(Q)$, which is the first inequality in our claim. Regarding the second inequality, since $m\geq 1$, we have that $\ell(Q)\leq \ell(R)<2^m\ell(R)$. Assume now that $\ell(Q)>\ell(R)$. Then,
    $$
    \text{dist}(Q,R) = \text{dist}(R, \partial Q) \leq \ell(R)^{\gamma}\ell(Q)^{1-\gamma}.
    $$
    Since $R$ is good, arguing analogously as before we find that $\ell(Q)<2^m\ell(R)$ and since $m\geq 1$, $2^{-m}\ell(R)<\ell(R)<\ell(Q)$.
\end{obs}

This new notion is what motivates us to separate the sum in (\ref{el_que_hem_dacotar}) as follows,
\begin{align}
    &\sum_{Q\in \mathcal{D}^{\text{tr},1}_1, R\in \mathcal{D}^{\text{tr},2}_2}\langle K_{\Theta}(\Delta_{1,Q}f), \Delta_{2,R}g\rangle  = \sum_{\substack{Q\in \mathcal{D}^{\text{tr},1}_1, R\in \mathcal{D}^{\text{tr},2}_2\\ Q\cap R = \varnothing }}\cdots + \sum_{\substack{Q\in \mathcal{D}^{\text{tr},1}_1, R\in \mathcal{D}^{\text{tr},2}_2\\ Q\cap R \neq  \varnothing }}\cdots\nonumber \\
    &= \sum_{\substack{Q\in \mathcal{D}^{\text{tr},1}_1, R\in \mathcal{D}^{\text{tr},2}_2\\ Q,R\text{ distant} }} + \sum_{\substack{Q\in \mathcal{D}^{\text{tr},1}_1, R\in \mathcal{D}^{\text{tr},2}_2\\ Q\cap R = \varnothing\\Q,R\text{ not distant} }} + \sum_{\substack{Q\in \mathcal{D}^{\text{tr},1}_1, R\in \mathcal{D}^{\text{tr},2}_2\\ Q\cap R \neq \varnothing\\ 2^{-m}\ell(R)\leq \ell(Q)\leq 2^m\ell(R) }} + \sum_{\substack{Q\in \mathcal{D}^{\text{tr},1}_1, R\in \mathcal{D}^{\text{tr},2}_2\\ Q\cap R = \varnothing\\ \ell(Q)<2^{-m}\ell(R)\text{ or }\ell(R)<2^{-m}\ell(Q)}} \nonumber \\
    & =: S_1 + S_2 + S_3 + S_4.\label{4_termes}
\end{align}
Our strategy now to complete the proof of Lemma \ref{lemma:good_functions} will be to show that each of the sums above is bounded by $c\|f\|_{L^2(\mu)}\|g\|_{L^2(\mu)}$. We are going to do this in sections \ref{sec:s1} through \ref{sec:s2s3}.

\subsection{Estimate of the sum $S_1$: distant cubes}

\label{sec:s1}

We can split the sum $S_1$ in two by considering
\begin{align*}
    \left|\sum_{\substack{Q\in \mathcal{D}^{\text{tr},1}_1, R\in \mathcal{D}^{\text{tr},2}_2\\Q,R \text{ distant }}}\langle K_{\Theta}(\Delta_{1,Q}f),\Delta_{2,R}g\rangle \right|&\leq \sum_{\substack{Q\in \mathcal{D}^{\text{tr},1}_1, R\in \mathcal{D}^{\text{tr},2}_2\\Q,R \text{ distant }\\ \ell(Q)\leq \ell(R)}}|\langle K_{\Theta}(\Delta_{1,Q}f), \Delta_{2,R}g\rangle |\\
    &\quad + \sum_{\substack{Q\in \mathcal{D}^{\text{tr},1}_1, R\in \mathcal{D}^{\text{tr},2}_2\\Q,R \text{ distant }\\ \ell(Q)> \ell(R)}}|\langle K_{\Theta}(\Delta_{1,Q}f), \Delta_{2,R}g\rangle |\\
    &=: S_{1,1} + S_{1,2}.
\end{align*}

First we will deal with the sum $S_{1,1}$ and once we have obtained the right bound for it we will argue that by analogous arguments, we obtain the same bound for $S_{1,2}$.

A quantity that will be useful to us is the following,
$$
D(Q,R):= \text{dist}(Q, R) + \ell(Q) + \ell(R),
$$
which we call the \textbf{long distance} between $Q$ and $R$. Using this notion, we have the following lemma.

\begin{lemma}
    \label{lema:phi_psi} Let $Q, R \subset \R^d$ be disjoint cubes and let $\varphi_Q, \psi_R\in L^2(\mu)$ be functions supported on $Q$ and $R$, respectively. If $\text{dist}(Q,\text{supp}(\psi_R))\geq \ell(Q)$ and $\int \varphi_Q\,d\mu=0$, then
    \begin{equation}
        |\langle K_{\Theta}\varphi_Q, \psi_R\rangle |\leq c\frac{\ell(Q)^{\eta}}{\text{dist}(Q, \text{supp}(\psi_R))^{n+\eta}}\|\varphi_Q\|_{L^1(\mu)}\|\psi_R\|_{L^1(\mu)}. \label{primera_estimacio_phi_psi}
    \end{equation}
    If, in addition, $\ell(Q)\leq \ell(R)$ and for some constant $c_6>0$,
    \begin{equation}
        \text{dist}(Q,\text{supp}(\psi_R))\geq c_6\ell(Q)^{\gamma}\ell(R)^{1-\gamma}, \label{segona_condicio}
    \end{equation}
    then we also have
    \begin{equation}
        |\langle K_{\Theta}\varphi_Q, \psi_R \rangle| \leq c\frac{\ell(Q)^{\frac{\eta}{2}}\ell(R)^{\frac{\eta}{2}}}{D(Q,R)^{n+\eta}}\|\varphi_Q\|_{L^1(\mu)}\|\psi_R\|_{L^1(\mu)}.\label{segona_estimacio_phi_psi}
    \end{equation}
\end{lemma}

\begin{proof}
    Let us denote by $z_Q$ the center of the cube $Q$. From the assumption, we infer that if $s\in Q$ and $t\in \text{supp}(\psi_R)$ are arbitrary points, then
    $$
    |z_Q-t|\approx |s-t| \geq \text{dist}(Q,\text{supp}(\psi_R))\geq \ell(Q)\geq c_d |s-z_Q|,
    $$
    where $c_d$ is a constant that only depends on the dimension $d$, because $s\in Q$. Thus, we can use the Calderón-Zygmund estimate we know for the kernel $\widetilde{k}_{\Theta}$ to obtain
    \begin{align*}
        |\langle K_{\Theta}\varphi_Q,\psi_R\rangle | &= \left|\iint \widetilde{k}_{\Theta}(t,s)\varphi_Q(s)\psi_R(t)\,d\mu(s)\,d\mu(t)\right|\\
        &= \left|\iint\left(\widetilde{k}_{\Theta}(t,s)-\widetilde{k}_{\Theta}(t,z_Q)\right)\varphi_Q(s)\psi_R(t)\,d\mu(s)\,d\mu(t)\right|\\
        &\leq c\iint\frac{|s-z_Q|^\eta}{|t-z_Q|^{n+\eta}}\varphi_Q(s)\psi_R(t)\,d\mu(s)\,d\mu(t)\\
        &\leq c\frac{\ell(Q)^{\eta}}{\text{dist}(Q,\text{supp}(\psi_R))^{n+\eta}}\|\varphi_Q\|_{L^1(\mu)}\|\psi_R\|_{L^1(\mu)},
    \end{align*}
    which is exactly (\ref{primera_estimacio_phi_psi}). Suppose now that $\ell(Q)\leq \ell(R)$ and that the condition (\ref{segona_condicio}) holds. If $\text{dist}(Q,\text{supp}(\psi_R))\geq \ell(R)$, the previous inequality becomes
    \begin{equation}
    |\langle K_{\Theta}\varphi_Q, \psi_R\rangle| \leq c\frac{\ell(Q)^{\frac{\eta}{2}}\ell(R)^{\frac{\eta}{2}}}{\text{dist}(Q,\text{supp}(\psi_R))^{n+\eta}}\|\varphi_Q\|_{L^1(\mu)}\|\psi_R\|_{L^1(\mu)}.\label{plug}
    \end{equation}
    We claim that in this situation, the following three quantities are comparable,
    \begin{equation}
    \text{dist}(Q,\text{supp}(\psi_R))\approx \text{dist}(Q,\text{supp}(\psi_R)) + \ell(Q) + \ell(R) \approx D(Q,R).\label{coses_comparables}
    \end{equation}
    Indeed, for the first one, one inequality is obvious, and for the other, just notice that by our hypothesis, we have
    $$
    3\, \text{dist}(Q, \text{supp}(\psi_R)) \geq \text{dist}(Q,\text{supp}(\psi_R)) + \ell(Q) + \ell(R).
    $$
    To see the second comparability in (\ref{coses_comparables}), since $\text{supp}(\psi_R)\subseteq R$, it is obvious that
    $$
    \text{dist}(Q,\text{supp}(\psi_R))+ \ell(Q)+\ell(R) \geq \text{dist}(Q,R) + \ell(Q)+\ell(R)= D(Q,R).
    $$
    For the other inequality, note that 
    $$
    \text{dist}(Q,\text{supp}(\psi_R)) \leq \text{dist}(Q,R) + \text{diam}(R) \leq c(\text{dist(Q,R)} + \ell(R)),
    $$
    where $c$ depends only on the dimension, $d$. Hence, 
    $$
    \text{dist}(Q,\text{supp}(\psi_R)) + \ell(Q)+\ell(R) \leq c' D(Q,R).
    $$ 
    If we plug this in (\ref{plug}), we get (\ref{segona_estimacio_phi_psi}). Now, if $\text{dist}(Q,\text{supp}(\psi_R))\leq \ell(R)$, then
    \begin{align*}
        \frac{\ell(Q)^{\eta}}{\text{dist}(Q, \text{supp}(\psi_R))^{n+\eta}}\leq \frac{\ell(Q)^{\eta}}{(c_6\ell(Q)^{\gamma}\ell(R)^{1-\gamma})^{n+\eta}} = \frac{c_6^{-n-\eta}\ell(Q)^{\frac{\eta}{2}}\ell(R)^{\frac{\eta}{2}}}{\ell(R)^{n+\eta}}\leq c\frac{\ell(Q)^{\frac{\eta}{2}}\ell(R)^{\frac{\eta}{2}}}{D(Q,R)^{n+\eta}},
    \end{align*}
    where we have used that in this case $D(Q,R)\leq 3 \ell(R)$.
\end{proof}

This lemma, along with the following, will be the essential ingredient in dealing with the term $S_{1,1}$.

\begin{lemma}
    \label{lema:schur} (Schur's lemma) Let $I$ be some set of indices, and for each $i\in I$ a number $w_i>0$. Suppose that, for some constant $a\geq 0$, the matrix $\{T_{i,j}\}_{i,j\in I}$ satisfies
    $$
    \sum_{j\in I}|T_{i,j}|\,w_j\leq a\,w_i, \quad \text{for each }\,i,
    $$
    and 
    $$
    \sum_{i\in I}|T_{i,j}|\, w_i \leq a\,w_j, \quad \text{for each }\,j.
    $$
    Then, the matrix $\{T_{i,j}\}_{i,j\in I}$ defines a bounded operator in $\ell^2(I)$ with norm bounded above by $a$.
\end{lemma}
\begin{proof}
    Let $\{x_{i}\}_{i\in I}\subset \ell^2(I)$ and set $y_i = \sum_{j\in I }T_{i,j}x_j$. We have to show that 
    $$
    \sum_{i\in I}|y_i|^2 \leq a^2 \sum_{i\in I}|x_i|^2.
    $$
    It will be useful to write
    $$
    |T_{i,j}||x_j| = \left(|T_{i,j}|^{\frac{1}{2}}w_j^{\frac{1}{2}}\right)\left(|T_{i,j}|^{\frac{1}{2}}w_j^{-\frac{1}{2}}|x_j|\right),
    $$
    to which we apply Cauchy-Schwarz and we get
    $$
    |y_i|^2 \leq \left(\sum_{j\in I}|T_{i,j}|w_j\right)\left(\sum_{j\in I}|T_{i,j}|w_j^{-1}|x_j|^2\right)\leq a\,w_i\sum_{j\in I}|T_{i,j}|w_j^{-1}|x_j|^2,
    $$
    where we have used the first hypothesis. Now, summing over $i\in I$ and interchanging the order of the sums, the second hypothesis yields
    $$
    \sum_{i\in I}|y_i|^2 \leq a\sum_{j\in I}\left(\sum_{i\in I}|T_{i,j}|w_i\right)w_j^{-1}|x_j|^2\leq a^2 \sum_{j\in I}|x_j|^2,
    $$
    as desired.
\end{proof}

Combining the two previous lemmas we will be able to bound the term $S_{1,1}$.

\begin{lemma}
    \label{lema:s11} If $f,g\in L^2(\mu)$ are as in Lemma \ref{lemma:good_functions} and they are very good, then we have
    \begin{align}
        S_{1,1} &\leq c \sum_{\substack{Q\in \mathcal{D}^{\text{tr},1}_1, R\in \mathcal{D}^{\text{tr},2}_2\\ \ell(Q)\leq \ell(R)}}\frac{\ell(Q)^{\frac{\eta}{2}}\ell(R)^{\frac{\eta}{2}}}{D(Q,R)^{n+\eta}}\mu(Q)^{\frac{1}{2}}\mu(R)^{\frac{1}{2}}\|\Delta_{1,Q}f\|_{L^2(\mu)}\|\Delta_{2,R}g\|_{L^2(\mu)}\nonumber \\
        &\leq c\,\|f\|_{L^2(\mu)}\|g\|_{L^2(\mu)}\label{seogona_desigualtat_s11}
    \end{align}
\end{lemma}

\begin{proof}
    Recall that the cubes $Q,R$ that appear in the sum $S_{1,1}$ are distant and moreover $\ell(Q)\leq \ell(R)$. This means that
    $$
    \text{dist}(Q, \text{supp}(\Delta_{2,R}g))\geq \text{dist}(Q,R)\geq \ell(Q)^{\gamma}\ell(R)^{1-\gamma} \geq \ell(Q).
    $$
    Hence, we can apply the second estimate from Lemma \ref{lema:phi_psi} and we obtain that
    \begin{align*}
    |\langle  K_{\Theta}(\Delta_{1,Q}f), \Delta_{2,R}g\rangle| &\leq c \frac{\ell(Q)^{\frac{\eta}{2}}\ell(R)^{\frac{\eta}{2}}}{D(Q,R)^{n+\eta}}\|\Delta_{1,Q}f\|_{L^1(\mu)}\|\Delta_{2,R}g\|_{L^1(\mu)}\\
    &\leq c \frac{\ell(Q)^{\frac{\eta}{2}}\ell(R)^{\frac{\eta}{2}}}{D(Q,R)^{n+\eta}}\mu(Q)^{\frac{1}{2}}\mu(R)^{\frac{1}{2}}\|\Delta_{1,Q}f\|_{L^2(\mu)}\|\Delta_{2,R}g\|_{L^2(\mu)},
    \end{align*}
    where in the second inequality we have simply used Cauchy-Schwarz. Summing over $Q\in \mathcal{D}^{\text{tr},1}_1$, $R\in \mathcal{D}^{\text{tr},2}_2$ and $\ell(Q)\leq \ell(R)$, we obtain the first inequality in (\ref{seogona_desigualtat_s11}). To prove the second inequality in (\ref{seogona_desigualtat_s11}), we want to use Schur's lemma with the matrix $\{T_{Q,R}, {Q\in \mathcal{D}^{\text{tr},1}_1, R\in \mathcal{D}^{\text{tr},2}_2}\}$ defined by
    $$
    T_{Q,R} = \begin{cases}\displaystyle
        \frac{\ell(Q)^{\frac{\eta}{2}}\ell(R)^{\frac{\eta}{2}}}{D(Q,R)^{n+\eta}}\mu(Q)^{\frac{1}{2}}\mu(R)^{\frac{1}{2}}, \quad &\text{if }\,(Q,R)\in \mathcal{D}^{\text{tr},1}_1\times\mathcal{D}^{\text{tr},2}_2\,\text{ and }\, \ell(Q)\leq \ell(R),\\
        0&\text{otherwise.}
    \end{cases}
    $$
    Recall that by Lemma \ref{lema_descomposicio_L2},
    $$
    \sum_{Q\in \mathcal{D}^{\text{tr},1}_{1}}\|\Delta_{1,Q}f\|_{L^2(\mu)}^2\leq c\,\|f\|^2_{L^2(\mu)},\quad \sum_{R\in \mathcal{D}^{\text{tr},2}_2}\|\Delta_{2,R}g\|^2_{L^2(\mu)}\leq c\,\|g\|^2_{L^2(\mu)},
    $$
    so we only have to show that the matrix defined above generates a bounded operator in $\ell^2$, in the sense that
    $$
    \sum_{Q\in \mathcal{D}^{\text{tr},1}_1, R\in \mathcal{D}^{\text{tr},2}_2}T_{Q,R}\,a_Q\,b_R \leq \left(\sum_{Q\in \mathcal{D}^{\text{tr},1}_1}a_Q^2\right)^{\frac{1}{2}}\left(\sum_{R\in \mathcal{D}^{\text{tr},2}_2}b_Q^2\right)^{\frac{1}{2}},
    $$
    for all $\ell^2$ sequences $\{a_Q\}_{Q\in \mathcal{D}^{\text{tr},1}_1}$, $\{b_R\}_{R\in \mathcal{D}^{\text{tr},2}_2}$. Choosing 
    $$
    a_Q = \|\Delta_{1,Q}f\|_{L^2(\mu)} \quad \text{and}\quad b_R = \|\Delta_{2,R}g\|_{L^2(\mu)},
    $$
    we will obtain (\ref{seogona_desigualtat_s11}). Using Schur's lemma, it is enough to prove that
    \begin{equation}
    \sum_{Q\in \mathcal{D}^{\text{tr},1}_1: \ell(Q)\leq \ell(R)}T_{Q,R}\mu(Q)^{\frac{1}{2}}\leq c\, \mu(R)^{\frac{1}{2}}, \quad \text{for all }\, R\in \mathcal{D}^{\text{tr},2}_2,\label{primera_condicio_schur}
    \end{equation}
    and 
    \begin{equation}
    \sum_{R\in \mathcal{D}^{\text{tr},2}_2: \ell(Q)\leq \ell(R)}T_{Q,R}\mu(R)^{\frac{1}{2}}\leq c\, \mu(Q)^{\frac{1}{2}}, \quad \text{for all }\, R\in \mathcal{D}^{\text{tr},1}_1,\label{segona_condicio_schur}
    \end{equation}
    that is, we have chosen the numbers $w_Q = \mu(Q)>0$ that appear in the statement of Schur's lemma. First we prove (\ref{primera_condicio_schur}). Fix $R\in \mathcal{D}^{\text{tr},2}_2$ and write
    \begin{align*}
        \sum_{Q\in \mathcal{D}^{\text{tr},1}_1: \ell(Q)\leq \ell(R)}T_{Q,R}\mu(Q)^{\frac{1}{2}} &= \mu(R)^{\frac{1}{2}}\sum_{Q\in \mathcal{D}^{\text{tr},1}_1: \ell(Q)\leq \ell(R)}\frac{\ell(Q)^{\frac{\eta}{2}}\ell(R)^{\frac{\eta}{2}}}{D(Q,R)^{n+\eta}}\mu(Q)\\
        &=\mu(R)^{\frac{1}{2}}\sum_{k\geq 0}2^{-\frac{k\eta}{2}}\left(\sum_{Q\in \mathcal{D}^{\text{tr},1}_1: \ell(Q)=2^{-k}\ell(R)}\frac{\ell(R)^{\eta}}{D(Q,R)^{n+\eta}}\mu(Q)\right).
    \end{align*}
    We can rewrite the last sum as follows. If we denote $2^{-1}R=\varnothing$,
    \begin{align}
         \sum_{Q\in \mathcal{D}^{\text{tr},1}_1: \ell(Q)= 2^{-k}\ell(R)} \frac{\ell(R)^{\eta}}{D(Q,R)^{n+\eta}}\mu(Q) = \sum_{j\geq 0}\sum_{\substack{Q\in \mathcal{D}^{\text{tr},1}_1: Q\cap 2^jR\neq \varnothing\\
         Q\cap 2^{j-1}R= \varnothing\\
         \ell(Q)= 2^{-k}\ell(R)}}\frac{\ell(R)^{\eta}}{D(Q,R)^{n+\eta}}\mu(Q),\label{divisio_j}
    \end{align}
    that is, we are adding the contribution of each $Q$ depending on how big we have to make the fixed cube $R$ in order to meet $Q$. We claim that in this situation, for each $j\geq 0$, $D(Q,R)\geq c \ell(2^jR)$. Indeed, from the condition $Q\cap 2^{j-1}R=\varnothing$, we have that, if $z_R$ denotes the center of $R$, $\|x-z_R\|_{\infty}\geq \frac{1}{2}\ell(2^{j-1}R)$ for all $x\in Q$. Hence, for $x\in Q$ and $y\in R$,
    $$
    |x-y|\geq \|x-y\|_{\infty} \geq \|x-z_R\|_{\infty}- \|y-z_R\|_{\infty} \geq \frac{1}{2}\ell(2^{j-1}R)-\frac{1}{2}\ell(R).
    $$
    Therefore, taking infimum we obtain
    \begin{align*}
    D(Q,R) &= \text{dist}(Q,R) + \ell(Q) + \ell(R) \\
    &\geq \frac{1}{2}\ell(2^{j-1}R) - \frac{1}{2}\ell(R) + \ell(Q) + \ell(R)\geq \frac{1}{2}\ell(2^{j-1}R) = \frac{1}{4}\ell(2^jR).
    \end{align*}
    Hence, the right-hand side in (\ref{divisio_j}) is bounded above by
    \begin{align}
    c \sum_{j\geq 0}\sum_{\substack{Q\in \mathcal{D}^{\text{tr},1}_1: Q\cap 2^jR\neq \varnothing\\
         Q\cap 2^{j-1}R= \varnothing\\
         \ell(Q)= 2^{-k}\ell(R)}} \frac{\ell(R)^{\eta}}{\ell(2^{j}R)^{n+\eta}}\mu(Q) &= c \sum_{j\geq 0} \frac{\ell(R)^{\eta}}{\ell(2^jR)^{n+\eta}}\sum_{\substack{Q\in \mathcal{D}^{\text{tr},1}_1: Q\cap 2^jR\neq \varnothing\\
         Q\cap 2^{j-1}R= \varnothing\\
         \ell(Q)= 2^{-k}\ell(R)}} \mu(Q)\nonumber \\
         &\leq c\sum_{j\geq 0}\frac{\ell(R)^{\eta}}{\ell(2^jR)^{n+\eta}}\mu(2^{j+2}R) \nonumber \\
         & = c\sum_{j\geq 0}\frac{\mu(2^{j+2}R)}{\ell(2^jR)^n2^{j\eta}}.\label{suma_j}
    \end{align}
    Recall that by Lemma \ref{lema:comportament_cubs_transit}, since $R$ is transit,
    $$
    \mu(2^{j+2}R)\leq c_0 \,\ell(2^{j+2}R)^n=c \,\ell(2^jR)^n,
    $$
    so the sum in (\ref{suma_j}) is bounded above by some constant that depends on $c_0$ and $\eta$. If we plug this inequality in (\ref{divisio_j}), we obtain (\ref{primera_condicio_schur}). The last thing left is to prove (\ref{segona_condicio_schur}), which we will do by similar arguments as what we have just done. Fix $Q\in \mathcal{D}^{\text{tr},1}_1$ and write, as before,
    \begin{align*}
        \sum_{R\in \mathcal{D}^{\text{tr},2}_2: \ell(Q)\leq \ell(R)}T_{Q,R}\mu(R)^{\frac{1}{2}} &= \mu(Q)^{\frac{1}{2}}\sum_{R\in \mathcal{D}^{\text{tr},2}_2: \ell(Q)\leq \ell(R)}\frac{\ell(Q)^{\frac{\eta}{2}}\ell(R)^{\frac{\eta}{2}}}{D(Q,R)^{n+\eta}}\mu(R)\\
        &= \mu(Q)^{\frac{1}{2}}\sum_{k\geq 0}2^{-\frac{k\eta}{2}}\left(\sum_{R\in \mathcal{D}^{\text{tr},2}_2:\ell(Q)=2^{-k}\ell(R)}\frac{\ell(R)^{\eta}}{D(Q,R)^{n+\eta}}\mu(R)\right).
    \end{align*}
    Now, for each $k\geq0$, call $Q^k$ a cube concentric with $Q$ whose side length is $2^k\ell(Q)=\ell(R)$. Then, arguing exactly as before,
    \begin{align*}
         \sum_{R\in \mathcal{D}^{\text{tr},2}_2: \ell(R)= 2^{k}\ell(Q)} \frac{\ell(R)^{\eta}}{D(Q,R)^{n+\eta}}\mu(R) &= \sum_{j\geq 0}\sum_{\substack{R\in \mathcal{D}^{\text{tr},2}_2: R\cap 2^jQ^k\neq \varnothing\\
         R\cap 2^{j-1}Q^k= \varnothing\\
         \ell(R)= 2^{k}\ell(Q)}}\frac{\ell(R)^{\eta}}{D(Q,R)^{n+\eta}}\mu(R),\\
         &\leq c \sum_{j\geq 0} \frac{\ell(Q^k)^{\eta}}{\ell(2^jQ^k)^{n+\eta}}\sum_{\substack{R\in \mathcal{D}^{\text{tr},2}_2: R\cap 2^jQ^k\neq \varnothing\\
         R\cap 2^{j-1}Q^k= \varnothing\\
         \ell(R)= 2^{k}\ell(Q)}}\mu(R)\\
         &\leq c\sum_{j\geq 0}\frac{\mu(2^{j+2}Q^k)}{\ell(2^jQ^k)^n2^{j\eta}}.
    \end{align*}
    Again by Lemma \ref{lema:comportament_cubs_transit}, $\mu(2^{j+2}Q^k)\leq c\, \ell(2^jQ^k)^n$, so as before we obtain (\ref{segona_condicio_schur}).
\end{proof}

Now we recover the sum $S_{1,2}$. Recall that
$$
S_{1,2} = \sum_{\substack{Q\in \mathcal{D}^{\text{tr},1}_1, R\in \mathcal{D}^{\text{tr},2}_2\\Q,R \text{ distant }\\ \ell(Q)> \ell(R)}}|\langle K_{\Theta}(\Delta_{1,Q}f), \Delta_{2,R}g\rangle | = \sum_{\substack{Q\in \mathcal{D}^{\text{tr},1}_1, R\in \mathcal{D}^{\text{tr},2}_2\\Q,R \text{ distant }\\ \ell(Q)> \ell(R)}}|\langle \Delta_{1,Q}f, K_{\Theta}^*(\Delta_{2,R}g)\rangle |.
$$
From this we see that we need the analogous of Lemmas \ref{lema:phi_psi}, but interchanging the roles of $Q$ and $R$ and $K_{\Theta}$ and $K^*_{\Theta}$. One can go over the proof and see that in that case the same estimates hold by exactly the same arguments, only by changing the names of the cubes and using that the SIO $K^*_{\Theta}$ has the same relevant properties as $K_{\Theta}$. That is, we have the following result. 

\begin{lemma}
    \label{lema_phi_psi_2} Let $Q, R\subset \R^d$ be disjoint cubes and let $\varphi_Q,\psi_R\in L^2(\mu)$ be functions supported on $Q$ and $R$, respectively. If $\text{dist}(R, \text{supp}(\varphi_Q))\geq \ell(R)$ and $\int \psi_R\,d\mu =0$, then
    $$
    |\langle \varphi_Q, K^*_{\Theta}(\psi_R) \rangle | \leq c\frac{\ell(R)^{\eta}}{\text{dist}(R, \text{supp}(\varphi_Q))}\|\varphi_Q\|_{L^1(\mu)}\|\psi_R\|_{L^1(\mu)}.
    $$
    If, in addition, $\ell(R)\leq \ell(Q)$ and for some constant $c_6>0$,
    $$
    \text{dist}(R, \text{supp}(\varphi_Q))\geq c_6 \ell(R)^{\gamma}\ell(Q)^{1-\gamma},
    $$
    then we also have 
    $$
    |\langle \varphi_Q, K^*_{\Theta}(\psi_R)|\leq c\frac{\ell(Q)^{\frac{\eta}{2}}\ell(R)^{\frac{\eta}{2}}}{D(Q,R)^{n+\eta}}\|\varphi_Q\|_{L^1(\mu)}\|\psi_R\|_{L^1(\mu)}.
    $$
\end{lemma}

Now, recall that the cubes $Q,R$ that appear in the sum $S_{1,2}$ are distant and moreover $\ell(Q)> \ell(R)$. This means that
    $$
    \text{dist}(R, \text{supp}(\Delta_{1,Q}f))\geq \text{dist}(Q,R)\geq \ell(R)^{\gamma}\ell(Q)^{1-\gamma} \geq \ell(R).
    $$
    Hence, we can apply the second estimate from the previous Lemma and we obtain that
    \begin{align*}
    |\langle  \Delta_{1,Q}f, K^*_{\Theta}(\Delta_{2,R}g)\rangle| &\leq c \frac{\ell(Q)^{\frac{\eta}{2}}\ell(R)^{\frac{\eta}{2}}}{D(Q,R)^{n+\eta}}\|\Delta_{1,Q}f\|_{L^1(\mu)}\|\Delta_{2,R}g\|_{L^1(\mu)}\\
    &\leq c \frac{\ell(Q)^{\frac{\eta}{2}}\ell(R)^{\frac{\eta}{2}}}{D(Q,R)^{n+\eta}}\mu(Q)^{\frac{1}{2}}\mu(R)^{\frac{1}{2}}\|\Delta_{1,Q}f\|_{L^2(\mu)}\|\Delta_{2,R}g\|_{L^2(\mu)}.
    \end{align*}
Since the last bound is symmetric in $Q$ and $R$ (note that here it does not matter whether the cubes are terminal of the first or second king), by the second estimate in Lemma \ref{lema:s11}, it is controlled above by $c\,\|f\|_{L^2(\mu)}\|g\|_{L^2(\mu)}$, which finishes the estimate of the sum $S_1$.

Before moving on to the next estimates, let us make a final remark about how we bounded the sum $S_1$. The fact that we were dealing only with transit cubes (of either kind) was crucial. Indeed, it allowed us to use Lemma \ref{lema:comportament_cubs_transit} and bound for every transit cube $S$ and $\lambda \geq 1$, $\mu(\lambda S)\leq c_0\ell(\lambda S)^n$, which would not have been possible if there had been terminal cubes, for which we do not have such a type of control, in the martingale decomposition given by Lemma \ref{lema_descomposicio_L2}.

\subsection{Estimate of the sum $S_4$}

As we did with the sum $S_1$, we can split the sum $S_4$ in two by considering
\begin{align*}
    S_4 = \sum_{\substack{Q\in \mathcal{D}^{\text{tr},1}_1, R\in \mathcal{D}^{\text{tr},2}_2\\ Q\cap R \neq \varnothing\\
    \ell(Q)<2^{-m}\ell(R)\text{ or }\ell(R)<2^{-m}\ell(Q)}}\langle K_{\Theta}(\Delta_{1,Q}f),\Delta_Rg\rangle &= \sum_{\substack{Q\in \mathcal{D}^{\text{tr},1}_1, R\in \mathcal{D}^{\text{tr},2}_2\\ Q\cap R \neq \varnothing\\
    \ell(Q)<2^{-m}\ell(R)}}\cdots  + \sum_{\substack{Q\in \mathcal{D}^{\text{tr},1}_1, R\in \mathcal{D}^{\text{tr},2}_2\\ Q\cap R \neq \varnothing\\
    \ell(R)<2^{-m}\ell(Q)}}\cdots\\
    &=:S_{4,1} + S_{4,2}.
\end{align*}
By duality, we can write the second sum as
$$
S_{4,2} = \sum_{\substack{Q\in \mathcal{D}^{\text{tr},1}_1, R\in \mathcal{D}^{\text{tr},2}_2\\ Q\cap R \neq \varnothing\\
    \ell(R)<2^{-m}\ell(Q)}}\langle \Delta_{1,Q}f, K^*_{\Theta}(\Delta_{2,R}g)\rangle.
$$
Since the assumptions that $K_{\Theta}^*$ satisfies are analogous to those for $K_{\Theta}$ (for example, recall how Lemma \ref{lema:phi_psi} has its corresponding version with $K^*_{\Theta}$, Lemma \ref{lema_phi_psi_2}), it is enough to prove the desired estimate only for $S_{4,1}$, which we are going to do in this section. 

It is important to recall that all these estimates are part of the proof of Lemma \ref{lemma:good_functions}, and so the functions $f$ and $g$ are supposed to be good, each with respect to the indicated lattices. This means that the cubes $Q\in \mathcal{D}_1$ appearing in the sum $S_{4,1}$ are $1$-$\mathcal{D}_1$-good with respect to $\mathcal{D}_2$, and analogously for the cubes $R\in \mathcal{D}_2$ with respect to $\mathcal{D}_2$. In particular, recalling condition (a) in the definition of bad cubes, we see that every $Q\in \mathcal{D}_1$ in the sum $S_{4,1}$ must satisfy that
\begin{equation}
\text{dist}(Q, \partial S) >\ell(Q)^{\gamma}\ell(R)^{1-\gamma}, \quad \text{for every }\, S\in \mathcal{D}_2\,\text{ such that }\, \ell(S)\geq 2^m\ell(Q).\label{condicio_bona}
\end{equation}
Since every cube $R$ in the sum $S_{4,1}$ satisfies that $\ell(R)\geq 2^{m+1}\ell(Q)$, we can choose the cube $S$ to be any of the $2^d$ children of $R$. Moreover, since $Q\cap R \neq \varnothing$ and the distance to the boundary is strictly positive, we see that $Q$ must be completely contained in one of these $2^d$ children, which we will denote from now on by $R_Q$.

Now, we can use the distinction between transit and terminal cubes to split further the sum $S_{4,1}$, by considering
\begin{align}
    S_{4,1} &= \sum_{\substack{Q\in \mathcal{D}_1^{\text{tr},1}, R \in \mathcal{D}^{\text{tr},2}_2\\
    Q\subset R\\
    \ell(Q)<2^{-m}\ell(R)\\
    R_Q\in \mathcal{D}^{\text{tr},2}_2}}\langle K_{\Theta}(\Delta_{1,Q}f),\Delta_{2,R}g\rangle + \sum_{\substack{Q\in \mathcal{D}_1^{\text{tr},1}, R \in \mathcal{D}^{\text{tr},2}_2\\
    Q\subset R\\
    \ell(Q)<2^{-m}\ell(R)\\
    R_Q\in \mathcal{D}^{\text{term},2}_2}}\langle K_{\Theta}(\Delta_{1,Q}f),\Delta_{2,R}g\rangle \label{s41}\\
    &:= S^{\text{tr}}_{4,1} + S^{\text{term}}_{4,1}.\nonumber
\end{align}

\subsubsection{Estimate of $S^{\text{\normalfont tr}}_{4,1}$}

We can use duality to write $S_{4,1}^{\text{tr}}$ as
\begin{align}
    S^{\text{tr}}_{4,1} &= \sum_{\substack{Q\in \mathcal{D}_1^{\text{tr},1}, R \in \mathcal{D}^{\text{tr},2}_2\\
    Q\subset R\\
    \ell(Q)<2^{-m}\ell(R)\\
    R_Q\in \mathcal{D}^{\text{tr},2}_2}} \langle \Delta_{1,Q}f, K^*_{\Theta}(\Delta_{2,R}g)\rangle \nonumber \\
    &= \sum_{\substack{Q\in \mathcal{D}_1^{\text{tr},1}, R \in \mathcal{D}^{\text{tr},2}_2\\
    Q\subset R\\
    \ell(Q)<2^{-m}\ell(R)\\
    R_Q\in \mathcal{D}^{\text{tr},2}_2}} \langle \Delta_{1,Q}f, K^*_{\Theta}(\chi_{R_Q}\Delta_{2,R}g)\rangle + \sum_{\substack{Q\in \mathcal{D}_1^{\text{tr},1}, R \in \mathcal{D}^{\text{tr},2}_2\\
    Q\subset R\\
    \ell(Q)<2^{-m}\ell(R)\\
    R_Q\in \mathcal{D}^{\text{tr},2}_2}}\langle \Delta_{1,Q}f, K^*_{\Theta}(\chi_{R\setminus R_Q}\Delta_{2,R}g)\rangle.\label{dues_sumes_s_tr_41}
\end{align}
The last sum on the right-hand side above will be easy to deal with, using our work from the previous section, as shown in the next lemma.

\begin{lemma}
    Let $f,g\in L^2(\mu)$ be very good. Then,
    \begin{equation}
    \sum_{\substack{Q\in \mathcal{D}_1^{\text{tr},1}, R \in \mathcal{D}^{\text{tr},2}_2\\
    Q\subset R\\
    \ell(Q)<2^{-m}\ell(R)\\
    R_Q\in \mathcal{D}^{\text{tr},2}_2}}|\langle \Delta_{1,Q}f, K_{\Theta}^*(\chi_{R\setminus R_Q}\Delta_{2,R}g)\rangle | \leq c\, \|f\|_{L^2(\mu)}\|g\|_{L^2(\mu)}.\label{segona_suma_s41_tr}
    \end{equation}
\end{lemma}

\begin{proof}
    Let $Q$ and $R$ be as above and good. Moreover, since $\ell(Q)<2^{-m}\ell(R)$, we have
    \begin{align*}
        \text{dist}(Q, \text{supp}(\chi_{R\setminus R_Q}\Delta_{2,R}g))&\geq \text{dist}(Q,\partial R_{Q})\geq \ell(Q)^{\gamma}\ell(R_Q)^{1-\gamma} = 2^{\gamma-1}\ell(Q)^\gamma\ell(R)^{1-\gamma}.
    \end{align*}
    We can order the children of $R$ that do not contain $Q$ as $\{R_j\}_{j=1}^{2^d-1}$. Then,
    $$
    \left|\langle \Delta_{1,Q}f, K^*_{\Theta}(\chi_{R\setminus R_Q}\Delta_{2,R}g)\rangle\right| \leq \sum_{j=1}^{2^d-1}|\langle \Delta_{1,Q}f, K^*_{\Theta}(\chi_{R_j}\Delta_{2,Q}g)\rangle|.
    $$
    For each $1\leq j\leq 2^d-1$, since $R_j\subset R\setminus R_Q$,
    \begin{align*}
        \text{dist}(Q,\text{supp}(\chi_{R_j}\Delta_{2,R}g))\geq \text{dist}(Q, \text{supp}(\chi_{R\setminus R_Q}\Delta_{2,R}g))&\geq 2^{\gamma-1}\ell(Q)^{\gamma}\ell(R)^{1-\gamma}\\
        &\geq c \ell(Q)^{\gamma}\ell(R_j)^{1-\gamma}.
    \end{align*}
    So, for each $R_j$ we can apply Lemma \ref{lema:phi_psi} and we obtain that
    \begin{align}
        |\langle \Delta_{1,Q}f,K^{*}_{\Theta}(\chi_{R_j}\Delta_{2,R}g)\rangle| &=|\langle K_{\Theta}(\Delta_{1,Q}f), \chi_{R_j}\Delta_{2,R}g\rangle |\nonumber \\
        &\leq c\frac{\ell(Q)^{\frac{\eta}{2}}\ell(R_j)^{\frac{\eta}{2}}}{D(Q,R_j)^{n+\eta}}\|\Delta_{1,Q}f\|_{L^1(\mu)}\|\chi_{R_j}\Delta_{2,R}g\|_{L^2(\mu)}.\label{fill}
    \end{align}
    Taking into account that $\ell(R_j)^{\frac{\eta}{2}}=2^{-\frac{\eta}{2}}\ell(R)^{\frac{\eta}{2}}$, $\|\chi_{R_j}\Delta_{2,R}g\|_{L^1(\mu)}\leq \|\Delta_{2,R}g\|_{L^1(\mu)}$ and that
    $$
    D(Q,R_j)= \text{dist}(Q,R_j)+\ell(Q)+\ell(R_j)\geq \underbrace{\text{dist}(Q,R)}_{=0}+\ell(Q)+\frac{1}{2}\ell(R_j)\geq \frac{1}{2}D(Q,R),
    $$
    we have that the right-hand side of (\ref{fill}) is bounded above by
    \begin{align*}
        c'\frac{\ell(Q)^{\frac{\eta}{2}}\ell(R)^{\frac{\eta}{2}}}{D(Q,R)^{n+\eta}}\mu(Q)^{\frac{1}{2}}\mu(R)^{\frac{1}{2}}\|\Delta_{1,Q}f\|_{L^2(\mu)}\|\Delta_{2,R}g\|_{L^2(\mu)}.
    \end{align*}
    Summing over the children of $R$, we obtain that the sum in (\ref{segona_suma_s41_tr}) is bounded by a constant times the sum in (\ref{seogona_desigualtat_s11}), from the previous section. Hence, by Lemma \ref{lema:s11}, it is bounded by $c\,\|f\|_{L^2(\mu)}\|g\|_{L^2(\mu)}$, which is what we wanted to prove.
\end{proof}

Now it will be useful to consider, for $g\in L^2(\mu)$ as above, the number
$$
c_{R,Q}(g) = \frac{\langle g\rangle_{R_Q}}{\langle b_2\rangle_{R_Q}}- \frac{\langle g\rangle_{R_Q}}{\langle b_2\rangle_{R}},
$$
which is chosen so that $\chi_{R_{Q}}\Delta_{2,R}g = c_{R,Q}(g)\chi_{R_Q}b_2$ (recall that the cubes $R$ that appear in the sum $S^{\text{tr}}_{4,1}$ are in $\mathcal{D}^{\text{tr},2}_2$). With this, we can write the first sum on the right-hand side of (\ref{dues_sumes_s_tr_41}) as
\begin{align}
    \sum_{\substack{Q\in \mathcal{D}^{\text{tr},1}_1, R\in \mathcal{D}^{\text{tr,2}}_2\\ Q\subset R\\ \ell(Q)<2^{-m}\ell(R)\\ R_{Q}\text{ transit}}}c_{R,Q}(g)\langle \Delta_{1,Q}f, K^{*}_{\Theta}(\chi_{R_Q}b_2)\rangle &= \sum_{\substack{Q\in \mathcal{D}^{\text{tr},1}_1, R\in \mathcal{D}^{\text{tr,2}}_2\\ Q\subset R\\ \ell(Q)<2^{-m}\ell(R)\\ R_{Q}\text{ transit}}}c_{R,Q}(g)\langle \Delta_{1,Q}f, K^*_{\Theta}b_2\rangle \nonumber \\
    &\quad - \sum_{\substack{Q\in \mathcal{D}^{\text{tr},1}_1, R\in \mathcal{D}^{\text{tr,2}}_2\\ Q\subset R\\ \ell(Q)<2^{-m}\ell(R)\\ R_{Q}\text{ transit}}}c_{R,Q}(g)\langle \Delta_{1,Q}f, K^*_{\Theta}(\chi_{\R^d\setminus R_Q}b_2)\rangle .\label{pre_paraproduct}
\end{align}
Our next strategy will be to bound each of the two sums on the right-hand side above by $c\,\|f\|_{L^2(\mu)}\|g\|_{L^2(\mu)}$. Once we have done that, we will have proved that
$$
|S^{\text{tr}}_{4,1}|\leq c\,\|f\|_{L^2(\mu)}\|g\|_{L^2(\mu)}.
$$
Let us start by taking care of the second term in (\ref{pre_paraproduct}), which we do in the following lemma.
\begin{lemma}
    Let $f,g\in L^2(\mu)$ be very good. Then,
    $$
    \sum_{\substack{Q\in \mathcal{D}^{\text{tr},1}_1, R\in \mathcal{D}^{\text{tr,2}}_2\\ Q\subset R\\ \ell(Q)<2^{-m}\ell(R)\\ R_{Q}\text{ transit}}}|c_{R,Q}(g)\langle \Delta_{1,Q}f, K^*_{\Theta}(\chi_{\R^d\setminus R_Q}b_2)\rangle|\leq c\,\|f\|_{L^2(\mu)}\|g\|_{L^2(\mu)}.
    $$
\end{lemma}
\begin{proof}
    First, we will give bounds for the coefficients $c_{R,Q}(g)$ and $|\langle \Delta_{1,Q}f, K^*_{\Theta}(\chi_{\R^d\setminus R_Q}b_2)\rangle|$ separately. We have
    \begin{align*}
        \|\Delta_{2,R}f\|^2_{L^2(\mu)} \geq \int_{R_Q}|\Delta_{2,R}g|^2\,d\mu = |c_{R,Q}(g)|^2\int_{R_Q}|b_2|^2\,d\mu &\geq |c_{R,Q}(g)|^2 |\langle b_2\rangle_{R_Q}|^2 \mu(R_Q)\\
        &\geq  c_{\text{acc}}^{-2}|c_{R,Q}(g)|^2 \mu(R_Q),
    \end{align*}
    because the cube $R_Q$ is transit. So, we get that
    \begin{equation}
    |c_{R,Q}(g)|\leq c\,\|\Delta_{2,R}g\|_{L^2(\mu)}\mu(R_Q)^{-\frac{1}{2}}.\label{estimacio_crq}
    \end{equation}
    To estimate the terms $|\langle \Delta_{1,Q}f, K^*_{\Theta}(\chi_{\R^d\setminus R_Q}b_2)|$, we will argue in a similar way to what we did in the proof of Lemma \ref{lema:phi_psi}. Denote by $z_Q$ the center of $Q$. We have that for $y\in Q$ and $x\in \R^d\setminus R_Q$,
    $$
    |z_Q-x| \approx |y-x| \geq c\,\ell(Q) \geq c\, |y-z_Q|.
    $$
    Therefore,
    \begin{align*}
        |K_{\Theta}(\Delta_{1,Q}f)(x)| &= \left|\int \widetilde{k}_{\Theta}(x,y)\Delta_{1,Q}f(y)\,d\mu(y)\right|\\
        &= \left|\int\left(\widetilde{k}_{\Theta}(x,y)-\widetilde{k}_{\Theta}(x,z_Q)\right)\Delta_{1,Q}f(y)\,d\mu(y)\right|\\
        & \leq c\int\frac{|y-z_Q|^{\eta}}{|x-z_Q|^{n+\eta}}\Delta_{1,Q}f(y)\,d\mu(y)\\
        &\leq c\frac{\ell(Q)^{\eta}}{|x-z_Q|^{n+\eta}}\|\Delta_{1,Q}f\|_{L^1(\mu)}.
    \end{align*}
   Integrating this inequality, we find that
    \begin{align*}
        |\langle \Delta_{1,Q}f, K^*_{\Theta}(\chi_{\R^d\setminus R_Q}b_2)\rangle | &= |\langle K_{\Theta}(\Delta_{1,Q}f), \chi_{\R^d\setminus R_Q}b_2\rangle|\\
        &\leq c\,\|\Delta_{1,Q}f\|_{L^2(\mu)}\int_{\R^d\setminus R_Q} \frac{\ell(Q)^{\eta}}{|x-z_Q|^{n+\eta}}|b_2(x)|\,d\mu(x).
    \end{align*}
    Now we use the fact that $b_2$ is a bounded function and integrate over annuli to get that
    \begin{align}
        \int_{\R^d\setminus \R_Q}\frac{\ell(Q)^{\eta}}{|x-z_Q|^{n+\eta}}|b_2(x)|\,d\mu(x) &\leq c_b\sum_{k\geq 0}\int_{2^{k+1}R_Q\setminus 2^kR_Q}\frac{\ell(Q)^{\eta}}{|x-z_Q|^{n+\eta}}\,d\mu(x)\nonumber \\
        &\leq c\frac{\ell(Q)^{\eta}\mu(2R_Q)}{\text{dist}(z_Q, \partial R_Q)^{n+\eta}} + c\sum_{k\geq 1}\frac{\ell(Q)^{\eta}\mu(2^{k+1}R_Q)}{\ell(2^kR_Q)^{n+\eta}},\label{separacio_k}
    \end{align}
    where the first term on the right-hand side of the second line corresponds to $k=0$ in the sum in the first line. Recall that by Lemma \ref{lema:comportament_cubs_transit}, since $R_Q$ is transit, $\mu(\lambda R_Q)\leq c_0\,\ell(\lambda R_Q)^n$ for all $\lambda\geq 1$. Moreover, recall that by (\ref{condicio_bona}),
    $$
    \text{dist}(Q, \partial R_Q)\geq \ell(Q)^{\gamma}\ell(R_Q)^{1-\gamma} = 2^{\gamma-1}\ell(Q)^{\gamma}\ell(R)^{1-\gamma} \geq 2^{\gamma-1}\ell(Q).
    $$
    Using this, we can bound the right-hand side of (\ref{separacio_k}) by
    \begin{align*}
    c\frac{\ell(Q)^{\eta}\ell(R_Q)^n}{(\ell(Q)^\gamma\ell(R_Q)^{1-\gamma})^{n+\eta}} + c\sum_{k\geq 1}\frac{\ell(Q)^{\eta}\ell(2^{k+1}R_Q)^n}{\ell(2^kR_Q)^{n+\eta}} &= c\frac{\ell(Q)^{\frac{\eta}{2}}}{\ell(R_Q)^{\frac{\eta}{2}}} + c\,2^{\eta}\frac{\ell(Q)^{\eta}}{\ell(R_Q)^{\eta}}\sum_{k\geq 1} \frac{1}{2^{k\eta}}\\
    &\leq c\frac{\ell(Q)^{\frac{\eta}{2}}}{\ell(R_Q)^{\frac{\eta}{2}}}\left(1+ \frac{\ell(Q)^{\frac{\eta}{2}}}{\ell(R_{Q})^{\frac{\eta}{2}}}\right)\\
    &\leq 2c\frac{\ell(Q)^{\frac{\eta}{2}}}{\ell(R_Q)^{\frac{\eta}{2}}},
    \end{align*}
    where in the last line we have used that $\ell(Q)\leq \ell(R_q)$, as $Q$ is contained in $R_Q$. This leads to the bound
    $$
    |\langle \Delta_{1,Q}f, K^*_{\Theta}(\chi_{\R^d\setminus R_Q}b_2)\rangle |\leq c\, \|\Delta_{1,Q}f\|_{L^2(\mu)}\frac{\ell(Q)^{\frac{\eta}{2}}}{\ell(R_Q)^{\frac{\eta}{2}}}\leq c\frac{\ell(Q)^{\frac{\eta}{2}}}{\ell(R_Q)^{\frac{\eta}{2}}}\mu(Q)^{\frac{1}{2}}\|\Delta_{1,Q}f\|_{L^2(\mu)}.
    $$
    Combining this with (\ref{estimacio_crq}), we obtain
    $$
    |c_{R,Q}(g)\langle \Delta_{1,Q}f, K^*_{\Theta}(\chi_{\R^d\setminus R_Q}b_2)\rangle | \leq c\frac{\ell(Q)^{\frac{\eta}{2}}}{\ell(R_Q)^{\frac{\eta}{2}}} \frac{\mu(Q)^{\frac{1}{2}}}{\mu(R_Q)^{\frac{1}{2}}}\|\Delta_{1,Q}f\|_{L^2(\mu)}\|\Delta_{2,R}g\|_{L^2(\mu)}.
    $$
    Therefore, summing over the admissible cubes, if we write $Q\in I(R)$ to indicate that $Q$ is contained in one of the children of $R$,
    \begin{align}
        &\sum_{\substack{Q\in \mathcal{D}^{\text{tr},1}_1, R\in \mathcal{D}^{\text{tr,2}}_2\\ Q\subset R\\ \ell(Q)<2^{-m}\ell(R)\\ R_{Q}\text{ transit}}}|c_{R,Q}(g)\langle \Delta_{1,Q}f, K^*_{\Theta}(\chi_{\R^d\setminus R_Q}b_2)\rangle|\nonumber\\
        &\qquad\qquad  \leq c\sum_{\substack{Q\in \mathcal{D}^{\text{tr},1}_1, R\in \mathcal{D}^{\text{tr,2}}_2\\ Q\in I(R)}} \frac{\ell(Q)^{\frac{\eta}{2}}}{\ell(R_Q)^{\frac{\eta}{2}}} \frac{\mu(Q)^{\frac{1}{2}}}{\mu(R_Q)^{\frac{1}{2}}}\|\Delta_{1,Q}f\|_{L^2(\mu)}\|\Delta_{2,R}g\|_{L^2(\mu)}\nonumber\\
        &\qquad \qquad \leq c\sum_{R\in \mathcal{D}^{\text{tr},2}_2}\|\Delta_{2,R}g\|_{L^2(\mu)}\sum_{\substack{Q\in \mathcal{D}^{\text{tr},1}_1\\ Q\in I(R)}}\frac{\ell(Q)^{\frac{\eta}{2}}}{\ell(R)^{\frac{\eta}{2}}} \frac{\mu(Q)^{\frac{1}{2}}}{\mu(R_Q)^{\frac{1}{2}}}\|\Delta_{1,Q}f\|_{L^2(\mu)}.\label{separacio_IR}
    \end{align}
    For each $R\in \mathcal{D}^{\text{tr},2}_2$, we can bound, using Cauchy-Schwarz, 
    \begin{align*}
        &\sum_{\substack{Q\in \mathcal{D}^{\text{tr},1}_1\\Q\in I(R)}}\frac{\ell(Q)^{\frac{\eta}{2}}}{\ell(R)^{\frac{\eta}{2}}}\frac{\mu(Q)^{\frac{1}{2}}}{\mu(R_Q)^{\frac{1}{2}}}\|\Delta_{1,Q}f\|_{L^2(\mu)}\\
        &\qquad\qquad \leq\Bigg(\sum_{\substack{Q\in \mathcal{D}^{\text{tr},1}_1\\ Q\in I(R)}}\frac{\ell(Q)^{\frac{\eta}{2}}}{\ell(R)^{\frac{\eta}{2}}}||\Delta_{1,Q}f\|^2_{L^2(\mu)}\Bigg)^{\frac{1}{2}}\Bigg(\sum_{\substack{Q\in \mathcal{D}^{\text{tr},1}_1\\ Q\in I(R)}}\frac{\ell(Q)^{\frac{\eta}{2}}}{\ell(R)^{\frac{\eta}{2}}}\frac{\mu(Q)^{\eta}}{\mu(R_Q)^{\eta}}\Bigg)^{\frac{1}{2}}.
    \end{align*}
    If we split the sum in the last factor above according to the size of the cubes $Q$, we can bound it by an absolute constant. Indeed,
    \begin{align*}
        \sum_{\substack{Q\in \mathcal{D}^{\text{tr},1}_1\\ Q\in I(R)}}\frac{\ell(Q)^{\frac{\eta}{2}}}{\ell(R)^{\frac{\eta}{2}}}\frac{\mu(Q)^{\eta}}{\mu(R_Q)^{\eta}} &= \sum_{k\geq 1}\sum_{\substack{ Q\in \mathcal{D}^{\text{tr},1}_1\\ Q \in I(R)\\ \ell(Q)=2^{-k}\ell(R)}}\frac{\ell(Q)^{\frac{\eta}{2}}}{\ell(R)^{\frac{\eta}{2}}}\frac{\mu(Q)}{\mu(R_Q)}\\
        &\leq \sum_{k\geq 1}2^{-\frac{k\eta}{2}}\Bigg(\sum_{j=1}^{2^d-1}\frac{1}{\mu(R_j)}\Bigg(\sum_{\substack{Q\in \mathcal{D}^{\text{tr},1}_1\\ Q\subset R_j,\, \ell(Q)=2^{-k}\ell(R)}}\mu(Q)\Bigg)\Bigg)\\
        &\leq 2^{d-1}\sum_{k\geq 1}2^{-\frac{k\eta}{2}}\leq c,
    \end{align*}
    where we have used that the cubes $Q$ appearing in the innermost sum are pairwise disjoint. This means that the sum on the right-hand side of (\ref{separacio_IR}) does not exceed, by Cauchy-Schwarz,
    \begin{align*}
        &c\sum_{R\in \mathcal{D}^{\text{tr},2}_2}\|\Delta_{2,R}g\|_{L^2(\mu)}\left(\sum_{\substack{Q\in \mathcal{D}^{\text{tr},1}_1: Q\subset R}}\frac{\ell(Q)^{\frac{\eta}{2}}}{\ell(R)^{\frac{\eta}{2}}}\|\Delta_{1,Q}f\|_{L^2(\mu)}^2\right)^{\frac{1}{2}}\\
        &\qquad \qquad \leq c\left(\sum_{R\in \mathcal{D}^{\text{tr},2}_2}\|\Delta_{2,R}g\|^2_{L^2(\mu)}\right)^{\frac{1}{2}}\left(\sum_{R\in \mathcal{D}^{\text{tr},2}_2}\sum_{Q\in \mathcal{D}^{\text{tr},1}_1: Q\subset R}\frac{\ell(Q)^{\frac{\eta}{2}}}{\ell(R)^{\frac{\eta}{2}}}\|\Delta_{1,Q}f\|^2_{L^2(\mu)}\right)^{\frac{1}{2}}.
    \end{align*}
    By Lemma \ref{lema_descomposicio_L2}, the first factor is bounded by $c\,\|g\|_{L^2(\mu)}$, which means that the proof will be finished if we are able to show that the second factor, the one with the double sum, is bounded above by $c\,\|f\|_{L^2(\mu)}$. This is indeed the case, because we can write it as
    \begin{align*}
    \sum_{Q\in \mathcal{D}^{\text{tr},1}_1}\|\Delta_{1,Q}f\|^2_{L^2(\mu)}\sum_{R\in \mathcal{D}^{\text{tr},2}_2: R\supset Q}\frac{\ell(Q)^{\frac{\eta}{2}}}{\ell(R)^{\frac{\eta}{2}}} &=  \sum_{Q\in \mathcal{D}^{\text{tr},1}_1}\|\Delta_{1,Q}f\|^2_{L^2(\mu)}\Bigg(\sum_{k\geq 1}\sum_{\substack{R\in \mathcal{D}^{\text{tr},2}_2, R\supset  Q\\ \ell(R)=2^k\ell(Q)}}2^{-\frac{k\eta}{2}}\Bigg)\\
    &\leq\sum_{Q\in \mathcal{D}^{\text{tr},1}_1}\|\Delta_{1,Q}f\|^2_{L^2(\mu)}\left(\sum_{k\geq 1}2^{-\frac{k\eta}{2}}\right)\leq c\,\|f\|_{L^2(\mu)}^2,
    \end{align*}
    where we have used, first, that for each $Q$ and $k$, there is at most one $R\in \mathcal{D}^{\text{tr},2}_2$ that contains $Q$ and $\ell(R)=2^k\ell(Q)$, and lastly, Lemma \ref{lema_descomposicio_L2} for $f$.
        
\end{proof}

Lastly, we need to bound the first term in (\ref{pre_paraproduct}). Namely, we should be able to prove that
$$
\Bigg|\sum_{\substack{Q\in \mathcal{D}^{\text{tr},1}_1, R\in \mathcal{D}^{\text{tr,2}}_2\\ Q\subset R\\ \ell(Q)<2^{-m}\ell(R)\\ R_{Q}\text{ transit}}}c_{R,Q}(g)\langle \Delta_{1,Q}f, K^*_{\Theta}b_2\rangle\Bigg|\leq c\,\|f\|_{L^2(\mu)}\|g\|_{L^2(\mu)}.
$$
When estimating the previous terms, we always used the triangle inequality and we actually bounded the sum of the modulus of the terms. Now, we will not use this technique. Observe that the terms $\langle\Delta_{1,Q}f, K^{*}_{\Theta}b_2\rangle $ do not depend on $R$. This lets us separate the previous sum in two, by considering
\begin{equation}
\sum_{\substack{Q\in \mathcal{D}^{\text{tr},1}_1, R\in \mathcal{D}^{\text{tr,2}}_2\\ Q\subset R\\ \ell(Q)<2^{-m}\ell(R)\\ R_{Q}\text{ transit}}}c_{R,Q}(g)\langle \Delta_{1,Q}f, K^*_{\Theta}b_2\rangle = \sum_{Q\in \mathcal{D}^{\text{tr},1}_1}\langle \Delta_{1,Q}f, K^*_{\Theta}b_2\rangle \Bigg(\sum_{\substack{R\in \mathcal{D}^{\text{tr},2}_2: R\supset Q\\ \ell(R)>2^m\ell(Q)\\R_Q\text{ transit}}}c_{R,Q}(g)\Bigg).\label{separar_Q_i_R}
\end{equation}
Let us now consider the function $\widetilde{g}=g-\Xi_2g$. Fix any good cube $Q\in \mathcal{D}^{\text{tr},1}_1$ with $\ell(Q)<2^{-m}\ell(R^0)$. Since 
$$
\frac{\langle \Xi_2 g\rangle_{R_Q}}{\langle b_2\rangle_{R_Q}} = \frac{1}{\langle b_2\rangle_{R_Q}}\frac{1}{\mu(R_Q)}\int_{R_Q}\frac{\langle g\rangle_{R^0}}{\langle b_2\rangle_{R^0}}b_2\,d\mu = \frac{\langle g\rangle_{R^0}}{\langle b_2\rangle_{R^0}} =  \frac{\langle \Xi_2 g\rangle _{R}}{\langle b_2\rangle_R},
$$
we have that
$$
c_{R,Q}(\widetilde{g}) = \frac{\langle \widetilde{g}\rangle_{R_Q}}{\langle b_2\rangle_{R_Q}}- \frac{\langle \widetilde{g}\rangle_{R}}{\langle b_2\rangle_{R}} = \frac{\langle g\rangle_{R_Q}}{\langle b_2\rangle_{R_Q}}- \frac{\langle g\rangle_{R}}{\langle b_2\rangle_{R}}  = c_{R,Q}(g).\\
$$
Since $\langle \widetilde{g}\rangle_{R^0}=0$, by a telescoping argument we have that
\begin{equation}
\sum_{\substack{R\in \mathcal{D}^{\text{tr},2}_2: R\supset Q\\ \ell(R)>2^m\ell(Q)\\ R_Q\text{ transit}}} c_{R,Q}(g) = \sum_{\substack{R\in \mathcal{D}^{\text{tr},2}_2: R\supset Q\\ \ell(R)>2^m\ell(Q)\\ R_Q\text{ transit}}} c_{R,Q}(\widetilde{g}) = \frac{\langle \widetilde{g}\rangle_{R(Q)}}{\langle b_2\rangle_{R(Q)}},\label{suma_telescopica} 
\end{equation}
where $R(Q)$ denotes the smallest transit cube $R\in \mathcal{D}^{\text{tr},2}_2$ such that $\ell(R)\geq 2^m\ell(Q)$. Such a cube always exists because $Q$ is completely contained in $R^0$. Indeed, since $\ell(R^0)\geq2^m\ell(Q)$ and $R^0$ is transit, we have that
$$
\text{dist}(Q,\partial R^0)>\ell(Q)^{\gamma}\ell(R)^{1-\gamma}>0,
$$
so $Q\cap \partial R^0=\varnothing$. Since $F\cap Q \neq \varnothing$, $Q$ must be entirely contained within $R^0$. Also, in the special case where the smallest such cube $R(Q)$ is precisely $R^0$, the identity (\ref{suma_telescopica}) also holds, since both sides of the equality are zero. Hence, using the expression (\ref{separar_Q_i_R}), 
$$
\sum_{\substack{Q\in \mathcal{D}^{\text{tr},1}_1, R\in \mathcal{D}^{\text{tr,2}}_2\\ Q\subset R\\ \ell(Q)<2^{-m}\ell(R)\\ R_{Q}\text{ transit}}}c_{R,Q}(g)\langle \Delta_{1,Q}f, K^*_{\Theta}b_2\rangle =\sum_{\substack{Q\in \mathcal{D}^{\text{tr},1}_1, R\in \mathcal{D}^{\text{tr,2}}_2\\ Q\subset R\\ \ell(Q)<2^{-m}\ell(R)\\ R_{Q}\text{ transit}}}\frac{\langle \widetilde{g}\rangle_{R(Q)}}{\langle b_2\rangle_{R(Q)}}\langle f, \Delta^*_{1,Q}(K^*_{\Theta}b_2)\rangle.
$$
Let us write the sum above in a more compact form. For a function $\psi\in L^2(\mu)$ such that $\Delta_{2,R}\psi =0$ except for a finite number of cubes $R\in \mathcal{D}^{\text{tr},2}_2$, we consider the following paraproduct,
$$
\Pi_{K_{\Theta}^*b_2}(\psi ) = \sum_{\substack{Q\in \mathcal{D}^{\text{tr},1}_1, Q\text{ good}\\ \ell(Q)<2^{-m}\ell(R^0)}}\frac{\langle \psi\rangle_{R(Q)}}{\langle b_2\rangle_{R(Q)}}\Delta_{1, Q}^*(K^*_{\Theta}b_2),
$$
so that 
\begin{equation}
\sum_{\substack{Q\in \mathcal{D}^{\text{tr},1}_1, R\in \mathcal{D}^{\text{tr,2}}_2\\ Q\subset R\\ \ell(Q)<2^{-m}\ell(R)\\ R_{Q}\text{ transit}}}c_{R,Q}(g)\langle \Delta_{1,Q}f, K^*_{\Theta}b_2\rangle = \langle f, \Pi_{K_{\Theta}^*b_2}(\widetilde{g})\rangle.\label{prodesc_g_tilde}
\end{equation}

\begin{obs}
    The classes of functions 
    \begin{align*}
    \mathcal{T}_1 &= \{\varphi \in L^2(\mu): \Delta_{1,Q}\varphi\neq 0 \,\text{ for finitely many }\, Q\in \mathcal{D}^{\text{tr},1}_1\},\\
    \mathcal{T}_2&= \{\psi \in L^2(\mu): \Delta_{2,R}\psi \neq 0 \,\text{ for finitely many }\,R\in \mathcal{D}^{\text{tr},2}_2\},
    \end{align*}
    are both dense in $L^2(\mu)$. This is justified, in both cases, by Lemma \ref{lema_descomposicio_L2}, which tells us that we can write, for $i=1,2$, and any $f\in L^2(\mu)$,
    $$
    f = \Xi_i f + \sum_{Q\in \mathcal{D}^{\text{tr},i}_i}\Delta_{i,Q}f,
    $$
    with convergence in $L^2(\mu)$. This means that $f$ is the limit in $L^2(\mu)$ of functions of the form 
    $$
    f_k = \Xi_if + \sum_{j\in I_k}\Delta_{i,Q_j}f,
    $$
    where $I_k\subset \mathcal{D}^{\text{tr},i}_i$ is some finite subset. By the properties of $\Delta_{i,Q}$ (d) and (e), from Lemma \ref{lema:propietats_delta}, we have that
    $$
    \Delta_{i,Q}f_k = \begin{cases}
        \Delta_{i, Q_j}f\quad &\text{ if }\, Q = Q_j,\,\text{ for some }\, Q_j\in I_k,\\
        0 &\text{ otherwise},
    \end{cases}
    $$
    that is, $\Delta_{i,Q}f_k\neq 0$ for finitely many $Q\in \mathcal{D}^{\text{tr},i}_i$, equivalently, $f_k\in \mathcal{T}_i$.
\end{obs}

This last remark will be key in proving that the paraproduct that we have just defined extends to a bounded operator in $L^2(\mu)$.

\begin{lemma}
    The paraproduct $\Pi_{K^*_{\Theta}b_2}$ extends to a bounded operator in $L^2(\mu)$.
\end{lemma}

\begin{proof}
    Let $\varphi\in \mathcal{T}_1$ and $\psi\in \mathcal{T}_2$. These subsets, by the previous remark, are dense in $L^2(\mu)$. Then, we have
    $$
    \langle \varphi, \Pi_{K^*_{\Theta}b_2}(\psi)\rangle = \sum_{\substack{Q\in \mathcal{D}^{\text{tr},1}_1,\, Q \text{ good}\\ \ell(Q)<2^{-m}\ell(R)}}\frac{\langle \psi\rangle_{R(Q)}}{\langle b_2\rangle_{R(Q)}}\langle \Delta_{1,Q}\varphi, K_{\Theta}^*b_2\rangle.
    $$
    Of course, we can restrict the cubes $Q$ that appear in the sum above so that $\Delta_{1,Q}\varphi\not\equiv0$. Then, we can write, using that $R(Q)$ is transit,
    \begin{align}
        &|\langle \varphi, \Pi_{K_{\Theta}^*b_2}(\psi)\rangle | \leq \frac{1}{c_{\text{acc}}}\sum_{\substack{Q\in \mathcal{D}^{\text{tr},1}_1, Q\text{ good}\\ \ell(Q)<2^{-m}\ell(R^0)\\
        \Delta_{1,Q}\varphi\not\equiv0}}|\langle\psi\rangle_{R(Q)}|\frac{|\langle \Delta_{1,Q}\varphi, K^*_{\Theta}b_2\rangle|}{\|\Delta_{1,Q}\varphi\|_{L^2(\mu)}}\|\Delta_{1,Q}\varphi\|_{L^2(\mu)}\nonumber\\
        &\quad \leq \frac{1}{c_{\text{acc}}}\Bigg(\sum_{\substack{Q\in \mathcal{D}^{\text{tr},1}_1, Q\text{ good}\\ \ell(Q)<2^{-m}\ell(R^0)\\
        \Delta_{1,Q}\varphi\not\equiv0}} |\langle \psi\rangle_{R(Q)}|^2 \frac{|\langle \Delta_{1,Q}\varphi, K^*_{\Theta}b_2\rangle|^2}{\|\Delta_{1,Q}\varphi\|_{L^2(\mu)}^2}\Bigg)^{\frac{1}{2}}\Bigg(\sum_{Q\in \mathcal{D}^{\text{tr},1}_1}\|\Delta_{1,Q}\varphi\|^2_{L^2(\mu)}\Bigg)^{\frac{1}{2}}.\label{sumatori_alt}
    \end{align}
    Since $\varphi\in L^2(\mu)$, by Lemma \ref{lema_descomposicio_L2}, we know that the last factor on the right-hand side above does not exceed $c\,\|\varphi\|_{L^2(\mu)}$, so to prove the lemma it suffices to show that the first factor is bounded by $c\,\|\psi\|_{L^2(\mu)}$. To do so, we will denote, for $R\in \mathcal{D}^{\text{tr},2}_2$,
    $$
    F(R) = \{Q\in \mathcal{D}^{\text{tr},1}_1: Q\,\text{ is }1\text{-}\mathcal{D}_1\text{-good w.r.t. }\, \mathcal{D}_2, \Delta_{1,Q}\varphi \not\equiv 0,\,  R(Q)= R\}.
    $$
    With this notation, we can rewrite the first factor in (\ref{sumatori_alt}) as
    $$
    \sum_{R\in \mathcal{D}^{\text{tr},2}_2}|\langle \psi\rangle_R|^2 \sum_{Q\in F(R)}\frac{|\langle \Delta_{1,Q}\varphi, K^*_{\Theta}b_2\rangle|^2}{\|\Delta_{1,Q}\varphi\|^2_{L^2(\mu)}},
    $$
    which we want to bound by $c\,\|\varphi\|_{L^2(\mu)}^2$. By the dyadic Carleson embedding theorem \ref{carleson_diadic}, it is enough to check that the coefficients
    $$
    a_R:=\sum_{Q\in F(R)}\frac{|\langle \Delta_{1,Q}\varphi, K^*_{\Theta}b_2\rangle|^2}{\|\Delta_{1,Q}\varphi\|_{L^2}^2}
    $$
    satisfy the condition (\ref{condicio_carleson}). Since $\Delta_{1,Q}=\Delta^2_{1,Q}$, we have
    $$
    a_R=\frac{|\langle \Delta_{1,Q}\varphi, K^*_{\Theta}b_2\rangle |^2}{\|\Delta_{1,Q}\varphi\|_{L^2(\mu)}^2} = \frac{|\langle \Delta_{1,Q}\varphi, \Delta^*_{1,Q}(K^*_\Theta b_2)\rangle |^2}{\|\Delta_{1,Q}\varphi\|_{L^2(\mu)}^2}\leq \|\Delta^*_{1,Q}(K^*_{\Theta}b_2)\|^2_{L^2(\mu)}.
    $$
    Observe that the families $F(R)$ are not overlapping. This is due to the minimality of the side length of each $R(Q)$ and the fact that the cubes $R$ that we consider are from the same dyadic lattice. Moreover, since all the cubes from $F(R)$ are contained in $R$, we have that for every $S\in \mathcal{D}_2$,
    \begin{align*}
        \sum_{R\in \mathcal{D}^{\text{tr},2}_2:\, R\subset S}a_R&\leq \sum_{R\in \mathcal{D}^{\text{tr},2}_2:\, R\subset S}\sum_{Q\in F(R)}\|\Delta^*_{1,Q}(K^*_{\Theta}b_2)\|^2_{L^2(\mu)}\\
        &\leq \sum_{Q\in\mathcal{D}^{\text{tr},1}_1:\, Q\subset S}\|\Delta^*_{1,Q}(\chi_SK^*_{\Theta}b_2)\|^2_{L^2(\mu)}.
    \end{align*}
    By (\ref{desigualtats_descomposicio_transposat}), the last sum is bounded above by $c\,\|\chi_SK^*_{\Theta}b_2\|^2_{L^2(\mu)}$. Moreover, using that $\|K^*_{\Theta}b_2\|_{\infty}\leq c_4$ (recall Lemma \ref{lema_distancia}), we deduce that
    $$
    \sum_{R\in \mathcal{D}^{\text{tr},2}_2:\, R\subset S} a_R\leq c\,\|\chi_SK^*_{\Theta}b_2\|_{L^2(\mu)}^2 \leq c\,\mu(S),
    $$
    which verifies the Carleson condition (\ref{condicio_carleson}).
\end{proof}

\begin{lemma}
    If $f,g\in L^2(\mu)$ are as in Lemma \ref{lemma:good_functions} and very good, then we have
    $$
    |S^{\text{tr}}_{4,1}|\leq c\,\|f\|_{L^2(\mu)}\|g\|_{L^2(\mu)}.
    $$
\end{lemma}

\begin{proof}
    As remarked before, from our previous work, the statement in the lemma is equivalent to proving the estimate
    $$
\Bigg|\sum_{\substack{Q\in \mathcal{D}^{\text{tr},1}_1, R\in \mathcal{D}^{\text{tr,2}}_2\\ Q\subset R\\ \ell(Q)<2^{-m}\ell(R)\\ R_{Q}\text{ transit}}}c_{R,Q}(g)\langle \Delta_{1,Q}f, K^*_{\Theta}b_2\rangle\Bigg|\leq c\,\|f\|_{L^2(\mu)}\|g\|_{L^2(\mu)}.
$$
Remember that by our definition of the numbers $c_{R,Q}(g)$, the expression on the left-hand side of the previous inequality is, as noted in (\ref{prodesc_g_tilde}), $|\langle f, \Pi_{K^*_\Theta b_2}(\widetilde{g})\rangle |$. Using that, by the previous lemma, the paraproduct is bounded in $L^2(\mu)$, we obtain that
$$
|\langle f, \Pi_{K^*_{\Theta}b_2}(\widetilde{g})\rangle |\leq c\,\|f\|_{L^2(\mu)}\|\widetilde{g}\|_{L^2(\mu)} \leq c\,\|f\|_{L^2(\mu)}\|g\|_{L^2(\mu)},
$$
because $\Xi_2$ is bounded in $L^2(\mu)$.
\end{proof}

\subsubsection{Estimate of $S^{\text{\normalfont term}}_{4,1}$}

The sum $S^{\text{term}}_{4,1}$, which we defined in (\ref{s41}), can be rewritten, by changing $R_Q$ to $R$ and $R$ to $\widehat{R}$ (the parent of $R$), respectively, as
\begin{equation}
S^{\text{term}}_{4,1} = \sum_{\substack{Q\in \mathcal{D}^{\text{tr},1}_1, R\in \mathcal{D}^{\text{term},2}_2\\ \ell(Q)\leq 2^{-m}\ell(R)\\ Q\subset R}}\langle K_{\Theta}(\Delta_{1,Q}f), \Delta_{2,\widehat{R}}g\rangle,\label{s41term}
\end{equation}
where, to simplify the notation, we have omitted the condition that $\widehat{R}\in \mathcal{D}^{\text{tr},2}_2$, which we are still considering, since $\Delta_{2,\widehat{R}}$ has not been defined for terminal cubes.

For any cube $R$ as in the sum above, we consider a Whitney decomposition of such a cube. That is, a family $W(R)$ of cubes from $\mathcal{D}_2$ such that
\begin{enumerate}
    \item the family covers $R$, that is $\displaystyle R\subseteq \bigcup_{Q\in W(R)}Q$,
    \item if $Q_1,Q_2\in W(R)$, then either $Q_1=Q_2$ or $Q_1\cap Q_2=\varnothing$,
    \item for every $S\in W(R)$, $\text{dist}(S, \partial R)=\ell(S)$,
    \item for every $S\in W(R)$, $3S\subset R$,
    \item the family of dilated cubes, $\{2S\}_{S\in W(R)}$ has bounded overlap, that is, there exists an absolute constant $c>0$ such that
    $$
    \sum_{S\in W(R)}\chi_{2S}\leq c.
    $$
\end{enumerate}
Moreover, for each $S\in W(R)$, we will consider the following two functions,
$$
g_{R,S} = \chi_{2S}\Delta_{2,\widehat{R}}g, \qquad \widetilde{g}_{R,S} = \chi_{\widehat{R}\setminus 2S}\Delta_{2,\widehat{R}}g,
$$
so that $\Delta_{2,\widehat{R}}g = g_{R,S} + \widetilde{g}_{R,S}$. With all this in mind, we can write, for every $R\in \mathcal{D}^{\text{term},2}_2$, 
\begin{align*}
    &\sum_{\substack{Q\in \mathcal{D}^{\text{tr},1}_1\\ Q\subset R\\ \ell(Q)\leq 2^{-m}\ell(R)}}\langle K_{\Theta}(\Delta_{1,Q}f), \Delta_{2,\widehat{R}}g\rangle = \sum_{S\in W(R)}\sum_{\substack{Q\in \mathcal{D}^{\text{tr},1}_1\\Q\subset R,\, z_Q\in S\\
    \ell(Q)\leq 2^{-m}\ell(R)}} \langle K_{\Theta}(\Delta_{1,Q}f), \Delta_{2,\widehat{R}}g\rangle\\
    &=\sum_{S\in W(R)}\sum_{\substack{Q\in \mathcal{D}^{\text{tr},1}_1\\Q\subset R,\, z_Q\in S\\
    \ell(Q)\leq 2^{-m}\ell(R)}} \langle K_{\Theta}(\Delta_{1,Q}f), \Delta_{2,\widehat{R}}g_{R,S}\rangle + \sum_{S\in W(R)}\sum_{\substack{Q\in \mathcal{D}^{\text{tr},1}_1\\Q\subset R,\, z_Q\in S\\
    \ell(Q)\leq 2^{-m}\ell(R)}} \langle K_{\Theta}(\Delta_{1,Q}f), \Delta_{2,\widehat{R}}\widetilde{g}_{R,S}\rangle.
\end{align*}
The goal of this section is to provide a bound for each of the two sums above, after summing $R\in \mathcal{D}^{\text{term},2}_2$. We will do so in two lemmas, each one dealing with only one of the sums.

\begin{lemma}
    \label{lema:s41term1} If $f,g\in L^2(\mu)$ are as in Lemma \ref{lemma:good_functions}, very good and $m$, from the definition of bad cubes, is chosen big enough, we have that
    $$
    \sum_{R\in \mathcal{D}^{\text{term},2}_2}\sum_{S\in W(R)}\sum_{\substack{Q\in \mathcal{D}^{\text{tr},1}_1\\ Q\subset R,\,z_Q\in S\\\ell(Q)\leq 2^{-m}\ell(R)}} |\langle K_{\Theta}(\Delta_{1,Q}f), \widetilde{g}_{R,S}\rangle| \leq c\,\|f\|_{L^2(\mu)}\|g\|_{L^2(\mu)}.
    $$
\end{lemma}
\begin{proof}
    For the sake of clarity, we will follow three distinct steps.

    \underline{\textbf{Step 1:} recall Lemma \ref{lema:phi_psi}}, in which we considered $Q,R\subset \R^d$ disjoint cubes and $\varphi_Q, \psi_R$ two functions in $L^2(\mu)$  supported on $Q$ and $R$, respectively. If $\text{dist}(Q,\text{supp}(\psi_R))\geq \ell(Q)$, $\int\varphi_Q=0$, $\ell(Q)\leq \ell(R)$ and for some constant $c_6>0$,
    \begin{equation*}
        \text{dist}(Q,\text{supp}(\psi_R))\geq c_6\,\ell(Q)^{\gamma}\ell(R)^{1-\gamma}, 
    \end{equation*}
    then
    \begin{equation}
        |\langle K_{\Theta}\varphi_Q, \psi_R \rangle| \leq c\frac{\ell(Q)^{\frac{\eta}{2}}\ell(R)^{\frac{\eta}{2}}}{D(Q,R)^{n+\eta}}\|\varphi_Q\|_{L^1(\mu)}\|\psi_R\|_{L^1(\mu)}.\label{segona_estimacio_phi_psi}
    \end{equation}
    Of course, we would like to use the Lemma with $\varphi_Q = \Delta_{1,Q}f$ and $\psi_R = \widetilde{g}_{R,S}$. We have that $\Delta_{1,Q}f$ is supported in $Q$, it has zero integral, and $\widetilde{g}_{R,S}$ is supported in $\widehat{R}$. The obstacle now is that $Q$ and $\widehat{R}$ are not disjoint. However, we claim, and will prove later, that in this situation we still have
    \begin{equation}
    \text{dist}(Q, \text{supp}(\widetilde{g}_{R,S}))\geq \text{dist}(Q,\partial 2S)\geq c\,\ell(Q)^{\gamma}\ell(\widehat{R})^{1-\gamma},\label{cosa_que_provare_despres}
    \end{equation}
    which is what we illustrate in Figure \ref{fig:lema_s41_term_2} (in dimension $d=2$, for simplicity).
    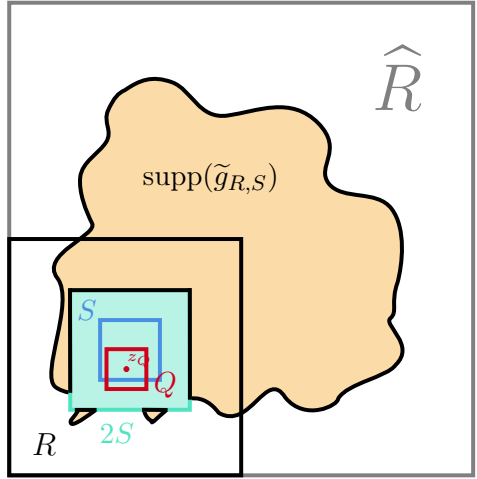
\begin{figure}[h]
    \centering

\tikzset{every picture/.style={line width=0.75pt}} 

\begin{tikzpicture}[x=0.75pt,y=0.75pt,yscale=-1,xscale=1]

\draw [color={rgb, 255:red, 0; green, 0; blue, 0 }  ,draw opacity=1 ][fill={rgb, 255:red, 245; green, 166; blue, 35 }  ,fill opacity=0.4 ][line width=1.5] [line join = round][line cap = round]   (42.47,24.65) .. controls (49.66,19.67) and (56.11,17.06) .. (66.13,20.73) .. controls (76.15,24.41) and (79.31,39.53) .. (94.06,37.38) .. controls (108.82,35.23) and (125.87,27.41) .. (139.39,41.21) .. controls (152.9,55.01) and (137.83,62.93) .. (144.78,71.61) .. controls (151.74,80.29) and (168.76,74.29) .. (176.24,84.21) .. controls (183.72,94.12) and (181.6,112.24) .. (179.21,121.15) .. controls (175.56,131.3) and (179.52,135.84) .. (175.62,139.11) .. controls (171.73,142.37) and (188.8,152.53) .. (178.18,162.56) .. controls (174.15,168.55) and (154.23,171.56) .. (149.45,179.36) .. controls (149.53,182.83) and (132.88,191.1) .. (122.89,178.84) .. controls (113.34,163.67) and (105.38,195.83) .. (89.5,186.53) .. controls (73.71,165.8) and (54.62,196.34) .. (55.88,190.61) .. controls (48.85,187.74) and (53.33,171.49) .. (41.59,177.31) .. controls (39.83,173.51) and (13.46,198.91) .. (15.29,190.37) .. controls (29.3,173.12) and (15.88,179.65) .. (7.68,173.66) .. controls (-0.53,167.67) and (17.68,131.01) .. (8.54,118.96) .. controls (-0.59,106.92) and (29.02,97.56) .. (25.97,91.49) .. controls (22.92,85.42) and (19.14,75.89) .. (20.03,68.26) .. controls (20.93,60.64) and (29.15,56.65) .. (32.9,50.9) .. controls (36.66,45.16) and (31.21,41.02) .. (34,34.8) .. controls (36.8,28.59) and (39.13,26.05) .. (42.47,24.65) -- cycle ;
\draw  [fill={rgb, 255:red, 255; green, 255; blue, 255 }  ,fill opacity=1 ] (15,125) -- (75,125) -- (75,185) -- (15,185) -- cycle ;
\draw  [color={rgb, 255:red, 128; green, 128; blue, 128 }  ,draw opacity=1 ][line width=1.5]  (-15.6,-19.2) -- (217.15,-19.2) -- (217.15,217.9) -- (-15.6,217.9) -- cycle ;
\draw  [line width=1.5]  (-15.6,99.35) -- (100.77,99.35) -- (100.77,217.9) -- (-15.6,217.9) -- cycle ;
\draw  [color={rgb, 255:red, 80; green, 227; blue, 194 }  ,draw opacity=1 ][fill={rgb, 255:red, 80; green, 227; blue, 194 }  ,fill opacity=0.4 ][line width=1.5]  (15,125) -- (75,125) -- (75,185) -- (15,185) -- cycle ;
\draw  [color={rgb, 255:red, 208; green, 2; blue, 27 }  ,draw opacity=1 ][fill={rgb, 255:red, 208; green, 2; blue, 27 }  ,fill opacity=1 ] (42.2,164.6) .. controls (42.2,164.05) and (42.65,163.6) .. (43.2,163.6) .. controls (43.75,163.6) and (44.2,164.05) .. (44.2,164.6) .. controls (44.2,165.15) and (43.75,165.6) .. (43.2,165.6) .. controls (42.65,165.6) and (42.2,165.15) .. (42.2,164.6) -- cycle ;
\draw [line width=1.5]    (15,125) -- (75,125) ;
\draw [line width=1.5]    (14.87,124) -- (15,176.5) ;
\draw [line width=1.5]    (17.31,185.08) -- (28.23,185.08) ;
\draw [line width=1.5]    (75,123.94) -- (75.06,179.94) ;
\draw [line width=1.5]    (51,185) -- (63.84,185.11) ;
\draw  [color={rgb, 255:red, 74; green, 144; blue, 226 }  ,draw opacity=1 ][line width=1.5]  (30,140) -- (60,140) -- (60,170) -- (30,170) -- cycle ;
\draw  [color={rgb, 255:red, 208; green, 2; blue, 27 }  ,draw opacity=1 ][line width=1.5]  (33.2,154.6) -- (53.2,154.6) -- (53.2,174.6) -- (33.2,174.6) -- cycle ;

\draw (165,0.79) node [anchor=north west][inner sep=0.75pt]  [color={rgb, 255:red, 128; green, 128; blue, 128 }  ,opacity=1 ]  {\Huge $\widehat{R}$};
\draw (-5.6,195.3) node [anchor=north west][inner sep=0.75pt]    {\large $R$};
\draw (17,128.4) node [anchor=north west][inner sep=0.75pt]  [color={rgb, 255:red, 74; green, 144; blue, 226 }  ,opacity=1 ]  {$S$};
\draw (28.4,190) node [anchor=north west][inner sep=0.75pt]  [color={rgb, 255:red, 80; green, 227; blue, 194 }  ,opacity=1 ]  {$2S$};
\draw (42,155.8) node [anchor=north west][inner sep=0.75pt]  [font=\tiny,color={rgb, 255:red, 208; green, 2; blue, 27 }  ,opacity=1 ]  {$z_{Q}$};
\draw (55.5,164.4) node [anchor=north west][inner sep=0.75pt]  [color={rgb, 255:red, 208; green, 2; blue, 27 }  ,opacity=1 ]  {$Q$};

\draw (50,60) node [anchor=north west][inner sep=0.75pt]  [opacity=1 ]  {$\text{supp}(\widetilde{g}_{R,S})$};

\end{tikzpicture}

    \caption{The cube $Q$ and $\text{supp}(\widetilde{g}_{R,S})$ are disjoint, whereas $Q$ and $\widehat{R}$ are not. Of course, the side length of $S$ in this picture is bigger than it should be, that is, $\text{dist}(S,\partial R)$, because of property $3$ from the Whitney decomposition. This is due to legibility reasons.}
    \label{fig:lema_s41_term_2}
\end{figure}

    Going over the proof of Lemma \ref{lema:phi_psi}, one can see that the same arguments apply in this case. Hence, we obtain that
    $$
    |\langle K_{\Theta}(\Delta_{1,Q}f), \widetilde{g}_{R,S}\rangle | \leq c\frac{\ell(Q)^{\frac{\eta}{2}}\ell(\widehat{R})^{\frac{\eta}{2}}}{D(Q, \widehat{R})^{n+\eta}}\|\Delta_{1,Q}f\|_{L^1(\mu)}\|\widetilde{g}_{R,S}\|_{L^1(\mu)}.
    $$
    \underline{\textbf{Step 2:} take the sum over the corresponding cubes}, also, using $\|\widetilde{g}_{R,S}\|_{L^1(\mu)}\leq \|\Delta_{2,\widehat{R}}g\|_{L^1(\mu)}$,
    \begin{align*}
        &\sum_{R\in \mathcal{D}^{\text{term},2}_2}\sum_{S\in W(R)}\sum_{\substack{Q\in \mathcal{D}^{\text{tr},1}_1\\ Q\subset R, \, z_Q\in S\\ \ell(Q)\leq 2^{-m}\ell(R)}}|\langle K_{\Theta}(\Delta_{1,Q}f), \widetilde{g}_{R,S}\rangle|\\
        &\quad  \leq c \sum_{R\in \mathcal{D}^{\text{term},2}_2}\sum_{S\in W(R)}\sum_{\substack{Q\in \mathcal{D}^{\text{tr},1}_1\\ Q\subset R, \, z_Q\in S\\ \ell(Q)\leq 2^{-m}\ell(R)}}\frac{\ell(Q)^{\frac{\eta}{2}}\ell(\widehat{R})^{\frac{\eta}{2}}}{D(Q, \widehat{R})^{n+\eta}}\|\Delta_{1,Q}f\|_{L^1(\mu)}\|\widetilde{g}_{R,S}\|_{L^1(\mu)}\\
        &\quad \leq c\sum_{\substack{Q\in \mathcal{D}^{\text{tr},1}_1, R\in \mathcal{D}^{\text{term},2}_2\\ \ell(Q)\leq 2^{-m}\ell(R)\\Q \subset R}}\frac{\ell(Q)^{\frac{\eta}{2}}\ell(\widehat{R})^{\frac{\eta}{2}}}{D(Q,\widehat{R})^{n+\eta}}\mu(Q)^{\frac{1}{2}}\mu(\widehat{R})^{\frac{1}{2}}\|\Delta_{1,Q}f\|_{L^2(\mu)}\|\Delta_{2,\widehat{R}}g\|_{L^2(\mu)}.
    \end{align*}
    Recall that $\widehat{R}\in \mathcal{D}^{\text{tr},2}_2$, so the right-hand side of the last inequality is bounded by the left-hand side of the first line in (\ref{seogona_desigualtat_s11}) from Lemma \ref{lema:s11}. By that same lemma, it is bounded by $c\,\|f\|_{L^2(\mu)}\|g\|_{L^2(\mu)}$.

    \underline{\textbf{Step 3:} proof of (\ref{cosa_que_provare_despres})}. First, we have that 
    $$
    \ell(Q) =  \ell(Q)^{\gamma}\ell(Q)^{1-\gamma} \leq \ell(Q)^{\gamma}2^{(1-\gamma)m}\,\ell(R)^{1-\gamma}.
    $$
    Hence, choosing $m$ big enough we have that
    $$
    8\,\ell(Q)\leq \ell(Q)^{\gamma}\ell(R)^{1-\gamma}\leq \text{dist}(Q,\partial R)\leq \text{dist}(z_Q, \partial R) \leq \ell(S),
    $$
    where we have used that $Q$ is good. Moreover, for any $x\in Q$ and $y\in \partial 2S$, since $z_Q\in S$,
    $$
    |x-y|\geq \|x-y\|_{\infty}\geq \|y-z_Q\|_{\infty} - \|z_Q-x\|_{\infty}\geq \frac{\ell(s)}{2}- \frac{\ell(Q)}{2},
    $$
    from which we deduce that
    $$
    \text{dist}(Q, \text{supp}(\widetilde{g}_{R,S}))\geq \text{dist}(Q,\partial 2S) \geq \frac{\ell(S)-\ell(Q)}{2} \geq \frac{\ell(S)}{4}\geq c\,\ell(Q)^{\gamma}\ell(\widehat{R})^{1-\gamma},
    $$
    which is (\ref{cosa_que_provare_despres}).
\end{proof}

Remark that in this proof we have used properties 1 and 2 of the Whitney decomposition of $R$, namely that the cubes from this decomposition cover $R$ and are pairwise disjoint. In the next Lemma, which will complete the proof of the bound for $S^{\text{term}}_{4,1}$ and hence for $S_4$, we will use, in addition, properties 3, 4 and 5.

\begin{lemma}
    \label{lema:s4g} If $f,g\in L^2(\mu)$ are as in Lemma \ref{lemma:good_functions}, very good and $m$ is chosen big enough, then
    $$
    \sum_{R\in \mathcal{D}^{\text{term},2}_2}\sum_{S\in W(R)}\sum_{\substack{Q\in \mathcal{D}^{\text{tr},1}_1\\Q\subset R, \,z_Q\in S\\ \ell(Q)\leq 2^{-m}\ell(R)}}|\langle K_{\Theta}(\Delta_{1,Q}f), g_{R,S}\rangle | \leq c\,\|f\|_{L^2(\mu)}\|g\|_{L^2(\mu)},
    $$
    and thus $|S^{\text{term}}_{4,1}|\leq c\,\|f\|_{L^2(\mu)}\|g\|_{L^2(\mu)}$.
\end{lemma}

\begin{proof}
    As we have remarked before stating the Lemma, the second statement follows from the previous lemma and the first inequality, which we prove now. 

    By the second inequality in (\ref{cosa_que_provare_despres}) from the previous lemma, if $Q$ and $S\subset R$ are cubes from the sum in the lemma and $z_Q\in S$, it turns out that $\ell(Q)\leq \ell(S)/8$ and $Q\subset 2S$. Moreover, since $R$ is terminal and $2S\subset R$, for any $y\in 2S$ we have
    \begin{align*}
    \Theta(y)\geq \text{dist}(y, \R^d\setminus W_{\mathcal{D}_2}) \geq \text{dist}(y, \partial R) &\geq \text{dist}(2S, \partial R)\\
    &= \text{dist}(S, \partial R) - \frac{\ell(S)}{2} = \frac{\ell(S)}{2}.
    \end{align*}
    Hence, for all $y\in 2S$, $x\in \R^d$, recalling the pointwise estimate from Lemma \ref{desigualtat_maxim_k_tilde},
    $$
    \left|\widetilde{k}_{\Theta}(y,x)\right|\leq \frac{c}{\Theta(y)^n}\leq \frac{c}{\ell(S)^n}.
    $$
    Hence,
    \begin{equation}
    \left|K_{\Theta}^*g_{R,S}(x)\right| \leq c\,\frac{\|g_{R,S}\|_{L^1(\mu)}}{\ell(S)^n} \leq c\, \frac{\|f_{R,S}\|_{L^2(\mu)}\mu(2S)^{\frac{1}{2}}}{\ell(S)^n}.\label{desigualtat_whitney}
    \end{equation}
    Using that $\Delta^2_{1,Q}= \Delta_{1,Q}$, this leads us to
    \begin{align}
        &\sum_{\substack{Q\in \mathcal{D}^{\text{tr},1}_1\\ Q\subset R,\, z_Q\in S\\ \ell(Q)\leq 2^{-m}\ell(R)}}\left|\langle K_{\Theta}(\Delta_{1,Q}f), g_{R,S}\rangle \right| = \sum_{\substack{Q\in \mathcal{D}^{\text{tr},1}_1\\ Q\subset R,\, z_Q\in S\\ \ell(Q)\leq 2^{-m}\ell(R)}}\left|\langle \Delta_{1,Q}f, \Delta^*_{1,Q}(K^*_{\Theta}(g_{R,S}))\right| \label{cs_zq}\\
        &\quad \leq\Bigg(\sum_{Q\in \mathcal{D}^{\text{tr},1}_1: \,Q\subset 2S}\|\Delta_{1,Q}f\|^2_{L^2(\mu)}\Bigg)^{\frac{1}{2}}\Bigg(\sum_{Q\in \mathcal{D}^{\text{tr},1}_1, \, Q\subset 2S}\|\Delta^*_{1,Q}(K_{\Theta}^*(g_{R,S}))\|_{L^2(\mu)}^2\Bigg)^{\frac{1}{2}}.\nonumber
    \end{align}

    By Remark \ref{desigualtats_descomposicio_transposat} and using the inequality in (\ref{desigualtat_whitney}), we get
    \begin{align*}
        \sum_{Q\in \mathcal{D}^{\text{tr},1}_1:\, Q\subset 2S}\|\Delta^*_{1,Q}(K_{\Theta}^*(g_{R,S}))\|^2_{L^2(\mu)} &= \sum_{Q\in \mathcal{D}^{\text{tr},1}_1:\, Q\subset 2S}\|\Delta^*_{1,Q}(\chi_{2S}K^*_{\Theta}(g_{R,S}))\|^2_{L^2(\mu)}\\
        &\leq \|\chi_{2S}K_{\Theta}^*(g_{R,S})\|^2_{L^2(\mu)}\leq c\,\frac{\|g_{R,S}\|^2_{L^2(\mu)}\mu(2S)^2}{\ell(S)^{2n}}.
    \end{align*}
    Note that if for a cube $2S$ appearing in the inequality above there are no transit cubes $Q\in \mathcal{D}^{\text{tr},1}_1$, the summand is zero. Hence, we can assume that any cube $2S$ as above contains a transit cube $Q\in \mathcal{D}^{\text{tr},1}_1$. By Lemma \ref{lema:comportament_cubs_transit}, for such $Q$ we have $\mu(\lambda Q)\leq c_0\,\ell(\lambda Q)^n$, for all $\lambda \geq 1$, which implies that $\mu(2S)\leq c\,\ell(2S)^n$ (possibly with a different constant). Therefore, using this in the last inequality above,
    $$
    \sum_{Q\in \mathcal{D}^{\text{tr},1}_1: \, Q\subset 2S}\|\Delta^*_{1,Q}(K_{\Theta}^*(g_{R,S}))\|^2_{L^2(\mu)} \leq c\,\|g_{R,S}\|^2_{L^2(\mu)}.
    $$
    Applying this in (\ref{cs_zq}) and summing over $S\in W(R)$,
    \begin{align*}
        &\sum_{S\in W(R)}\sum_{\substack{Q\in \mathcal{D}^{\text{tr},1}_1\\ Q\subset R, \, z_Q\in S\\ \ell(Q)\leq 2^{-m}\ell(R)}}|\langle K_{\Theta}(\Delta_{1,Q}f), g_{R,S}\rangle|\\
        &\qquad \qquad \leq c\, \sum_{S\in W(R)}\Bigg[\Bigg(\sum_{Q\in \mathcal{D}^{\text{tr},1}_1: \, Q\subset 2S}\|\Delta_{1,Q}f\|^2_{L^2(\mu)}\Bigg)^{\frac{1}{2}}\|g_{R,S}\|_{L^2(\mu)}\Bigg]\\
        &\qquad \qquad \leq c\Bigg(\sum_{S\in W(R)}\sum_{Q\in \mathcal{D}^{\text{tr},1}_1:\, Q\subset 2S}\|\Delta_{1,Q}f\|^2_{L^2(\mu)}\Bigg)^{\frac{1}{2}}\Bigg(\sum_{S\in W(R)}\|g_{R,S}\|^2_{L^2(\mu)}\Bigg)^{\frac{1}{2}}.
    \end{align*}
    Remember that our goal is to bound this quantity, after summing over $R\in \mathcal{D}^{\text{term},2}_2$, by $c\,\|f\|_{L^2(\mu)}\|g\|_{L^2(\mu)}$. We look at each of the two factors above separately. First, using that the cubes $2S$ have bounded overlap,
    $$
    \sum_{S\in W(R)}\sum_{Q\in \mathcal{D}^{\text{tr},1}_1: \, Q\subset 2S}\|\Delta_{1,Q}f\|^2_{L^2(\mu)}  \leq c \sum_{Q\in \mathcal{D}^{\text{tr},1}_1: \,Q\subset R} \|\Delta_{1,Q}f\|^2_{L^2(\mu)},
    $$
    and also,
    $$
    \sum_{S\in W(R)}\|g_{R,S}\|^2_{L^2(\mu)} = \sum_{S\in W(R)}\int_{2S}|\Delta_{2,\widehat{R}}g|^2\,d\mu \leq c\|\Delta_{2,\widehat{R}}g\|^2_{L^2(\mu)},
    $$
    which means that
    \begin{align*}
    &\sum_{R\in \mathcal{D}^{\text{term},2}_2}\sum_{S\in W(R)}\sum_{\substack{Q\in \mathcal{D}^{\text{tr},1}_1\\ Q\subset R,\, z_Q\in S\\ \ell(Q)\leq 2^{-m}\ell(R)}}|\langle K_{\Theta}(\Delta_{1,Q}f), g_{R,S}\rangle |\\
    &\qquad \qquad \leq c\sum_{R\in \mathcal{D}^{\text{term},2}_2}\Bigg[\Bigg(\sum_{Q\in \mathcal{D}^{\text{tr},1}_1:\, Q\subset R}\|\Delta_{1,Q}f\|^2_{L^2(\mu)}\Bigg)^{\frac{1}{2}}\|\Delta_{2,\widehat{R}}g\|^2_{L^2(\mu)}\Bigg]\\
    &\qquad \qquad \leq c\Bigg(\sum_{R\in \mathcal{D}^{\text{term},2}_2}\sum_{Q\in \mathcal{D}^{\text{tr},1}_1:\,Q\subset R}\|\Delta_{1,Q}f\|^2_{L^2(\mu)}\Bigg)^{\frac{1}{2}}\Bigg(\sum_{R\in \mathcal{D}^{\text{term},2}_2}\|\Delta_{2,\widehat{R}}g\|^2_{L^2(\mu)}\Bigg)^{\frac{1}{2}}.
    \end{align*}
    Since the cubes $R\in \mathcal{D}^{\text{term},2}_2$ with $\widehat{R}\in \mathcal{D}^{\text{tr},2}_2$ are clearly pairwise disjoint, we have, for the first factor on the right-hand side above,
    $$
    \sum_{R\in \mathcal{D}^{\text{term},2}_2}\sum_{Q\in \mathcal{D}^{\text{tr},1}_1, \,Q\subset R}\|\Delta_{1,Q}f\|^2_{L^2(\mu)} \leq \sum_{Q\in \mathcal{D}^{\text{tr},1}_1}\|\Delta_{1,Q}f\|^2_{L^2(\mu)}\leq c\,\|f\|_{L^2(\mu)}^2.
    $$
    Lastly, since each cube has $2^d$ children and $\widehat{R}\in \mathcal{D}^{\text{tr},2}_2$,
    $$
    \sum_{R\in \mathcal{D}^{\text{term},2}_2}\|\Delta_{2,\widehat{R}}g\|^2_{L^2(\mu)}\leq 2^d \sum_{R\in \mathcal{D}^{\text{tr},2}_2}\|\Delta_{2,R}g\|^2_{L^2(\mu)}\leq c\,2^d\,\|g\|^2_{L^2(\mu)},
    $$
    thus concluding the proof.
\end{proof}

As a final remark for this section, let us highlight that, while the Whitney decomposition was key in the proof of the preceding Lemma \ref{lema:s4g}, the role of the suppressed kernel $\widetilde{k}_{\Theta}$ and our hypothesis on the function $\Theta$ were also essential. Indeed, combining all these, we were able to control
$$
\left|\widetilde{k}_{\Theta}(y,x)\right| \leq \frac{c}{\ell(S)^n}, \quad y\in 2S, x\in \R^d,
$$
which is the cornerstone of the argument used above. In contrast, the main inequality in Lemma \ref{lema:s41term1} was
$$
\text{dist}(Q, \text{supp}(\widetilde{g}_{R,S}))\geq c\,\ell(Q)^{\gamma}\ell(\widehat{R})^{1-\gamma}.
$$
This enabled us to apply a version of Lemma \ref{lema:phi_psi}, which depends not on the suppression of the kernel, but rather on the Calderón-Zygmund estimates that it satisfies.

\subsection{Bound for $S_2+ S_3$}
\label{sec:s2s3}

To finish the proof of Lemma \ref{lemma:good_functions}, recall that the two remaining terms from (\ref{4_termes}) are
$$
S_2+ S_3 = \sum_{\substack{Q\in \mathcal{D}^{\text{tr},1}_1, R\in \mathcal{D}^{\text{tr},2}_2\\ Q\cap R = \varnothing\\Q,R\text{ not distant} }}\langle K_{\Theta}(\Delta_{1,Q}f), \Delta_{2,R}g\rangle + \sum_{\substack{Q\in \mathcal{D}^{\text{tr},1}_1, R\in \mathcal{D}^{\text{tr},2}_2\\ Q\cap R \neq \varnothing\\ 2^{-m}\ell(R)\leq \ell(Q)\leq 2^m\ell(R) }}\langle K_{\Theta}(\Delta_{1,Q}f), \Delta_{2,R}g\rangle.
$$
Remember that the functions $f$ and $g$ that we consider are supposed to be good, which means that the cubes $Q$ and $R$ that appear in the sums above are good. Then, by Remark \ref{remark:not_distant_good}, the fact that the cubes $Q,R$ in $S_2$ are disjoint and not distant implies that $2^{-m}\ell(R)<\ell(Q)<2^m\ell(R)$. Hence, we can put together the two sums above and we get that
\begin{align}
|S_2+S_3|&\leq \sum_{\substack{Q\in \mathcal{D}^{\text{tr},1}_1, \,R\in \mathcal{D}^{\text{tr},2}_2\\ Q, R\text{ not distant}\\ 2^{-m}\ell(R)\leq \ell(Q)\leq 2^m\ell(R)}}|\langle K_{\Theta}(\Delta_{1,Q}f), \Delta_{2,R}g\rangle|\nonumber \\
&\leq \sum_{\substack{Q\in \mathcal{D}^{\text{tr},1}_1, \, R\in \mathcal{D}^{\text{tr},2}_2\\ \text{dist}(Q,R)<\max(\ell(Q),\ell(R))\\ 2^{-m}\ell(R)\leq \ell(Q)\leq 2^m\ell(R)}}|\langle K_{\Theta}(\Delta_{1,Q}f), \Delta_{2,R}g\rangle |,\label{s2+s3}
\end{align}
where in the last inequality we have used that if $Q,R$ are not distant,
$$
\text{dist}(Q,R)<\min(\ell(Q), \ell(R))^{\gamma}\max(\ell(Q), \ell(R))^{1-\gamma}\leq \max(\ell(Q), \ell(R)).
$$
The term in (\ref{s2+s3}) is exactly the bound that appears in the statement of Lemma \ref{lemma:good_functions}, thus finishing the proof.

\subsection{Partial bound for the extra term for good functions}
\label{sec:partial_bound}

The aim of this section is to control the additional term that appeared in the bound for the lemma about good functions \ref{lemma:good_functions}, which we recall below,
\begin{equation}
A:=\sum_{\substack{Q\in \mathcal{D}^{\text{tr},1}_1, \, R\in \mathcal{D}^{\text{tr},2}_2\\ \text{dist}(Q,R)<\max(\ell(Q),\ell(R))\\ 2^{-m}\ell(R)\leq \ell(Q)\leq 2^m\ell(R)}}|\langle K_{\Theta}(\Delta_{1,Q}f), \Delta_{2,R}g\rangle |.  \label{suma_A}
\end{equation}
Of course, we would like to find a bound of the type $A\leq c\,\|f\|_{L^2(\mu)}\|g\|_{L^2(\mu)}$. However, this goal is overly optimistic. We are going to propose a splitting of $A$, similarly to what we did in (\ref{4_termes}), and we will encounter a term that would vanish if our operator were antisymmetric. Since this is not our case, we will have to introduce a probabilistic argument to deal with it, which will be done in Section \ref{sec:acotacio_ai}.

For now, notice that in the sum in (\ref{suma_A}), each cube $Q$ can interact with a number of cubes $R$ bounded above by some constant depending only on the dimension $d$ and the number $m$. This is because we have control of the distance between $Q$ and $R$ and of the side length of $R$ in terms of $2^m\ell(Q)$, which is what we illustrate in Figure \ref{fig:distancies_entre_cubs} (for simplicity, in dimension $d=2$).

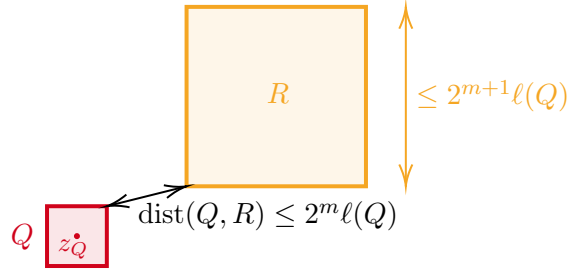
\begin{figure}[h]
    \centering

\tikzset{every picture/.style={line width=0.75pt}} 

\begin{tikzpicture}[x=0.75pt,y=0.75pt,yscale=-1,xscale=1]

\draw  [color={rgb, 255:red, 208; green, 2; blue, 27 }  ,draw opacity=1 ][fill={rgb, 255:red, 208; green, 2; blue, 27 }  ,fill opacity=0.1 ][line width=1.5]  (0,100) -- (30,100) -- (30,130) -- (0,130) -- cycle ;
\draw  [color={rgb, 255:red, 245; green, 166; blue, 35 }  ,draw opacity=1 ][fill={rgb, 255:red, 245; green, 166; blue, 35 }  ,fill opacity=0.1 ][line width=1.5]  (70,0) -- (160,0) -- (160,90) -- (70,90) -- cycle ;
\draw    (31.94,99.51) -- (68.06,90.49) ;
\draw [shift={(70,90)}, rotate = 165.96] [color={rgb, 255:red, 0; green, 0; blue, 0 }  ][line width=0.75]    (10.93,-3.29) .. controls (6.95,-1.4) and (3.31,-0.3) .. (0,0) .. controls (3.31,0.3) and (6.95,1.4) .. (10.93,3.29)   ;
\draw [shift={(30,100)}, rotate = 345.96] [color={rgb, 255:red, 0; green, 0; blue, 0 }  ][line width=0.75]    (10.93,-3.29) .. controls (6.95,-1.4) and (3.31,-0.3) .. (0,0) .. controls (3.31,0.3) and (6.95,1.4) .. (10.93,3.29)   ;
\draw [color={rgb, 255:red, 245; green, 166; blue, 35 }  ,draw opacity=1 ]   (180,2) -- (180,88) ;
\draw [shift={(180,90)}, rotate = 270] [color={rgb, 255:red, 245; green, 166; blue, 35 }  ,draw opacity=1 ][line width=0.75]    (10.93,-3.29) .. controls (6.95,-1.4) and (3.31,-0.3) .. (0,0) .. controls (3.31,0.3) and (6.95,1.4) .. (10.93,3.29)   ;
\draw [shift={(180,0)}, rotate = 90] [color={rgb, 255:red, 245; green, 166; blue, 35 }  ,draw opacity=1 ][line width=0.75]    (10.93,-3.29) .. controls (6.95,-1.4) and (3.31,-0.3) .. (0,0) .. controls (3.31,0.3) and (6.95,1.4) .. (10.93,3.29)   ;
\draw  [color={rgb, 255:red, 208; green, 2; blue, 27 }  ,draw opacity=1 ][fill={rgb, 255:red, 208; green, 2; blue, 27 }  ,fill opacity=1 ] (14,115) .. controls (14,114.45) and (14.45,114) .. (15,114) .. controls (15.55,114) and (16,114.45) .. (16,115) .. controls (16,115.55) and (15.55,116) .. (15,116) .. controls (14.45,116) and (14,115.55) .. (14,115) -- cycle ;

\draw (44.19,96.16) node [anchor=north west][inner sep=0.75pt]    {$\text{dist}( Q,R) \leq 2^{m} \ell ( Q)$};
\draw (183.71,35.11) node [anchor=north west][inner sep=0.75pt]  [color={rgb, 255:red, 245; green, 166; blue, 35 }  ,opacity=1 ]  {$\leq 2^{m+1} \ell ( Q)$};
\draw (4,115) node [anchor=north west][inner sep=0.75pt]  [color={rgb, 255:red, 208; green, 2; blue, 27 }  ,opacity=1 ]  {\small $z_{Q}$};
\draw (-19.6,107.34) node [anchor=north west][inner sep=0.75pt]  [color={rgb, 255:red, 208; green, 2; blue, 27 }  ,opacity=1 ]  {$Q$};
\draw (108.8,36.94) node [anchor=north west][inner sep=0.75pt]  [color={rgb, 255:red, 245; green, 166; blue, 35 }  ,opacity=1 ]  {$R$};

\end{tikzpicture}

    \caption{Once we fix $Q\in \mathcal{D}^{\text{tr},1}_1$, the admissible cubes $R\in \mathcal{D}_2$ cannot be too far away from it.}
    \label{fig:distancies_entre_cubs}
\end{figure}

The same reasoning applies if instead of fixing $Q\in \mathcal{D}^{\text{tr},1}_1$, we fix $R\in \mathcal{D}^{\text{tr},2}_2$. Hence, by Cauchy-Schwarz and Lemma \ref{lema_descomposicio_L2}, one way of proving the desired bound $A\leq c\,\|f\|_{L^2(\mu)}\|g\|_{L^2(\mu)}$ would be to show that for all $Q$ and $R$ in the sum in (\ref{s2+s3}),
$$
|\langle K_{\Theta}(\Delta_{1,Q}f), \Delta_{2,R}g\rangle| \leq c\, \|\Delta_{1,Q}f\|_{L^2(\mu)}\|\Delta_{2,R}g\|_{L^2(\mu)}.
$$
We can separate the right hand side above according to the children of the cubes that appear,
$$
\langle K_{\Theta}(\Delta_{1,Q}f), \Delta_{2,R}g\rangle = \sum_{P\in \mathcal{CH}(Q),\,S\in \mathcal{CH}(R)}\langle K_{\Theta}(\chi_P\Delta_Q f), \chi_S\Delta_R g\rangle.
$$
For each pair $P,S$, we can separate further,
\begin{align}
    \langle K_{\Theta}(\chi_P\Delta_{1,Q}f), \chi_S \Delta_{2,R}g\rangle &= \langle K_{\Theta}(\chi_{P\setminus S}\Delta_{1,Q}f), \chi_S\Delta_{2,R}g\rangle + \langle K_{\Theta}(\chi_{P\cap S}\Delta_{1,Q}f), \chi_{S\setminus P}\Delta_{2,R}g\rangle \nonumber\\
    &\quad + \langle K_{\Theta}(\chi_{P\cap S}\Delta_{1,Q}f), \chi_{P\cap S}\Delta_{2,R}g\rangle. \label{separacio_fills}
\end{align}
Our strategy now will be to distinguish whether the cubes $P,S$ appearing above are transit or terminal. In the first case, we will see that the first two terms can be bounded by $c\,\|\Delta_{1,Q}f\|_2\,\|\Delta_{2,R}g\|_2$, which, as we mentioned above, is enough for our purpose. In the second case, we will have to introduce a probabilistic argument. Let us first deal with the transit case.

\begin{lemma}\label{els_dos_primers_transit}
    Let $Q,R, P$ and $S$ be as in (\ref{separacio_fills}). If $P\in \mathcal{D}^{\text{tr},1}_1$ and $S\in \mathcal{D}^{\text{tr},2}_2$, then
    \begin{align*}
    |\langle K_{\Theta}(\chi_{P\setminus S}\Delta_{1,Q}f), \chi_S\Delta_{2,R}g\rangle|&\leq c\, \|\Delta_{1,Q}f\|_2\,\|\Delta_{2,R}g\|_2,\\
    |\langle K_{\Theta}(\chi_{P\cap S}\Delta_{1,Q}f), \chi_{S\setminus P}\Delta_{2,R}g\rangle|&\leq c\, \|\Delta_{1,Q}f\|_2\,\|\Delta_{2,R}g\|_2.
    \end{align*}
\end{lemma}

\begin{proof}
We will only prove the first inequality because of the symmetry of the conditions on $P$ and $S$: they are both transit cubes, their side lengths are comparable and they are not too far away from each other. 

First, note that if we assume that both $P$ and $S$ are transit, we can write
\begin{equation}\label{coeficients_cp}
\begin{aligned}
\chi_P \Delta_{1,Q}f &= c_P(f)b_1, \qquad c_P(f) = \frac{\langle f\rangle_P}{\langle b_1\rangle_P} - \frac{\langle f\rangle_Q}{\langle b_1\rangle_Q},\\
\chi_S\Delta_{2,R}g&=c_S(g)b_2, \qquad c_S(g) = \frac{\langle g\rangle_S }{\langle b_2\rangle_S} - \frac{\langle g\rangle_R}{\langle b_2\rangle_R}.
\end{aligned}
\end{equation}
Hence, the first term from (\ref{separacio_fills}) becomes
\begin{alignat}{2}
    &&|\langle K_{\Theta}(\chi_{P\setminus S}\Delta_{1,Q}f),\chi_S\Delta_{2,R}g\rangle| &= |\langle K_{\Theta}(\chi_{P\setminus S}\,c_P(f)\,b_1), \chi_S \,c_S\,b_2\rangle|\nonumber\\
    && &\leq c_b^2\,|c_P(f)|\,|c_S(g)|\int_{y\in P\setminus S}\bigg(\int_{x\in S} |\widetilde{k}_{\Theta}(x,y)|\,d\mu(x)\bigg)\,d\mu(y).\label{primer_terme}
\end{alignat}
Using the properties of the suppressed kernel, we see that we need to obtain a bound of the sort
\begin{equation}
I:=\int_{y\in P\setminus S}\bigg(\int_{x\in S} \frac{1}{|x-y|^n+\Theta(x)^n+\Theta(y)^n}\,d\mu(x)\bigg)\,d\mu(y)\leq c\,\mu(P)^{\frac{1}{2}}\,\mu(S)^{\frac{1}{2}}.\label{acotacio_mu_mu}
\end{equation}
Indeed, since the cube $P$ is transit, 
$$
\mu(P)\leq c_\text{acc}\,|\nu_1(P)| = c_{\text{acc}}\bigg|\int_P b_1\,d\mu\,\bigg| \leq c_\text{acc}\,\|\chi_P\,b_1\|_{L^2(\mu)}\,\mu(P)^{\frac{1}{2}},
$$
so we find that
$$
|c_P(f)|\,\mu(P)^{\frac{1}{2}}\leq c_\text{acc}\, |c_P(f)|\,\|\chi_P\,b_1\|_{L^2(\mu)} = c_{\text{acc}}\, \|\chi_P\,\Delta_{1,Q}f\|_{L^2(\mu)}\leq c_{\text{acc}}\,\|\Delta_{1,Q}f\|_{L^2(\mu)},
$$
and the same holds for the coefficient $c_S(g)$ and the function $\Delta_{2,R}g$. So, let us prove (\ref{acotacio_mu_mu}). Recall that since both $Q$ and $R$ are good cubes, we have that for any $\lambda>0$,
\begin{align}
    &\mu\left(\left\{x\in S: \text{dist}(x,\partial P)\leq \lambda \ell(P)\right\}\right)\leq \,\lambda\, M \,\mu(S),\label{primera_condicio_small}\\
    & \mu\left(\{x\in P : \text{dist}(x, \partial S)\leq \lambda \ell(S)\}\right)\leq \lambda \, M\,\mu(P). \label{condicio_small}
\end{align}
We start with the inner integral in (\ref{primer_terme}). Let $y\in P\setminus S$, then
\begin{equation}
\int_{x\in S} \frac{1}{|x-y|^n + \Theta(x)^n + \Theta(y)^n}d\mu(x) \leq \int_{\text{dist}(y,\partial S)\leq |x-y|\leq c'\ell(S)}\frac{1}{|x-y|^n + \Theta(y)^n}d\mu(x).\label{integral_theta}
\end{equation}
Assume first that $\Theta(y)\leq \text{dist}(y, \partial S)$. We claim that the integral on the right-hand side above is bounded by $c\, \log \left(\frac{c'\ell(S)}{\text{dist}(y, \partial S)}\right)$, where $c$ is not necessarily the same constant as in the domain of integration above. Indeed, we can integrate over annuli, calling $r=\text{dist}(y,\partial S)$ and $R = c'\ell(S)$, and we get that
\begin{align}
\int_{r\leq |x-y|\leq R} \frac{1}{|x-y|^n}d\mu(x) = \sum_{k=0}^N\int_{2^k r\leq |x-y|\leq 2^{k+1}r}\frac{1}{|x-y|^n}d\mu(x)\leq \sum_{k=0}^N\frac{\mu(B(x, 2^{k+1}r))}{(2^k r)^n},\label{anells}
\end{align}
where $N$ is the biggest integer such that $2^Nr\leq R$, from which we see that $N \leq \log_2\left(\frac{R}{r}\right)$. Using our assumption,
$$
\text{dist}(y, \partial S)\geq \Theta(y) \geq \text{dist}(y, \R^d\setminus H_{\mathcal{D}_1}),
$$
so we see that $\mu(B(y, 2^{k+1}r))\leq c_0(2^{k+1}r)^n$ for each $k$, from which we get that the right hand side in the last inequality from (\ref{anells}) is bounded above by 
$$
c\left(\log_2\left(\frac{R}{r}\right)+1\right)\leq c\log \left(\frac{c' \ell(S)}{\text{dist}(y, \partial S)}\right),
$$
as wished. Assume now that $\Theta(y)\geq \text{dist}(y, \partial S)$. In this case, we can separate the integral in (\ref{integral_theta}) as
\begin{align*}
    \int_{\subalign{&\,x\in S, \\ &\text{dist}(y, \partial S)\leq |x-y|}}\frac{1}{|x-y|^n + \Theta(y)^n}d\mu(x) &= \int_{\subalign{&\,x\in S,\\& \text{dist}(y,\partial S) \leq |x-y|\leq \Theta(y)}}\frac{1}{|x-y|^n + \Theta(y)^n}d\mu(x) \\
    &\qquad + \int_{\subalign{&\,x\in S,\\ &\Theta(y)<|x-y|}} \frac{1}{|x-y|^n+\Theta(y)^n}d\mu(x):= J_1 + J_2.
\end{align*}
For $J_1$ we have an easy bound,
$$
J_1 \leq \int_{|x-y|\leq \Theta(y)}\frac{1}{\Theta(y)^n}d\mu(x) \leq \frac{\mu(B(y,\Theta(y)))}{\Theta(y)^n}\leq c_0,
$$
where we have used that, since $\Theta(y)\geq \text{dist}(y, \R^d\setminus H_{\mathcal{D}_1})$, the ball appearing above is not contained in $H_{\mathcal{D}_1}$. For $J_2$, we can obtain, through analogous techniques, a similar bound to that obtained for the case that $\Theta(y)\leq \text{dist}(y,\partial S)$. Indeed,
\begin{align*}
    J_2 =\int_{\subalign{\, &x\in S,\\ &\Theta(y)<|x-y|}}\frac{1}{|x-y|^n + \Theta(y)^n}d\mu(x) &\leq \int_{\Theta(y)<|x-y|\leq c'\ell(S)} \frac{1}{|x-y|^n}d\mu(x)\\
    &= \sum_{k=0}^N \int_{2^k\Theta(y) < |x-y|\leq 2^{k+1}\Theta(y)} \frac{1}{|x-y|^n}d\mu(x)\\
    &\leq \sum_{k=0}^{N}\frac{\mu(B(y, 2^{k+1}\Theta(y)))}{(2^k\Theta(y))^n}\leq (N+1)2^n c_0.
\end{align*}
Now, the bound we obtain for $N$ is slightly different than before, but we can retrieve the same bound easily,
$$
N \leq \log_2\left(\frac{c'\ell(S)}{\Theta(y)}\right) \leq c\log\left(\frac{c'\ell(S)}{\text{dist}(y, \partial S)}\right).
$$
Hence, the integral on the left-hand side of (\ref{acotacio_mu_mu}) is bounded above by
\begin{align*}
c\int_{P\setminus S} \log\left(\frac{c'\ell(S)}{\text{dist}(y, \partial S)}\right) \,d\mu(y) &= c\sum_{k\geq 0} \int_{\subalign{&\,y\in P,\\ &2^{k-1}\ell(P)\leq \text{dist}(y, \partial S)\leq 2^{-k}\ell(P)}} \log\left(\frac{c'\ell(S)}{\text{dist}(y, \partial S)}\right)\,d\mu(y)\\
& \leq c\sum_{k\geq 0} k \mu\left(\left\{y\in P : \text{dist}(y, \partial S)\leq 2^{-k}\ell(P)\right\}\right)\\
&\leq c\sum_{k\geq 0} k \mu\left(\left\{y\in P : \text{dist}(y, \partial S)\leq c\,2^{-k}\ell(S)\right\}\right)\\
&\leq c M\mu(P)\sum_{k\geq 0} k 2^{-k} \leq c M\mu(P),
\end{align*}
where we have used that $\ell(P)$ and $\ell(S)$ are comparable and the inequality (\ref{condicio_small}), which we obtained from the fact that $Q = \widehat{P}$ is good with respect to $\mathcal{D}_2$. However, this is not quite the bound that we were aiming for on the right-hand side of (\ref{acotacio_mu_mu}): we need both $\mu(P)$ and $\mu(S)$ to appear instead of only one of them. To obtain such a bound, we are going to use a similar strategy to what we have done now to estimate the same integral, but in terms of $\mu(S)$. Afterwards, we will take an average of these estimates, which will yield (\ref{acotacio_mu_mu}).

Let us write
\begin{align*}
    I &= \int_{x\in S \cap P}\bigg(\int_{y\in P\setminus S} \frac{1}{|x-y|^n+\Theta(x)^n+\Theta(y)^n}\,d\mu(y)\bigg)\,d\mu(x) \\
    &\qquad + \int_{x\in S \setminus P}\bigg(\int_{y\in P\setminus S} \frac{1}{|x-y|^n+\Theta(x)^n+\Theta(y)^n}\,d\mu(y)\bigg)\,d\mu(x) := J_3 + J_4.
\end{align*}
For $J_3$, via arguments analogous to those we used to obtain the first bound for $I$, we have that if $s\in S\cap P$, then
$$
\int_{y\in P\setminus S}\frac{1}{|x-y|^n+\Theta(x)^n + \Theta(y)^n}d\mu(y) \leq c\log\left(\frac{c'\ell(P)}{\text{dist}(x,\partial P)}\right).
$$
Thus,
$$
J_3 \leq c\,\int_{S\cap P} \log\left(\frac{c'\ell(P)}{\text{dist}(x,\partial P)}\right)\,d\mu(x)\leq c\mu(S\cap P)^{\frac{1}{2}}\left(\int_{S\cap P}\log\left(\frac{c'\ell(P)}{\text{dist}(x, \partial P)}\right)^2\,d\mu(x)\right)^{\frac{1}{2}}. 
$$
Using the small boundary condition, this time from (\ref{primera_condicio_small}), the last integral can be estimated by
\begin{align*}
    \int_{S\cap P}\log\left(\frac{c'\ell(P)}{\text{dist}(x,\partial P)}\right)^2 d\mu(x) &\leq \int_S\log\left(\frac{c'\ell(P)}{\text{dist}(x,\partial P)}\right)^2 d\mu(x) \\
    & = \sum_{k\geq 0}\int_{\subalign{&\, y\in S,\\ &2^{-k-1}\ell(S)\leq \text{dist}(x,\partial P)\leq 2^{-k}\ell(S)}}\log\left(\frac{c'\ell(P)}{\text{dist}(x,\partial P)}\right)^2 d\mu(x)\\
    &\leq c\sum_{k\geq 0} k^2 \mu\left(\left\{ x\in S : \text{dist}(x,\partial P)\leq 2^{-k}\ell(S)\right\}\right)\\
    &\leq c\sum_{k\geq 0}k^2 \mu\left(\left\{ x\in S : \text{dist}(x,\partial P)\leq c2^{-k}\ell(P)\right\}\right)\\
    &\leq cM\mu(S)\sum_{k\geq 0}k^22^{-k} \leq c M \mu(S).
\end{align*}
So we see that $J_3\leq cM \mu(P)^{\frac{1}{2}}\mu(S)^{\frac{1}{2}}$. Let us turn our attention to $J_4$. Arguing analogously as before, one bounds the inner integral, for $x\in S\setminus P$, by the logarithmic term involving $\ell(P)$ and $\text{dist}(x,\partial P)$. Afterwards, the small boundary condition (\ref{primera_condicio_small}) yields that $J_4\leq  cM \mu(S)$. 

Putting together all these estimates, we have found that both $I\leq cM \mu(P)$ and $I\leq c\mu(P)^{\frac{1}{2}}\mu(S)^{\frac{1}{2}} + c\mu(S)$. Distinguishing whether $\mu(P)\leq \mu(S)$ or the converse inequality is true, we obtain the desired bound (\ref{acotacio_mu_mu}), thus finishing the proof.
\end{proof}

The last term form (\ref{separacio_fills}) turns out to be more complicated, because the small boundary condition is not helpful when dealing with the intersections. Assuming that both $P$ and $S$ are transit, we can write (see (\ref{coeficients_cp})) 
$$
\langle K_{\Theta}(\chi_{P\cap S}\Delta_{1,Q}f), \chi_{P\cap S}\Delta_{2,R}g\rangle = c_P(f)c_S(g)\,\langle K_\Theta(\chi_{P\cap S}\,b_1), \chi_{P\cap S}\,b_2\rangle.
$$
Note that if our operator $K_\Theta$ were antisymmetric and instead of two functions $b_1$ and $b_2$ we were working with only one function, $b$, this problematic term would vanish.

Let us denote $\Delta = P\cap S$. Let $x_\Delta$ be its center, so we can write it as
$$
\Delta = x_\Delta +  \prod_{j=1}^d\bigg[-\frac{a_j}{2}, \frac{a_j}{2}\bigg],
$$
where $a_j>0$ for each $j = 1,\dots, d$. Now, for a (small) fixed number $0<\varepsilon_a$, we denote $\ell_{\varepsilon_a}= \varepsilon_a\,\min(\ell(P), \ell(S))$ and consider a smaller concentric hyperrectangle $\widetilde{\Delta}_{\varepsilon_a}$, defined by
$$
\widetilde{\Delta}_{\varepsilon_a} = x_\Delta + \prod_{j=1}^d[-b_j, b_j], \quad \text{where }\, b_j = \frac{a_j}{2}-\ell_{\varepsilon_a},
$$
(see Figure \ref{fig:Delta_epsilon_a}, which, for simplicity, illustrates the situation in dimension $d=2$).

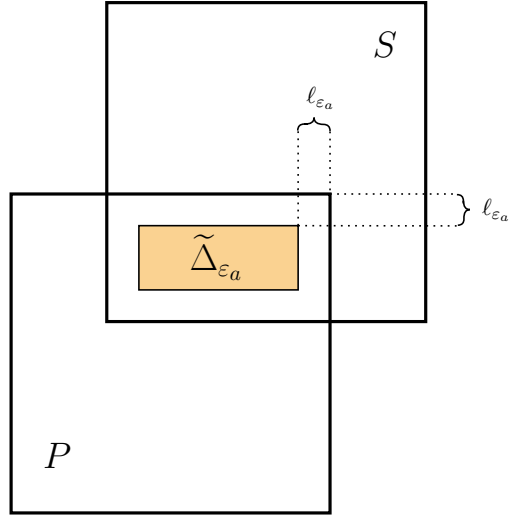
\begin{figure}[h]
    \centering

\tikzset{every picture/.style={line width=0.75pt}} 

\scalebox{0.8}{
\begin{tikzpicture}[x=0.75pt,y=0.75pt,yscale=-1,xscale=1]

\draw   (60,120) -- (200,120) -- (200,200) -- (60,200) -- cycle ;
\draw  [fill={rgb, 255:red, 245; green, 166; blue, 35 }  ,fill opacity=0.5 ] (80.14,139.86) -- (180,139.86) -- (180,180.14) -- (80.14,180.14) -- cycle ;
\draw  [line width=1.5]  (60,0) -- (260,0) -- (260,200) -- (60,200) -- cycle ;
\draw  [line width=1.5]  (0,120) -- (200,120) -- (200,320) -- (0,320) -- cycle ;
\draw  [dash pattern={on 0.84pt off 2.51pt}]  (200,120) -- (280,120) ;
\draw  [dash pattern={on 0.84pt off 2.51pt}]  (180,139.86) -- (280,140) ;
\draw   (280.14,139.57) .. controls (282.77,139.49) and (284.04,138.14) .. (283.97,135.51) -- (283.97,135.51) .. controls (283.86,131.76) and (285.11,129.84) .. (287.74,129.76) .. controls (285.11,129.84) and (283.74,128) .. (283.63,124.25)(283.68,125.94) -- (283.63,124.25) .. controls (283.55,121.62) and (282.2,120.35) .. (279.57,120.43) ;
\draw  [dash pattern={on 0.84pt off 2.51pt}]  (180,139.86) -- (180,80) ;
\draw  [dash pattern={on 0.84pt off 2.51pt}]  (200,140) -- (200,80.14) ;
\draw   (200,80.6) .. controls (200.05,77.82) and (198.68,76.41) .. (195.9,76.36) -- (195.9,76.36) .. controls (191.93,76.29) and (189.97,74.87) .. (190.02,72.09) .. controls (189.97,74.87) and (187.96,76.22) .. (183.99,76.15)(185.78,76.18) -- (183.99,76.15) .. controls (181.21,76.1) and (179.8,77.47) .. (179.75,80.25) ;

\draw (110.6,144) node [anchor=north west][inner sep=0.75pt]  [font=\LARGE]  {$\widetilde{\Delta}_{\varepsilon _{a}}$};
\draw (18.67,274.4) node [anchor=north west][inner sep=0.75pt]  [font=\LARGE]  {$P$};
\draw (225.33,17.07) node [anchor=north west][inner sep=0.75pt]  [font=\LARGE]  {$S$};
\draw (293.94,121.62) node [anchor=north west][inner sep=0.75pt]    {$\ell _{\varepsilon _{a}}$};
\draw (183.69,51.37) node [anchor=north west][inner sep=0.75pt]    {$\ell _{\varepsilon _{a}}$};

\end{tikzpicture}}

    \caption{The set $\Delta_{\varepsilon_a}$ inside $\Delta$.}
    \label{fig:Delta_epsilon_a}
\end{figure}

Now, we claim that for each $\Delta$ and $\widetilde{\Delta}_{\varepsilon_a}$, there are numbers $\gamma_j\in [b_j, b_j+ \ell_{\varepsilon_a}/2]$ such that the hyperrectangle
$$
{\Delta}_{\varepsilon_a} := x_{\Delta} + \prod_{j=1}^d[-\gamma_j, \gamma_j],
$$
satisfies that there is some $t_{\Delta, \varepsilon_a}>0$, depending on $\varepsilon_a$ and $\Delta$, such that for any $\lambda>0$,
\begin{equation}
\mu\big(\big\{x\in \Delta : \text{dist}(x,\partial\, {\Delta}_{\varepsilon_a})\leq \lambda \,\text{diam}({\Delta}_{\varepsilon_a}\big)\big\}\big) \leq \, t_{\Delta,{\varepsilon_a}}\lambda \mu(\Delta).\label{t-small}
\end{equation}
The existence of such a hyperrectangle $\widetilde{\Delta}_{\varepsilon_a}\subset \Delta_{\varepsilon_a} \subset \Delta$ is justified by the following lemma, which is proved by a simple modification of Lemma 9.43 in \cite{llibre_xavi}. We omit the details.
\begin{lemma}
    Let $\mu$ be a Radon measure on $\R^d$ and $a>0$. Let $t_a$ be some constant big enough (depending only on $a$ and $d$) and let $R\subset \R^d$ be any fixed hyperrectangle. Then, there exists a concentric hyperrectangle $R'$ with $R\subset R'\subset aR$ such that for any $\lambda>0$,
    $$
    \mu\left(\left\{x\in aR: \text{dist}(x,\partial R')\leq \lambda \text{diam}(R')\right\}\right) \leq t_a\,\lambda\, \mu(aR).
    $$
\end{lemma}
With ${\Delta}_{\varepsilon_a}$ satisfying (\ref{t-small}), we consider two more subsets of $\Delta$ defined by the following relations. First, consider $\widetilde{P}$ a cube, concentric with $P$ and contained inside $P$, such that its boundary hyperplanes inside $\Delta$ go along the boundary hyperplanes of ${\Delta}_{\varepsilon_a}$. Then, we call $S_\partial := \Delta \setminus \widetilde{P}$, see Figure \ref{fig:definicio_S_partial}.

\begin{figure}[h]
    \centering

 
\tikzset{
pattern size/.store in=\mcSize, 
pattern size = 5pt,
pattern thickness/.store in=\mcThickness, 
pattern thickness = 0.3pt,
pattern radius/.store in=\mcRadius, 
pattern radius = 1pt}
\makeatletter
\pgfutil@ifundefined{pgf@pattern@name@_om93exylf}{
\pgfdeclarepatternformonly[\mcThickness,\mcSize]{_om93exylf}
{\pgfqpoint{0pt}{0pt}}
{\pgfpoint{\mcSize+\mcThickness}{\mcSize+\mcThickness}}
{\pgfpoint{\mcSize}{\mcSize}}
{
\pgfsetcolor{\tikz@pattern@color}
\pgfsetlinewidth{\mcThickness}
\pgfpathmoveto{\pgfqpoint{0pt}{0pt}}
\pgfpathlineto{\pgfpoint{\mcSize+\mcThickness}{\mcSize+\mcThickness}}
\pgfusepath{stroke}
}}
\makeatother
\tikzset{every picture/.style={line width=0.75pt}} 

\begin{tikzpicture}[x=0.75pt,y=0.75pt,yscale=-1,xscale=1]

\draw  [color={rgb, 255:red, 65; green, 117; blue, 5 }  ,draw opacity=1 ][fill={rgb, 255:red, 184; green, 233; blue, 134 }  ,fill opacity=0.4 ][dash pattern={on 1.69pt off 2.76pt}][line width=1.5]  (10,10) -- (190,10) -- (190,190) -- (10,190) -- cycle ;
\draw  [line width=1.5]  (0,0) -- (200,0) -- (200,200) -- (0,200) -- cycle ;
\draw  [line width=1.5]  (80,0) -- (200,0) -- (200,80) -- (80,80) -- cycle ;
\draw  [color={rgb, 255:red, 189; green, 16; blue, 224 }  ,draw opacity=1 ][fill={rgb, 255:red, 189; green, 16; blue, 224 }  ,fill opacity=0.35 ][line width=1.5]  (90,10) -- (190,10) -- (190,70) -- (90,70) -- cycle ;
\draw  [color={rgb, 255:red, 208; green, 2; blue, 27 }  ,draw opacity=1 ][pattern=_om93exylf,pattern size=6pt,pattern thickness=0.75pt,pattern radius=0pt, pattern color={rgb, 255:red, 208; green, 2; blue, 27}] (80,0) -- (200,0) -- (200,80) -- (190,80) -- (190,10) -- (80,10) -- cycle ;

\draw (127.6,24.8) node [anchor=north west][inner sep=0.75pt]   {\Large ${\Delta }_{\varepsilon _{a}}$};
\draw (84.4,-25) node [anchor=north west][inner sep=0.75pt]    {\Large $\Delta $};
\draw (209.2,182) node [anchor=north west][inner sep=0.75pt]    {\Large $P$};
\draw (162,157.4) node [anchor=north west][inner sep=0.75pt]  [color={rgb, 255:red, 65; green, 117; blue, 5 }  ,opacity=1 ]  {\Large $\widetilde{P}$};
\draw (210,20) node [anchor=north west][inner sep=0.75pt]  [color={rgb, 255:red, 208; green, 2; blue, 27 }  ,opacity=1 ]  {\Large $S_{\partial }$};

\end{tikzpicture}

    \caption{The definition of $S_\partial$.}
    \label{fig:definicio_S_partial}
\end{figure}
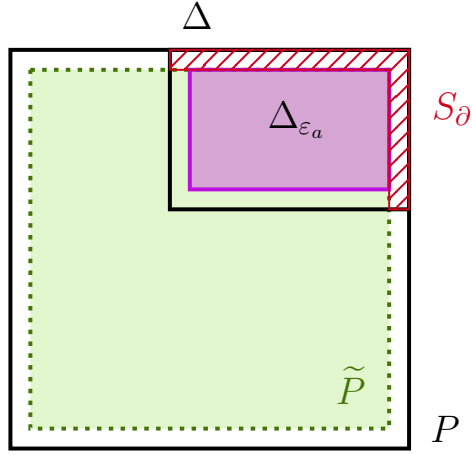

Analogously, we consider another cube, $\widetilde{S}\subset S$, concentric with $S$, asking for the same condition concerning its boundary hyperplanes and those of ${\Delta}_{\varepsilon_a}$, and now we define $P_\partial := \Delta \setminus \widetilde{S}$. Putting all this together, we can write
$$
\Delta = {\Delta}_{\varepsilon_a} \cup P_\partial \cup S_\partial,
$$
where $S_\partial$ and $P_\partial$ are both disjoint from ${\Delta}_{\varepsilon_a}$, but $P_\partial \cap S_\partial \neq \varnothing$ (see Figure \ref{fig:separacio_interseccio}).

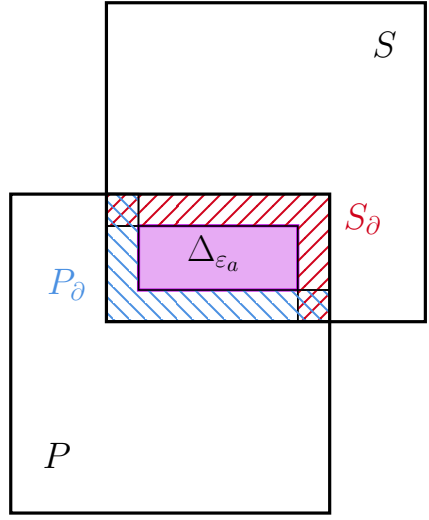
\begin{figure}[h]

\centering

 
\tikzset{
pattern size/.store in=\mcSize, 
pattern size = 5pt,
pattern thickness/.store in=\mcThickness, 
pattern thickness = 0.3pt,
pattern radius/.store in=\mcRadius, 
pattern radius = 1pt}
\makeatletter
\pgfutil@ifundefined{pgf@pattern@name@_k31oeb2z9}{
\pgfdeclarepatternformonly[\mcThickness,\mcSize]{_k31oeb2z9}
{\pgfqpoint{0pt}{0pt}}
{\pgfpoint{\mcSize+\mcThickness}{\mcSize+\mcThickness}}
{\pgfpoint{\mcSize}{\mcSize}}
{
\pgfsetcolor{\tikz@pattern@color}
\pgfsetlinewidth{\mcThickness}
\pgfpathmoveto{\pgfqpoint{0pt}{0pt}}
\pgfpathlineto{\pgfpoint{\mcSize+\mcThickness}{\mcSize+\mcThickness}}
\pgfusepath{stroke}
}}
\makeatother

 
\tikzset{
pattern size/.store in=\mcSize, 
pattern size = 5pt,
pattern thickness/.store in=\mcThickness, 
pattern thickness = 0.3pt,
pattern radius/.store in=\mcRadius, 
pattern radius = 1pt}
\makeatletter
\pgfutil@ifundefined{pgf@pattern@name@_pfjxw0hn0}{
\pgfdeclarepatternformonly[\mcThickness,\mcSize]{_pfjxw0hn0}
{\pgfqpoint{0pt}{-\mcThickness}}
{\pgfpoint{\mcSize}{\mcSize}}
{\pgfpoint{\mcSize}{\mcSize}}
{
\pgfsetcolor{\tikz@pattern@color}
\pgfsetlinewidth{\mcThickness}
\pgfpathmoveto{\pgfqpoint{0pt}{\mcSize}}
\pgfpathlineto{\pgfpoint{\mcSize+\mcThickness}{-\mcThickness}}
\pgfusepath{stroke}
}}
\makeatother
\tikzset{every picture/.style={line width=0.75pt}} 

\scalebox{0.8}{
\begin{tikzpicture}[x=0.75pt,y=0.75pt,yscale=-1,xscale=1]

\draw   (60,120) -- (200,120) -- (200,200) -- (60,200) -- cycle ;
\draw    (60,140) -- (180,140) ;
\draw    (180,140) -- (180,200) ;
\draw    (80.14,119.71) -- (79.86,180.14) ;
\draw    (79.86,180.14) -- (200.14,180.14) ;
\draw  [color={rgb, 255:red, 189; green, 16; blue, 224 }  ,draw opacity=1 ][fill={rgb, 255:red, 189; green, 16; blue, 224 }  ,fill opacity=0.35 ][line width=1.5] (80.14,139.86) -- (180,139.86) -- (180,180.14) -- (80.14,180.14) -- cycle ;
\draw  [pattern=_k31oeb2z9,pattern size=6pt,pattern thickness=0.75pt,pattern radius=0pt, pattern color={rgb, 255:red, 208; green, 2; blue, 27}] (200,120) -- (200,200) -- (180,200) -- (180,140) -- (60,140) -- (60,120) -- cycle ;
\draw  [pattern=_pfjxw0hn0,pattern size=6pt,pattern thickness=0.75pt,pattern radius=0pt, pattern color={rgb, 255:red, 74; green, 144; blue, 226}] (60,119.57) -- (80.14,119.71) -- (80.14,180.14) -- (200.14,180.14) -- (200,199.57) -- (60,199.57) -- cycle ;
\draw  [line width=1.5]  (60,0) -- (260,0) -- (260,200) -- (60,200) -- cycle ;
\draw  [line width=1.5]  (0,120) -- (200,120) -- (200,320) -- (0,320) -- cycle ;

\draw (109.1,145) node [anchor=north west][inner sep=0.75pt]  [font=\LARGE]  {$\Delta _{\varepsilon_a }$};
\draw (18.67,274.4) node [anchor=north west][inner sep=0.75pt]  [font=\LARGE]  {$P$};
\draw (225.33,17.07) node [anchor=north west][inner sep=0.75pt]  [font=\LARGE]  {$S$};
\draw (21.33,165.4) node [anchor=north west][inner sep=0.75pt]  [font=\LARGE,color={rgb, 255:red, 74; green, 144; blue, 226 }  ,opacity=1 ]  {$P_{\partial }$};
\draw (207.33,124.4) node [anchor=north west][inner sep=0.75pt]  [font=\LARGE,color={rgb, 255:red, 208; green, 2; blue, 27 }  ,opacity=1 ]  {$S_{\partial }$};

\end{tikzpicture}}

\caption{The sets $P_\partial, S_\partial$ and $\Delta_{\varepsilon_a}$ inside $\Delta$.}

\label{fig:separacio_interseccio}

\end{figure}

This enables us to separate the additional term even further,
\begin{equation}
\langle K_\Theta(\chi_{P\cap S}\,b_1), \chi_{P\cap S}\,b_2\rangle = \langle K_{\Theta}(\chi_\Delta b_1), \chi_{\Delta_{\varepsilon_a}}b_2\rangle + \langle K_{\Theta}(\chi_{\Delta}b_1), \chi_{\Delta\setminus \Delta_{\varepsilon_a}}b_2\rangle.\label{proposta_xavi}
\end{equation}
In the following lemma, we will obtain a bound for the first term, after multiplication by the appropriate coefficients, which will depend on $\varepsilon_a$. This dependence will cause no problems in the final argument, since we will fix $\varepsilon_a$ and it will simply be a constant.
\begin{lemma}\label{lema:arreglo_xavi}
    Let $P,S,Q,R$ and $\Delta, \Delta_{\varepsilon_a}$ be as above. Then, we have,
    $$
    |\langle K_{\Theta}(\chi_\Delta \Delta_{1,Q}f), \chi_{\Delta_{\varepsilon_a}}\Delta_{2,R}g\rangle | \leq c_a \|\Delta_{1,Q}f\|_{L^2(\mu)}\|\Delta_{2,R}g\|_{L^2(\mu)},
    $$
    where $c_a$ is a constant that depends on $\varepsilon_a$ and on the constant $\alpha$ that we fixed in the last paragraph from Section \ref{sec:exceptional_set}, but neither on $P$ nor on $S$.
\end{lemma}

\begin{obs}
    Before proving the lemma, we make two easy but useful observations. If $\Theta$ is as in Lemma \ref{lemma:good_functions}, we have that
    $$
    \sup_{r\geq 2\Theta(x)}\frac{\mu(B(x,r))}{r^n}\leq c_0, \quad \text{and}\quad \sup_{\delta \geq 2\Theta(x)}|T_{\delta}(\nu_1)(x)|\leq \alpha,
    $$
    where $\alpha$ is the number that we fix in the definition of the exceptional set $S$, in Section \ref{sec:exceptional_set}.

    To prove the first assertion, we have to use the property that $\Theta(x)\geq \text{dist}(x,\R^d\setminus H_{\mathcal{D}_1})$, and distinguish two cases.
    \begin{itemize}
        \item If $\Theta(x)>0$, for $r>\Theta(x)$ we will have that $r> \text{dist}(x, \R^d\setminus H_{\mathcal{D}_1})$. Hence, $B(x,r)\not\subset H_{\mathcal{D}_1}\subset \cap_{w\in \R^d}H_{\mathcal{D}(w)}$. By hypothesis (a) from Theorem \ref{TEOREMA}, it follows that $\mu(B(x,r))\leq c_0r^n$.
        \item If $\Theta(x)=0$, we deduce that $\text{dist}(x, \R^d\setminus H_{\mathcal{D}_1})$, so for any $r>0$, $B(x,r)\cap (\R^d\setminus H_{\mathcal{D}_1})\neq \varnothing$ and we have the same control for $\mu(B(x,r))$.
    \end{itemize}
    Now we turn our attention to the second assertion. We have to show that for any $\delta\geq 2\Theta(x)$, we have that $|T_{\delta}(\nu_1)(x)|\leq \alpha$. Again, we will separate two cases. First, if $x\not\in S^1_0$, the inequality that we want to prove is clear. If $x\in S^1_0$, we can use the fact that $\Theta(x)\geq \text{dist}(x, \R^d\setminus S)\geq e_1(x)$. In this case, if $\Theta(x)=0$, then $e_1(x)=0$ as well and so obviously $|T_{\delta}(\nu_1)(x)|\leq \alpha$. If $\Theta(x)>0$, we have that $\delta >e_1(x)$, and so by definition of supremum it holds that $|T_{\delta}(\nu_1)(x)|\leq \alpha$. Since this works for any $\delta \geq 2\Theta$, taking supremum we obtain the inequality above.
\end{obs}

\begin{proof}[Proof of Lemma \ref{lema:arreglo_xavi}]
    First of all, note that it suffices to show that for all $x\in \Delta_{\varepsilon_a}$, we can bound
    \begin{equation}
    |K_{\Theta}(\chi_\Delta b_1)(x)| \leq \frac{c}{\varepsilon^n_a}.\label{cosa_que_ens_arregla_la_vida}
    \end{equation}
    Indeed, using this and the fact that both $b_1$ and $b_2$ are bounded functions, we obtain
    $$
    |\langle K_{\Theta}(\chi_\Delta b_1), \chi_{\Delta_{\varepsilon_a}}b_2\rangle| \leq \frac{c}{\varepsilon_a^n}\mu(\Delta_{\varepsilon_a}) \leq \frac{c}{\varepsilon_a^n}\mu(P\cap S).
    $$
    Hence,
    \begin{align*}
    |c_P(f)c_S(g)\langle K_\Theta(\chi_\Delta b_1), \chi_{\Delta_\varepsilon}b_2\rangle| &\leq \frac{c}{\varepsilon^n_a}|c_P(f)|\,|c_S(g)|\,\mu(P\cap S)\\
    &\leq \frac{c}{\varepsilon^n_a}|c_S(g)|\,\mu(S)^{\frac{1}{2}}\,|c_P(f)|\,\mu(P)^{\frac{1}{2}}\\
    &\leq \frac{c}{\varepsilon^n_a}\|\Delta_{1,Q}f\|_{L^2(\mu)}\|\Delta_{2,R}g\|_{L^2(\mu)}.
    \end{align*}
    This is the inequality that we wanted to prove. It remains to show (\ref{cosa_que_ens_arregla_la_vida}). Let $x\in \Delta_{\varepsilon_a}$, we have that 
    $$
    |K_{\Theta}(\chi_{\Delta}b_1)(x)- T_{2\Theta(x)}(\chi_{\Delta}b_1)(x)| \leq c\sup_{r\geq 2\Theta(x)}\frac{\mu(B(x,r))}{r^n}  <\infty.
    $$
    The first inequality is proved with an analogous reasoning as what we argued in Lemma \ref{diferencia_truncat_suppressed}, taking $\varepsilon=2\Theta(x)$. The fact that the supremum above is finite is the first part of the discussion in the previous remark. So, we obtain that
    $$
    |K_{\Theta}(\chi_{\Delta}b_1)(x)| \leq |T_{2\Theta(x)}(\chi_{\Delta} b_1)(x)| + c \leq \sup_{\delta \geq 2\Theta(x)}|T_{\delta}(\chi_{\Delta}b_1)(x)| +c.
    $$
    Now we focus on estimating the supremum on the right-hand side above. We can separate it by writing
    $$
    \sup_{\delta \geq 2\Theta(x)}|T_{\delta}(\chi_{\Delta}b_1)(x)| \leq \sup_{\delta\geq 2\Theta(x)}|T_{\delta}(\nu_1)(x)| + \sup_{\delta \geq 2\Theta(x)}|T_{\delta}(\chi_{\Delta^c}\nu_1)(x)|.
    $$
    The first supremum on the right-hand side is finite (in fact, it is bounded above by $\alpha$) by the second part of the discussion in the previous remark. To finish the proof, we have to show that the second term is bounded by $c/\varepsilon^n_a$. Let $\delta >0$. By denoting $\delta'=\text{diam}(\Delta) + \ell(P)$, we can separate
    \begin{equation}
    |T_{\delta}(\chi_{\Delta^c}\nu_1)(x)| \leq |T_{\delta}(\chi_{B(x, \delta')\setminus \Delta}\nu_1)(x)| + |T_{\delta}(\chi_{B(x, \delta')^c}\nu_1)(x)|.\label{separacio_delta_prima}
    \end{equation}
    The second term above can be bounded by
    \begin{align*}
    |T_{\delta}(\chi_{B(x, \delta')^c}\nu_1)(x)|&= \left|\int_{|x-y|\geq \delta}k(x,y)\chi_{B(x, \delta')^c}(y)\,d\nu_1(y)\right|\\
    &=\left|\int_{|x-y|\geq \max(\delta, \delta')}k(x,y)\,d\nu_1(y)\right|\\
    &\leq \left|T_{\max(\delta, \delta')}(\nu_1)(x)\right| \leq \sup_{r\geq 2\Theta(x)}|T_r(\nu_1)(x)|,
    \end{align*}
    which we already know to be bounded above, by the previous remark. It only remains to give a bound for the first term on the right-hand side of (\ref{separacio_delta_prima}). As usual, we will separate two cases and use an appropriate bound for each case.
    
    Assume first that $\Theta(x)\leq \varepsilon_a\,\ell(P)$. Using that for $y\in \Delta^c$ and $x\in \Delta_{\varepsilon_a}$, we have that $|x-y|\geq c\,\varepsilon_a \ell(P)$, we obtain
        \begin{align*}
            |T_\delta (\chi_{B(x, \delta')\setminus \Delta}\nu_1)(x)| &\leq \int_{|x-y|\geq \delta}|k(x,y)||\chi_{B(x,\delta')\setminus \Delta}(y)|\,d|\nu_1|(y)\\
            &\leq \frac{c}{(\varepsilon_a \ell(P))^n}|\nu_1|(B(x,\delta'))\leq \frac{c}{\varepsilon_a^n}\frac{\mu(B(x,\delta'))}{(\delta')^n}\leq \frac{c}{\varepsilon^n_a},
        \end{align*}
        because on the one hand, $\ell(P)\geq \frac{1}{2}\delta'$ and on the other hand, $\delta'\geq \ell \geq \frac{1}{\varepsilon_a}\Theta(x)>2\Theta(x)$, by choosing $\varepsilon_a<\frac{1}{2}$.

    Now, suppose that $\Theta(x)>\varepsilon_a\, \ell(P)$. In this case, $\delta'\leq 2\,\ell(P)<\frac{1}{2\varepsilon_a}\Theta(x)$. Hence, since $\delta \geq 2\Theta(x)$,
    \begin{align*}
        |T_{\delta}(\chi_{B(x, \delta')\setminus \Delta}\nu_1)(x)|&\leq \int_{|x-y|\geq \delta} |k(x,y)||\chi_{B(x, \delta')\setminus \Delta}(y)|\, d|\nu_1|(y)\\
        &\leq c\,\frac{|\nu_1|(B(x, \delta'))}{\delta^n} \leq \frac{c\mu(B(x, \frac{1}{2\varepsilon_a}\Theta(x)))}{\varepsilon^n_a\left(\frac{1}{2\varepsilon_a}\Theta(x)\right)^n} \leq \frac{c}{\varepsilon_a^n},
    \end{align*}
    choosing $\varepsilon_a<\frac{1}{4}$, so that $\frac{1}{2\varepsilon}\Theta(x)\geq 2\Theta(x)$ and so our previous reasoning can apply.
\end{proof}

Let us put together all the estimates from this section. For the sake of readability, let us denote
\begin{align*}
\mathcal{A}&= \{(Q,R)\in \mathcal{D}_1^{\text{tr},1}\times \mathcal{D}_2^{\text{tr},2}: \text{dist}(Q,R)<\max (\ell(Q),\ell(R)), \,2^{-m}\ell(R)\leq \ell(Q)\leq 2^m\ell(R)\}.
\end{align*}
Moreover, we can classify according to what type of cubes are the children of those $(Q,R)\in \mathcal{A}$, by letting
\begin{align*}
\mathcal{A}_{1,1}&= \{(P,S)\in \mathcal{CH}(Q)\times \mathcal{CH}(R): (Q,R)\in \mathcal{A},\, P\in \mathcal{D}^{\text{tr},1}_1, \,S\in \mathcal{D}^{\text{tr},2}_2 \},
\end{align*}
and, similarly, we define $\mathcal{A}_{2,1}$ if $P$ is terminal and $S$ is transit, $\mathcal{A}_{1,2}$ for $P$ transit and $S$ terminal and $\mathcal{A}_{2,2}$ for both $P$ and $S$ terminal. Combining Lemmas \ref{els_dos_primers_transit} and \ref{lema:arreglo_xavi}, we have found that
\begin{alignat}{2}
|S_2 + S_3| &\leq c_{a}\,\|f\|_{L^2(\mu)}\|g\|_{L^2(\mu)}&& + \sum_{\mathcal{A}_{1,2}\cup \mathcal{A}_{2,1}\cup \mathcal{A}_{2,2}} |\langle K_{\Theta}(\chi_{P\setminus S}\Delta_{1,Q}f), \chi_S\Delta_{2,R}g\rangle|\nonumber\\
&&&+ \sum_{\mathcal{A}_{1,2}\cup \mathcal{A}_{2,1}\cup \mathcal{A}_{2,2}} |\langle K_{\Theta}(\chi_{P\cap S}\Delta_{1,Q}f), \chi_{S\setminus P}\Delta_{2,R}g\rangle|\nonumber\\
& && + \sum_{\mathcal{A}_{1,2}\cup A_{2,1}\cup A_{2,2}}|\langle K_{\Theta}(\chi_{P\cap S}\Delta_{1,Q}f), \chi_{P\cap S}\Delta_{2,R}g\rangle|\nonumber\\
&&& +  \sum_{\mathcal{A}_{1,1}}|\langle K_{\Theta}(\chi_{\Delta_{P,S}}b_1), \chi_{\Delta_{P,S}\setminus (\Delta_{P,S})_{\varepsilon_a}}b_2\rangle|\nonumber\\
&=: c_{a}\,\|f\|_{L^2(\mu)}\|g\|_{L^2(\mu)} &&+ A_1 + A_2 + A_3+ A_4,\label{A_i}
\end{alignat}
where, of course, we denote $\Delta_{P,S}=P\cap S$ for each of the cubes $P,S$ appearing in the sum above. An upper bound for the sums $A_i$ will be obtained via the probabilistic argument in Section \ref{sec:acotacio_ai}.

\section{A Cotlar-type inequality}
\label{chap:cotlar}

Recall that for $x\in \R^d$, its $n$-Ahlfors radius is defined as
\begin{align*}
    \mathcal{R}(x) &= \sup\{r>0 : \mu(B(x,r))>c_0r^n\},
\end{align*}
and we put $\mathcal{R}(x)=0$ if there does not exist any $r>0$ satisfying the condition that defines the set above. Moreover, we set
\begin{equation}
H= \bigcup_{x\in \R^d\,:\, \mathcal{R}(x)>0} B(x, \mathcal{R}(x)).\label{definicio_conjunt_H}
\end{equation}
By assumption (a) from Theorem \ref{TEOREMA}, we have that $H\subset H_{\mathcal{D}(w)}$ for each $w\in \Omega$. In this section we are going to prove the following result, which will be key to finish the proof of Theorem \ref{TEOREMA} in the next section.

\begin{lemma}\label{lema:primer_cotlar}
    Let $\varepsilon>0$ and let $T$ be a Calderón-Zygmund operator of degree $n$ defined in $\R^d\times \R^d \setminus \{(x,y)\in \R^d\times \R^d: x=y\}$ with kernel $t(x,y)$ satisfying
    \begin{equation}
    |t(x,y)| \leq \frac{1}{\max(\varepsilon, \text{dist}(x, \R^d\setminus H), \text{dist}(y, \R^d\setminus H))^n}, \quad \text{for all }\, x,y\in \R^d.\label{condicio_nucli_t}
    \end{equation}
    Then, we have that
    $$
    \|T_*\|_{L^2(\mu)\to L^2(\mu)} \leq c(1+ \|T\|_{L^2(\mu)\to L^2(\mu)}),
    $$
    where $c$ depends only on $c_0$ and on the Calderón-Zygmund constants of the kernel $t(x,y)$.
\end{lemma}

\begin{obs}
    Since for any $x\in \R^d$, we have that $\text{dist}(x, \R^d\setminus H)\geq \mathcal{R}(x)$. if the kernel $t(x,y)$ satisfies the condition (\ref{condicio_nucli_t}), then
    $$
    |t(x,y)| \leq \frac{1}{\max(\varepsilon, \mathcal{R}(x), \mathcal{R}(y))^n}.
    $$
\end{obs}

The proof of Lemma \ref{lema:primer_cotlar} will require a Cotlar-type inequality, which we state and prove below.
\begin{lemma}\label{lema:segon_cotlar}
    Let $T$ be a singular integral operator satisfying the assumptions of Lemma \ref{lema:primer_cotlar}. Suppose that $T$ is bounded from $L^1(\mu)$ into $L^{1,\infty}(\mu)$. Then, for every $f\in L^2(\mu)$ and $x\in \R^d$, we have
    $$
    T_{*}f(x) \leq c \, M_{\mu}(Tf)(x) + c\left(\|T\|_{L^1(\mu)\to L^{1,\infty}(\mu)}+1\right) M_{\mu}f(x),
    $$
    where $c$ depends only on $c_0$ and on the Calderón-Zygmund constants of the kernel $t(x,y)$.
\end{lemma}

Using the boundedness of $M_\mu$ in $L^2(\mu)$ (recall Theorem \ref{HL}), to prove Lemma \ref{lema:primer_cotlar}, it suffices to apply the preceding result, after showing that
\begin{equation}
\|T\|_{L^1(\mu)\to L^{1,\infty}(\mu)} \leq c\left(\|T\|_{L^2(\mu)\to L^2(\mu)}+1\right).\label{cota_extra}
\end{equation}
\begin{proof}[Proof of Lemma \ref{lema:segon_cotlar}]
    Let $\varepsilon>0$. By Theorem \ref{cotlar_referenciar} (Cotlar's inequality), we have that if $T$ is bounded from $M(\R^d)$ to $L^{1,\infty}(\mu)$, and if
    \begin{equation}
    \mu(B(x,r)) \leq c_0r^n, \quad \text{for }\ r\geq \varepsilon,\label{condicio_que_ens_volem_ventilar}
    \end{equation}
    then
    $$
    |T_\varepsilon \nu(x)| \leq c M_\mu(T \nu)(x) + c\left(\|T\|_{M(\R^d)\to L^{1,\infty}(\mu)}+1\right) M_\mu\nu(x),
    $$
    for all $\nu\in M(\R^d)$. If we assume that $T$, instead of being bounded from $M(\R^d)$ into $L^{1,\infty}$, is of weak type (1,1) with respect to $\mu$, the same proof gives that for any $f\in L^1(\mu)$,
    \begin{equation}
    |T_\varepsilon f(x)| \leq c\, M_\mu(Tf)(x) + c\left(\|T\|_{L^1(\mu)\to L^{1,\infty}(\mu)}+1\right) M_\mu f(x).\label{conclusio_que_volem_amb_menys_hipotesis}
    \end{equation}
    So, in order to prove the lemma, we should be able to show that the preceding estimate also holds for our particular operator $T$ without the growth condition (\ref{condicio_que_ens_volem_ventilar}). Fix $x\in \R^d$. If $\mathcal{R}(x)=0$, then condition (\ref{condicio_que_ens_volem_ventilar}) holds and there is nothing to do. Hence, we can assume that $\mathcal{R}(x)>0$, and we have to show that (\ref{conclusio_que_volem_amb_menys_hipotesis}) holds for any $0<\varepsilon<\mathcal{R}(x)$. To do this, we set
    $$
    |T_\varepsilon f(x)| \leq |T_{\mathcal{R}(x)}f(x)| + \int_{0<|x-y|<\mathcal{R}(x)}|t(x,y)|\, |f(y)|\,d\mu(y).
    $$
    The first term, $|T_{\mathcal{R}(x)} f(x)|$, is bounded by the right-hand side of (\ref{conclusio_que_volem_amb_menys_hipotesis}). For the second term, we have to use the fact that $|t(x,y)|\leq \frac{1}{\mathcal{R}(x)^n}$, and so
    \begin{align*}
        \int_{0<|x-y|<\mathcal{R}(x)}|t(x,y)|\, |f(y)|\, d\mu(y) &\leq \frac{1}{\mathcal{R}(x)^n}\int_{B(x, \mathcal{R}(x))}|f(y)|\, d\mu(y)\\
        &\leq \frac{c_0}{\mu(B(x, \mathcal{R}(x)))}\int_{B(x,\mathcal{R}(x))}|f(y)|\,d\mu(y)\leq c_0\, M_\mu f(x),
    \end{align*}
    where we have used that $\mu(B(x, \mathcal{R}(x)))\leq c_0\mathcal{R}(x)^n$. Thus, (\ref{conclusio_que_volem_amb_menys_hipotesis}) holds for all $\varepsilon>0$ in our case. 
\end{proof}

The only task remaining to finish the proof of Lemma \ref{lema:primer_cotlar} is to prove (\ref{cota_extra}), which is done in the following lemma. 
\begin{lemma}
    Let $T$ be a singular integral operator satisfying the assumptions of Lemma \ref{lema:primer_cotlar}. If $T$ is bounded in $L^2(\mu)$, then it is also bounded from $L^1(\mu)$ to $L^{1,\infty}(\mu)$. Moreover, we have
    $$
    \|T\|_{L^1(\mu)\to L^{1,\infty}(\mu)} \leq c\left(\|T\|_{L^2(\mu)\to L^2(\mu)} +1\right).
    $$
\end{lemma}

\begin{proof}
    We have to show that for any $f\in L^1(\mu)$ and $\lambda>0$,
    $$
    \mu(\{x\in \R^d : |Tf(x)|>\lambda\}) \leq c\left(\|T\|_{L^2(\mu)\to L^2(\mu)}+1\right)\frac{\|f\|_{L^1(\mu)}}{\lambda}.
    $$
    Clearly, we may assume that $\lambda > 8 \|f\|_{L^1(\mu)}/\|\mu\|$, because otherwise the inequality above is obviously true. 

    Consider $\{Q_i\}_i$ the Calderón-Zygmund decomposition at level $\lambda$, given by Lemma \ref{descomposicio_calderon_zygmund}, for the measure $\nu = f\mu$. Moreover, let $R_i$ be the smallest cube of the of the form $R_i=6^k Q_i, \, k\geq 1$, such that $\mu(6R_i)\leq 6^{d+1}\mu(R_i)$ (recall Lemma \ref{lema_cubs_doblants}). Then, we write $f = g+b$, where
    $$
    g = \chi_{\R^d \setminus \bigcup_i Q_i}f+ \sum_i \varphi_i,
    $$
    and
    $$
    b= \sum_i b_i := \sum_i (w_i\,f-\varphi_i),
    $$
    where the functions $\varphi_i$ satisfy conditions (\ref{2.16}), (\ref{2.17}) and (\ref{2.18}) of Lemma \ref{descomposicio_calderon_zygmund}, and $w_i = \frac{\chi_{Q_i}}{\sum_k \chi_{Q_k}}$.

    Since each $Q_i$ satisfies that $|f\mu|(Q_i)>\frac{\lambda}{2^{d+1}}\mu(2Q_i)$, we have that
    $$
    \mu\bigg(\bigcup_i 2Q_i\bigg) \leq \sum_i \mu(2Q_i) \leq \sum_i \frac{2^{d+1}}{\lambda}|f\mu|(Q_i) \leq \frac{c}{\lambda}\sum_i \int_{Q_i}|f|\,d\mu \leq \frac{c}{\lambda}\|f\|_{L^1(\mu)},
    $$
    because the cubes from the family $\{Q_i\}_i$ are almost disjoint. So, it suffices to show that 
    \begin{equation}
    \mu\bigg(\bigg\{x\in \R^d\setminus \bigcup_i 2Q_i: |Tf(x)|>\lambda\bigg\}\bigg) \leq c\left(\|T\|_{L^2(\mu)\to L^2(\mu)}+1\right)\frac{\|f\|_{L^1(\mu)}}{\lambda}.\label{cotlar_el_que_provarem}
    \end{equation}
    Using the decomposition $f = g+b$, the right hand side of the inequality above does not exceed
    \begin{equation}
    \mu \bigg(\bigg\{\ x\in \R^d: |Tg(x)|>\frac{\lambda}{2}\bigg\}\bigg) + \mu\bigg(\bigg\{x\in \R^d \setminus \bigcup_i 2Q_i: |Tb(x)|>\frac{\lambda}{2}\bigg\}\bigg).\label{mesures_g_i_b}
    \end{equation}
    To estimate the first term above, first we use that $T$ is bounded in $L^2(\mu)$ and afterwards, that $|g|\leq c\lambda$, to get
    \begin{align*}
        \mu \bigg(\bigg\{\ x\in \R^d: |Tg(x)|>\frac{\lambda}{2}\bigg\}\bigg)  &= \mu \bigg(\bigg\{x\in \R^d : |Tg(x)|^2 >\frac{\lambda^2}{4}\bigg\}\bigg)\\
        &\leq \frac{4}{\lambda^2}\int_{\R^d}|Tg|^2 \,d\mu \leq \frac{4\|T\|^2_{L^2(\mu)\to L^2(\mu)}}{\lambda^2}\int_{\R^d}|g|^2\,d\mu.
    \end{align*}
    Furthermore, again using the properties of the cubes from the Calderón-Zygmund decomposition,
    \begin{align*}
        \int_{\R^d}|g|\,d\mu \leq\int_{\R^d\setminus \bigcup_i Q_i}|f|\,d\mu + \sum_i\int_{\R^d}|\varphi_i|\,d\mu  \leq \|f\|_{L^1(\mu)} + \sum_i\|f\chi_{Q_i}\|_{L^1(\mu)} \leq c\|f\|_{L^1(\mu)}.
    \end{align*}
    Hence,
    $$
    \mu \bigg(\bigg\{\ x\in \R^d: |Tg(x)|>\frac{\lambda}{2}\bigg\}\bigg) \leq \frac{c\|T\|_{L^2(\mu)\to L^2(\mu)}^2}{\lambda}\|f\|_{L^1(\mu)}.
    $$
    Let us examine now the second term in (\ref{mesures_g_i_b}). By Chevishev, we have that
    \begin{align}
        \mu\bigg(\bigg\{x\in \R^d \setminus \bigcup_i 2Q_i: |Tb(x)|>\frac{\lambda}{2}\bigg\}\bigg) \leq \frac{2}{\lambda}\int_{\R^d\setminus \bigcup_i 2Q_i}|Tb|\,d\mu \leq \frac{2}{\lambda}\sum_i \int_{\R^d\setminus 2Q_i}|Tb_i|\,d\mu. \label{cheby}
    \end{align}
    Now, we claim that for each $i$, the following inequality holds,
    \begin{equation}
        \int_{\R^d\setminus 2Q_i}|Tb_i|\,d\mu \leq c\,\|f\chi_{Q_i}\|_{L^1(\mu)}. \label{claim_per_cada_i}
    \end{equation}
    Notice that from this claim and (\ref{cheby}), we obtain that 
    $$
    \mu\bigg(\bigg\{x\in \R^d \setminus \bigcup_i 2Q_i: |Tb(x)|>\frac{\lambda}{2}\bigg\}\bigg) \leq \frac{c}{\lambda}\sum_{i}\|f\chi_{Q_i}\|_{L^1(\mu)} \leq \frac{c\,\|f\|_{L^1(\mu)}}{\lambda}.
    $$
    Let us prove (\ref{claim_per_cada_i}). Denote by $x_{Q_i}$ the center of the cube $Q_i$, and let $B_i = B(x_{Q_i}, \mathcal{R}(x_{Q_i})/2)$. We can write
    \begin{equation}
        \int_{\R^d \setminus 2Q_i}|Tb_i|\,d\mu \leq \int_{B_i}|Tb_i|\,d\mu + \int_{2R_i \setminus (2Q_i\cup B_i)}|Tb_i|\, d\mu + \int_{\R^d\setminus (2R_i \cup B_i)}|Tb_i|\,d\mu.\label{separacio_tres_integrals}
    \end{equation}
    We are now going to study each of the integrals on the right-hand side above separately. For the first one, notice that if $x\in B_i$, then $\text{dist}(x, \R^d\setminus H)\geq \mathcal{R}(x_{Q_i})/2$, so
    $$
    |t(x,y)| \leq \frac{2}{\mathcal{R}(x_{Q_i})^n}.
    $$
    Thus, using that for each $i$, we have $\|b_i\|_{L^1(\mu)}\leq c\, \|f\chi_{Q_i}\|_{L^1(\mu)}$ (this is a direct consequence of the definition of $b_i$ and the condition (\ref{2.18})), this yields
    $$
    |Tb_i(x)| \leq \frac{2}{\mathcal{R}(x_{Q_i})^n}\|b_i\|_{L^1(\mu)} \leq \frac{c}{\mathcal{R}(x_{Q_i})^n}\|f\chi_{Q_i}\|_{L^1(\mu)}.
    $$
    Lastly, using that $\mu(B_i)\leq \mu(B(x_{Q_i}, \mathcal{R}(x_{Q_i})))\leq c_0\,\mathcal{R}(x_{Q_i})^n$, we have that the first integral on the right-hand side of (\ref{separacio_tres_integrals}) is bounded by 
    $$
    \int_{B_i}|Tb_i|\,d\mu \leq  \frac{c}{\mathcal{R}(x_{Q_i})^n}\|f\chi_{Q_i}\|_{L^1(\mu)}\mu(B_i) \leq c\,\|f\chi_{Q_i}\|_{L^1(\mu)}.
    $$
    Now we turn our attention to the second integral on the right-hand side of (\ref{separacio_tres_integrals}). Using (\ref{2.18}), the $L^2(\mu)$ boundedness of $T$ and that $\mu(6R_i)\leq 6^{d+1}\mu(R_i)$, we have that
    \begin{align}
        \int_{2R_i\setminus (2Q_i \cup B_i)}|T\varphi_i|\,d\mu \leq \left(\int_{2R_i}|T\varphi|^2 \, d\mu\right)^{\frac{1}{2}}\mu(2R_i)^{\frac{1}{2}}&\leq c\left(\int_{\R^d}|\varphi_i|^2\,d\mu\right)^{\frac{1}{2}}\mu(R_i)^{\frac{1}{2}}\nonumber \\
        &\leq c\,\|f\chi_{Q_i}\|_{L^1(\mu)}.\label{la_segona_integral_part1}
    \end{align}
    Furthermore, since $\text{supp}(w_if)\subset Q_i$, if $x\in 2R_i\setminus 2Q_i$ and $y\in Q_i$, then $|x-x_{Q_i}|\leq c\,|x-y|$, so we have that
    \begin{align*}
        |T(w_if)(x)| &\leq \int_{\R^d}|t(x,y)|\, |w_i(y)|\,|f(y)|\,d\mu(y) \leq \int_{Q_i}|t(x,y)|\,|f(y)|\,d\mu(y)\\
        &\leq \frac{c}{|x-x_{Q_i}|^n}\int_{Q_i}|f(y)|\,d\mu(y) = \frac{c\,\|f\chi_{Q_i}\|_{L^1(\mu)}}{|x-x_{Q_i}|^n}.
    \end{align*}
    Integrating this inequality, 
    \begin{equation}
    \int_{2R_i \setminus (2Q_i \cup B_i)} |T(w_i f)|\,d\mu \leq c\,\|f\chi_{Q_i}\|_{L^1(\mu)} \int_{2R_i \setminus (2Q_i \cup B_i)}\frac{1}{|x-x_{Q_i}|^n}\,d\mu(x). \label{la_segona_integral}
    \end{equation}
    By Lemma \ref{lema:2.15}, since by our choice of $R_i$, there are no $(6,6^{d+1})$-doubling cubes of the form $6^kQ_i$ between $6Q_i$ and $R_i$, we have
    $$
    \int_{2R_i \setminus (2Q_i \cup B_i)} \frac{1}{|x-x_{Q_i}|^n}\,d\mu(x) \leq c\frac{\mu(2R_i)}{\ell(2R_i)^n}.
    $$
    Also, remark that if the integral above is non-zero, then we must have that $2R_i\not\subset B_i$, that is, if $\mathcal{R}(x_{Q_i})<\sqrt{d}\ell(2R_i)$. Hence, by the definition of $\mathcal{R}(x_{Q_i})$, it must be that
    $$
    \mu(2R_i) \leq \mu\left(B(x_{Q_i},\sqrt{d}\,\ell(2R_i) )\right) \leq c_0\left(\sqrt{d}\,\ell(2R_i)\right)^n,
    $$
    so we can bound the integral above by
    $$
    \int_{2R_i\setminus (2Q_i\cup B_i)} \frac{1}{|x-x_{Q_i}|^n}\,d\mu(x) \leq c\, c_0\,d^{\frac{n}{2}}.
    $$
    So, combining this with (\ref{la_segona_integral_part1}) and (\ref{la_segona_integral}), we have that 
    $$
    \int_{2R_i \setminus (2Q_i\cup B_i)} |Tb_i|\,d\mu \leq \int_{2R_i\setminus (2Q_i \cup B_i)}|T(w_i\,f)|\,d\mu + \int_{2R_i \setminus (2Q_i\cup B_i)}|T\varphi_i|\,d\mu \leq c\,\|f\chi_{Q_i}\|_{L^1(\mu)}.
    $$
    Lastly, we estimate the rightmost integral in (\ref{separacio_tres_integrals}). We can use the fact that $\int b_i \,d\mu=0$, $\text{supp}(b_i)\subset R_i$ and $\|b_i\|_{L^1(\mu)}\leq c\,\|f\chi_{Q_i}\|_{L^1(\mu)}$. This yields, for $x\not\in 2R_i$,
    \begin{align*}
    |Tb_i(x)| \leq \int |k(x,y)-k(x,x_{Q_i})|\,|b_i(y)|\, d\mu(y) &\leq \int c\,\frac{|x_{Q_i}-y|^{\eta}}{|x-x_{Q_i}|^{n+\eta}}\,|b_i(y)| \,d\mu(y) \\
    &\leq \frac{c\,\ell(R_i)^{\eta}}{|x-x_{Q_i}|^{n+\eta}}\|f\chi_{Q_i}\|_{L^1(\mu)}.
    \end{align*}
    Integrating the last inequality, 
    \begin{align}
        \int_{\R^d\setminus (2R_i \cup B_i)} |Tb_i|\,d\mu & \leq c\,\ell(R_i)^{\eta}\,\|f\chi_{Q_i}\|_{L^1(\mu)}\int_{\R^d\setminus (2R_i \cup B_i)}\frac{1}{|x-x_{Q_i}|^{n+\eta}}\,d\mu(x)\nonumber \\
        &\leq c\,\ell(R_i)^{\eta}\|f\chi_{Q_i}\|_{L^1(\mu)}\int_{|x-x_{Q_i}|>\max(\ell(R_i), \frac{1}{2}\mathcal{R}(x_{Q_i}))}\frac{1}{|x-x_{Q_i}|^{n+\eta}}\,d\mu(x).\label{terceraintegral}
    \end{align}
    As usual, we can estimate the last integral splitting the domain of integration into annuli. Denoting $\alpha_i = \max(\ell(R_i), \frac{1}{2}\mathcal{R}(x_{Q_i}))$, 
    $$
    \int_{|x-x_{Q_i}|>\alpha_i}\frac{1}{|x-x_{Q_i}|^{n+\eta}}\,d\mu(x) \leq \sum_{k\geq 0} \frac{\mu(B(x_{Q_i}, 2^{k+1}\alpha_i))}{(2^{k}\alpha_i)^{n+\eta}}.
    $$
    For each $k\geq 0$, $2^{k+1}\alpha_i\geq 2^k \mathcal{R}(x_{Q_i})\geq \mathcal{R}(x_{Q_i})$, so we deduce that the sum above does not exceed
    $$
    \sum_{k\geq 0}\frac{c_0\, (2^{k+1}\alpha_i)^n}{(2^k\alpha_i)^{n+\eta}} = \frac{2^n \,c_0}{\max(\ell(R_i), \frac{1}{2}\mathcal{R}(x_{Q_i}))^\eta}\sum_{k\geq 0}\frac{1}{2^{k\eta}}\leq \frac{c}{\ell(R_i)^{\eta}}.
    $$
    Plugging this in (\ref{terceraintegral}), we have proved (\ref{claim_per_cada_i}).
\end{proof}

\section{A probabilistic argument}

\label{chap:prob}

\subsection{Low probability of bad cubes and functions}

Let us remark that, up until now, all our efforts have been aimed towards estimating the behavior of the suppressed operators $K_{\Theta}$ when they act on good functions. Recall that our motivation for obtaining such estimates, as mentioned in the beginning of Section \ref{sec:good_functions}, was that the probability that a transit cube was bad with respect to a fixed dyadic lattice could be made arbitrarily small. The aim of this section is to, first, give a more precise statement of such claim and, afterwards, prove it. 

Remember that for a transit cube $Q\in \mathcal{D}^{\text{tr},1}_1$, we say that it is bad with respect to the dyadic lattice $\mathcal{D}_2$ if either
\begin{enumerate}[label=(\alph*)]
    \item there exists a cube $R\in \mathcal{D}_2$ such that $\text{dist}(Q,\partial R)\leq \ell(Q)^{\gamma}\ell(R)^{1-\gamma}$ and $2^m\ell(Q)\leq \ell(R)\leq 2^N$ (where $m\geq 1$ is some integer that we will fix below), or
    \item there exists a transit cube $R\in \mathcal{D}^{\text{tr}}_2$ with $2^{-m}\ell(Q)\leq \ell(R)\leq 2^m \ell(Q)$ and $\text{dist}(Q,R)\leq 2^m \ell(Q)$, such that for some $P\in \mathcal{CH}(Q)$, there is $S\in \mathcal{CH}(R)$ such that $\partial P$ is not $(\mu, M,S)$-small.
\end{enumerate}
Before estimating the probability that a given cube is bad with respect to another dyadic lattice, we need a couple of technical lemmas.

\begin{lemma}\label{lema_tecnic_probabilistic_1}
    Let $Q\in \mathcal{D}(w_1)= \mathcal{D}_1$ with $w_1\in \Omega$ and $k\geq m$ be fixed. For $w_2\in \Omega$, let $R(w_2)\in \mathcal{D}(w_2)=\mathcal{D}_2$ be the cube with side length $2^k\ell(Q)$ containing the center $z_Q$. Suppose that $2^k\ell(Q)\leq 2^N$. Denote by $z_{R(w_2)}$ the center of $R(w_2)$. Then, for any subset $A\subset 2^kQ$,
    $$
    P^{\Omega}\left(\{w_2\in \Omega : z_{R(w_2)}\in A\}\right) \leq c\,\frac{\mathcal{L}^d(A)}{\mathcal{L}^d(2^kQ)}.
    $$
\end{lemma}

\begin{proof}
    Before proving the inequality in the statement, let us remark that $z_{R(w_2)}$ always belongs to $\overline{2^kQ}$. This is because $z_Q\in R(w_2)$ and so we have that
    $$
    \|z_Q-z_{R(w_2)}\|_{\infty} \leq \frac{\ell(R(w_2))}{2} = \frac{2^k}{2}\ell(Q) = \frac{\ell(2^kQ)}{2}.
    $$
    Now, let us recall some useful information concerning the set $\Omega$, which was defined in Section \ref{sec:random_dyadic_lattices}. Remember that we called $S^0= [0,2^N]^d$ and we assumed that $F\subset \frac{1}{8}S^0$. Then, $\Omega$ is defined as
    $$
    \Omega = [-2^{N-4}, 2^{N-4}]^d,
    $$
    and for $w\in \Omega$, we write $Q^0(w)=w+S^0$. By Lemma \ref{cub_1/4}, since $\ell(\Omega)= 2^{N-3}$, for each $w\in \Omega$ we have that $F\subset \frac{1}{4}Q^0(w)$. Lastly, we denote by $P^\Omega$ the normalized Lebesgue measure on the cube $\Omega$. We see that $\ell(\Omega) = 2^{N-3}$ could be smaller than $2^k\ell(Q)$. To fix this imbalance, we consider $\Omega':=[-2^{N-1}, 2^{N-1}]^d$, and we write
    \begin{align}
        P^{\Omega}\left(\{w_2\in \Omega : z_{R(w_2)}\in A\}\right) &= \frac{\mathcal{L}^d(\{w_2\in \Omega : z_{R(w_2)}\in A\})}{\mathcal{L}^d(\Omega)}\nonumber \\
        &\leq 2^{3d}\,\frac{\mathcal{L}^d(\{w_2\in \Omega': z_{R(w_2)}\in A\})}{\mathcal{L}^d(\Omega')},\label{pas_1_probabilitat}
    \end{align}
    because $\mathcal{L}^d(\Omega')=2^{3d}\mathcal{L}^d(\Omega)$. To estimate the measure above, we can think of the assignation $w_2\mapsto z_{R(w_2)}$ as a map defined not only on $\Omega$, but for $w_2\in \R^d$, $\text{z}\,\colon \R^d\to 2^kQ$, i.e. $\text{z}(w_2)= z_{R(w_2)}$. Using this, we can write
    $$
    \mathcal{L}^d(\{w_2\in \Omega': z_{R(w_2)}\in A\}) = \mathcal{L}^d(\Omega'\cap \text{z}^{-1}(A)).
    $$
    Let us now study the pre-image $\text{z}^{-1}(A)$. First, a simple computation shows that if we write, using coordinates, $z_Q = (z_Q^1,\dots, z_Q^d)$ and $w_2 = (w_2^1,\dots, w_2^d)$,
    \begin{align*}
    \text{z}(w_2) &= w_2 + 2^k\ell(Q)\left(\left\lfloor \frac{z_{Q}^1-w_2^1}{2^k\ell(Q)}\right\rfloor, \dots , \left\lfloor \frac{z_{Q}^d-w_2^d}{2^k\ell(Q)}\right\rfloor\right) + \frac{2^k\ell(Q)}{2}(1,\dots, 1)\\
    &= w_2 + 2^k\ell(Q)(n_1(w_2), \dots, n_d(w_2)) + 2^{k-1}\ell(Q)(1,\dots, 1).
    \end{align*}
    Hence,
    \begin{align*}
        \text{z}^{-1}(A) &=\left\{w_2: 2^k\ell(Q)(n_1(w_2), \dots, n_d(w_2)) + 2^{k-1}\ell(Q)(1,\dots, 1)\in A\right\} \\
        &=\bigcup_{(n_1, \dots, n_d)\in \mathbb{Z}^d}\left( A- 2^k\ell(Q)(n_1,\dots, n_d) - 2^{k-1}\ell(Q)(1,\dots, 1)\right),
    \end{align*}
    and the sets appearing in the last union are pairwise disjoint. We are interested in taking their intersection with $\Omega'$. Given any $y\in 2^kQ$, we have that $y=\text{z}(w_2)$ for some $w_2\in \Omega'$ if and only if there exist $n_1,\dots, n_d\in \mathbb{Z}$ such that
    $$
    y- 2^k\ell(Q)(n_1,\dots, n_d) -2^{k-1}\ell(Q)(1,\dots,1)\in \Omega'.
    $$
    Since $\ell(2^kQ)= 2^k\ell(Q)\leq 2^N = \ell(\Omega')$, it is clear that we can always choose such integers $n_1,\dots, n_d$ (possibly there is more than one choice). Write $\ell(R(w_2))=2^j$, with $j\leq N$. Then, we can divide $\Omega'$ into $\alpha = (2^{N-j})^d$ cubes of side length $2^j$. In this case, the image of each of these cubes via the function $\text{z}$ is $2^kQ$, which means that
    $$
    \text{z}^{-1}(A) \cap \Omega' = \bigcup_{l=1}^{\alpha} \left(A-2^k\ell(Q)(n_1^l, \dots, n_d^l)-2^{k-1}\ell(Q)(1,\dots, 1)\right) :=\bigcup_{l=1}^\alpha A_l.
    $$
    Since the sets $\{A_l\}_{l=1}^{\alpha}$ are pairwise disjoint translates of $A$, we find that
    $$
    \mathcal{L}^d(\text{z}^{-1}(A)\cap \Omega') = \sum_{l=1}^\alpha\mathcal{L}^d(A_l) = \alpha\, \mathcal{L}^d(A).
    $$
    Taking $A = 2^kQ$, we obtain $\mathcal{L}^d(\text{z}^{-1}(A)\cap \Omega')= \mathcal{L}^d(\Omega') = \alpha \,\mathcal{L}^d(2^kQ)$. Using this in (\ref{pas_1_probabilitat}), we obtain
    $$
    P^{\Omega}\left(\{w_2\in \Omega : z_{R(w_2)}\in A\}\right) \leq 2^{3d}\frac{\alpha \,\mathcal{L}^d(A)}{\alpha \,\mathcal{L}^d(2^kQ)} = c\,\frac{\mathcal{L}^d(A)}{\mathcal{L}^d(2^kQ)},
    $$
    which is what we wanted to prove.
\end{proof}

\begin{lemma}\label{lema_tecnic_probabilistic_2}
    Let $Q,k$ be as in Lemma \ref{lema_tecnic_probabilistic_1}. Using the same notation as in the lemma, the following are equivalent,
    \begin{enumerate}[label=(\arabic*)]
        \item there exists some $R'\in \mathcal{D}(w_2)$ with side length $2^k\ell(Q)$ such that $\text{dist}(Q,\partial R')\leq \ell(Q)^{\gamma}\ell(R')^{1-\gamma}$,
        \item $z_Q\in \overline{\mathcal{U}_s(\partial R(w_2))\cap R(w_2)}$, where $s = \ell(Q)^{\gamma}\ell(R(w_2))^{1-\gamma} + \frac{1}{2}\ell(Q)$,
        \item $z_{R(w_2)}\in \overline{\mathcal{U}_s(\partial (2^kQ))\cap 2^kQ}$.
    \end{enumerate}
\end{lemma}

\begin{proof}
    We start by showing that $(2)$ implies $(1)$. Assume that $z_Q\in \overline{\mathcal{U}_s(\partial R(w_2))\cap R(w_2)}$. We can use that
    $$
    \text{dist}(Q,\partial R(w_2))  \leq \text{dist}(z_Q, \partial R(w_2)) \leq \ell(Q)^{\gamma}\ell(R(w_2))^{1-\gamma}. 
    $$
    Hence, there is $R'$ some neighbor of $R(w_2)$ such that
    $$
    \text{dist}(Q,\partial R') = \text{dist}(Q, \partial R(w_2)) \leq \ell(Q)^{\gamma}\ell(R(w_2))^{1-\gamma} = \ell(Q)^{\gamma}\ell(R')^{1-\gamma},
    $$
    so we can take $R'$ to be the cube from $\mathcal{D}(w_2)$ appearing in condition $(1)$. 
    
    In the converse direction, assume that there is $R'\in \mathcal{D}(w_2)$ with $\ell(R')=2^k\ell(Q)$ such that $\text{dist}(Q,\partial R')\leq \ell(Q)^{\gamma}\ell(R')^{1-\gamma}$. To check that $(2)$ holds, since we already know that $z_Q\in \overline{R(w_2)}$, we only need to check that $\text{dist}(z_Q, \partial R(w_2))\leq s$. 
    
    There are two possibilities for the cube $R'$. Either $R'=R(w_2)$ or it is a neighbor of $R(w_2)$. In either case, $\text{dist}(Q, \partial R')=\text{dist}(Q, \partial R(w_2)) \leq \ell(Q)^{\gamma}\ell(R(w_2))^{1-\gamma}$. Since our cubes have sides parallel to the coordinate axes, we find that
    $$
    \text{dist}(z_Q, \partial R(w_2)) = \text{dist}(Q, \partial R(w_2)) + \frac{1}{2}\ell(Q) \leq \ell(Q)^{\gamma}\ell(R(w_2))^{1-\gamma} + \frac{1}{2}\ell(Q)=s,
    $$
    which means that $(2)$ holds. Lastly, to show that $(2)$ and $(3)$ are equivalent, we will show that they are both equivalent to a fourth condition,
    $$
    (4)\quad 2^{k-1}\ell(Q)-s\leq \|z_Q-z_{R(w_2)}\|_{\infty} \leq 2^{k-1}\ell(Q).
    $$
    Assume that $(2)$ holds. In particular, $z_Q\in \overline{R(w_2)}$, so $\|z_Q-z_{R(w_2)}\|_{\infty} = 2^{k-1}\ell(Q)$. Since $z_Q\in \overline{\mathcal{U}_s(\partial R(w_2))}$, we can choose $x\in \partial R(w_2)$ such that $\|x-z_Q\|_{\infty}\leq s$. Using this auxiliary point,
    \begin{align*}
        \|z_Q-z_{R(w_2)}\|_{\infty} &\geq \|z_{R(w_2)}-x\|_\infty - \|x-z_{Q}\|_{\infty}\geq 2^{k-1}\ell(Q) -s.
    \end{align*}
    Conversely, assume that $(4)$ holds. Since by construction $z_Q\in R(w_2)$, we only have to check that $\text{dist}(z_Q, \partial R(w_2))\leq s$. Again, since the sides of our cubes are parallel to the axes, 
    $$
    \text{dist}(z_Q, \partial R(w_2)) = 2^{k-1}\ell(Q) - \|z_{R(w_2)}- z_Q\|_{\infty} \leq 2^{k-1}\ell(Q) - 2^{k-1}\ell(Q)+s = s.
    $$
    Lastly, to see that $(3)$ is equivalent to $(4)$, we argue as above, taking $x\in \partial(2^kQ)$ instead of $x\in \partial R(w_2)$, and the arguments are analogous.
\end{proof}

The following lemma is the main point of this section. It assures us that we can choose $m$ and $M$ in the definition of bad cubes so that the probability that a cube is bad is as small as we want.

\begin{lemma}\label{lema:probabilitat}
    Let $0<\varepsilon_b<1$ be fixed. Suppose that the constants $m$ and $M$ are big enough, depending only on $\varepsilon_b$, in the definition of bad cubes. Let $\mathcal{D}_1=\mathcal{D}(w_1)$, with $w_1\in \Omega$, be any fixed dyadic lattice. Then, for each fixed $Q\in \mathcal{D}^{\text{tr},1}_1$, we have that
    \begin{equation}
    P^{\Omega}(\{w_2\in \Omega : Q\,\text{ is bad with respect to }\, \mathcal{D}(w_2)=\mathcal{D}_2\})\leq \varepsilon_b.\label{probabilitat}
    \end{equation}
\end{lemma}

\begin{proof}
    Fix $Q\in \mathcal{D}^{\text{tr},1}_1$. Since it is bad if either (a) or (b) above in the definition happens, the probability in (\ref{probabilitat}) is controlled by 
    $$
    P^{\Omega}(\{Q\,\text{is bad w.r.t.} \mathcal{D}_2 \,\text{due to (a)}\}) + P^{\Omega}(\{Q\,\text{is bad w.r.t.} \mathcal{D}_2 \,\text{due to (b)}\}):=P_{{\text{(a)}}} + P_{{\text{(b)}}}.
    $$
    Our aim now is to choose the constants $m$ and $M$ so that each of the two probabilities above is small. Note that condition (a) only takes into account the integer $m$, while condition (b) depends on both $m$ and $M$. Due to this, our strategy will be to make $P_{\text{(a)}}$ small by choosing $m$ big enough, and for a fixed value of $m$, choose $M$ big enough so that $P_{\text{(b)}}$ is also small.

    {\underline{\textbf{Choice of} $m$:}} we start by estimating
    $$
    P_{\text{(a)}} = P^{\Omega}(\{w_2\in \Omega : \exists R\in \mathcal{D}_2, \, \text{dist}(Q, \partial R)\leq \ell(Q)^{\gamma}\ell(R)^{1-\gamma},\, 2^m\ell(Q)\leq \ell(R)\leq 2^N\}).
    $$
    For each $k\geq m$, by Lemmas \ref{lema_tecnic_probabilistic_2} and \ref{lema_tecnic_probabilistic_1}, we have that
    \begin{align}
        &P^{\Omega}\left(\left\{\omega_2\in \Omega : \exists R\in \mathcal{D}_2, \text{dist}(Q,\partial R)\leq \ell(Q)^{\gamma}\ell(R)^{1-\gamma }, \ell(R)=2^k\ell(Q)\right\}\right)\label{probabilitat_m} \\
        &\qquad \qquad = P^{\Omega}\left(\left\{w_2\in \Omega: z_{R(w_2)}\in \overline{\mathcal{U}_s(\partial (2^kQ))\cap 2^kQ}\right\}\right) \nonumber \\
        &\qquad \qquad \leq c\,\frac{\mathcal{L}^d(\overline{\mathcal{U}_s(\partial(2^kQ))\cap 2^kQ)})}{\mathcal{L}^d(2^kQ)} \leq c\, \frac{s^{d-1}\,2^k\ell(Q)}{(2^k\ell(Q))^d}.\nonumber
    \end{align}
    Since $s \leq 2\ell(Q)^{\gamma}\ell(R(w_2))^{1-\gamma} = 2^{k(1-\gamma)+1}\ell(Q)$, we have that the probability in (\ref{probabilitat_m}) is bounded above by
    $$
    c\,\frac{(\ell(Q)2^{k(1-\gamma)+1})^{d-1}2^k\ell(Q)}{(2^k \ell(Q))^d} = c\, 2^{-k\gamma(d-1)}.
    $$
    This means that
    $$
    P_{\text{(a)}}\leq \sum_{k\geq m}c\, 2^{-k\gamma(d-1)} = c\, \frac{2^{-m\gamma(d-1)}}{1-2^{-\gamma(d-1)}}\xrightarrow{m\to +\infty}0,
    $$
    so we choose $m$ so that $P_{\text{(a)}}\leq \frac{\varepsilon_b}{2}$.
    
    \underline{\textbf{Choice of }$M$:} we write 
    \begin{align*}
        P_{\text{(b)}} = P^{\Omega}(\{ w_2\in \Omega :\,& \exists R\in \mathcal{D}(w_2), 2^{-m}\ell(Q)\leq \ell(R)\leq 2^m\ell(Q), \text{dist}(Q,R)\leq 2^m\ell(Q)\\
        &\text{and for some }P\in \mathcal{CH}(Q),\text{ there is }S\in \mathcal{CH}(R) \text{ such that }\partial P\\
        & {\text{is not }(\mu,M,S)}\text{-small}\}).
    \end{align*}
    By writing $Q = \cup_{i=1}^{2^d}P_i$ as the union of its children, we have that the probability above is bounded by
    \begin{alignat}{2}
        &\sum_{i=1}^{2^d} P^{\Omega}(\{ w_2\in \Omega :\,&& \exists S\in \mathcal{D}(w_2), 2^{-m-1}\ell(Q)\leq \ell(S)\leq \min(2^m\ell(Q), 2^N),\nonumber\\ & &&\text{dist}(P_i,S)\leq 2^{m+2}\ell(Q)
        \text{ such that }\partial P_i\text{ is not }(\mu,M,S)\text{-small}\}).\label{probabilitat_fills}
    \end{alignat}
    Let us write the sum above as $\sum_{i=1}^{2^d}p_i$, where for each $i$, $p_i$ denotes the probability associated to the $i$-th child of $Q$ above. Since there is no canonical ordering of $\mathcal{CH}(Q)$, we will only estimate $p_1$ and obtain a bound that does not depend on that specific child of $Q$. For convenience, we will denote from now on $P:=P_1$. Note that the boundary of $P$ is the union of $2d$ $d-1$ dimensional cubes, its ``faces'', that is,
    $$
    \partial P = \bigcup_{i=1}^{2d} H_i (P),
    $$
    where each $H_i(P)$ is contained in a hyperplane from $\R^d$. Since
    $$
    \mu\left(\left\{x\in S : \text{dist}(x,\partial P)\leq \lambda \ell(P)\right\}\right) \leq \sum_{i=1}^{2d} \mu\left(\left\{x\in S : \text{dist}(x,H_i(P))\leq \lambda \ell(P)\right\}\right),
    $$
    we see that if $\partial P$ is not $(\mu,M,S)$-small, there is some $H_{i_0}(P)$ that is not $(\mu, M,S)$-small. This is the essential observation in order to show that $P_{\text{(b)}}$ can be made arbitrarily small by choosing $M$ large enough. Our strategy will be to show that, for a fixed $S$ satisfying (\ref{probabilitat_fills}), there are very few $H_i(P)$ that are not $(\mu,M,S)$-small. This will be enough to complete our estimate, because using the fact that both $\mathcal{D}_1$ and $\mathcal{D}_2$ are translates of the same dyadic lattice, fixing $Q$ and looking for $S$ as in (\ref{probabilitat_fills}) is the same as fixing $S$ and looking for a suitable $Q$.
    
    Since our dyadic cubes have sides parallel to the coordinate axes, the directions that are normal to their faces are those given by the canonical basis of $\R^d$, which we denote by $\{e_1,\dots,e_d\}$. It is enough to show that there are very few $H_i(P)$ with normal direction $e_1$. 
    
    Consider all the hyperplanes from $\R^d$ with normal direction $e_1$ such that they intersect $S$, and call this family $H$. Consider $\sigma $ the projection of $\mu \lfloor S$ into the line $L_1$ of direction $e_1$. That is, for $E\subset L_1$ a measurable set, if $\pi\,\colon\R^d\to \R $ denotes the projection into the first coordinate,
    $$
    \sigma(A) = \mu\lfloor S(\pi^{-1}(E)) = \mu(S\cap \pi^{-1}(E)) = \mu(\{x\in S: x_1\in E\}).
    $$
    The key observation now is that for a hyperplane $h\in H$,
    \begin{align*}
    \{x\in S : \text{dist}(x,h)\leq \lambda \ell(P)\} &= \{x\in S: |x_1-x_h|\leq \lambda \ell(P)\} \\
    &=\pi^{-1}([x_h-\lambda\ell(P), x_h+\lambda\ell(P)])\cap S,
    \end{align*}
    where we have denoted $\pi(h)=x_h$. From this, we see that
    $$
    \mu\left(\{x\in S : \text{dist}(x,h)\leq \lambda \ell(P)\}\right) = \sigma(\{y\in L_1: |y-x_h|\leq \lambda \ell(P)\}).
    $$
    Hence, if the hyperplane $h\in H$ is not $(\mu, M,S)$-small, then the point $x_h\in L_1$ satisfies that there is $\lambda>0$ such that
    $$
    \sigma (\{y\in L_1 : |y-x_h|\leq \lambda \ell(P)\})> \lambda M\mu(S).
    $$
    This means that $h$ is not $(\mu, M,S)$-small if and only if
    $$
    M_R\sigma(x_h) := \sup_{r>0}\frac{\sigma(B(x_h, r))}{r}> \frac{M\mu(S)}{\ell(P)} \geq \frac{c}{2^N}M\mu(S)=c M\mu(S),
    $$
    because we always only consider cubes with side length bounded above by $2^N$, which we absorb inside the constant $c$. Recall that the maximal operator $M_R$ is bounded from $M(\R^d)$ to $L^{1,\infty}(\R)$, and so
    \begin{align*}
        \mathcal{L}^1\left(\left\{x : \exists\, h \subset \pi^{-1}(\{x\}) \text{ not }(\mu, M, S)\text{-small}\right\}\right) &\leq \mathcal{L}^1\left(\left\{x : M_{R}\sigma(x)>c M\mu(S)\right\}\right)\\
        &\leq \frac{A}{c M \mu(S)}\|\sigma\| \leq \frac{A}{c M}.
    \end{align*}
    This reasoning works fixing any other of the $e_i$ as the normal direction to our hyperplanes. We will use the same estimate for the sum of all of them, since there are $d$ possibilities for $e_i$. 
    
    Let us call $SH$ the set of hyperplanes, parallel to one of the faces of $S$, intersecting $S$ and such that they are not $(\mu, M,S)$-thin. We can consider on each of them the projection of the Lebesgue measure onto the first coordinate, and from our reasoning above, we have that the measure of $SH$ is bounded above by $\frac{A}{c M}$.

    By periodicity of the dyadic lattice $\mathcal{D}_1$ at the fixed scale $\ell(S)$, the probability with respect to $P^{\Omega}$ that there is $P\in \mathcal{D}_1$ satisfying the condition from (\ref{probabilitat_fills}) has one of its faces in $SH$ is, at most, $\frac{A}{ cM}$. However, we also need to take into account the cubes $P$ satisfying $2^{-m-1}\ell(R)\leq \ell(P)\leq 2^{m+1}\ell(R)$. Again by periodicity of the lattices, for $\ell(P)>\ell(R)$, we have the same bound for the probability. For the smaller cubes, the bound is $\frac{2^{m+1}A}{c M}$. If we add the probabilities over all the possible sizes, we obtain that 
    $$
    P_{\text{(b)}} \leq \frac{2^{2m+2}A}{cM}.
    $$
    Since $m$ is fixed from the estimate of $P_{\text{(a)}}$, we can always choose $M$ large enough so that $P_{\text{(b)}}\leq \frac{\varepsilon_b}{2}$. 
\end{proof}

Using the previous lemma, we are going to obtain a useful estimate of the expected value of the norm in $L^2(\mu)$ of the bad part of a function. We give the precise statement in the lemma below.

\begin{lemma}\label{lema:esperança_part_dolenta}
    Let $f\in L^2(\mu)$. For $\text{w}=(w_1,w_2)\in \Omega^2$, set
    $$
    f_{\text{bad}}(\text{w}) = \sum_{Q\in \mathcal{D}^{\text{tr},1}(w_1)\cap \text{bad}(w_2)}\Delta_{1,Q}f,
    $$
    where ``$Q\in \text{bad}(w_2)$'' means that $Q$ is bad with respect to $\mathcal{D}(w_2)$. Then, we have that
    $$
    \mathbb{E}_{P^{\Omega^2}}\left(\|f_{\text{bad}}(\text{w})\|^2_{L^2(\mu)}\right) \leq c_3^2 \, \varepsilon_b \,\|f\|^2_{L^2(\mu)},
    $$
    where we have denoted by $\mathbb{E}_{P^{\Omega^2}}$ the expectation with respect to $P^{\Omega^2}=P^\Omega\times P^{\Omega}$.
\end{lemma}
\begin{proof}
    From the decomposition in $L^2(\mu)$, Lemma \ref{lema_descomposicio_L2}, we have that
    $$
    \|f_{\text{bad}}(\text{w})\|^2_{L^2(\mu)} \leq c_3 \sum_{Q\in \mathcal{D}^{\text{tr},1}(w_1)\cap \text{bad}(w_2)}\|\Delta_{1,Q}f\|^2_{L^2(\mu)}.
    $$
    Then, using the preceding lemma,
    \begin{align*}
        \mathbb{E}_{P^{\Omega^2}}\big(\|f_{\text{bad}}(&\text{w})\|^2_{L^2(\mu)}\big) \leq c_3\int\bigg(\int \sum_{Q\in \mathcal{D}^{\text{tr},1}(w_1)\cap \text{bad}(w_2)}\|\Delta_{1,Q}f\|^2_{L^2(\mu)}\,dP^{\Omega}(w_2)\bigg)\,dP^{\Omega}(w_1) \\
        &= c_3\int \sum_{Q\in \mathcal{D}^{\text{tr},1}(w_1)}\|\Delta_{1,Q}f\|^2_{L^2(\mu)}\bigg(\int \chi_{\{Q\in \text{bad}(w_2)\}}(w_2)\,dP^{\Omega}(w_2)\bigg)\,dP^{\Omega}(w_2)\\
        &= c_3\int\sum_{Q\in \mathcal{D}^{\text{tr},1}(w_1)}\|\Delta_{1,Q}f\|^2_{L^2(\mu)} \,P^{\Omega}(\{w_2: Q\in \text{bad}(w_2)\})\,dP^{\Omega}(w_1)\\
        &\leq c_3\,\varepsilon_b\int\sum_{Q\in \mathcal{D}^{\text{tr},1}(w_1)}\|\Delta_{1,Q}f\|^2_{L^2(\mu)}\, dP^{\Omega}(w_1)\leq c_3^2\, \varepsilon_b\|f\|^2_{L^2(\mu)},
    \end{align*}
    where in the last inequality we have used Lemma \ref{lema_descomposicio_L2} again.
\end{proof}

The estimate above will be key in the final argument, where we will take an average over the set $\Omega$.

Recall now that we obtained a partial bound for $|S_2+S_3|$ in terms of the sums $A_i$ in (\ref{suma_A}). We mentioned that the full bound would have to wait until the last steps, since it will involve a probabilistic argument. Now, we introduce some new concepts, which will be the correct tools to deal with the additional term that we found when trying to bound $|S_2+S_3|$.

For a cube $Q\subset \R^d$ and $\varepsilon>0$, we denote $\delta_Q =Q\setminus (1-2\varepsilon)Q$. That is, $\delta_Q$ consists of the points from $Q$ that are at distance greater than $\varepsilon$ from the center of $Q$, with respect to the distance induced by the $\infty$-norm in $\R^d$ (see Figure \ref{fig:delta_q}).

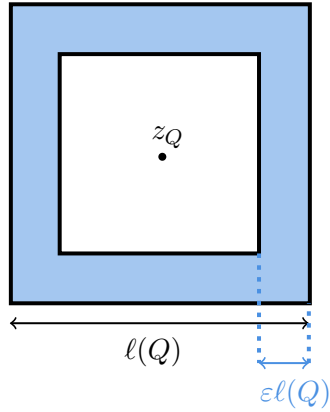
\begin{figure}[h]
    \centering

 
\tikzset{
pattern size/.store in=\mcSize, 
pattern size = 5pt,
pattern thickness/.store in=\mcThickness, 
pattern thickness = 0.3pt,
pattern radius/.store in=\mcRadius, 
pattern radius = 1pt}
\makeatletter
\pgfutil@ifundefined{pgf@pattern@name@_pfxydi5av}{
\pgfdeclarepatternformonly[\mcThickness,\mcSize]{_pfxydi5av}
{\pgfqpoint{-\mcThickness}{-\mcThickness}}
{\pgfpoint{\mcSize}{\mcSize}}
{\pgfpoint{\mcSize}{\mcSize}}
{
\pgfsetcolor{\tikz@pattern@color}
\pgfsetlinewidth{\mcThickness}
\pgfpathmoveto{\pgfpointorigin}
\pgfpathlineto{\pgfpoint{0}{\mcSize}}
\pgfusepath{stroke}
}}
\makeatother

 
\tikzset{
pattern size/.store in=\mcSize, 
pattern size = 5pt,
pattern thickness/.store in=\mcThickness, 
pattern thickness = 0.3pt,
pattern radius/.store in=\mcRadius, 
pattern radius = 1pt}
\makeatletter
\pgfutil@ifundefined{pgf@pattern@name@_n39x23tej}{
\pgfdeclarepatternformonly[\mcThickness,\mcSize]{_n39x23tej}
{\pgfqpoint{-\mcThickness}{-\mcThickness}}
{\pgfpoint{\mcSize}{\mcSize}}
{\pgfpoint{\mcSize}{\mcSize}}
{
\pgfsetcolor{\tikz@pattern@color}
\pgfsetlinewidth{\mcThickness}
\pgfpathmoveto{\pgfpointorigin}
\pgfpathlineto{\pgfpoint{0}{\mcSize}}
\pgfusepath{stroke}
}}
\makeatother
\tikzset{every picture/.style={line width=0.75pt}} 

\begin{tikzpicture}[x=0.75pt,y=0.75pt,yscale=-1,xscale=1]

\draw  [fill={rgb, 255:red, 74; green, 144; blue, 226 }  ,fill opacity=0.5 ][line width=1.5]  (0.5,0) -- (150.5,0) -- (150.5,150) -- (0.5,150) -- cycle ;
\draw  [fill={rgb, 255:red, 255; green, 255; blue, 255 }  ,fill opacity=1 ][line width=1.5]  (25,25) -- (125,25) -- (125,125) -- (25,125) -- cycle ;
\draw  [fill={rgb, 255:red, 0; green, 0; blue, 0 }  ,fill opacity=1 ] (75,76.5) .. controls (75,75.67) and (75.67,75) .. (76.5,75) .. controls (77.33,75) and (78,75.67) .. (78,76.5) .. controls (78,77.33) and (77.33,78) .. (76.5,78) .. controls (75.67,78) and (75,77.33) .. (75,76.5) -- cycle ;
\draw[<->]    (0,160) -- (150,160) ;

\draw [color={rgb, 255:red, 74; green, 144; blue, 226 }  ,draw opacity=1, <-> ]   (150,180) -- (125,180) ;

\draw [color={rgb, 255:red, 74; green, 144; blue, 226 }  ,draw opacity=1 ][pattern=_pfxydi5av,pattern size=6pt,pattern thickness=0.75pt,pattern radius=0pt, pattern color={rgb, 255:red, 0; green, 0; blue, 0}][line width=1.5]  [dash pattern={on 1.69pt off 2.76pt}]  (125,125) -- (125,180) ;
\draw [color={rgb, 255:red, 74; green, 144; blue, 226 }  ,draw opacity=1 ][pattern=_n39x23tej,pattern size=6pt,pattern thickness=0.75pt,pattern radius=0pt, pattern color={rgb, 255:red, 0; green, 0; blue, 0}][line width=1.5]  [dash pattern={on 1.69pt off 2.76pt}]  (150,180) -- (150,150) ;

\draw (69.5,60) node [anchor=north west][inner sep=0.75pt]    {$z_{Q}$};
\draw (56,165.4) node [anchor=north west][inner sep=0.75pt]    {$\ell ( Q)$};
\draw (119.25,186.65) node [anchor=north west][inner sep=0.75pt]  [color={rgb, 255:red, 74; green, 144; blue, 226 }  ,opacity=1 ]  {$\ \varepsilon \ell ( Q)$};

\end{tikzpicture}

    \caption{In dimension $d=2$, the set $\delta_Q$ is shaded in blue.}
    \label{fig:delta_q}
\end{figure}

Now, fix any $x\in \R^d$ and $k\in \mathbb{Z}$. We denote by $p_\varepsilon$ the probability, with respect to $w\in \Omega$, that $x\in \delta_R$ for some $R\in \mathcal{D}(w)^{\text{tr},2}$ and $2^{k-m}\leq \ell(R)\leq 2^{k+m}$, that is,
\begin{equation}
p_\varepsilon = \frac{\mathcal{L}^d(\{w\in \Omega: \exists R\in \mathcal{D}(w)^{\text{tr},2}, \, 2^{k-m}\leq \ell(R)\leq 2^{k+m}, \, x\in \delta_R\})}{\mathcal{L}^d(\Omega)}.\label{p_epsilon}
\end{equation}
Let us now show how $p_\varepsilon$ behaves as $\varepsilon\to0^+$.

\begin{lemma}
    For $x\in \R^d, \,k\in \mathbb{Z}$ and $\varepsilon>0$, define $p_\varepsilon$ as above. Then, $p_\varepsilon$ has limit $0$ as $\varepsilon\to 0^+$, independently of $x$ and $k$.\label{lemma:p_epsilon_a}
\end{lemma}

\begin{proof}
    Our strategy will be to reduce the problem to the one-dimensional case, where cubes will be intervals and for an interval $I$ the set $I_\delta$ will be two smaller intervals sharing one endpoint each with the bigger one $I$. In this simpler setting, we will be able to compute the set we are measuring in the numerator in (\ref{p_epsilon}).

    If we denote by $\pi_j$ the projection into the $j$-th coordinate, that is, $\pi_j(x)=x_j$, we obviously have that
    \begin{align*}
        p_\varepsilon &\leq \frac{\mathcal{L}^d(\{w\in \Omega: \exists R\in \mathcal{D}(w), \, 2^{k-m}\leq \ell(R)\leq \min(2^{k+m}, 2^N), \, x\in \delta_R\})}{\mathcal{L}^d(\Omega)}\\
        &\leq \frac{1}{\mathcal{L}^d(\Omega)}\sum_{j=k-m}^{\min(k+m, 2^N)}\mathcal{L}^d(\{w\in \Omega: \exists R\in \mathcal{D}(w), \ell(R)=2^j, x\in \delta_R\})\\
        &\leq \frac{1}{\mathcal{L}^d(\Omega)}\sum_{j=k-m}^{\min(k+m, 2^N)}\sum_{i=1}^d\mathcal{L}^d(\{w\in \Omega: \exists R\in \mathcal{D}(w), \ell(R)=2^j, x_i\in \delta_{\pi_i(R)}\})\\
        &=: \frac{1}{\mathcal{L}^d(\Omega)}\sum_{j=k-m}^{\min(k+m,2^N)}\sum_{i=1}^d p_\varepsilon^{i,j}.
    \end{align*}
    It suffices to find a common bound, in terms of $\varepsilon$, for each $p_\varepsilon^i$, since $\Omega$ is a fixed cube and there are always $2md$ terms in the sum above, independently of $k$ and of $x$. We will only estimate $p_\varepsilon^{1,j}$, because all the others are analogous. We consider $y=x_1 \text{ mod } 2^j$, so that $y\in [0, 2^{j})$. If we denote by $\mathcal{D}_{0,j}^1$ the usual dyadic intervals of side length $2^j$ in the real line, we have that for $R\in \mathcal{D}(w)$, $\pi_1(R)\in \mathcal{D}_{0,j}^1+\pi_1(w)$. This means that 
    \begin{equation}
    p_{\varepsilon}^{1,j} = \mathcal{L}^{d-1}([-2^{N-4}, 2^{N-4}])\cdot \mathcal{L}^1(\{w_1\in [-2^{N-4}, 2^{N-4}]: \exists \,I\in \mathcal{D}_{0,j}^1+w_1, \, x_1\in \delta_{I} \}).\label{p_epsilon_1j}
    \end{equation}
    Clearly, the measure appearing in the first factor is $(2^{N-3})^{d-1}$. Our goal now is to describe the set appearing in the second factor. For this, we will distinguish three cases. We start with the simplest of the three: when $j=N-4$. We assume that $\varepsilon<\frac{1}{2}$ and consider an open interval of length $2^{N-3}\,\varepsilon$ centered at $y$, which we call $J$.
    
    Assume first that this interval is completely contained inside $[0, 2^{N-4}]$. We can extend it periodically with period $2^{N-4}$, that is, we consider the intervals $J_k = J + k\,2^{N-4}$, with $k\in \mathbb{Z}$. Obviously, there are only two which have non-empty intersection with $[-2^{N-4}, 2^{N-4}]$, namely, $J_0=J$ and $J_{-1}$ (see Figure \ref{fig:intervals_J}).

    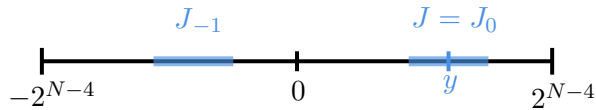
\begin{figure}[h]
        \centering

\tikzset{every picture/.style={line width=0.75pt}} 

\begin{tikzpicture}[x=0.75pt,y=0.75pt,yscale=-1,xscale=1]

\draw [line width=1.5]    (0,0) -- (256,0) ;
\draw [shift={(256,0)}, rotate = 180] [color={rgb, 255:red, 0; green, 0; blue, 0 }  ][line width=1.5]    (0,6.71) -- (0,-6.71)   ;
\draw [shift={(0,0)}, rotate = 180] [color={rgb, 255:red, 0; green, 0; blue, 0 }  ][line width=1.5]    (0,6.71) -- (0,-6.71)   ;
\draw [line width=1.5]    (128,-5) -- (128,5) ;
\draw [color={rgb, 255:red, 74; green, 144; blue, 226 }  ,draw opacity=1 ][line width=1.5]    (204,-5) -- (204,5) ;
\draw [color={rgb, 255:red, 74; green, 144; blue, 226 }  ,draw opacity=0.65 ][line width=3.75]    (184,0) -- (224,0) ;
\draw [color={rgb, 255:red, 74; green, 144; blue, 226 }  ,draw opacity=0.65 ][line width=3.75]    (56,0) -- (96,0) ;

\draw (123.6,8.4) node [anchor=north west][inner sep=0.75pt]    {$0$};
\draw (244,9.4) node [anchor=north west][inner sep=0.75pt]    {$2^{N-4}$};
\draw (-18.4,8.6) node [anchor=north west][inner sep=0.75pt]    {$-2^{N-4}$};
\draw (199.6,5.4) node [anchor=north west][inner sep=0.75pt]    {$\textcolor[rgb]{0.29,0.56,0.89}{y}$};
\draw (184,-28.6) node [anchor=north west][inner sep=0.75pt]    {$\textcolor[rgb]{0.29,0.56,0.89}{J=J_{0}}$};
\draw (64.8,-28.2) node [anchor=north west][inner sep=0.75pt]    {$\textcolor[rgb]{0.29,0.56,0.89}{J}\textcolor[rgb]{0.29,0.56,0.89}{_{-1}}$};

\end{tikzpicture}

        \caption{The point $y=x_1$ mod $2^j$ and the intervals $J, J_{-1}$.}
        \label{fig:intervals_J}
    \end{figure}

    We claim that, up to the set of measure zero consisting of the $4$ endpoints of the intervals, the set appearing in the second factor in (\ref{p_epsilon_1j}) is $J_{-1}\cup J$. Indeed, for any $t\in J_{-1}\cup J$, there is a unique $J^t\in \mathcal{D}_{0,j}^1+t$ containing $y$, which will be separated no less than $2^{N-3}\varepsilon$ from one of its endpoints (see Figure \ref{fig:J_t}).

    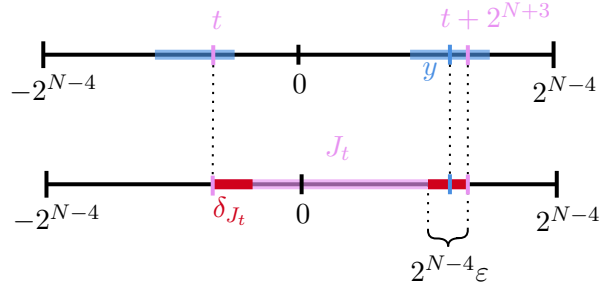
\begin{figure}[H]
        \centering

\tikzset{every picture/.style={line width=0.75pt}} 

\begin{tikzpicture}[x=0.75pt,y=0.75pt,yscale=-1,xscale=1]

\draw  [dash pattern={on 0.84pt off 2.51pt}]  (212.95,64.44) -- (212.99,70.13) -- (212.96,89.29) ;
\draw  [dash pattern={on 0.84pt off 2.51pt}]  (193,65) -- (193.04,70.69) -- (193,89.84) ;
\draw [line width=1.5]    (0,0) -- (256,0) ;
\draw [shift={(256,0)}, rotate = 180] [color={rgb, 255:red, 0; green, 0; blue, 0 }  ][line width=1.5]    (0,6.71) -- (0,-6.71)   ;
\draw [shift={(0,0)}, rotate = 180] [color={rgb, 255:red, 0; green, 0; blue, 0 }  ][line width=1.5]    (0,6.71) -- (0,-6.71)   ;
\draw [line width=1.5]    (128,-5) -- (128,5) ;
\draw [color={rgb, 255:red, 74; green, 144; blue, 226 }  ,draw opacity=1 ][line width=1.5]    (204,-5) -- (204,5) ;
\draw [color={rgb, 255:red, 74; green, 144; blue, 226 }  ,draw opacity=0.65 ][line width=3.75]    (184,0) -- (203.86,0) -- (224,0) ;
\draw [color={rgb, 255:red, 74; green, 144; blue, 226 }  ,draw opacity=0.65 ][line width=3.75]    (56,0) -- (96,0) ;
\draw [line width=1.5]    (1.6,65) -- (257.6,65) ;
\draw [shift={(257.6,65)}, rotate = 180] [color={rgb, 255:red, 0; green, 0; blue, 0 }  ][line width=1.5]    (0,6.71) -- (0,-6.71)   ;
\draw [shift={(1.6,65)}, rotate = 180] [color={rgb, 255:red, 0; green, 0; blue, 0 }  ][line width=1.5]    (0,6.71) -- (0,-6.71)   ;
\draw [color={rgb, 255:red, 231; green, 146; blue, 241 }  ,draw opacity=1 ][line width=1.5]    (85.27,-4.76) -- (85.27,5.24) ;
\draw  [dash pattern={on 0.84pt off 2.51pt}]  (85.27,5.24) -- (84.99,60.67) ;
\draw  [dash pattern={on 0.84pt off 2.51pt}]  (204,5) -- (204,60) ;
\draw [color={rgb, 255:red, 231; green, 146; blue, 241 }  ,draw opacity=0.65 ][line width=3.75]    (85,65) -- (213,65) ;
\draw [line width=1.5]    (129.6,60) -- (129.6,70) ;
\draw [color={rgb, 255:red, 231; green, 146; blue, 241 }  ,draw opacity=1 ][line width=1.5]    (212.99,-4.76) -- (212.99,5.24) ;
\draw  [dash pattern={on 0.84pt off 2.51pt}]  (212.99,5.24) -- (212.99,60.24) ;
\draw [color={rgb, 255:red, 208; green, 2; blue, 27 }  ,draw opacity=1 ][line width=3.75]    (85,65) -- (105,65) ;
\draw [color={rgb, 255:red, 208; green, 2; blue, 27 }  ,draw opacity=1 ][line width=3.75]    (193,65) -- (213,65) ;
\draw [color={rgb, 255:red, 74; green, 144; blue, 226 }  ,draw opacity=1 ][line width=1.5]    (204,60) -- (204,68.6) -- (204,70) ;
\draw [color={rgb, 255:red, 231; green, 146; blue, 241 }  ,draw opacity=1 ][line width=1.5]    (212.99,60.13) -- (212.99,70.13) ;
\draw [color={rgb, 255:red, 231; green, 146; blue, 241 }  ,draw opacity=1 ][line width=1.5]    (84.99,60.67) -- (84.99,70.67) ;
\draw   (193,88.16) .. controls (193.05,90.9) and (194.45,92.24) .. (197.18,92.19) -- (197.18,92.19) .. controls (201.09,92.12) and (203.08,93.45) .. (203.13,96.19) .. controls (203.08,93.45) and (205.01,92.05) .. (208.92,91.98)(207.16,92.01) -- (208.92,91.98) .. controls (211.66,91.93) and (213.01,90.53) .. (212.96,87.79) ;

\draw (123.6,8.4) node [anchor=north west][inner sep=0.75pt]    {$0$};
\draw (244,9.4) node [anchor=north west][inner sep=0.75pt]    {$2^{N-4}$};
\draw (-18.4,8.6) node [anchor=north west][inner sep=0.75pt]    {$-2^{N-4}$};
\draw (188.7,2.26) node [anchor=north west][inner sep=0.75pt]    {$\textcolor[rgb]{0.29,0.56,0.89}{y}$};
\draw (125.2,73.4) node [anchor=north west][inner sep=0.75pt]    {$0$};
\draw (245.6,74.4) node [anchor=north west][inner sep=0.75pt]    {$2^{N-4}$};
\draw (-16.8,73.6) node [anchor=north west][inner sep=0.75pt]    {$-2^{N-4}$};
\draw (83.03,-23.5) node [anchor=north west][inner sep=0.75pt]  [color={rgb, 255:red, 231; green, 146; blue, 241 }  ,opacity=1 ]  {$\textcolor[rgb]{0.91,0.57,0.95}{t}$};
\draw (197.7,-27.93) node [anchor=north west][inner sep=0.75pt]  [color={rgb, 255:red, 231; green, 146; blue, 241 }  ,opacity=1 ]  {$\textcolor[rgb]{0.91,0.57,0.95}{t+2^{N+3}}$};
\draw (139.76,39.22) node [anchor=north west][inner sep=0.75pt]  [color={rgb, 255:red, 231; green, 146; blue, 241 }  ,opacity=1 ]  {$J_{t}$};
\draw (182.75,99.29) node [anchor=north west][inner sep=0.75pt]    {$2^{N-4} \varepsilon $};

\draw (83.5,70) node [anchor=north west][inner sep=0.75pt]  [color={rgb, 255:red, 208; green, 2; blue, 27 }  ,opacity=1 ]  {$\delta_{J_t} $};

\end{tikzpicture}

        \caption{$J_t$ is the dyadic interval from $\mathcal{D}_{0,j}^1+t$ that contains $x$.}
        \label{fig:J_t}
    \end{figure}

    This means that in this case we have that
\begin{align*}
    \mathcal{L}^1(\{w_1\in [-2^{N-4}, 2^{N-4}]: \exists \,I\in \mathcal{D}_{0,j}^1+w_1, \, x_1\in \delta_{I} \}) &= \mathcal{L}^1(J_{-1}\cup J_0) =  2^{N-2}\varepsilon.
\end{align*}

If $J$ is not completely contained inside $[0,2^{N-4}]$, by considering its periodic extension of period $2^{N-4}$ and its intersection with $[-2^{N-4}, 2^{N-4}]$, we still obtain a set of $\mathcal{L}^1$-measure $2^{N-4}\varepsilon$ (see Figure \ref{fig:intervals_tallats}).

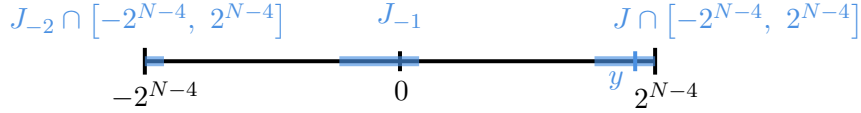
\begin{figure}[H]
    \centering

\tikzset{every picture/.style={line width=0.75pt}} 

\begin{tikzpicture}[x=0.75pt,y=0.75pt,yscale=-1,xscale=1]

\draw [line width=1.5]    (0.4,0.58) -- (256.4,0.58) ;
\draw [shift={(256.4,0.58)}, rotate = 180] [color={rgb, 255:red, 0; green, 0; blue, 0 }  ][line width=1.5]    (0,6.71) -- (0,-6.71)   ;
\draw [shift={(0.4,0.58)}, rotate = 180] [color={rgb, 255:red, 0; green, 0; blue, 0 }  ][line width=1.5]    (0,6.71) -- (0,-6.71)   ;
\draw [line width=1.5]    (128.4,-4.42) -- (128.4,5.58) ;
\draw [color={rgb, 255:red, 74; green, 144; blue, 226 }  ,draw opacity=1 ][line width=1.5]    (246.4,-4.42) -- (246.4,5.58) ;
\draw [color={rgb, 255:red, 74; green, 144; blue, 226 }  ,draw opacity=0.65 ][line width=3.75]    (226,0.58) -- (245.86,0.58) -- (256.4,0.58) ;
\draw [color={rgb, 255:red, 74; green, 144; blue, 226 }  ,draw opacity=0.65 ][line width=3.75]    (98,0.58) -- (138,0.58) ;
\draw [color={rgb, 255:red, 74; green, 144; blue, 226 }  ,draw opacity=0.65 ][line width=3.75]    (0.4,0.58) -- (10,0.58) ;

\draw (124,8.98) node [anchor=north west][inner sep=0.75pt]    {$0$};
\draw (244.4,9.98) node [anchor=north west][inner sep=0.75pt]    {$2^{N-4}$};
\draw (-18,9.18) node [anchor=north west][inner sep=0.75pt]    {$-2^{N-4}$};
\draw (231.59,4) node [anchor=north west][inner sep=0.75pt]    {$\textcolor[rgb]{0.29,0.56,0.89}{y}$};
\draw (233.6,-30) node [anchor=north west][inner sep=0.75pt]  [color={rgb, 255:red, 74; green, 144; blue, 226 }  ,opacity=1 ]  {$J\cap \left[ -2^{N-4} ,\ 2^{N-4}\right]$};
\draw (115.2,-30) node [anchor=north west][inner sep=0.75pt]  [color={rgb, 255:red, 74; green, 144; blue, 226 }  ,opacity=1 ]  {$J_{-1}$};
\draw (-69.2,-30) node [anchor=north west][inner sep=0.75pt]  [color={rgb, 255:red, 74; green, 144; blue, 226 }  ,opacity=1 ]  {$J_{-2} \cap \left[ -2^{N-4} ,\ 2^{N-4}\right]$};

\end{tikzpicture}

    \caption{The interval $J$ is not necessarily completely contained in $[-2^{N-4},2^{N-4}]$.}
    \label{fig:intervals_tallats}
\end{figure}

Now, if $j<N-4$, the set we obtain is morally the same, but at a much smaller scale and instead of having the measure of $2$ intervals ($J_{-1}$ and $J_0$) of length $2^{N-3}\varepsilon$, we have $2^{N-3-j}$ intervals of length $2^{j+1}\varepsilon$. Hence, in this case we find that
$$
\mathcal{L}^1(\{w_1\in [-2^{N-4}, 2^{N-4}]: \exists \,I\in \mathcal{D}_{0,j}^1+w_1, \, x_1\in \delta_{I} \}) = 2^{N-3-j}\cdot 2^{j+1}\varepsilon = 2^{N-2}\varepsilon,
$$
which is the same that we found in the case that $j=N-4$. Lastly, if $j>N-4$, for every fixed $\varepsilon>0$, the measure that we obtain is $\text{min}(2^{N-3}, 2^{j+1}\varepsilon)\leq 2^{N+1}\varepsilon$. This means that in any case, 
$$
p_{\varepsilon}^{1,j}\leq (2^{N-3})^{d-1}\, 2^{N+1}\varepsilon,
$$
which has limit $0$ as $\varepsilon\to0^+$, independently of $x$ and of $k$, as we wanted to prove.
\end{proof}

Along with the notion of $p_\varepsilon$, we will also need to define, for $Q\in \mathcal{D}^{\text{tr},1}_1$, its \textbf{bad part} as
\begin{equation}
Q_\text{b} = Q\cap \bigg( \bigcup_{\substack{R\in \mathcal{D}^{\text{tr},2}_2\\ 2^{-m}\ell(R)\leq \ell(Q)\leq 2^m \ell(R)}}\delta_R\bigg),\label{part_dolenta_dun_cub}
\end{equation}
and the same swapping the roles of $\mathcal{D}_1$ and $\mathcal{D}_2$. The reason for the size condition in the second line is the analogous one that appears in the third line of (\ref{suma_A}), which is what we are trying to control.

Now, we are going to deal with transit cubes whose parents are also transit. This motivates us to write, for any $k\in \mathbb{Z}$,
$$
f^k = \sum_{\substack{Q\in \mathcal{D}^{\text{tr},1}_1\\ \ell(Q) = 2^k}} \Delta_{1,Q}f, \quad f^k_{b_1} = \sum_{\substack{Q,\widehat{Q}\in \mathcal{D}^{\text{tr},1}_1\\ \ell(Q)= 2^k}}\chi_{Q_\text{b}}\Delta_{1,Q}f = \sum_{\substack{Q,\widehat{Q}\in \mathcal{D}^{\text{tr},1}_1\\ \ell(Q)=2^k}}c_Q(f)\chi_{Q_\text{b}}b_1,
$$
where the coefficients $c_Q(f)$ are defined by the same rule as in (\ref{coeficients_cp}). We define $g^k$ and $g^k_{b_2}$ analogously. By Lemma \ref{lema_descomposicio_L2}, we can write
\begin{align*}
    f &= \Xi_1f + \sum_{k\in \mathbb{Z}}f^k = \Xi_1 f + \sum_{k\leq N}f^k,\\
    g &= \Xi_2 g + \sum_{k\in \mathbb{Z}}g^k = \Xi_2 g + \sum_{k\leq N}g^k,
\end{align*}
in $L^2(\mu)$, because the series above are unconditionally convergent and, by definition, transit cubes have side length bounded above by $2^N$. The importance of the functions $f^k_{b_1}$ is reflected in the following lemma.

\begin{lemma}\label{lema:part_dolenta_k}
    Denote by $P^{\Omega^2}= P^\Omega\times P^{\Omega}$ and for each $k\in \mathbb{Z}$, let $f^k_{b_1}$ be as above. Then,
    $$
    \mathbb{E}_{P^{\Omega^2}}\bigg(\sum_{k\leq N}\|f^k_{b_1}\|^2_{L^2(\mu)}\bigg)  \leq p_\varepsilon \,c_3\|f\|^2_{L^2(\mu)}.
    $$
\end{lemma}

\begin{proof}
    The key step to prove the inequality above is to notice that for any $x\in \R^d$ and $Q\in \mathcal{D}^{\text{tr},1}_1$, $P^{\Omega}(x\in Q_{\text{b}})\leq p_{\varepsilon}$. This is because if $x\in Q_\text{b}$, then there is $R\in \mathcal{D}_2^{\text{tr},2}$ such that $x\in \delta_R$. Hence, taking $k=\log_2\ell(Q)$ in (\ref{p_epsilon}), we see that the desired inequality holds.

    For any $x\in \text{supp}(\mu)$ and $k\leq N$, we have
    $$
f^k_{b_1}(x) = \begin{cases}
    c_Q(f)\chi_{Q_{\text{b}}}(x)b_1(x), \quad &\text{if }\,x\in Q\in \mathcal{D}^{\text{tr},1}_1, \text{ with }\, \widehat{Q}\in \mathcal{D}^{\text{tr},1}_1, \\
    0, &\text{otherwise.}
\end{cases}
$$
Hence, taking expectation over $w_2\in \Omega$, for $x$ as in the first case above,
\begin{align*}
    \mathbb{E}_{P^\Omega}\left(|f^k_{b_1}(x)|^2\right) & =  \int_{\Omega}|c_Q(f)\chi_{Q_\text{b}}(x)b_1(x)|^2 \,dP^{\Omega}(w_2) = |c_Q(f)b_1(x)|^2 \int_{\Omega}\chi_{Q_\text{b}}(x)\,dP^\Omega(w_2) \\
    &= |c_Q(f)b_1(x)|^2\, P^\Omega\left(x\in Q_\text{b}\right) \leq p_\varepsilon\,|\Delta_{1,Q}f(x)|^2 = p_\varepsilon\ |f^k(x)|^2,
\end{align*}
because of our first remark. For $x$ as in the second case, the expectation is trivially zero. This shows that
\begin{align*}
    \mathbb{E}_{P^{\Omega}}\left(\sum_{ k\leq N}\|f^k_{b_1}\|^2_{L^2(\mu)}\right) &= \sum_{k\leq N}\int\mathbb{E}_{P^{\Omega}}\left(|f^k_{b_1}(x)|^2\right)\,d\mu(x)\\
    &\leq p_\varepsilon \sum_{k\leq N}\int|f^k(x)|^2d\,\mu(x)\\
    &= p_\varepsilon \sum_{ k \leq N}\|f^k\|^2_{L^2(\mu)}\leq   p_\varepsilon\, c_3\|f\|^2_{L^2(\mu)},
\end{align*}
where in the last inequality we have used again Lemma \ref{lema_descomposicio_L2}. Taking expectation once again, we obtain the inequality in the statement of the current lemma.
\end{proof}
Moreover, for each $k\leq N$, we can write 
\begin{align*}
    \|f^k_{b_1}\|^2_{L^2(\mu)} &= \int\bigg|\sum_{\substack{Q,\,\widehat{Q}\in \mathcal{D}^{\text{tr},1}_1\\ \ell(Q)=2^k}} c_Q(f)\chi_{Q_{\text{b}}}(x)b_1(x)\bigg|^2\,d\mu(x) = \sum_{\substack{Q, \, \widehat{Q}\in \mathcal{D}^{\text{tr},1}_1\\ \ell(Q)=2^k}}\int_{Q_\text{b}}|c_Q(f)b_1(x)|^2\,d\mu(x)\\
    &= \sum_{\substack{Q, \widehat{Q}\in \mathcal{D}^{\text{tr},1}_1\\ \ell(Q)=2^k}}|c_Q(f)|^2\|\chi_{Q_{\text{b}}}b_2\|^2_{L^2(\mu)}.
\end{align*}
Combining this identity with the inequality given in the preceding lemma, we have found that
\begin{equation*}
    \mathbb{E}_{P^{\Omega^2}}\bigg(\sum_{\substack{Q,\,\widehat{Q}\in \mathcal{D}^{\text{tr},1}_1}}|c_Q(f)|^2\|\chi_{Q_\text{b}}b_1\|^2_{L^2(\mu)}\bigg) = \mathbb{E}_{P
^{\Omega^2}}\bigg(\sum_{k\leq N}\|f^k_{b_1}\|^2_{L^2(\mu)}\bigg)\leq p_\varepsilon\, c_3\|f\|^2_{L^2(\mu)}.
\end{equation*}
Clearly, the analogous estimate for the function $g$ also holds,
\begin{equation}
    \mathbb{E}_{P^{\Omega^2}}\bigg(\sum_{\substack{R,\,\widehat{R}\in \mathcal{D}^{\text{tr},2}_2}}|c_R(g)|^2\|\chi_{R_\text{b}}b_2\|^2_{L^2(\mu)}\bigg) \leq p_\varepsilon\, c_3\|g\|^2_{L^2(\mu)}.\label{desigualtat_p_epsilon_g}
\end{equation}
These two identities are going to be essential to bound the last sum in (\ref{A_i}). 

\subsection{The good set $G$}

Recall that the conclusion from Theorem \ref{TEOREMA} was that there exists a set 

$$
G\subset F\setminus \bigcap_{w\in \R^d}\left(H_{\mathcal{D}(w)}\cup T^1_{\mathcal{D}(w)}\cup T^2_{\mathcal{D}(w)}\right)
$$
satisfying properties (i)-(iii). In this section we will give the precise definition of $G$. For this, let $W_{\mathcal{D}_1}$ and $W_{\mathcal{D}_2}$ be the total exceptional sets, as defined in Section \ref{subsec:main_lemma_good_functions}, corresponding to two dyadic lattices $\mathcal{D}_1=\mathcal{D}(w_1)$ and $\mathcal{D}_2=\mathcal{D}(w_2)$. Using (\ref{mu_htt}), we get that for some $0<\delta_1<1$ and for all $w\in \Omega$,
\begin{align*}
    \mu(W_{\mathcal{D}(w)}) &= \mu(H_{\mathcal{D}(w)}\cup T^1_{\mathcal{D}(w)}\cup T^2_{\mathcal{D}(w)}\cup S)\\
    &\leq \mu(H_{\mathcal{D}(w)}\cup T^1_{\mathcal{D}(w)}\cup T^2_{\mathcal{D}(w)}) + \mu(S \setminus H_{\mathcal{D}(w)}) \leq \delta_1\mu(F),
\end{align*}
from which we deduce that $\mu(F\setminus W_{\mathcal{D}(w)})\geq (1-\delta)\mu(F)$. Using this inequality, we are going to define $G$ using, once again, a probabilistic argument. For each $x\in F$, we consider the probabilities
\begin{align*}
p_0(x)&:= P^{\Omega}\left(\left\{w\in \Omega: x\in F\setminus W_{\mathcal{D}(w)}\right\}\right),\\
p(x)&:= P^{\Omega^2}\left(\left\{(w_1,w_2)\in \Omega\times \Omega: x\in F\setminus (W_{\mathcal{D}(w_1)}\cup W_{\mathcal{D}(w_2)})\right\}\right).
\end{align*}
Recall that $P^{\Omega^2}= P^{\Omega}\times P^{\Omega}$, which is the product of the normalized $d$-dimensional Lebesgue measure in $\Omega$ with itself, so applying Fubini's theorem we obtain the relation $p(x)=p_0(x)^2$. Moreover, integrating with respect to $\mu$, we find that
\begin{align}
    \int_F p_0(x)\,d\mu(x) &= \int_F \left(\int_{\Omega}\chi_{F\setminus W_{\mathcal{D}(w)}}(x)\, dP^{\Omega}(w)\right)\, d\mu(x)\nonumber\\
    &=\int_{\Omega} \mu(F\setminus W_{\mathcal{D}(w)})\,dP^{\Omega}(w)\geq (1-\delta_1)\mu(F).\label{integral_p0}
\end{align}
We define the sets
$$
G := \left\{x\in F : p_0(x)>\frac{(1-\delta_1)}{2}\right\},
$$
and $B:= F\setminus G$. The motivation for defining such a set $G$ is that for every $x\in G$, we have a strictly positive lower bound of $p(x)$, namely,
\begin{equation}
p(x) = p_0(x)^2 >\frac{(1-\delta_1)^2}{4}:=\beta>0.\label{definicio_beta}
\end{equation}
Informally, this number being strictly positive means that for a big proportion of dyadic lattices $\mathcal{D}(w_1), \mathcal{D}(w_2)$, the points from $G$ are in $F\setminus (W_{\mathcal{D}(w_1)}\cup W_{\mathcal{D}(w_2)})$.

We claim that $G$ is the desired set from Theorem \ref{TEOREMA}. First, if we take $x\in \bigcap_{w\in \mathbb{R}^d}(H_{\mathcal{D}(w)}\cup T^1_{\mathcal{D}(w)}\cup T^2_{\mathcal{D}(w)})$, in particular, $x\not \in F\setminus W_{\mathcal{D}(w)}$ for every $w\in \Omega$. Hence, $p_0(x)=0$. By definition, this means that $x\not\in G$. This proves that
$$
G\subset F \setminus \bigcap_{w\in \mathbb{R}^d}\left(H_{\mathcal{D}(w)}\cup T^1_{\mathcal{D}(w)}\cup T^2_{\mathcal{D}(w)}\right),
$$
as wished. Moreover, we have, for $x\in B$,
$$
p_0(x)\leq \frac{1-\delta_1}{2} \,\Leftrightarrow \, (1-p_0(x))\frac{2}{1+\delta_1} \geq 1,
$$
so we find, using the lower bound in (\ref{integral_p0}), that the $\mu$-measure of $B$ is controlled by
\begin{align*}
    \mu(B)&=\mu\left(\left\{x\in F : p_0(x)\leq \frac{1-\delta_1}{2}\right\}\right) \leq \frac{2}{1+\delta}\int_F(1-p_0(x))\,d\mu(x) \leq \frac{2\,\delta_1}{1+\delta_1}\mu(F).
\end{align*}
Hence, 
$$
\mu(G) = \mu(F)-\mu(B) \geq \frac{1-\delta_1}{1+\delta_1}\mu(F),
$$
which is condition (i) from Theorem \ref{TEOREMA}. Now, we define a function
\begin{alignat*}{2}
    &&\Omega\times \Omega\times \R^n &\longrightarrow [0,+\infty)\\
    &&(w_1,w_2,x)&\longmapsto \Phi_{(w_1,w_2)}(x)=\text{dist}(x, F\setminus(W_{\mathcal{D}(w_1)}\cup W_{\mathcal{D}(w_2)})).
\end{alignat*}
Observe that if $x\in G$, then
$$
P^{\Omega^2}(\{(w_1,w_2): \Phi_{(w_1,w_2)}(x)=0\}) \geq p(x)>\beta.
$$
From this, we deduce that 
$$
\mu(\{x\in F : p(x)>\beta\})\geq \mu(G) \geq \frac{1-\delta_1}{1+\delta_1}\mu(F),
$$
and
$$
\mu\left(\left\{x\in F : P^{\Omega^2}\left(\left\{(w_1,w_2): \Phi_{(w_1,w_2)}(x)=0\right\}\right)>\beta\right\}\right) \geq \frac{1-\delta_1}{1+\delta_1}\mu(F).
$$
The rest of the proof will be devoted to showing that $T$ is bounded in $L^2(\mu\lfloor G)$.

\subsection{The functions $\Phi, \Psi_{\text{w}}$ and the operators $\widetilde{K}$ and $\widetilde{T}$}

Recall that in our main lemma for good functions (Lemma \ref{lemma:good_functions}), we considered a function $\Theta$ satisfying certain conditions that allowed us to obtain the right bounds for $|\langle K_\Theta f, g\rangle|$, where $K_\Theta$ is the suppressed operator and $f$ and $g$ were good functions. In this section, we are going to choose an appropriate function $\Theta$, with which we will apply Lemma \ref{lemma:good_functions}. In addition, we will consider two operators, arising from the averages of our suppressed operator with respect to $P^{\Omega^2}$, which will be crucial in showing that $T$ is bounded in $L^2(\mu\lfloor G)$.

First, for a fixed $\varepsilon_0>0$, we define
\begin{equation}
    \Phi(x) = \varepsilon_0 + \inf_{\substack{B\subset \Omega^2\\ P^{\Omega^2}(B)=\beta}}\sup_{(w_1,w_2)\in B} \Phi_{(w_1,w_2)}(x)=: \varepsilon_0 + I(x) .\label{funcio_phi}
\end{equation}
In the following lemma, we show three relevant characteristics of this function.
\begin{lemma}\label{lema:propietats_phi}
    Let $\Phi$ be the function defined in (\ref{funcio_phi}). Then, the following three properties hold.
    \begin{enumerate}
        \item $\Phi$ is $1$-Lipschitz.
        \item $\Phi(x)=\varepsilon_0$ for all $x\in G$.
        \item For any $x\in F$,
        $$
        \Phi(x)\geq \max(\text{dist}(x, \mathbb{C}\setminus H), e_1(x),e_2(x), \varepsilon_0) \geq \max(\mathcal{R}(x), e_1(x), e_2(x), \varepsilon_0),
        $$
        where $e_1(x), e_2(x)$ are as we defined them in (\ref{e1}) and (\ref{e2}), respectively, and $H$ and $\mathcal{R}(x)$ in (\ref{definicio_conjunt_H}).
    \end{enumerate}
\end{lemma}
\begin{proof}
    For the first part, note that since for any fixed $(w_1,w_2)\in \Omega\times \Omega$, $x\mapsto \Phi_{(w_1,w_2)}(x)$ is $1$-Lipschitz, for a fixed $B\subset \Omega^2$, taking supremum in $(w_1,w_2)\in B$ still has the Lipschitz property. To see that it is also preserved after taking the infimum, for $x,y\in \R^d$ fixed and $\varepsilon>0$, let $B_y\subset \Omega^2$ be such that $P^{\Omega^2}(B)=\beta$ and 
$$
S_{B_y}(y) = \sup_{(w_1,w_2)\in B}\Phi_{(w_1,w_2)}(y) \leq \inf_{\substack{B\subset \Omega^2\\ P^{\Omega^2}(B)=\beta}}\sup_{(w_1,w_2)\in B} \Phi_{(w_1,w_2)}(y) +  \varepsilon.
$$
Then,
$$
I(x)-I(y) \leq S_{B_y}(x) - I(y) = (S_{B_y}(x) -S_{B_y}(y)) + (S_{B_y}(y) -I(y)) \leq |x-y|+\varepsilon.
$$
Since $\varepsilon$ was arbitrary, we obtain that $I(x)-I(y)\leq |x-y|$. Since $y,x$ were also arbitrary, the function $x\mapsto I(x)$ is $1$-Lipschitz, and so is $\Phi$.

Secondly, if $x\in G$, by definition we have that the set $A_x=\{(w_1,w_2): \Phi_{(w_1,w_2)}(x)>0\}$ satisfies that $P^{\Omega^2}(A_x)>0$. Recalling that $P^{\Omega^2}= P^{\Omega}\times P^{\Omega}$, and each of these is a normalization of $d$-dimensional Lebesgue measure, this implies that there exists some $B_x\subset A_x$ with $P(B_x)=\beta$. Then, it is clear that
$$
\sup_{(w_1,w_2)\in B_x}\Phi_{(w_1,w_2)}(x)=0,
$$
which implies that $I(x)=0$, as defined above during the proof of the first part. Hence, $\Phi(x)=0$ for such $x$. 

For the proof of the third point, it suffices to notice that for all $x\in F$ and $(w_1,w_2)\in \Omega\times \Omega$,
$$
\Phi_{(w_1,w_2)}(x) \geq \max(\text{dist}(x, \mathbb{C}\setminus H), e_1(x),e_2(x)).
$$
This is because if $B(x,r)> c_0r^n$, then this ball must be contained in $H_{\mathcal{D}}$ for any dyadic lattice $\mathcal{D}$, in particular, for $\mathcal{D}(w_1)$ and $\mathcal{D}(w_2)$, and, of course, the set $S$ is independent of $\mathcal{D}(w_1)$ and $\mathcal{D}(w_2)$.
\end{proof}
Given $\text{w}=(w_1,w_2)\in \Omega^2$, we set $\Psi_{\text{w}}(x) = \max(\Phi(x), \Phi_{\text{w}}(x))$, and we consider the operator obtained from the average
$$
\widetilde{K}f(x) = \int K_{\Psi_{\text{w}}}f(x)dP^{\Omega^2}(\text{w}).
$$
That is $\widetilde{K}$ is an integral operator with kernel given by
$$
\widetilde{k}(x,y) = \int \widetilde{k}_{\Psi_{\text{w}}}(x,y)\, dP^{\Omega^2}(\text{w}). 
$$
Analogously, we define an operator via averaging the truncated operators,
$$
\widetilde{T}f(x) = \int T_{\Psi_{\text{w}}(x)}f(x)\, dP^{\Omega^2}(\text{w}),
$$
with kernel
\begin{equation}
\widetilde{t}(x,y) = \int k_{\Psi_{\text{w}}}(x,y)\, dP^{\Omega^2}(\text{w}), \quad \text{where }\, k_{\Psi_\text{w}}(x,y) = k(x,y)\chi_{|x-y|\geq \Psi_{\text{w}}(x)}.\label{nucli_mitjana}
\end{equation}
Note that $\widetilde{k}(x,y)$ is a Calderón-Zygmund kernel, because it is the average of the kernels $k_{\Psi_{\text{w}}}(x,y)$, which we know that are also of this type. In contrast, this is not the case for the truncated kernel $\widetilde{t}(x,y)$, in general. 

Below, we are going to show, using the operators $\widetilde{K}$ and $\widetilde{T}$, that the suppressed operator $K_{\Phi}$ is bounded in $L^2(\mu)$, with norm independent of $\varepsilon_0$. The $L^2(\mu\lfloor G)$
boundedness of our SIO $T$ will be a consequence of the fact that $\Phi\equiv \varepsilon_0$ on the set $G$, which does not depend on $\varepsilon_0$. 

\begin{lemma}\label{lema:desigualtats_maximal}
    Consider the truncated kernels,
    $$
    k_\Phi(x,y) = k(x,y)\chi_{|x-y|\geq \Phi(x)}, \qquad k_{\Psi_{\text{w}}}(x,y) = k(x,y)\chi_{|x-y|\geq \Psi_{\text{w}}(x)},
    $$
    and $\widetilde{t}$ as in (\ref{nucli_mitjana}). Denote the associated operators by $T_\Phi$, $T_{\Psi_{\text{w}}}$, and $\widetilde{T}$, respectively. The maximal operator associated with the kernel $\widetilde{t}(x,y)$ is defined as
    $$
    \widetilde{T}_{*}f(x) = \sup_{\delta>0}|\widetilde{T}_\delta f(x)| = \sup_{\delta>0}\left| \int_{|x-y|\geq \delta}\widetilde{t}(x,y)\, f(x)\, d\mu(y)\right|, \quad \text{for }f\in L^2(\mu), x\in \R^d.
    $$
    Then, for any $f\in L^2(\mu)$ and $x\in \R^d$, we have that
    \begin{equation}
    |T_{\Phi}f(x)| \leq \frac{2}{\beta}\widetilde{T}_{*}f(x), \qquad |T_{\Psi_{\text{w}}}f(x)| \leq \frac{2}{\beta}\widetilde{T}_*f(x), \quad \text{for each }\, \text{w}\in \Omega^2,\label{desigualtats_maximal}
    \end{equation}
    where $\beta$ is the constant defined in (\ref{definicio_beta}).
\end{lemma}

\begin{proof}
    We start by proving the first inequality in (\ref{desigualtats_maximal}). To do so, we write the kernel $\widetilde{t}$ as
    \begin{equation}
    \widetilde{t}(x,y)  = k(x,y)P^{\Omega^2}\left(\left\{\text{w}\in \Omega^2 : |x-y|\geq \Psi_{\text{w}}(x)\right\}\right) = k(x,y)v_x(|x-y|),\label{re-escriptura_t}
    \end{equation}
    where we define
    $$
    v_x(t) = P^{\Omega^2}\left(\left\{ \text{w}\in \Omega^2 : t\geq \Psi_{\text{w}}(x) = \max(\Phi(x), \Phi_\text{w}(x))\right\}\right).
    $$
    We claim that the new function satisfies the following properties:
    \begin{enumerate}
        \item If $t<\Phi(x)$, then $v_x(t)=0$.
        \item $v_x$ is non-decreasing in $t$ and right-continuous.
        \item If $t\geq \Phi(x)$, then $v_x(t)\geq \beta$.
    \end{enumerate}
    The first property is clear: using that $\Phi(x)\leq \Psi_{\text{w}}$ for any $\text{w}\in \Omega^2$, if $t\leq \Phi(x)$, the probability in the definition of $v_x$ is obviously zero. For the second property, let $t\leq s$. Since
    $$
    \{\text{w}\in \Omega^2: s\geq \Psi_{\text{w}}(x)\} \supset \{\text{w}\in \Omega^2: t\geq \Psi_{\text{w}}(x)\},
    $$
    we have that $v_x(s)\geq v_x(t)$ and so $v_x$ is non-decreasing. The right-continuity follows from the fact that, for each $x$, $v_x$ is the distribution function of the random variable $\Psi_{\text{w}}(x)$. To see that the last property holds, note that if $t\geq \Phi(x)$, then
    \begin{equation}
    v_x(t) = P^{\Omega^2}\left(\left\{\text{w}\in \Omega^2 : t\geq \Phi_\text{w}(x)\right\}\right).\label{v_x}
    \end{equation}
    Moreover, recalling how we defined $\Phi(x)$ in (\ref{funcio_phi}), if $t\geq \Phi(x)$, then it also holds that
    $$
    t\geq \inf_{\substack{B\subset \Omega^2\\ P^{\Omega^2}(B)\geq \beta}}\sup_{\text{w}'\in B}\Phi_{\text{w}'}(x).
    $$
    Now, it can be that either the infimum is strictly smaller than $t$ or that it is equal to $t$. We will separate these two cases and give a different argument for each one in order to see that $v_x(t)\geq \beta$.
    \begin{itemize}
        \item If $t>I(x)$, let $0<\varepsilon<\frac{t-I(x)}{2}$. By definition of infimum, there is some $B_\varepsilon\subset \Omega^2$ with $P^{\Omega^2}(B_\varepsilon)\geq \beta$ such that
        $$
        \sup_{\text{w}'\in B_{\varepsilon}}\Phi_{\text{w}'}(x) \leq \inf_{\substack{B\in \Omega^2\\ P^{\Omega^2}(B)\geq \beta}}\sup_{\text{w}'\in B}\Phi_{\text{w}'}(x)+ \varepsilon =I(x)+\varepsilon <t.
        $$
        This means that $B_{\varepsilon}$ must be contained in the subset appearing on the right-hand side of (\ref{v_x}). Then, taking probabilities, we obtain that $v_x(t)\geq P^{\Omega^2}(B_\varepsilon)\geq \beta$. 
        \item If $t= I(x)$, let us call $s = \{s\in \R : v_x(s)\geq \beta\}$. We claim that $t= \inf S$. Indeed, if $s\in S$, we call $B_s = \{\text{w}\in\Omega^2: \Phi_{\text{w}}(x)\leq s\}$. Then, $P^{\Omega^2}(B_s)= v_x(s)\geq \beta$. Hence, $t\leq s$. For the converse inequality, let $B\subset \Omega^2$ be such that $P^{\Omega^2}(B)\geq \beta$, and denote $s_B= \sup_{\text{w}\in B}\Phi_\text{w}(x)$, so obviously $B\subset \{\text{w}\in \Omega^2 : \Phi_\text{w}(x)\leq s_B\}$. Taking probabilities we have that
        $$
        \beta \leq P^{\Omega^2}\left(\left\{\text{w}\in \Omega^2: \Phi_{\text{w}}(x)\leq s_B\right\}\right) = v_x(s_B),
        $$
        from which we have that $s_B\in S$. Since $t$ is the infimum of the $s_B$, we obtain the desired inequality. Lastly, by definition of infimum, we can take $(s_n)\subset S$ a decreasing sequence with $t = \lim_n s_n$. Since for all $n$, $v_x(s_n)\geq \beta$, by right-continuity of $v_x$ we obtain that $v_x(t)\geq \beta$, which is what we wanted to show.
    \end{itemize}
    As well as these three properties of the function $v_x$, we will require an additional identity, which we deduce now. From (\ref{re-escriptura_t}) and the first property above, we infer that
    $$
        \widetilde{t}(x,y)= \widetilde{t}(x,y)\chi_{|x-y|\geq \Phi(x)} = k(x,y)\,v_x(|x-y|)\chi_{|x-y|\geq \Phi(x)} = v_x(|x-y|)\,k_{\Phi}(x,y).
    $$
    Rearranging, 
    \begin{equation}
    k_{\Phi}(x,y) = v_x(|x-y|)^{-1}\,\widetilde{t}(x,y)\,\chi_{|x-y|\geq \Phi(x)}.\label{k_phi}
    \end{equation}
    Using this identity and the previous three properties, we will be able to write the kernel $k_\Phi$ as a convex combination of the kernels $\widetilde{t}(x,y)\chi_{|x-y|\geq t}$, for $t\geq \Phi(x)$, which we will integrate with respect to $\mu$ to obtain our result. To do so, since $v_x(t)^{-1}$, as a function of $t$, is continuous from the right and non-increasing, we can define a Lebesgue-Stieljes measure $\sigma$ on $\R$ by considering
    \begin{equation*}
        \begin{cases}
            \sigma(a,b] = v_x(a)^{-1}- v_x(b)^{-1}, \qquad \text{if }\,(a,b]\subset (\Phi(x),\infty),\\
            \sigma(-\infty, \Phi(x)] = 0.
        \end{cases}
    \end{equation*}
    From the non-increasing property, we see that this is indeed a positive measure. Moreover, it is finite, because
    $$
    \sigma(\R)\leq \sup_{t\geq \Phi(x)}\frac{1}{v_x(t)}\leq \beta,
    $$
    by the third property from the start of the proof. Now, we write
    \begin{align*}
    \int_{t>\Phi(x)}\widetilde{t}(x,y)\chi_{|x-y|\geq t}\, d\sigma(t) &= \widetilde{t}(x,y)\,\sigma((\Phi(x), |x-y|])\\
    &=\widetilde{t}(x,y)\,\chi_{|x-y|\geq \Phi(x)}(v_x(\Phi(x))^{-1}- v_x(|x-y|)^{-1}),
    \end{align*}
    which, using (\ref{k_phi}), is the same as
    $$
    k_{\Phi}(x,y) = v_x(\Phi(x))^{-1}\widetilde{t}(x,y)\,\chi_{|x-y|\geq \Phi(x)} -  \int_{t>\Phi(x)}\widetilde{t}(x,y)\chi_{|x-y|\geq t}\, d\sigma(t),
    $$
    which is the convex combination that we mentioned above. If we multiply the left hand side of this identity by $f(y)$ and integrate with respect to $\mu$ on $y$, we get $T_{\Phi}f(x)$. If we do the same for the right hand side, we obtain
    \begin{align*}
        &v_x(\Phi(x))^{-1}\int_{|x-y|\geq \Phi(x)} \widetilde{t}(x,y) f(y)\,d\mu(y) -\int_{t>\Phi(x)}\bigg(\int_{|x-y|\geq t}\widetilde{t}(x,y)f(y)\,d\mu(y)\bigg)\,d\sigma(t)\\
        &\qquad \qquad = v_x(\Phi(x))^{-1}\widetilde{T}_{\Phi(x)}f(x) -\int_{t>\Phi(x)}\widetilde{T}_{t}f(x)\,d\sigma(t),
    \end{align*}
    where $\widetilde{T}_t$ and $\widetilde{T}_{\Phi(x)}$ are the $t$-truncated and $\Phi(x)$-truncated versions of the operator $\widetilde{T}$ defined in the statement of the lemma, respectively. From this, since $\sigma(\R)=v_x(\Phi(x))^{-1}\leq  \beta^{-1}$, we obtain the first inequality in (\ref{desigualtats_maximal}). Indeed,
    $$
    |T_{\Phi}f(x)| \leq \frac{1}{\beta}\widetilde{T}_*f(x) + \int_{t>\Phi(x)}\widetilde{T}_*f(x)\,d\sigma(x) \leq \frac{2}{\beta}\,\widetilde{T}_*f(x).
    $$
    It remains to prove the second inequality in (\ref{desigualtats_maximal}). The strategy to do so will be analogous to what we have just argued. Indeed, notice that by the definition of $T_{\Psi_\text{w}}$, we have that
    $$
    k_{\Psi_\text{w}}(x,y) = k_{\Phi}(x,y)\,\chi_{|x-y|\geq \Psi_\text{w} (x)} = v_x(|x-y|)^{-1}\,\widetilde{t}(x,y)\,\chi_{|x-y|\geq \Psi_\text{w}(x)}.
    $$
    Observe that this is the same identity as (\ref{k_phi}), replacing $\Phi$ by $\Psi_\text{w}$. Arguing analogously as before, integrating with respect to $\sigma$ for $t>\Psi_\text{w}(x)$, we obtain that
    $$
    T_{\Psi_\text{w}}f(x) = v_x(\Psi_\text{w}(x))^{-1}\widetilde{T}_{\Psi_\text{w}(x)}f(x)- \int_{t>\Psi_\text{w}(x)}\widetilde{T}_{t}f(x)\,d\sigma(x).
    $$
    Therefore, 
    $$
    |T_{\Psi_\text{w}}f(x)| \leq \frac{1}{\beta}\widetilde{T}_{*}f(x) + \int_{t>\Psi_\text{w}(x)}\widetilde{T}_*f(x)\,d\sigma(t) \leq \frac{2}{\beta}\widetilde{T}_*f(x),
    $$
    which is exactly the second inequality in (\ref{desigualtats_maximal}).
\end{proof}

Notice that the operator $\widetilde{T}$ is the average of the operators $T_{\Psi_{\text{w}}}$ over $\text{w}\in \Omega^2$. From this, we see that (\ref{desigualtats_maximal}) asserts that each of the operators $T_{\Psi_{\text{w}}}$ is controlled by the maximal version of the average operator $\widetilde{T}$.

Lastly, we are going to see that we can control the operator norms of both $K_\Phi$ and $K_{\Psi_\text{w}}$ in $L^2(\mu)$ by that of $\widetilde{K}$.

\begin{lemma}\label{desigualtats_tilde}
    We have that
    \begin{equation}
    \|K_\Phi\|_{L^2(\mu)\to L^2(\mu)} \leq c\beta^{-1}\|\widetilde{K}\|_{L^2(\mu)\to L^2(\mu)} + c\beta^{-1},\label{desigualtat_k_phi_beta}
    \end{equation}
    and
    \begin{equation}
    \|K_{\Psi_\text{w}}\|_{L^2(\mu)\to L^2(\mu)} \leq c\beta^{-1}\|\widetilde{K}\|_{L^2(\mu)\to L^2(\mu)} + c\beta^{-1}, \quad \text{for each w}\in \Omega^2.\label{desigualtat_k_psi_beta}
    \end{equation}
\end{lemma}

\begin{proof}
    We start by applying Lemma \ref{diferencia_truncat_suppressed}, with the measure $\sigma = f\mu$, for $f\in L^2(\mu)$. If $\Theta\,\colon\R^d\to (0,+\infty) $ is a $1$-Lipschitz function, for $\varepsilon\geq \Theta(x)$, we have that
    \begin{equation}
    |T_{\varepsilon}\sigma(x)-K_{\Theta,\varepsilon}\sigma(x)| \leq c\sup_{r\geq \varepsilon}\frac{1}{r^n}\int_{B(x,r)}|f|\,d\mu,\label{deduccio_5.4}
    \end{equation}
    where we denote by $K_{\Theta, \varepsilon}$ the $\varepsilon$-truncated version of the suppressed operator $K_\Theta$. For an arbitrary $\varepsilon\geq 0$, we denote $T_{\Theta, \varepsilon}f(x) = T_{\max(\Theta(x), \varepsilon)}f(x)$. We claim that then,
    $$
    |K_{\Theta,\varepsilon}f(x)-T_{\Theta,\varepsilon}f(x)| \leq c\sup_{r\geq \Theta(x)}\frac{1}{r^n}\int_{B(x,r)}|f|\,d\mu.
    $$
    Indeed, for $\varepsilon\geq \Theta(x)$ the preceding inequality is simply (\ref{deduccio_5.4}). For $\varepsilon<\Theta(x)$, we can write
    \begin{align*}
        |K_{\Theta,\varepsilon}f(x)-T_{\Theta, \varepsilon}f(x)| &= |K_{\Theta, \varepsilon}f(x)-T_{\Theta(x)}f(x)|\\
        &\leq |K_{\Theta, \Theta(x)}f(x)-T_{\Theta(x)}f(x)| + |K_{\Theta, \varepsilon}f(x)-K_{\Theta, \Theta(x)}f(x)|\\
        &=|K_{\Theta, \Theta(x)}f(x)-T_{\Theta(x)}f(x)| + \int_{B(x, \Theta(x))}|\widetilde{k}_{\Theta}(x,y)|\,|f(y)|\,d\mu(y)\\
        &\leq c\,\sup_{r\geq \Theta(x)}\frac{1}{r^n}\int_{B(x,r)}|f(y)|\,d\mu(y) + \int_{B(x,\Theta(x))}\frac{c}{\Theta(x)^n}|f(y)|\,d\mu(y)\\
        &\leq c\sup_{r\geq \Theta(x)}\frac{1}{r^n}\int_{B(x,r)}|f|\,d\mu
    \end{align*}
    where we have used (\ref{deduccio_5.4}) with $\varepsilon=\Theta(x)$. This settles our claim. Suppose now that $\Theta$ is one of the Lipschitz functions $\Phi$ or $\Psi_\text{w}$, defined in the beginning of this section. Since, in any case, $\mu(B(x,r))\leq c_0r^n$ for any $r\geq \Theta(x)$, the supremum above is bounded by $c_0 M_{\mu}f(x)$, the maximal function. So, we have proved that
    \begin{equation}
    |K_{\Theta, \varepsilon}f(x)- T_{\Theta, \varepsilon}f(x)|\leq c\,c_0 M_{\mu}f(x), \quad \text{for any }\varepsilon\geq 0.\label{diferencia_hl}
    \end{equation}
    Choosing $\Theta=\Phi$ and $\varepsilon=0$, from Lemma \ref{lema:desigualtats_maximal}, we infer that
    \begin{equation}
    |K_{\Phi}f(x)| \leq |T_{\Phi}f(x)| + c\,c_0M_{\mu}f(x) \leq 2\beta^{-1}\widetilde{T}_*f(x) + c\,c_0M_\mu f(x).\label{desigualtat_k_phi_maximal}
    \end{equation}
    Thus, knowing that the centered maximal Hardy-Littlewood operator is bounded in $L^2(\mu)$, we obtain that
    $$
    \|K_{\Phi}\|_{L^2(\mu)\to L^2(\mu)} \leq 2\beta^{-1}\| \widetilde{T}_*\|_{L^2(\mu)\to L^2(\mu)} + c.
    $$
    Moreover, setting $\Theta= \Psi_{\text{w}}$ in (\ref{diferencia_hl}), and taking the mean on $\text{w}\in \Omega^2$, we get
    \begin{align*}
        |\widetilde{K}_{\varepsilon}f(x)- \widetilde{T}_\varepsilon f(x)| &=\bigg|\int_{\Omega^2}\left(K_{\Psi_{\text{w}},\varepsilon}f(x)- T_{\Psi_{\text{w}},\varepsilon}f(x)\right)dP^{\Omega^2}(\text{w})\bigg| \leq c\,c_0 M_\mu f(x).
    \end{align*}
    Taking supremum in $\varepsilon\geq 0$, we deduce that
    $$
    \widetilde{T}_*f(x) \leq \widetilde{K}_*f(x) + c\,c_0 M_\mu f(x),
    $$
    where we denote $\widetilde{K}_*f(x) = \sup_{\varepsilon>0}|\widetilde{K}_\varepsilon f(x)|$. If we plug this in (\ref{desigualtat_k_phi_maximal}), 
    $$
    |K_{\Phi}f(x)|\leq 2\beta^{-1}\widetilde{T}_*f(x) + c\,c_0M_\mu f(x)\leq 2\beta^{-1}\widetilde{K}_*f(x) + c\, c_0\,\beta^{-1} M_\mu f(x),
    $$
    where we also used the fact that $0<\beta<1$. Then, noticing that each $\widetilde{k}_{\Psi_\text{w}}$ satisfies condition (\ref{condicio_nucli_t}) with $\varepsilon=\varepsilon_0$ and so the same holds for the average kernel $\widetilde{k}$, applying Lemma \ref{lema:primer_cotlar}, we know that
    $$
    \|\widetilde{K}_*\|_{L^2(\mu)\to L^2(\mu)} \leq c\,\|\widetilde{K}\|_{L^2(\mu)\to L^2(\mu)}+c.
    $$
    Therefore,
    $$
    \|K_{\Phi}\|_{L^2(\mu)\to L^2(\mu)} \leq c\beta^{-1}\|\widetilde{K}\|_{L^2(\mu)\to L^2(\mu)} + c\beta^{-1},
    $$
    which is (\ref{desigualtat_k_phi_beta}). The proof of (\ref{desigualtat_k_psi_beta}) is analogous and we omit it.
\end{proof}

\subsection{The key estimates for $A_i$}

\label{sec:acotacio_ai}

Recall that for a function $\Theta$ and good functions $f,g\in L^2(\mu)$ satisfying the conditions of Lemma \ref{lemma:good_functions}, we were only able to bound
$$
|\langle K_{\Theta}f, g\rangle| \leq c_{a}\,\|f\|_{L^2(\mu)}\|g\|_{L^2(\mu)} + A_1 + A_2 + A_3+ A_4,
$$
where $c_a$ is a constant that depends on some small number $\varepsilon_a$, to be fixed later, and the constant $\alpha$ from Section \ref{sec:exceptional_set}, and the sums $A_{i}$ were defined in (\ref{A_i}). This section is devoted to giving a probabilistic bound for the terms $A_i$. That is, we are going to bound their expectation with respect to $P^{\Omega^2}$. From now on, we are going to choose $\Theta= \Psi_{\text{w}}$, for $\text{w}\in \Omega^2$. 

Remember that the terms $A_1, A_2$ and $A_3$ concern the case where one of the cubes is terminal, while in $A_4$, all cubes that appear are transit. We will now see that the estimates for the first three terms are very similar. Recall that for a cube $P\in \mathcal{D}_1$, we say that it is terminal (of the first kind) if either $2P\subset H_{\mathcal{D}_1}$ or if $P\subset T^1_{\mathcal{D}_1}$. In the first case, note that then, for any $x\in P$, $y\in \R^d$,
\begin{equation*}
    |\widetilde{k}_{\Psi_\text{w}}(x,y)| \leq \frac{c}{\Psi_{\text{w}}(x)^n}\leq \frac{c}{\Phi(x)^n} \leq \frac{c}{\text{dist}(x, \R^d\setminus H)^n} \leq \frac{c}{\ell(P)^n}.
\end{equation*}
Analogously, if $S\in \mathcal{D}_2$ is terminal because $2S\subset H_{\mathcal{D}_2}$,
$$
|\widetilde{k}_{\Psi_\text{w}}(x,y)| \leq \frac{c}{\ell(S)^n}, \quad \text{for all }x\in S, y\in \R^d.
$$
\begin{lemma}\label{lema:esperança_a1}
    Let $A_1$ be as in (\ref{A_i}), with $\Theta = \Psi_{\text{w}}$, for $\text{w}\in \Omega^2$. Then, for any $\gamma>0$, and $f,g$ as in Lemma \ref{lemma:good_functions},
    \begin{align*}
    \mathbb{E}_{P^{\Omega^2}}(A_1) &= \mathbb{E}_{P^{\Omega^2}}\bigg(\sum_{\mathcal{A}_{1,2}\cup \mathcal{A}_{2,1}\cup \mathcal{A}_{2,2}} |\langle K_{\Psi_{\text{w}}}(\chi_{P\setminus S}\Delta_{1,Q}f), \chi_S\Delta_{2,R}g\rangle|\bigg) \\
    &\leq c\,\sqrt{\gamma}\sup_{\text{w}\in \Omega^2}\|K_{\Psi_{\text{w}}}\|_{L^2(\mu)\to L^2(\mu)} \|f\|_{L^2(\mu)}\|g\|_{L^2(\mu)}.
    \end{align*}
\end{lemma}

\begin{proof}
    We will only show the details for $(P,S)\in \mathcal{A}_{1,2}$. All the other cases are analogous. Assume first that $S$ is terminal because $2S\subset H_{\mathcal{D}_2}$. Then,
    \begin{align*}
        |\langle K_{\Psi_\text{w}}(\chi_{P\setminus S}\Delta_{1,Q}f), \chi_S\Delta_{2,R}g\rangle| &\leq \int_{x\in S}|\Delta_{2,R}g(x)|\bigg(\int_{y\in P\setminus S}|\widetilde{k}_{\Psi_{\text{w}}}(x,y)|\,|\Delta_{1,Q}f(y)|\,d\mu(y)\bigg)d\mu(x)\\
        &\leq \frac{c}{\ell(S)^n}\|\Delta_{2,R}g\|_{L^2(\mu)}\|\Delta_{1,Q}f\|_{L^2(\mu)}\mu(Q)^{\frac{1}{2}}\mu(R)^{\frac{1}{2}}\\
        &\leq c \|\Delta_{2,R}g\|_{L^2(\mu)}\|\Delta_{1,Q}f\|_{L^2(\mu)}\frac{\mu(Q)^{\frac{1}{2}}\mu(R)^{\frac{1}{2}}}{\ell(Q)^{\frac{n}{2}}\ell(R)^{\frac{n}{2}}}\\
        &\leq c \|\Delta_{2,R}g\|_{L^2(\mu)}\|\Delta_{1,Q}f\|_{L^2(\mu)}.
    \end{align*}
    where we have used that $P\in\mathcal{CH}(Q), S\in \mathcal{CH}(R)$, $\ell(P)\approx \ell(S)$ and, lastly, that since both $Q$ and $R$ are transit, $\mu(Q)\leq c_0 \ell(Q)^n, \mu(R)\leq c_0\ell(R)^n$. Recalling that the cubes $Q$ can only interact with a uniformly bounded number of cubes $R$ and vice versa, we get that the sum over the terminal cubes $2S\subset H_{\mathcal{D}_2}$ is bounded above by $c\,\|f\|_{L^2(\mu)}\|g\|_{L^2(\mu)}$. The same holds if $2P \subset H_{\mathcal{D}_1}$.

    Now, assume that $S$ is terminal because $S\subset T^2_\mathcal{D}$. In this case, we divide $P\setminus S = P_0 \cup \delta_{PS}$, where (see Figure \ref{fig:P_menys_S}).
    $$
    \delta_{PS} = \mathcal{U}_{\gamma \ell(P)}(\partial S)\cap (P\setminus S).
    $$
    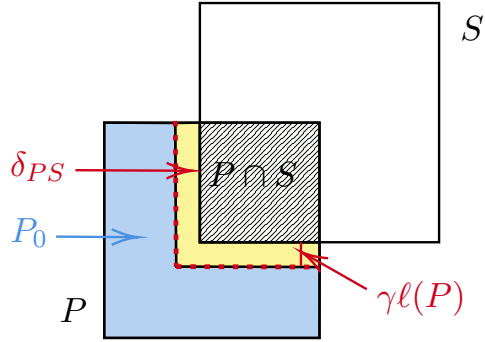
\begin{figure}[h]
        \centering

 
\tikzset{
pattern size/.store in=\mcSize, 
pattern size = 5pt,
pattern thickness/.store in=\mcThickness, 
pattern thickness = 0.3pt,
pattern radius/.store in=\mcRadius, 
pattern radius = 1pt}
\makeatletter
\pgfutil@ifundefined{pgf@pattern@name@_eqhpdhg45}{
\pgfdeclarepatternformonly[\mcThickness,\mcSize]{_eqhpdhg45}
{\pgfqpoint{0pt}{0pt}}
{\pgfpoint{\mcSize+\mcThickness}{\mcSize+\mcThickness}}
{\pgfpoint{\mcSize}{\mcSize}}
{
\pgfsetcolor{\tikz@pattern@color}
\pgfsetlinewidth{\mcThickness}
\pgfpathmoveto{\pgfqpoint{0pt}{0pt}}
\pgfpathlineto{\pgfpoint{\mcSize+\mcThickness}{\mcSize+\mcThickness}}
\pgfusepath{stroke}
}}
\makeatother
\tikzset{every picture/.style={line width=0.75pt}} 

\scalebox{1.2}{

\begin{tikzpicture}[x=0.75pt,y=0.75pt,yscale=-1,xscale=1]

\draw  [fill={rgb, 255:red, 74; green, 144; blue, 226 }  ,fill opacity=0.4 ] (99.73,50.18) -- (100.09,110.18) -- (160.09,110.36) -- (160,140) -- (70,140) -- (70,50) -- cycle ;
\draw  [fill={rgb, 255:red, 248; green, 231; blue, 28 }  ,fill opacity=0.45 ] (110,50) -- (110,100) -- (160,100) -- (160.09,110.36) -- (100.09,110.18) -- (99.73,50.18) -- cycle ;
\draw   (110,0) -- (210,0) -- (210,100) -- (110,100) -- cycle ;
\draw   (70,50) -- (160,50) -- (160,140) -- (70,140) -- cycle ;
\draw  [pattern=_eqhpdhg45,pattern size=2.0250000000000004pt,pattern thickness=0.1pt,pattern radius=0pt, pattern color={rgb, 211:red, 214; green, 219; blue, 0}] (110,50) -- (160,50) -- (160,100) -- (110,100) -- cycle ;
\draw [color={rgb, 255:red, 208; green, 2; blue, 27 }  ,draw opacity=1 ][line width=1.5]  [dash pattern={on 1.69pt off 2.76pt}]  (99.73,49.27) -- (100.09,110.18) ;
\draw [color={rgb, 255:red, 208; green, 2; blue, 27 }  ,draw opacity=1 ][line width=1.5]  [dash pattern={on 1.69pt off 2.76pt}]  (100.09,110.18) -- (160.09,110.36) ;
\draw [color={rgb, 255:red, 74; green, 144; blue, 226 }  ,draw opacity=1 ][fill={rgb, 255:red, 74; green, 144; blue, 226 }  ,fill opacity=1 ]   (49.9,97.64) -- (82.1,97.88) ;
\draw [shift={(84.1,97.9)}, rotate = 180.44] [color={rgb, 255:red, 74; green, 144; blue, 226 }  ,draw opacity=1 ][line width=0.75]    (10.93,-3.29) .. controls (6.95,-1.4) and (3.31,-0.3) .. (0,0) .. controls (3.31,0.3) and (6.95,1.4) .. (10.93,3.29)   ;
\draw [color={rgb, 255:red, 208; green, 2; blue, 27 }  ,draw opacity=1 ]   (58,69.88) -- (105,70.11) ;
\draw [shift={(107,70.13)}, rotate = 180.29] [color={rgb, 255:red, 208; green, 2; blue, 27 }  ,draw opacity=1 ][line width=0.75]    (10.93,-3.29) .. controls (6.95,-1.4) and (3.31,-0.3) .. (0,0) .. controls (3.31,0.3) and (6.95,1.4) .. (10.93,3.29)   ;
\draw [color={rgb, 255:red, 208; green, 2; blue, 27 }  ,draw opacity=1 ]   (152.25,100.05) -- (152.25,109.88) ;
\draw [color={rgb, 255:red, 208; green, 2; blue, 27 }  ,draw opacity=1 ]   (179.75,119.63) -- (154.24,105.11) ;
\draw [shift={(152.5,104.13)}, rotate = 29.63] [color={rgb, 255:red, 208; green, 2; blue, 27 }  ,draw opacity=1 ][line width=0.75]    (10.93,-3.29) .. controls (6.95,-1.4) and (3.31,-0.3) .. (0,0) .. controls (3.31,0.3) and (6.95,1.4) .. (10.93,3.29)   ;

\draw (50,121.73) node [anchor=north west][inner sep=0.75pt]    {$P$};
\draw (217.5,3.9) node [anchor=north west][inner sep=0.75pt]    {$S$};
\draw (112,64.89) node [anchor=north west][inner sep=0.75pt]    {$P\cap S$};
\draw (29.35,88.4) node [anchor=north west][inner sep=0.75pt]  [color={rgb, 255:red, 74; green, 144; blue, 226 }  ,opacity=1 ]  {$P_{0}$};
\draw (29.25,60.7) node [anchor=north west][inner sep=0.75pt]  [color={rgb, 255:red, 208; green, 2; blue, 27 }  ,opacity=1 ]  {$\delta _{PS}$};
\draw (182.5,114.45) node [anchor=north west][inner sep=0.75pt]  [color={rgb, 255:red, 208; green, 2; blue, 27 }  ,opacity=1 ]  {$\gamma \ell ( P)$};

\end{tikzpicture}}

        \caption{The sets $P_0$ and $\delta_{PS}$ inside $P\setminus S$.}
        \label{fig:P_menys_S}
    \end{figure}

    Using this way of writing $P\setminus S$ as a disjoint union of two subsets, we have that each term $|\langle K_{\Psi_{\text{w}}}(\chi_{P\setminus S}\Delta_{1,Q}f), \chi_S\Delta_{2,R}g\rangle|$ appearing in the sum $A_1$ is controlled by
    \begin{align*}
        &\int_{y\in P_0}\bigg(\int_{x\in S}|\widetilde{k}_{\Psi_{\text{w}}}(x,y)|\,|\Delta_{1,Q}f(y)|\,|\Delta_{2,R}g(x)|\,d\mu(x)\bigg)d\mu(y)\\
        &\qquad +\int_{y\in \delta_{PS}}\bigg(\int_{x\in S} |\widetilde{k}_{\Psi_{\text{w}}}(x,y)|\,|\Delta_{1,Q}f(y)|\,|\Delta_{2,R}g(x)|\,d\mu(x)\bigg)d\mu(y) := I_1 + I_2.
    \end{align*}
    The bound for $I_1$ is simple, due to the distance to the boundary of $S$ that we have created. Indeed, for $y\in P_0$ and $x\in S$, we have $|x-y|\geq \gamma \ell(P)$, so
    $$
    |\widetilde{k}_{\Psi_\text{w}}(x,y)| \leq \frac{c}{|x-y|^n} \leq \frac{c}{\gamma^n}\frac{1}{\ell(P)^n},
    $$
    and so we obtain the same bound as in the case that $2S\subset H_{\mathcal{D}_1}$, with a constant that, of course, depends on the constant $\gamma$, which will be fixed later. For the second integral, we have that
    $$
    I_2 = I_2(P,S)\leq\sup_{\text{w}\in \Omega^2} \|K_{\Psi_{\text{w}}}\|_{L^2(\mu)\to L^2(\mu)}\|\chi_{\delta_{PS}}\Delta_{1,Q}f\|_{L^2(\mu)}\|\chi_S\Delta_{2,R}g\|_{L^2(\mu)}.
    $$
    Let us take expectations over all these terms. Applying Cauchy-Scwharz twice, denoting $\mathbb{E}_{P^{\Omega^2}}= \mathbb{E}_{(w_1,w_2)}$,
    \begin{align*}
        &\mathbb{E}_{P^{\Omega^2}}\bigg(\sum_{(P,S)\in\mathcal{A}_{1,2}}\|\chi_{\delta_{PS}}\Delta_{1,Q}f\|_{L^2(\mu)}\|\chi_S\Delta_{2,R}g\|_{L^2(\mu)}\bigg)\\
        &\qquad \leq \bigg[\mathbb{E}_{(w_1,w_2)}\bigg(\sum_{(P,S)\in \mathcal{A}_{1,2}}\|\chi_{\delta_{PS}}\Delta_{1,Q}f\|_{L^2(\mu)}^2\bigg)\bigg]^{\frac{1}{2}}\bigg[\mathbb{E}_{(w_1,w_2)}\bigg(\sum_{(P,S)\in \mathcal{A}_{1,2}}\|\chi_{S}\Delta_{2,R}g\|_{L^2(\mu)}^2\bigg)\bigg]^{\frac{1}{2}}\\
        &\qquad =: A\cdot B.
    \end{align*}
    Our goal now is to bound $A \leq c\sqrt{\gamma}\,\|f\|_{L^2(\mu)}$ and $B \leq c\,\|g\|_{L^2(\mu)}$. We start with the inequality involving the term $B$. Notice that the cubes $S\in \mathcal{D}^{\text{term},2}_2$ whose parent is transit are pairwise disjoint. Moreover, since each cube $R$ can only interact with a bounded number of cubes $Q$ (satisfying the conditions defining $\mathcal{A}_{1,2}$), we can bound, for each pair $(w_1, w_2)\in \Omega^2$,
    \begin{align*}
        \sum_{(P,S)\in \mathcal{A}_{1,2}}\|\chi_S\Delta_{2,\widehat{S}}g\|_{L^2(\mu)}^2 &= \sum_{P, \widehat{P}\in \mathcal{D}^{\text{tr},1}_1} \sum_{S: (P,S)\in \mathcal{A}_{1,2}}\|\chi_S \Delta_{2,\widehat{S}}g\|^2_{L^2(\mu)}\\
        &= \sum_{\substack{S\in \mathcal{D}^{\text{term},2}_2\\ \widehat{S}\in \mathcal{D}^{\text{tr},2}_2}}\|\chi_S\Delta_{2,\widehat{S}}g\|_{L^2(\mu)}^2\bigg(\sum_{P: (P,S)\in \mathcal{A}_{1,2}}1\bigg)\\
        &\leq c\sum_{\substack{S\in \mathcal{D}^{\text{term},2}_2\\ \widehat{S}\in \mathcal{D}^{\text{tr},2}_2}}\|\chi_S\Delta_{2,\widehat{S}}g\|_{L^2(\mu)}^2 \\
        &\leq c\sum_{\substack{S\in \mathcal{D}^{\text{term},2}_2\\ \widehat{S}\in \mathcal{D}^{\text{tr},2}_2}}\|\Delta_{2,\widehat{S}}g\|_{L^2(\mu\lfloor S)}^2 \leq c\,\|g\|^2_{L^2(\mu)},
    \end{align*}
    from which we deduce that $B\leq c\,\|g\|_{L^2(\mu)}$. Let us turn our attention to the term $A$. For this one, it will be useful to write
    \begin{align}
        A^2 &= \mathbb{E}_{(w_1,w_2)}\bigg(\sum_{(P,S)\in \mathcal{A}_{1,2}}\|\chi_{\delta_{PS}}\Delta_{1,\widehat{P}}f\|^2_{L^2(\mu)}\bigg)\nonumber\\
        &\leq \mathbb{E}_{(w_1,w_2)}\bigg(\sum_{(P,S)\in \mathcal{A}_{1,2}}\|\chi_{P\cap \mathcal{U}_{\gamma \ell(P)}(\partial S)}\Delta_{1,\widehat{P}}f\|^2_{L^2(\mu)}\bigg) \nonumber\\
        &= \mathbb{E}_{w_1}\bigg(\sum_{P,\widehat{P}\in \mathcal{D}^{\text{tr},1}_1}\bigg[\mathbb{E}_{w_2}\bigg(\sum_{S: (P,S)\in \mathcal{A}_{1,2}}\|\chi_{P\cap \mathcal{U}_{\gamma \ell(P)}(\partial S)}\Delta_{1,\widehat{P}}f\|^2_{L^2(\mu)}\bigg)\bigg]\bigg).\label{desigualtat_A_QUADRAT}
    \end{align}
    We claim that for each $P, \widehat{P}\in \mathcal{D}^{\text{tr},1}_1$,
    $$
    \mathbb{E}_{w_2}\bigg(\sum_{S: (P,S)\in \mathcal{A}_{1,2}}\|\chi_{P\cap \mathcal{U}_{\gamma \ell(P)}(\partial S)}\Delta_{1,Q}f\|^2_{L^2(\mu)}\bigg) \leq c\,\gamma\,\|\Delta_{1,\widehat{P}}f\|_{L^2(\mu)}^2.
    $$
    For each $S$ in the sum above, we can write its boundary as the union of $2d$ $(d-1)$-dimensional faces, $\partial S = \cup_{k=1}^{2d}c_k$. Hence,
    $$
    \mathcal{U}_{\gamma \ell(P)}(\partial S)  = \bigcup_{k=1}^{2d}\mathcal{U}_{\gamma \ell(P)}(c_k).
    $$
    Therefore,
    $$
    \|\chi_{P\cap \mathcal{U}_{\gamma\ell(P)}(\partial S)}\Delta_{1,\widehat{P}}f\|_{L^2(\mu)}^2 \leq \sum_{k=1}^{2d}\|\chi_{P\cap \mathcal
    U_{\gamma\ell(P)}(c_k)}\Delta_{1,\widehat{P}}f\|^2_{L^2(\mu)}. 
    $$
    To estimate each of the terms in the sum above, let us consider, for each fixed $P,S$, the measure $d\sigma = |\chi_P\Delta_{1,\widehat{P}}f|^2\,d\mu$. With this notation,
    \begin{align*}
        \|\chi_{P\cap \mathcal{U}_{\gamma \ell(P)}(c_k)}\Delta_{1,\widehat{P}}f\|_{L^2(\mu)}^2 &= \int_{P\cap \mathcal{U}_{\gamma \ell(P)}(c_k)}|\Delta_{1,\widehat{P}}f(x)|^2\,d\mu(x) = \sigma (\mathcal{U}_{\gamma \ell(P)}(c_k)). 
    \end{align*}
    Again, since each $S$ in the sum that we are estimating can only interact with a finite number of cubes $P$, we can swap the order of the expectation and the sum, to obtain
    \begin{align*}
        \mathbb{E}_{w_2} \bigg(\sum_{S: (P,S)\in \mathcal{A}_{1,2}}\|\chi_{P\cap \mathcal{U}_{\gamma \ell(P)}(\partial S)}\Delta_{1,\widehat{P}}f\|_{L^2(\mu)}^2\bigg) &= \sum_{S: (P,S)\in \mathcal{A}_{1,2}}\sum_{k=1}^{2d}\mathbb{E}_{w_2}(\sigma(\mathcal{U}_{\gamma\ell(P)}(c_k))).
    \end{align*}
    Note that on the right-hand side above, the random term inside the expectation is $c_k$, which, strictly speaking, is $c_k(w_2)$, but for commodity we have never written it in this form. Fix $1\leq k\leq 2d$, we are going to estimate the corresponding expectation above. We can assume that $c_k$ is contained in the hyperplane $H_1(s) = \{x_1=s\}$, for $s\in \pi(P)=I$, where $\pi$ denotes the projection into the first coordinate. This enables us to parametrize $c_k(w_2)=c_k(s)$ and because of the periodicity of the dyadic lattice,
    $$
    dP^{\Omega}  \leq \frac{c}{\mathcal{L}^1(I)} d(\mathcal{L}^1\lfloor  I) = \frac{c}{\ell(P)}d(\mathcal{L}^1\lfloor I),
    $$
    and so we can compute the expectation as
    \begin{align*}
        \mathbb{E}_{w_2}\left[\sigma\left(\mathcal{U}_{\gamma \ell(P)}(c_k)\right)\right]&= \int_{\Omega}\sigma\left(\mathcal{U}_{\gamma \ell(P)}(c_k(w_2))\right)dP^{\Omega}(w_2)\\
        &\leq c\,\frac{1}{\ell(P)}\int_{x_1\in I}\bigg(\int_{\subalign{&\,y\in P\,:\\ &|x_1-y_1|\leq \gamma \ell(P) }}d\sigma(y)\bigg)\,dx_1\\
        &= c\,\frac{1}{\ell(P)}\int_{y\in P}\bigg(\int_{\subalign{&\, x_1\in I\,:\\ &|x_1-y_1|\leq \gamma \ell(P)}}dx_1\bigg)\,d\sigma(y)\\
        &= 2\gamma \int_{y\in P}d\sigma (y) =  2\gamma\, \|\sigma\| = 2\gamma\, \|\chi_P\Delta_{1,\widehat{P}}f\|^2_{L^2(\mu)}.
    \end{align*}
    This means that for each $P,\widehat{P}\in \mathcal{D}^{\text{tr},1}_1$,
    \begin{align*}
        &\mathbb{E}_{w_2} \bigg(\sum_{S: (P,S)\in \mathcal{A}_{1,2}}\|\chi_{P\cap \mathcal{U}_{\gamma \ell(P)}(\partial S)}\Delta_{1,\widehat{P}}f\|_{L^2(\mu)}^2\bigg)\\
        &\qquad\quad \leq \sum_{S: (P,S)\in \mathcal{A}_{1,2}}\sum_{k=1}^{2d}2\gamma\,\|\chi_P\Delta_{1,\widehat{P}}f\|_{L^2(\mu)}^2 \leq c\,\gamma \,\|\chi_P\Delta_{1,\widehat{P}}f\|_{L^2(\mu)}^2,
    \end{align*}
    because each $P$ can only interact with a finite number of cubes $S$. If we plug this in (\ref{desigualtat_A_QUADRAT}), we have been able to bound
    \begin{align*}
        A^2 \leq \mathbb{E}_{w_1}\bigg(\sum_{P,\widehat{P}\in \mathcal{D}^{\text{tr},1}_1}c\,\gamma\,\|\chi_P\Delta_{1,\widehat{P}}f\|^2_{L^2(\mu)}\bigg)  \leq c\,\gamma\, \mathbb{E}_{w_1}\bigg(\sum_{Q\in \mathcal{D}^{\text{tr},1}_1}\|\Delta_{1,Q}f\|^2_{L^2(\mu)}\bigg) \leq c\,\gamma\,\|f\|_{L^2(\mu)}^2,
    \end{align*}
    where we have used Lemma \ref{lema_descomposicio_L2}. This finishes the proof that $A\leq c\,\sqrt{\gamma}\,\|f\|_{L^2(\mu)}$ and thus that of Lemma \ref{lema:esperança_a1}, for the sum over the pairs $(P,S)\in \mathcal{A}_{1,2}$. For $(P,S)\in \mathcal{A}_{2,1}$, we would simply interchange the roles of $P$ and $S$ in the arguments above and we obtain the same bound. For $(P,S)\in \mathcal{A}_{2,2}$, instead of using Lemma \ref{lema_descomposicio_L2} in the last step, we would use the fact that the terminal cubes $P$ whose parent is transit are pairwise disjoint, exactly as we have done for the estimate of the term $B$.
\end{proof}

For the second and third terms $A_2+ A_3$, almost the same arguments, changing the roles of the cubes $P$ and $S$ whenever we need to, yield the same bound, which is the following lemma, the proof of which we omit in order to avoid unnecessary repetition.

\begin{lemma}\label{lema:esperança_a2_a3}
    Let $A_2$ and $A_3$ be as in (\ref{A_i}), with $\Theta = \Psi_{\text{w}}$, for $\text{w}\in \Omega^2$. Then, for any $\gamma>0$ and $f,g$ as in Lemma \ref{lemma:good_functions},
    \begin{align*}
    \mathbb{E}_{P^{\Omega^2}}(A_2) &= \mathbb{E}_{P^{\Omega^2}}\bigg(\sum_{\mathcal{A}_{1,2}\cup \mathcal{A}_{2,1}\cup \mathcal{A}_{2,2}} |\langle K_{\Psi_{\text{w}}}(\chi_{P\cap S}\Delta_{1,Q}f), \chi_{S\setminus P}\Delta_{2,R}g\rangle|\bigg) \\
    &\leq c\,\sqrt{\gamma}\sup_{\text{w}\in \Omega^2}\|K_{\Psi_{\text{w}}}\|_{L^2(\mu)\to L^2(\mu)} \|f\|_{L^2(\mu)}\|g\|_{L^2(\mu)}.
    \end{align*}
    Furthermore,
    \begin{align*}
    \mathbb{E}_{P^{\Omega^2}}(A_3) &= \mathbb{E}_{P^{\Omega^2}}\bigg(\sum_{\mathcal{A}_{1,2}\cup \mathcal{A}_{2,1}\cup \mathcal{A}_{2,2}} |\langle K_{\Psi_{\text{w}}}(\chi_{P\cap S}\Delta_{1,Q}f), \chi_{P\cap S}\Delta_{2,R}g\rangle|\bigg) \\
    &\leq c\,\sqrt{\gamma}\sup_{\text{w}\in \Omega^2}\|K_{\Psi_{\text{w}}}\|_{L^2(\mu)\to L^2(\mu)} \|f\|_{L^2(\mu)}\|g\|_{L^2(\mu)}.
    \end{align*}
\end{lemma}
Lastly, we deal with the sum $A_4$, in which all the cubes that appear are transit. Unfortunately, the same technique as in the two preceding lemmas fails, because the family of transit cubes whose parents are also transit is not pairwise disjoint, unlike in the case where the children are terminal and the parents are transit. To get an appropriate bound in this case, we will need to separate each summand even further. We do this in the following lemma.

\begin{lemma}\label{lema:esperança_a4}
    Let $A_4$ be as in (\ref{A_i}), with $\Theta=\Psi_{\text{w}}$ for $\text{w}\in \Omega^2$. Then, for $f,g$ as in Lemma \ref{lemma:good_functions},
    \begin{align*}
        \mathbb{E}_{P^{\Omega^2}}(A_4) &= \mathbb{E}_{P^{\Omega^2}}\bigg(\sum_{\mathcal{A}_{1,1}}|c_P(f)||c_S(g)||\langle K_{\Psi_{\text{w}}}(\chi_{\Delta_{P,S}}b_1), \chi_{\Delta_{P,S}\setminus (\Delta_{P,S})_{\varepsilon_a}}b_2\rangle|\bigg) \\
        &\leq c\,\|f\|_{L^2(\mu)}\|g\|_{L^2(\mu)}\left(p_{\varepsilon_a}^{\frac{1}{2}}\sup_{\text{w}\in \Omega^2}\|K_{\Psi_{\text{w}}}\|_{L^2(\mu)\to L^2(\mu)}+c_{\varepsilon_a}\right),
    \end{align*}
    where $p_{\varepsilon_a}\to 0$ as $\varepsilon_a\to 0$, as discussed in Lemma \ref{lemma:p_epsilon_a}, and $c_{\varepsilon_a}$ is a constant that depends on $\varepsilon_a$.
\end{lemma}

\begin{proof}
    For commodity of the reader, we recall that for $(P,S)\in \mathcal{A}_{1,1}$, we denote $\Delta_{P,S}=P\cap S$, but whenever there is no possible confusion, we simple write $\Delta$ instead of $\Delta_{P,S}$. Moreover, we recall that we considered three subsets of $\Delta$: $\Delta_{\varepsilon_a}$, $P_\partial$, and $S_\partial$ (see Figure \ref{fig:separacio_interseccio} and the discussion above).

This allows us to separate further, for each $(P,S)\in \mathcal{A}_{1,1}$,
\begin{equation}\label{separacio_final}
\begin{aligned}
    \langle K_{\Psi_{\text{w}}}(\chi_\Delta b_1), \chi_{\Delta\setminus \Delta_{\varepsilon_a}}b_2 \rangle &= \langle K_{\Psi_\text{w}}(\chi_\Delta b_1), \chi_{S_\partial}b_2\rangle + \langle K_{\Psi_{\text{w}}}(\chi_{\Delta}b_1), \chi_{P_\partial \setminus (P_\partial \cap S_\partial)}b_2\rangle\\
    &= \langle K_{\Psi_\text{w}}(\chi_\Delta b_1), \chi_{S_\partial}b_2\rangle + \langle K_{\Psi_\text{w}}(\chi_{P_\partial}b_1), \chi_{P_\partial \setminus (P_\partial \cap S_\partial )}b_2\rangle \\
    &\qquad + \langle K_{\Psi_\text{w}}(\chi_{\Delta\setminus P_\partial}b_1), \chi_{P_\partial \setminus (P_\partial \cap S_\partial)}b_2\rangle. 
\end{aligned}
\end{equation}
To bound the first two terms on the right-hand side of the last inequality, we are going to use Lemma \ref{lema:part_dolenta_k} (recall the definition of the bad part of a cube in (\ref{part_dolenta_dun_cub})). For the last term, we will introduce yet another probabilistic argument, which in this case will take advantage of the growth condition on our Calderón-Zygmund kernel $\widetilde{k}_{\Psi_{\text{w}}}$.

Note that we can bound the first term in (\ref{separacio_final}) by
$$
|\langle K_{\Psi_\text{w}}(\chi_\Delta b_1), \chi_{S_\partial}b_2\rangle| \leq \|K_{\Psi_{\text{w}}}\|_{L^2(\mu)\to L^2(\mu)}\|\chi_Pb_1\|_{L^2(\mu)}\|\chi_{S_{\text{b}}}b_2\|_{L^2(\mu)},
$$
because $S_\partial \subset S_\text{b}$ (remember our choice of the width $\ell_{\varepsilon_a}$). We take the sum over the cubes that we are considering,
\begin{align}
    &\sum_{(P,S)\in \mathcal{A}_{1,1}}|c_P(f)|\, |c_S(g)|\,\|K_{\Psi_\text{w}}\|_{L^2(\mu)\to L^2(\mu)}\|\chi_Pb_1\|_{L^2(\mu)}\|\chi_{S_\text{b}}g\|_{L^2(\mu)}\nonumber \\
    &\quad \leq \|K_{\Psi_{\text{w}}}\|_{L^2(\mu)\to L^2(\mu)}\bigg(\sum_{P,\widehat{P}\in \mathcal{D}^{\text{tr},1}_1} |c_P(f)|^2\|\chi_Pb_1\|^2_{L^2(\mu)}\bigg)^{\frac{1}{2}}\bigg(\sum_{S,\widehat{S}\in \mathcal{D}^{\text{tr},2}_2} |c_S(g)|^2\|\chi_{S_{\text{b}}}b_2\|^2_{L^2(\mu)}\bigg)^{\frac{1}{2}}. \label{desigualtat_sb}
\end{align}
By Lemma \ref{lema_descomposicio_L2}, the middle factor is controlled by
$$
\sum_{P, \widehat{P}\in \mathcal{D}^{\text{tr},1}_1}|c_P(f)|^2\|\chi_Pb_1\|^2_{L^2(\mu)}\leq \sum_{Q\in \mathcal{D}^{\text{tr},1}_1}\|\Delta_{1,Q}f\|^2_{L^2(\mu)}\leq c_3\|f\|^2_{L^2(\mu)},
$$
Hence, taking expectation in (\ref{desigualtat_sb}), and using the inequality (\ref{desigualtat_p_epsilon_g}), we have
\begin{align}
&\mathbb{E}_{P^{\Omega^2}}\bigg(\sum_{(P,S)\in \mathcal{A}_{1,1}}|c_P(f)||c_S(g)||\langle K_{\Psi_{\text{w}}}(\chi_{\Delta}b_1), \chi_{S_\partial}b_2\rangle|\bigg)\nonumber \\
&\qquad \qquad\leq c\, p_{\varepsilon_a}^{\frac{1}{2}}\, \sup_{\text{w}\in \Omega^2}\|K_{\Psi_{\text{w}}}\|_{L^2(\mu)\to L^2(\mu)}\|f\|_{L^2(\mu)}\|g\|_{L^2(\mu)},\label{cota_que_encaixa}
\end{align}
The second term in (\ref{separacio_final}) can be bounded by
$$
|\langle K_{\Psi_\text{w}}(\chi_{P_\partial}b_1), \chi_{P_\partial \setminus (P_\partial \cap S_\partial )}b_2\rangle|\leq \|K_{\Psi_{\text{w}}}\|_{L^2(\mu)\to L^2(\mu)}\|\chi_{P_\text{b}}b_1\|_{L^2(\mu)}\|\chi_Sb_2\|_{L^2(\mu)}.
$$
Arguing analogously as above, taking sums and then expectation, we obtain the same bound for it as in (\ref{cota_que_encaixa}). Let us turn our attention to term from the last line of (\ref{separacio_final}), which is
\begin{equation}
\langle K_{\Psi_\text{w}}(\chi_{\Delta\setminus P_\partial}b_1), \chi_{P_\partial \setminus (P_\partial \cap S_\partial)}b_2\rangle,\label{ultima_cosa_que_he_de_separar}
\end{equation}
with $\Delta = P\cap S$ and $P_\partial$ and $S_\partial$ as chosen in Section \ref{sec:partial_bound}. Clearly, this term is bounded above, in absolute value, by
$$
c_b^2 \int_{P_\partial \setminus S_\partial}\bigg(\int_{\Delta\setminus P_\partial }|\widetilde{k}_\Theta(x,y)|\,d\mu(y)\bigg)\,d\mu(x).
$$
Since, as we have argued in previous lemmas, for any fixed $S$ as above, the number of cubes $P$ with which it can interact is always bounded, and the same is true swapping the roles of $S$ and $P$, it is enough to bound the integral above by $c_{\varepsilon_a}\,\mu(P)^{\frac{1}{2}}\mu(S)^{\frac{1}{2}}$. This bound is going to be obtained using the small boundary condition that $\Delta_{\varepsilon_a}$ satisfies. This condition was stated in (\ref{t-small}), but for commodity of the reader we recall it now. We have that $\Delta_{\varepsilon_a}$ satisfies that for any $\lambda>0$,
\begin{equation}
\mu\big(\big\{x\in \Delta : \text{dist}(x,\partial\, {\Delta}_{\varepsilon_a})\leq \lambda \,\text{diam}({\Delta}_{\varepsilon_a}\big)\big\}\big) \leq \, t_{\Delta,{\varepsilon_a}}\lambda \mu(\Delta).\label{t-small2}
\end{equation}
The technique that we will use is almost the same as what we did in Lemma \ref{els_dos_primers_transit}. First, we fix $x\in P_\partial \setminus S_\partial$, and we need to estimate
\begin{align*}
&\int_{y\in \Delta \setminus P_\partial} \frac{1}{|x-y|^n + \Theta(x)^n+ \Theta(y)^n}d\mu(y)\\
&\qquad \qquad \leq \int_{\text{dist}(x, \partial \Delta_{\varepsilon_a})\leq |x-y|\leq c'\text{diam}(\Delta_{\varepsilon_a})} \frac{1}{|x-y|^n + \Theta(x)^n}d\mu(y).
\end{align*}
We can estimate the last integral by arguments analogous to the bounds in (\ref{integral_theta}), and so we obtain that
\begin{align*}
    \int_{P_\partial \setminus S_\partial} \bigg(\int_{\Delta\setminus P_\partial}|\widetilde{k}_{\Theta}(x,y)|\,&d\mu(y)\bigg)\,d\mu(x) \leq c\int_{P_\partial \setminus S_\partial} \log\bigg(\frac{c'\text{diam}(\Delta_{\varepsilon_a})}{\text{dist}(x, \partial \Delta_{\varepsilon_a})}\bigg)\,d\mu(x)\\
    & \leq c\sum_{k\geq 0} k\,\mu\left(\left\{y\in P_\partial \setminus S_\partial : \text{dist}(x, \partial \Delta_{\varepsilon_a})\leq 2^{-k}\text{diam}(\Delta_{\varepsilon_a})\right\}\right)\\
    &\leq c\sum_{k\geq 0} k\,2^{-k}\,t_{\Delta_{\varepsilon_a}}\,\mu(\Delta) = c_{\varepsilon_a}\mu(\Delta) \leq c_{\varepsilon_a}\, \mu(P)^{\frac{1}{2}}\mu(S)^{\frac{1}{2}}.
\end{align*}

Combining this inequality with the ones that we obtained for the sums of the corresponding other terms in (\ref{separacio_final}), we have obtained that
$$
\mathbb{E}_{P^{\Omega^2}}(A_4) \leq c\, \|f\|_{L^2(\mu)}\|g\|_{L^2(\mu)}\left(p_{\varepsilon_a}^{\frac{1}{2}}\sup_{\text{w}\in \Omega^2}\|K_{\Psi_{\text{w}}}\|_{L^2(\mu)\to L^2(\mu)}+c_{\varepsilon_a}\right),
$$
which proves the assertion of the lemma.
\end{proof}

\subsection{The final step}

In this section, we are going to complete the proof of Theorem \ref{TEOREMA}. The main ingredient in this final step will be to show that the operator $\widetilde{K}$, whose kernel was the average of the kernels $\widetilde{k}_{\Psi_{\text{w}}}$, with respect to $P^{\Omega^2}$, is bounded in $L^2(\mu)$. The key tools to achieve this will be, of course, our result for good functions, Lemma \ref{lemma:good_functions}, combined with the probabilistic estimates from the preceding section.

\begin{lemma}
    The operator $\widetilde{K}$ is bounded in $L^2(\mu)$.
\end{lemma}

\begin{proof}
    Let $f,g\in L^2(\mu)$. Our goal is to estimate $\langle \widetilde{K}f, g\rangle$, which, by definition, is
    $$
    \langle \widetilde{K}f, g\rangle = \int_{\R^d} \bigg(\int_{\Omega^2} K_{\Psi_\text{w}}f(x)\,dP^{\Omega^2}(\text{w})\bigg)g(x)\,d\mu(x) = \mathbb{E}_{P^{\Omega^2}}(\langle K_{\Psi_\text{w}}f,g \rangle).
    $$
    Let us estimate, for each $\text{w}= (w_1,w_2)\in \Omega^2$, the term inside the expectation on the right-hand side above. As we mentioned before the definition of bad cubes, we are going to write
    $$
    f_{\text{good}}(\text{w}) = \Xi_1f + \sum_{Q\in \mathcal{D}^{\text{tr},1}(w_1)\cap \text{good}(w_2)}\Delta_{1,Q}f, \qquad f_{\text{bad}}(\text{w}) = \sum_{Q\in\mathcal{D}^{\text{tr},1}(w_1)\cap \text{bad}(w_2)} \Delta_{1,Q}f,
    $$
    and
    $$
    g_{\text{good}}(\text{w}) = \Xi_2g + \sum_{R\in \mathcal{D}^{\text{tr},2}(w_2)\cap \text{good}(w_1)}\Delta_{2,R}g, \qquad g_{\text{bad}}(\text{w}) = \sum_{R\in\mathcal{D}^{\text{tr},2}(w_2)\cap \text{bad}(w_1)} \Delta_{2,R}g,
    $$
    where ``good$(w_i)$'' means ``good with respect to the dyadic lattice $\mathcal{D}(w_i)$'', and the analogous notion for ``bad$(w_i)$''. Since, by Lemma \ref{lema_descomposicio_L2}, we have that $f = f_{\text{good}}(\text{w}) + f_{\text{bad}}(\text{w})$ and $g = g_{\text{good}}(\text{w}) +g_{\text{bad}}(\text{w})$ in $L^2(\mu)$, we can write
    \begin{align}
        \langle K_{\Psi_\text{w}}f,g\rangle &= \langle K_{\Psi_\text{w}}f_{\text{good}}(\text{w}), g_{\text{good}}(\text{w})\rangle + \langle K_{\Psi_\text{w}}f_{\text{good}}(\text{w}), g_{\text{bad}}(\text{w})\rangle\label{separacio_good_bad}\\
        &\qquad + \langle K_{\Psi_\text{w}}f_{\text{bad}}(\text{w}), g\rangle. \nonumber
    \end{align}
    Of course, the plan now will be to apply Lemma \ref{lemma:good_functions} for the first term and, for the other two, we can take advantage of the fact that the probability of bad functions can be made arbitrarily small. Indeed, by Lemma \ref{lemma:good_functions} and the results from Section \ref{sec:partial_bound}, we have that for each $\text{w}\in \Omega^2$, and for $\varepsilon_a>0$ small,
    $$
    |\langle K_{\Psi_\text{w}}f_{\text{good}}(\text{w}), g_{\text{good}}(\text{w})\rangle| \leq c_{\varepsilon_a}\|f\|_{L^2(\mu)}\|g\|_{L^2(\mu)} + R(\text{w}),
    $$
    where $R(\text{w})$ corresponds to the additional terms that we were not able to bound in Section \ref{sec:partial_bound}. For these terms, gathering the results from Lemmas \ref{lema:esperança_a1}, \ref{lema:esperança_a2_a3} and \ref{lema:esperança_a4}, we have that for any $\gamma>0$,
    $$
    \mathbb{E}_{P^{\Omega^2}}(R) \leq c\, \|f\|_{L^2(\mu)}\|g\|_{L^2(\mu)} \left[\left(\sqrt{\gamma}+ p_{\varepsilon_a}^{\frac{1}{2}}\right)\sup_{\text{w}\in \Omega^2}\|K_{\Psi_\text{w}}\|_{L^2(\mu)\to L^2(\mu)}+ c_{\varepsilon_a}\right], 
    $$
    where $p_{\varepsilon_a}$ has limit $0$ as $\varepsilon_a$ approaches $0$. Hence, 
    $$
    \mathbb{E}_{P^{\Omega^2}}\left(|\langle K_{\Psi_\text{w}}f_{\text{good}}(\text{w}), g_{\text{good}}(\text{w})\rangle|\right) \leq c\,\|f\|_{L^2(\mu)}\|g\|_{L^2(\mu)}\left(\text{E}(\varepsilon_a,\gamma)\sup_{\text{w}\in \Omega^2}\|K_{\Psi_{\text{w}}}\|_{L^2(\mu)\to L^2(\mu)} + c_{\varepsilon_a}\right),
    $$
    where $\text{E}(\varepsilon_a,\gamma)\to 0$ as both $\varepsilon_a$ and $\gamma$ tend to $0$. For the second term in (\ref{separacio_good_bad}), we have
    $$
    |\langle K_{\Psi_\text{w}}f_{\text{good}}(\text{w}), g_{\text{bad}}(\text{w})\rangle| \leq c_3^2\,\|K_{\Psi_{\text{w}}}\|_{L^2(\mu)\to L^2(\mu)} \|f\|_{L^2(\mu)}\|g_\text{bad}(\text{w})\|_{L^2(\mu)},
    $$
    because, applying Lemma \ref{lemma:good_functions} two times, 
    \begin{align*}
        c_{3}^{-1}\|f_{\text{good}}(\text{w})\|_{L^2(\mu)}^2 &\leq \|\Xi_1f_{\text{good}}(\text{w})\|^2_{L^2(\mu)} + \sum_{Q\in \mathcal{D}^{\text{tr},1}(w_1)}\|\Delta_{1,Q}f_{\text{good}}(\text{w})\|^2_{L^2(\mu)}\\
        &\leq \|\Xi_1f\|_{L^2(\mu)}^2 + \sum_{Q\in \mathcal{D}^{\text{tr},1}(w_1)\cap \text{good}(w_2)}\|\Delta_{1,Q}f\|^2_{L^2(\mu)}\leq c_3 \,\|f\|_{L^2(\mu)}^2.
    \end{align*}
    Moreover, by Lemma \ref{lema:esperança_part_dolenta},
    $$
    \mathbb{E}_{P^{\Omega^2}}\left(\|g_{\text{bad}}(\text{w})\|^2_{L^2(\mu)}\right) \leq c_3^2\, \varepsilon_b\,\|g\|_{L^2(\mu)}^2.
    $$
    Therefore, 
    \begin{align*}
        |\mathbb{E}_{P^{\Omega^2}}\left(\langle K_{\Psi_\text{w}}f_{\text{good}}(\text{w}), g_{\text{bad}}(\text{w}\rangle )\right)| &\leq \mathbb{E}_{P^{\Omega^2}}\left(\|K_{\Psi_{\text{w}}}f_{\text{good}}(\text{w}) \|_{L^2(\mu)}^2\right)^{\frac{1}{2}}\mathbb{E}_{P^{\Omega^2}}\left(\|g_\text{bad}(\text{w})\|_{L^2(\mu)}^2\right)^{\frac{1}{2}}\\
        &\leq c_3^3\,\varepsilon_b^{\frac{1}{2}}\sup_{\text{w}\in \Omega^2}\|K_{\Psi_{\text{w}}}\|_{L^2(\mu)\to L^2(\mu)}\|f\|_{L^2(\mu)}\|g\|_{L^2(\mu)}.
    \end{align*}
    Arguing analogously, for the last term in (\ref{separacio_good_bad}), we have that
    $$
    |\mathbb{E}_{P^{\Omega^2}}\left(\langle K_{\Psi_\text{w}}f_{\text{bad}}(\text{w}), g\rangle\right)| \leq c_3\,\varepsilon_b^{\frac{1}{2}}\sup_{\text{w}\in \Omega^2}\|K_{\Psi_\text{w}}\|_{L^2(\mu)\to L^2(\mu)}\|f\|_{L^2(\mu)}\|g\|_{L^2(\mu)}.
    $$
    Since, by the inequality in (\ref{desigualtat_k_psi_beta}),
    $$
    \sup_{\text{w}\in \Omega^2}\|K_{\Psi_\text{w}}\|_{L^2(\mu)\to L^2(\mu)} \leq c\,\beta^{-1}\|\widetilde{K}\|_{L^2(\mu)\to L^2(\mu)} + c\beta^{-1},
    $$
    from the previous estimates, taking averages in (\ref{separacio_good_bad}) over $\text{w}\in \Omega^2$, we obtain
    $$
    |\langle \widetilde{K}f, g\rangle| \leq c\,\|f\|_{L^2(\mu)}\|g\|_{L^2(\mu)}\left(\beta^{-1}\|\widetilde{K}\|_{L^2(\mu)\to L^2(\mu)}\widetilde{\text{E}}(\varepsilon_a,\varepsilon_b,\gamma)+ \beta^{-1}\widetilde{\text{E}}(\varepsilon_a,\varepsilon_b,\gamma)+ c_{\varepsilon_a}\right),
    $$
    where, for commodity, we have written $\widetilde{\text{E}}(\varepsilon_a,\varepsilon_b,\gamma) = \text{E}(\varepsilon_a,\gamma) +c_3^2\varepsilon_b^{\frac{1}{2}}+ c_3\varepsilon_b^{\frac{1}{2}}$, so that, of course, $\widetilde{\text{E}}(\varepsilon_a, \varepsilon_b,\gamma)$ has limit $0$ as $\varepsilon_a,\varepsilon_b$ and $\gamma$ tend to $0$. Taking suprema in $f,g\in L^2(\mu)$, we infer that
    $$
    \|\widetilde{K}\|_{L^2(\mu)\to L^2(\mu)} \leq c\,\left(\beta^{-1}\|\widetilde{K}\|_{L^2(\mu)\to L^2(\mu)}\widetilde{\text{E}}(\varepsilon_a,\varepsilon_b,\gamma)+ \beta^{-1}\widetilde{\text{E}}(\varepsilon_a,\varepsilon_b,\gamma)+ c_{\varepsilon_a}\right).
    $$
    From this, recalling that $\beta$ is independent of $\varepsilon_a,\varepsilon_b$ and $\gamma$, we see that if we take $\varepsilon_a, \varepsilon_b$ and $\gamma$ small enough,
    $$
    \|\widetilde{K}\|_{L^2(\mu)\to L^2(\mu)} \leq 2c(1+c_{\varepsilon_a}),
    $$
    which is finite, as we wanted to show.
\end{proof}

Using the fact that $\widetilde{K}\,\colon L^2(\mu)\to L^2(\mu)$ is bounded, in the next lemma we finish the proof of Theorem \ref{TEOREMA}.

\begin{lemma}
    The operator $K_{\Phi}$ is bounded in $L^2(\mu)$, with a bound independent of $\varepsilon_0$, and thus the SIO $T$ is bounded in $L^2(\mu\lfloor G)$.
\end{lemma}

\begin{proof}
    First, by Lemma \ref{diferencia_truncat_suppressed}, we have that 
    $$
    |T_{\varepsilon_0}f(x) - K_{\Phi, \varepsilon_0}f(x)| \leq c\,\sup_{r\geq \varepsilon_0}\frac{1}{r^n}\int_{B(x,r)}|f|\,d(\mu\lfloor G).
    $$
    Arguing as in Lemma \ref{desigualtats_tilde}, the right hand side above is bounded by $c\, c_0\, M_{\mu\lfloor G}f(x)$. Since the maximal operator is bounded in $L^2(\mu\lfloor G)$, we deduce that $T_{\varepsilon_0}$ is bounded in $L^2(\mu\lfloor G)$ if and only if $K_{\Phi,\varepsilon_0}$ is bounded in $L^2(\mu\lfloor G)$. 

    Moreover, for $x\in G$,
    \begin{align*}
        |K_{\Phi, \varepsilon_0}f(x) - K_{\Phi}f(x)|&\leq \int_{|x-y|\leq \varepsilon_0}|\widetilde{k}_{\Phi}(x,y)|\,|f(y)|\,d\mu(y)\\
        & \leq \frac{1}{\varepsilon_0^n}\int_{|x-y|\leq \varepsilon_0}|f(y)|\,d\mu(y)\leq c\, c_0 \,M_{\mu}f(x),
    \end{align*}
    again, arguing as in Lemma \ref{desigualtats_tilde} and using that for $x\in G$, $\Phi(x)=\varepsilon_0$ (recall Lemma \ref{lema:propietats_phi}), combined with the fact that $|\widetilde{k}_{\Phi}(x,y)|\leq \frac{1}{\Phi(x)^n}$. By the boundedness of the maximal operator, we see that $K_{\Phi,\varepsilon_0}$ is bounded in $L^2(\mu)$ if and only if $K_{\Phi}$ is bounded in $L^2(\mu)$.

    Furthermore, combining inequality (\ref{desigualtat_k_phi_beta}) with the previous lemma, we see that $K_{\Phi}$ is bounded in $L^2(\mu)$ independently of $\varepsilon_0$, and so, by our previous arguments, $T_{\varepsilon_0}$ is bounded in $L^2(\mu\lfloor G)$, independently of $\varepsilon_0$. Equivalently, $T\,\colon L^2(\mu\lfloor G)\to L^2(\mu\lfloor G)$ is bounded, which is what we wanted to prove.
\end{proof}

\newpage

\pagestyle{plain}
\printbibliography
\addcontentsline{toc}{section}{References}

\end{document}